\documentclass{amsart}

\usepackage[utf8]{inputenc}
\usepackage[T1]{fontenc} 

\usepackage{standalone}

\usepackage{tikz}
\usetikzlibrary{shapes,positioning,fit,calc,math,patterns,shapes.misc,shapes.geometric,decorations.markings,external}
\usepackage{tikz-cd}
\usepackage{graphicx}
\usepackage{subfig}
\usepackage{caption}
\usepackage{comment}
\usepackage{amssymb}
\usepackage{amsmath}
\usepackage{amsthm}
\usepackage{enumitem}
\usepackage{relsize}
\usepackage{fullpage}
\usepackage{colonequals}
\usepackage{mathrsfs}
\usepackage{mathtools}

\usepackage[
backend=biber,
uniquename=false,
doi=false,
isbn=false,
citestyle=alphabetic,
bibstyle=alphabetic,
sorting=nyt,
url=false
]{biblatex}
\AtEveryBibitem{%
  \clearlist{language}%
  \clearfield{eid}%
}

\tikzset{
curarrow/.style={
rounded corners=8pt,
execute at begin to={every node/.style={fill=red}},
to path={-- ([xshift=50pt]\tikztostart.center)
  |- (#1) node[fill=white] {$\scriptstyle d^*$}
  -| ([xshift=-40pt]\tikztotarget.center)
  -- (\tikztotarget)}
  }
}

\usepackage[hidelinks,pdfencoding=auto, psdextra]{hyperref}
\usepackage[noabbrev,capitalize]{cleveref}

\graphicspath{{tikz-pdfs/}}

\newtheorem{theorem}{Theorem}[section]
\newtheorem{lemma}[theorem]{Lemma}
\newtheorem{proposition}[theorem]{Proposition}
\newtheorem{corollary}[theorem]{Corollary}

\crefname{prop}{proposition}{propositions}
\Crefname{prop}{Proposition}{Propositions}

\theoremstyle{definition}
\newtheorem{definition}[theorem]{Definition}

\newtheorem{remark}[theorem]{Remark}

\numberwithin{equation}{section}

\newcommand{\smon}{N}

\newcommand{\mon}[1][\smon]{\text{MConf}(#1)}

\newcommand{\modc}[2][\smon]{\text{Mod-Conf}^{\, 1}_{#1}(#2)}

\newcommand{\rimodc}[3][M]{\text{Mod-Conf}^{\, #3}_{#1}(#2)}

\newcommand{\fconf}[1][M]{\text{FConf}(#1)}

\newcommand{\bmon}[2][\smon]{\text{MConf}(#1)_{#2}}
\newcommand{\bmodc}[3][\smon]{\modc[#1]{#2}_{#3}}
\newcommand{\dmodel}{Y}

\newcommand{\restr}[1]{|_{#1}}

\newcommand{\iterxbarc}[5]{\mathcal{X}^{#5}_{#4}({#1},{#2},{#3})}
\newcommand{\piterxbarc}[6]{\mathcal{X}_{#6,#4}^{#5}({#1},{#2},{#3})}
\newcommand{\iterxpbarc}[5]{\mathcal{X}^{#5}_{#4}({#1},{#2},{#3})^{+}}

\newcommand{\wpbarc}[4]{W_{#4}({#1},{#2},{#3})}

\newcommand{\citerxbarc}[6]{C\mathcal{X}^{#5}_{#4}({#1},{#2},{#3};{#6})}
\newcommand{\cpiterxbarc}[7]{C\mathcal{X}^{#5}_{#7, #4}({#1},{#2},{#3};{#6})}

\newcommand{\tot}[1]{\text{Tot}(#1)}
\newcommand{\srange}{m}
\newcommand{\partic}{k}

\newcommand{\neck}{\lambda}
\newcommand{\newstab}[1]{t_{#1 ,c}}

\newcommand{\lnewstab}[1]{\hat{t}_{#1 ,c}}
\newcommand{\moconfigone}{\mathcal{C}}
\newcommand{\moconfigtwo}{\mathcal{C}}
\newcommand{\bdisk}{\modc[S^{n-1}]{\bar{D}^{n}}^{\leq 1}}

\newcommand{\fieldc}{R}
\newcommand{\nch}[1]{C_{*}(#1;\fieldc)}

\newcommand{\conffilt}{\mathcal{G}}

\newcommand{\browd}[2]{\left [#1, #2\right ]}
\newcommand{\embed}{\iota}

\newcommand{\iteratedmap}[3]{\text{Cone}^{#3}(#1, (#2))}
\newcommand{\pariteratedmap}[4]{\text{Cone}_{#4}^{#3}(#1, (#2))}
\newcommand{\jiteratedmap}[4]{\text{Cone}^{#3,#4}(#1, (#2))}
\newcommand{\jpariteratedmap}[5]{\text{Cone}_{#4}^{#3,#5}(#1, (#2))}

\newcommand{\punc}[1]{Z^{\srange}_{#1}(M)}
\newcommand{\ppunc}[1]{Z^{\srange}_{\partic, #1}(M)}
\newcommand{\cpunc}[2][A^{I}_{*}]{CZ^{\srange}_{ #2}(#1,M;\fieldc)}
\newcommand{\cppunc}[3][A^{I}_{*}]{CZ^{\srange}_{#3, #2}(#1,M;\fieldc)}

\usepackage{microtype}

\begin{document}
\title{Secondary Homological Stability for Unordered Configuration Spaces}

\author{Zachary Himes}
\address{Department of Mathematics, Purdue University, 150 N. University Street, West Lafayette, Indiana, 47906}
%\curraddr{Department of Mathematics and Statistics, Case Western Reserve University, Cleveland, Ohio 43403}
\email{himesz@purdue.edu}
\thanks{Zachary Himes was supported in part by a Frederick N. Andrews Fellowship, Department of Mathematics, Purdue University.}
%\author{Author Two}
%\address{Mathematical Research Section, School of Mathematical Sciences,Australian National University, Canberra ACT 2601, Australia}

%    General info
%\subjclass[2000]{Primary 54C40, 14E20; Secondary 46E25, 20C20}

%\date{January 1, 2001 and, in revised form, June 22, 2001.}


%\keywords{Differential geometry, algebraic geometry}

\begin{abstract}
%Secondary homological stability is a recently discovered stability pattern for the homology of a sequence of  spaces exhibiting homological stability in a range where homological stability does not hold.
Secondary homological stability is a recently discovered stability pattern for a sequence of spaces exhibiting homological stability and it holds outside the range where the homology stabilizes.
%The first example of secondary homological stability is a result of \textcite{MR4028513} on mapping class groups. 
We prove secondary homological stability for the homology of the unordered configuration spaces of a connected manifold. The main difficulty is the case that the manifold is closed because there are no obvious maps inducing stability and the homology 
%exhibits periodic homological stability instead of homological stability. 
eventually is periodic instead of stable.
We resolve this issue by constructing a new chain-level stabilization map for configuration spaces.
\end{abstract}

\maketitle
\tableofcontents

\section{Introduction}\label{intro}

%\subsection{Homological Stability for the Configuration Space of a Non-compact Manifold}
The unordered configuration space of $k$ points in a manifold $M$ is the quotient space
\begin{center}
$\text{Conf}_{\partic}(M)\colonequals \{(x_1 ,\ldots , x_{\partic })\in M^{\partic} :  x_i \neq x_j \text{ if } i\neq j \}/\Sigma_{\partic}.$
%\{\{x_1 ,\ldots , x_{\partic } \}\subset M :  x_i \neq x_j \text{ for } i\neq j \}$
\end{center}
The subset $\{(x_1 ,\ldots , x_{\partic })\in M^{\partic} :  x_i \neq x_j \text{ if } i\neq j \}\subset M^{\partic}$ has the subspace topology and the symmetric group $\Sigma_{\partic}$ acts on it by permuting the order of the coordinates.
%obtained by quotienting out a subspace of $M^{\partic}$ by the action of the symmetric group $\Sigma_{k}$
%Given a manifold $M$, let 
%\begin{center}
%$\text{Conf}_{\partic}(M)\colonequals \{(x_1 ,\ldots , x_{\partic })\in M^{k} :  x_i \neq x_j \text{ if } i\neq j \}/\Sigma_{k}$
%\{\{x_1 ,\ldots , x_{\partic } \}\subset M :  x_i \neq x_j \text{ for } i\neq j \}$
%\end{center}
%denote the unordered configuration space of $k$  distinct points in $M$.
In this paper, we study stability patterns for the homology groups $H_{*} (\text{Conf}_{\partic }(M);\fieldc)$ for a connected topological manifold $M$. When the homology is over the finite field $\mathbb{F}_{p}\colonequals \mathbb{Z}/p\mathbb{Z}$, we prove that there are two stability patterns, depending on whether the manifold is open
%\footnote{You suggested I replace open/closed with non-compact/compact. I don't think I should do this because an open manifold is a non-compact manifold without boundary. I realize it doesn't ultimately matter whether the manifold has boundary or not, but I'm worried that replacing open with non-compact might lead to the paper containing statements that aren't technically correct.} 
or closed. When the manifold is open, the homology of the configuration space exhibits secondary homological stability, in the sense of \textcite[Theorem A]{MR4028513}. When the manifold is closed, the homology of the configuration space has a new pattern called periodic secondary homological stability.
%We prove that the configuration space of a non-compact 
%When the manifold $M$ is a non-compact surface and we consider homology with coefficients in the field $\mathbb{F}_{p}$, we prove that there is a map $$s_{2}\colon H_{i}(\text{Conf}_{\partic}(M),\text{Conf}_{\partic}(M);\mathbb{F}_{p})\to  H_{i+2p-2}(\text{Conf}_{\partic}(M),\text{Conf}_{\partic}(M);\mathbb{F}_{p})$$ that is an isomorphism in a range. When the manifold is closed, we construct a stabilization map of chain complexes $$t_{1}\colon C_{i}(\text{Conf}_{\partic}(M);\mathbb{F}_{p})\to  C_{i}(\text{Conf}_{\partic+p}(M);\mathbb{F}_{p})$$ and show that this map is an isomorphism in a range.
%We first consider the case that the manifold is non-compact.
\subsection{Homological Stability}
%Homological Stability for the Configuration Space of an Open Manifold
%If the manifold $M$ is non-compact, there is an open embedding $\embed\colon  M\sqcup \mathbb{R}^{n}\hookrightarrow M$. 
Given an embedding $$\embed\colon  L\sqcup M\hookrightarrow P,$$ there is an induced map $$\text{Conf}(\embed)\colon \text{Conf}_{j}(L)\times \text{Conf}_{\partic}(M)\to  \text{Conf}_{j+\partic}(P).$$ 
By the K\"unneth theorem,
%$H_{i}\big (\text{Conf}(M)\times \text{Conf}(\smon)\big )$ is isomorphic to $\oplus_{i=a+b}H_{a}\big (\text{Conf}(M)\big)\otimes H_{b}\big (\text{Conf}(\smon)\big)$ and so 
there is an induced map 
\begin{align*}
     \text{Conf}(\embed)_{*}\colon \bigoplus_{i=a+b}H_{a} (\text{Conf}_{j}(L);\fieldc)\otimes H_{b} (\text{Conf}_{\partic}(M);\fieldc)  &\to   H_{i} (\text{Conf}_{j+k}(P);\fieldc)\\
    \omega\otimes\tau & \mapsto  \text{Conf}(\embed)_{*}(\omega \otimes \tau).
\end{align*}
Fixing a homology class $x\in H_{a}(\text{Conf}_{j}(L);\fieldc)$ and a representative cycle $x'\in C_{a}(\text{Conf}_{j}(L);\fieldc)$ of $x$ yields a map
\begin{align*}
     t_{x}\colon H_{b}(\text{Conf}_{\partic}(M);\fieldc) &\to   H_{a+b}(\text{Conf}_{j+\partic}(P);\fieldc)\\
    \omega &\mapsto  \text{Conf}(\embed)_{*}(\omega \otimes x').
\end{align*}
%When the manifold is open (an open manifold is an
%a non-compact manifold without boundary)
When the manifold $M$ is a non-compact and without boundary (we will call such manifolds \emph{open manifolds}), there is an embedding $\embed\colon \mathbb{R}^{n}\sqcup M \to  M$ (see e.g. \textcite[Lemma 2.4]{MR3344444}). In this case, when $x\in H_{0} (\text{Conf}_{1}(\mathbb{R}^{n});\fieldc)$ is the class of a point, the map,
$$t_{1}\colonequals t_{x}\colon H_{i} (\text{Conf}_{\partic}(M);\fieldc)\to  H_{i} (\text{Conf}_{\partic+1}(M);\fieldc),$$ which depends on 
the choice of an end of $M$,
%the isotopy class of $\embed$, 
is known as the \textit{stabilization map} (see \cref{fig:stabilization map}).
%\footnote{need to check/figure out how to do maps for tikz diagrams}
%\footnote{I need different notation for the stabilization map since this clashes with the notation for the height-maybe replace height $t_{1}$ with $v_{1}$}
%$H_{i}\big (\text{Conf}(M)\times \text{Conf}(\mathbb{R}^{n})\big)\cong \oplus_{i=a+b} H_{a}\big (\text{Conf}(M)\big )\otimes H_{b}\big (\text{Conf}(\mathbb{R}^{n})\big )$

\begin{figure}[htb]
  \centering
  \includegraphics{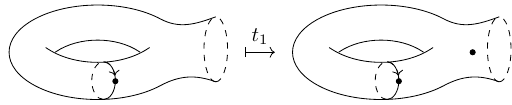}
  \caption{The stabilization map $t_{1}\colon H_{i} (\text{Conf}_{\partic}(M))\to  H_{i} (\text{Conf}_{\partic+1}(M))$ }
  %     without .tex extension
  % or use \input{mytikz} $t_{1}\colon H_{i}\big (\text{Conf}_{\partic}(M)\big )\to  H_{i}\big (\text{Conf}_{\partic}(M)\big )$}
  \label{fig:stabilization map}
\end{figure}

In \textcite[p.~104]{MR358766}, McDuff proved that the stabilization map $t_{1}$ is an isomorphism in integral homology  
%$(t_{1})_{*}\colon H_{i}(\text{Conf}_{\partic}(M);\mathbb{Z})\xrightarrow{\sim} H_{i}(\text{Conf}_{\partic}(M);\mathbb{Z})$
for all $k\gg i$ if $M$ is a smooth manifold of finite type. The isomorphism $H_{i} (\text{Conf}_{\partic}(M);\fieldc)\cong H_{i} (\text{Conf}_{\partic+1}(M);\fieldc)$ for all $k\gg i$ is known as \textit{homological stability}. The \textit{stable range} is the range where $t_{1}$ is an isomorphism and has been explicitly computed for homology with various coefficients
%in many cases and not just for integral homology 
(the current best known ranges for various coefficients are \textcite[Proposition A.1]{MR533892}, \textcite[Corollary 3]{MR2909770}, \textcite[Theorem B]{MR3032101}, \textcite[Theorem 1.4]{MR3344444}, \textcite[Theorem 1.3]{MR3704255}, and \cref{homological stability in dimension 2}, where we improve the stable range for odd torsion when the manifold is a surface). Since the stabilization map $t_{1}$ is also induced from the map of chain complexes 
\begin{align*}
    C_{*}(\text{Conf}_{\partic}(M);\fieldc)&\to   C_{*} (\text{Conf}_{\partic+1}(M);\fieldc)\\
    \xi &\mapsto  \text{Conf}(\embed)_{*}(\xi\otimes x)
\end{align*}
%there is a stabilization map $t_{1}$ inducing homological stability, 
one can make sense of relative homology. The stable range is the same as the range where the relative homology $H_{i}(\text{Conf}_{\partic+1}(M),\text{Conf}_{\partic}(M);\fieldc)$ vanishes because the stabilization map $t_{1}$ is injective by \textcite[p.~103]{MR358766}.
\subsection{Secondary Homological Stability}
%For this subsection, assume that the manifold $M$ is non-compact.
\begin{comment}
By the K\"unneth theorem,\footnote{Should I elaborate what I mean here?}
%$H_{i}\big (\text{Conf}(M)\times \text{Conf}(\mathbb{R}^{n})\big )\cong \oplus_{i=a+b} H_{a}\big (\text{Conf}(M)\big )\otimes H_{b}\big (\text{Conf}(\mathbb{R}^{n})\big )$
the stabilization map $(t_{1})_{*}\colon H_{*}\big (\text{Conf}_{\partic}(M)\big )\to  H_{*}\big (\text{Conf}_{\partic}(M)\big )$ can be interpreted as 
the restriction of the map
\begin{align*}
     \text{Conf}(\embed)_{*}\colon H_{i}\big (\text{Conf}(M)\times \text{Conf}(\mathbb{R}^{n})\big )  &\to   H_{i}\big (\text{Conf}(M)\big )\\
    \omega\otimes\tau &\mapsto  \text{Conf}(\embed)_{*}(\omega \otimes \tau)
\end{align*}
to $H_{i}\big (\text{Conf}_{\partic}(M)\big )\otimes H_{0}\big (\text{Conf}_{1}(\mathbb{R}^{n})\big )\subset H_{i}\big (\text{Conf}_{\partic}(M)\times \text{Conf}(\mathbb{R}^{n})\big ) $.
\end{comment}
%the composition of $$H_{i}\big (\text{Conf}(M)\times \text{Conf}(\mathbb{R}^{n})\big )\cong \oplus_{i+j=k} H_{i}\big (\text{Conf}(M)\big )\otimes H_{j}\big (\text{Conf}(\mathbb{R}^{n})\big )$$

%The map $(t_{1})_{*}$ is induced from the following map of chain complexes
%\begin{align*}
%    C_{*}\big (\text{Conf}_{\partic}(M)\big )&\to  C_{*}\big (\text{Conf}_{\partic}(M)\big )\\
%    \omega &\mapsto \text{Conf}(\embed)_{*}(\omega \otimes \xi).
%\end{align*}
If we replace $x\in H_{0} (\text{Conf}_{1}(\mathbb{R}^{n});\fieldc)$ with another homology class $z\in H_{a} (\text{Conf}_{b}(\mathbb{R}^{n});\fieldc)$ we get a new map of homology groups
\begin{align*}
    t_{z}\colon H_{*}(\text{Conf}_{\partic}(M);\fieldc)&\to  H_{*+a}(\text{Conf}_{\partic+b}(M);\fieldc).
    %\\
    %\omega &\mapsto \text{Conf}(\embed)_{*}(\omega \otimes z).
\end{align*}
It is natural to ask if this map yields a different stability pattern for $H_{*}(\text{Conf}_{\partic}(M);\fieldc)$.
In general, the map $t_{z}\colon H_{*}(\text{Conf}_{\partic}(M);\fieldc)\to  H_{*+a}(\text{Conf}_{\partic+b}(M);\fieldc)$ is not an isomorphism in a range for any choice of $z$ because it misses some classes in the image of $t_{1}$. But, we will show that $t_{z}$ induces a map of relative homology groups $$t_{z}\colon H_{*}(\text{Conf}_{\partic}(M), \text{Conf}_{\partic-1}(M);\fieldc)\to  H_{*+a} (\text{Conf}_{\partic+b}(M), \text{Conf}_{\partic-1+b}(M);\fieldc).$$
%where $\text{Conf}_{\partic}(M)\subset  \text{Conf}_{\partic}(M)$ means the image $t_{1} \big (\text{Conf}_{\partic}(M)\big )$, 
This induced map of relative homology groups will be an isomorphism in a new range for certain choices of $z\in H_{a}(\text{Conf}_{b}(\mathbb{R}^{n});\fieldc)$.
In the range that
the homology $H_{*}(\text{Conf}_{\partic}(M);\fieldc)$ 
stabilizes, the relative homology groups $H_{*} (\text{Conf}_{\partic}(M), \text{Conf}_{\partic-1}(M);\fieldc)$ and $H_{*+a} (\text{Conf}_{\partic +b}(M), \text{Conf}_{\partic -1+b}(M);\fieldc)$ are isomorphic because both groups are zero. In this paper, we will show that the map $t_{z}$ is an isomorphism in a range where the relative homology groups are potentially nonzero.
%From calculations of the stable range and our choice of $z\in H$, if $H_{*}\big (\text{Conf}_{\partic}(M), \text{Conf}_{\partic}(M)\big )$ is zero, then $H_{*+a}\big (\text{Conf}_{\partic+b}(M), \text{Conf}_{\partic+b}(M)\big )$ is zero as well. Therefore, this new range can be larger than the stable range.
%if $H_{*}\big (\text{Conf}_{\partic}(M), \text{Conf}_{\partic}(M)\big )$ is zero and $H_{*+a}\big (\text{Conf}_{\partic+b}(M), \text{Conf}_{\partic+b}(M)\big )$ will both be zero in the stable range,  

Isomorphisms of the form $$H_{*}(\text{Conf}_{\partic}(M), \text{Conf}_{\partic-1}(M);\fieldc)\cong  H_{*+a} (\text{Conf}_{\partic +b}(M), \text{Conf}_{\partic -1+b}(M);\fieldc)$$ are called \textit{secondary homological stability}. We prove the following secondary homological stability result (as well as higher-order stability results-see \cref{higher stab open surface} and \cref{higher stab open higher dimensional manifold}) for the configuration spaces of an open manifold.

\begin{theorem}
\label{secondary stability open case}
Let $M$ be an open connected manifold of dimension $n$. Let $R$ be either $\mathbb{Z}$ or the field $\mathbb{F}_{p}$.
%Let $k$ denote a nonnegative integer.
%Given a ring $R$ and a nonnegative integer $k$, let $f(R,n,k), g(R,n,k)$ denote the constants. 
%We have that $H_{i}( \text{Conf}_{\partic}(M),\text{Conf}_{\partic}(M);\fieldc )=0$ for $i\leq f(n,k,R)$.
%\hfill\begin{enumerate}
    %\item\label{homological stability open case} Let $f(R,n,k)$ denote the constant
    %\[f(R,n,k)= 
    %\begin{cases}
     % k & \text{if } R=\mathbb{F}_{p} \text{ for } p\text{ an odd prime} \text{ and } n> 2\\
      %\frac{k(p-1)-(p-2)}{p} & \text{if } R=\mathbb{F}_{p} \text{ for } p\text{ an odd prime} \text{ and } n=2\\
      %???? & \text{if } R=\mathbb{F}_{2}\\
      %k/2 & \text{if } R=\mathbb{Z}.
   %\end{cases}
   %\]
    %The relative homology $H_{i}(\text{Conf}_{\partic}(M),\text{Conf}_{\partic}(M);\fieldc)$ vanishes for $i\leq f(R,n,k)$.
    %\item\label{secondary stability open case}
    Let $g(R,n,k)$ denote the constant
    \[g(R,n,k)= 
    \begin{cases}
      \frac{p^{2}-1}{p^2}k - \frac{2p^{2}-2p-2}{p^{2}} & \text{if } R=\mathbb{F}_{p} \text{ for } p\text{ an odd prime} \text{ and } n=2\\
      (2/3)k & \text{if } R=\mathbb{Z} \text{ and } n=2   \\
      (3/4)k & \text{if } R=\mathbb{F}_{2} \text{ or if } n>2 \text{ and } R=\mathbb{Z}. \\
   \end{cases}
   \]
   There exists a homology class $z\in H_{a}(\text{Conf}_{b}(\mathbb{R}^{n});\fieldc)$ with
   \[(a,b)=
    \begin{cases}
      (2p-2,2p) & \text{if } \fieldc=\mathbb{F}_{p} \text{ for } p\text{ an odd prime and } n=2\\
      (1,2)  & \text{if } \fieldc=\mathbb{F}_{2} \text{ or if } \fieldc=\mathbb{Z}.\\
   \end{cases}
   \]
   such that the stabilization map $$t_{z}\colon H_{i}(\text{Conf}_{\partic}(M),\text{Conf}_{\partic -1}(M);\fieldc)\to  H_{i+a}(\text{Conf}_{\partic+b}(M),\text{Conf}_{\partic-1+b}(M);\fieldc ) $$ is an isomorphism for $i< g(R,n,k)$ and a surjection for $i=g(R,n,k)$. %\footnote{fix thm for case $R=\mathbb{Z}$ and $n=2$}
%\end{enumerate}

%\hfill\begin{enumerate}
%   \item $p=2$ case\colon 
%    \item Let $p$ be an odd prime. There is a homology class $y_{1}\in H$ and a map    $$s_{2} \colon  H_{i}( \text{Conf}_{\partic}(M),  \text{Conf}_{\partic}(M);\mathbb{F}_p )\to  H_{i+2p-2}(\text{Conf}_{\partic+2p}(M),\text{Conf}_{\partic+2p}(M);\mathbb{F}_p )$$  which is an isomorphism for $i \leq \frac{p^2 -1}{p^2}(k+1)-\frac{2p^2 -2p -2}{p^2}$.
%   \item Integral homology
%\end{enumerate}

%There is a map
%\begin{comment}
%$$s_{2} \colon  H_{i}(     \text{Conf}_{\partic}(M),     \text{Conf}_{\partic}(M);\mathbb{F}_p )\to  H_{i+2p-2}(\text{Conf}_{\partic+2p}(M),\text{Conf}_{\partic+2p}(M);\mathbb{F}_p )$$ which is an isomorphism for $i \leq \frac{p^2 -1}{p^2}(k+1)-\frac{2p^2 -2p -2}{p^2}$.
    %and a surjection for $i= \frac{p^2 -1}{p^2}(k+1)-\frac{2p^2 -2p -2}{p^2}$.
%    \end{comment}
\end{theorem}
The ranges in this theorem are optimal---they are achieved when $M=\mathbb{R}^{n}$. All of the proofs in this paper assume that $n \geq 2$ but many of the statements are vacuously true when $n=1$.
The assumption that $n=2$ when $p$ is odd in \cref{secondary stability open case} may be dropped. But in the case $n\neq 2$ and $p$ is odd, the relative homology groups $H_{i}(\text{Conf}_{\partic}(M),\text{Conf}_{\partic-1}(M);\mathbb{F}_{p})$ vanish in a range including the above range by \textcite[Proposition 3.3]{MR3344444}. Hence, these groups do not exhibit nontrivial secondary stability in this range.
%A similar result holds when the prime $p$ is odd and the dimension $n$ is greater than $2$ by a generalization of the case the prime $p$ is odd and the dimension $n$ equals $2$, but in this case the result is trivially true because the relative homology groups are zero by \textcite[Proposition 3.3]{KupMil}.
%When the dimension $n$ is greater than $2$ and the ring $R$ is $\mathbb{F}_{p}$ for $p$ an odd prime, this result is trivially true because the relative homology groups vanish in this range by \textcite[Proposition 3.3]{KupMil}.
When the manifold $M$ is a surface and $R=\mathbb{F}_{p}$ for $p$ an odd prime, we also show that the relative homology $H_{i}(\text{Conf}_{\partic}(M),\text{Conf}_{\partic -1}(M);R)$ vanishes for $i\leq \frac{k(p-1)-(p-2)}{p}$ and that this range is optimal (\cref{homological stability in dimension 2}).
%The only new range in \ref{homological stability open case} is the case $R=\mathbb{F}_{p}$ for $p$ an odd prime and $n=2$. The other ranges in \ref{homological stability open case} follow from Proposition 3.3 of \textcite{MR3344444}, Proposition A.1 of \textcite{MR533892}, and explicit calculations of $H_{i}(\text{Conf}(M);\mathbb{F}_{2})$.

The homology classes in \cref{secondary stability open case} are explicitly given in \cref{sec:stability for open}. The case $R=\mathbb{F}_{2}$ or $\mathbb{Z}$ is a folk theorem known to some experts and follows from calculations in \textcite{MR991102}. When the ring $R$ is $\mathbb{F}_{2}$, a picture for the homology class $z$ is given by two points ``doing a half twist'' (see \cref{fig:secondary stabilization map for F2}).
\begin{figure}[htb]
  \includegraphics{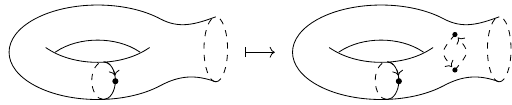}
  %     without .tex extension
  % or use \input{mytikz}
  \caption{The secondary stabilization map for homology with $\mathbb{F}_{2}$-coefficients}
  \label{fig:secondary stabilization map for F2}
\end{figure}
This class is similar to the class in \textcite{MR4019898}---in their case, the homology class inducing secondary representation stability for ordered configuration spaces is given by two ordered points doing a full twist. When the ring $R$ is $\mathbb{F}_{p}$ for $p$ an odd prime, the picture for $z$ is a $(2p-2)$-parameter family of $2p$ spinning points and hence is harder to illustrate.

The first example of secondary homological stability is in \textcite[Theorem A]{MR4028513} which states that %first obtained secondary homological stability results for 
the relative homology of the mapping class groups of a surface of genus $g$ with one boundary component $H_{i}(\text{Mod}_{g,1}, \text{Mod}_{g-1,1};\mathbb{Z})$ stabilizes as the genus $g$ increases by $3$ and the homological degree $i$ increases by $2$. A secondary stability result in the setting of representation stability for the ordered configuration space of an open connected manifold of dimension at least $2$ can be found in \textcite[Theorem 1.2]{MR4019898}. For other secondary stability results, see also \textcite[Theorem B]{miller2019representation}, \textcite[Theorem C]{MR4546785},  \textcite[Theorems B and C]{MR4099630}, and \textcite{ho2020higher}.

\subsection{Periodic Secondary Homological Stability}
We now discuss similar results when the manifold is a compact manifold without boundary (we will call such manifolds \emph{closed manifolds}).
%Suppose that $M$ is now a closed connected manifold.
%Periodic homological stability is a stability pattern that comes from studying the homology $H_{*}\big (\text{Conf}(M)\big )$ when $M$ is a closed connected manifold (closed here means...). 
When $M$ is a closed manifold there is no obvious stabilization map $\text{Conf}_{\partic}(M)\to  \text{Conf}_{\partic+1}(M)$.
%For example, when $M$ is the closed ball $\bar{D}^{n}$, by Theorems B and C of \textcite{chen2018adding}, there exists a continuous map $\text{Conf}_{\partic}(\bar{D}^{n})\to  \text{Conf}_{\partic}(\bar{D}^{n})$ if and only if $k$ or $n$ is equal to $2$.
Moreover, direct calculations show that the homology $H_{*}(\text{Conf}_{\partic}(M);\fieldc)$ does not stabilize in general---for example, $H_{1}(\text{Conf}_{\partic}(S^{2});\mathbb{Z})$ is isomorphic to $ \mathbb{Z}/(2k-2)\mathbb{Z}$ for $k\geq 2$ by a presentation of the spherical braid group given in \textcite[p.~255]{MR141128}.

Even though the integral homology does not stabilize in general, if instead one considers homology with coefficients over a different ring $R$, then there are still stability patterns for $H_{i}(\text{Conf}_{\partic}(M);R)$. For example, if the manifold $M$ is orientable and the ring $R$ is $\mathbb{Q}$, \textcite[Corollary 3]{MR2909770} states that $H_{i}(\text{Conf}_{\partic}(M);\mathbb{Q})$ is isomorphic to $H_{i}(\text{Conf}_{\partic+1}(M);\mathbb{Q})$ for $i<\partic$. Randal-Williams and Knudsen later on showed that the assumption that $M$ is orientable can be dropped (\textcite[Theorem C]{MR3032101} and \textcite[Theorem 1.3]{MR3704255}). In the same paper, Randal-Williams also showed that $H_{i}(\text{Conf}_{\partic}(M);\mathbb{F}_{2})$ is isomorphic to $H_{i}(\text{Conf}_{\partic+1}(M);\mathbb{F}_{2})$ for $i\leq k/2$.
%Periodicity of the homology of $\text{Conf}_{\partic}(M)$ also follows as a result ofWhen the ring is $\mathbb{F}_{2}$, Randal-Williams showed that the homology $H_{i}(\text{Conf}_{\partic}(M);\mathbb{F}_{2})$ is isomorphic to $H_{i}(\text{Conf}_{\partic}(M);\mathbb{F}_{2})$ in a range (\textcite[Theorem C]{RandalWilliamsUnordered}).
Nagpal (and independently \textcite[Theorem D]{MR3398727}) proved that if the ring $R$ is the field  $\mathbb{F}_{p}$ with $p$ an odd prime, then 
%and Cantero-Palmer each independently proved that, 
under various conditions on $M$, the homology $H_{i}(\text{Conf}_{\partic}(M);\mathbb{F}_{p})$ is periodic with period a power of $p$ in an explicit range (\textcite[Theorem F]{MR3358218}). Currently, the strongest periodicity result is  \textcite[Theorem A]{MR3556286} which states that when the dimension of $M$ is even and $p$ is an odd prime, the homology $H_{i}(\text{Conf}_{\partic}(M);\mathbb{F}_{p})$ is isomorphic to $H_{i}(\text{Conf}_{\partic+p}(M);\mathbb{F}_{p})$ in an explicit range 
%$m\in\mathbb{\smon}$ and $d$ dividing $2m$, the homology $H_{i}(\text{Conf}_{\partic}(M);\mathbb{Z}/d\mathbb{Z})$ is isomorphic to $H_{i}(\text{Conf}_{\partic+m}(M);\mathbb{Z}/d\mathbb{Z})$ in the stable range 
and when the dimension of $M$ is odd, $H_{i}(\text{Conf}_{\partic}(M);\mathbb{Z})$ is isomorphic to $H_{i}(\text{Conf}_{\partic+1}(M);\mathbb{Z})$ for $i<\partic/2$. For other stability patterns in the homology of configuration spaces, see \textcite[Theorem 4.7]{MR3261967}, \textcite{MR3398727}, and \textcite[Theorem 5.1]{MR3805056}.
%\textcite[Corollary F]{MR3398727}
%Results of Nagpal-Snowden on divided power algebras obtained information about the torsion of the integral homology $H_{*}(\text{Conf}(M);\mathbb{Z})$ 
%In this paper we prove that there is a map inducing periodic stability, as well as \textit{periodic secondary homological stability} when the dimension of $M$ is $2$.

We prove a secondary stability result (as well as higher-order stability results---see \cref{stability results for closed manifold}) for configuration spaces of closed manifolds. Since an embedding $M\sqcup \mathbb{R}^{n}\to  M$ does not exist when $M$ is a closed manifold, it is not obvious what relative homology should even mean. Although Kupers--Miller and Randal-Williams were able to prove stability patterns for the homology groups of the configuration spaces of closed manifolds, 
they only show that the homology groups are abstractly isomorphic and they do not even produce maps at the level of homology groups.
%so we cannot use their work to define relative homology.
%we cannot use their work to define relative homology because they only prove that the homology groups are abstractly isomorphic
%without giving a stabilization map inducing this isomorphism. 
%show that $H_{i}(\text{Conf}_{\partic}(M);\mathbb{F}_{p})$ is isomorphic to $H_{i}(\text{Conf}_{\partic+p}(M);\mathbb{F}_{p})$ when $n$ is even and $H_{*}(\text{Conf}_{\partic}(M);\mathbb{F}_{p} ) $ is isomorphic to $H_{*}(\text{Conf}_{\partic+1}(M);\mathbb{F}_{p} )$ when $n$ is odd or when $p=2$, they do not give 
We resolve this issue by constructing new chain-level stabilization maps that yield the stability results of \textcite[Theorem C]{MR3032101} and 
\textcite[Theorem A]{MR3556286}. 
These new stabilization maps allow us to define relative homology groups and we will also show that they induce stabilization maps of relative homology groups.
%Although Kupers--Miller build a stabilization map $$H_{i}(\text{Conf}_{\partic}(M);\mathbb{Z}[1/2])\to  H_{i}(\text{Conf}_{\partic+1}(M);\mathbb{Z}[1/2]) $$ via a zig-zag of maps of spaces when the dimension of the manifold $M$ is odd (Theorem B of \textcite{MR3556286}), our chain-level map $C_{*}(\text{Conf}_{\partic}(M);\mathbb{F}_{p} )\to  C_{*}(\text{Conf}_{\partic+1}(M);\mathbb{F}_{p} )$ is constructed differently from theirs.
%In the case that the homology 

\begin{theorem}\label{closed periodic secondary stability}
%need to fix this label
Let $M$ be a closed connected 
%surface
manifold of dimension $n$. Let $R$ be either $\mathbb{Z}$ or the field $\mathbb{F}_{p}$.
%and let $p$ be a prime number. 
\hfill\begin{enumerate}
    \item\label{closed secondary stability for dim M or p odd} Suppose that $n$ is odd or that $R$ is $\mathbb{F}_{2}$. There is a class $z\in H_{1}(\text{Conf}_{2}(\mathbb{R}^{n});R)$ and a stabilization map
    $$\newstab{z}\colon H_{i}(\text{Conf}_{\partic}(M),\text{Conf}_{\partic -1}(M);R)\to  H_{i+1}(\text{Conf}_{\partic +2}(M),\text{Conf}_{\partic +1}(M);R)$$ that is an isomorphism for $i< (3/4)k$ and a surjection for $i= (3/4)k$.
    \item\label{closed secondary stability} Suppose that $n$ is $2$ and $R$ is $\mathbb{F}_{p}$ with $p>2$. There is a class $z\in H_{2p-2}(\text{Conf}_{2p}(\mathbb{R}^{2});\mathbb{F}_{p})$ and a stabilization map
    $$\newstab{z}\colon H_{i}(\text{Conf}_{\partic}(M),\text{Conf}_{\partic -p}(M);R)\to  H_{i+2p-2}(\text{Conf}_{\partic +2p}(M),\text{Conf}_{\partic +p}(M);R)$$ that is an isomorphism for $i< \frac{p^{2}-1}{p^{2}}k-\frac{2p^{2}-2p-2}{p^{2}}$ and a surjection for $i=\frac{p^{2}-1}{p^{2}}k-\frac{2p^{2}-2p-2}{p^{2}}$.
\end{enumerate}
\end{theorem}

In the case that the ring $R$ is $\mathbb{F}_{2}$, the results of this paper mostly follow from \textcite[Theorem A and Theorem B]{MR991102} (see \cref{remark on applying BCT}). In this case, B\"{o}digheimer--Cohen--Taylor explicitly calculate $H_{*}(\text{Conf}_{\partic}(M);\mathbb{F}_{2})$  for all $\partic$ from $H_{*}(M;\mathbb{F}_{2})$. 
%To our knowledge, their framework does not extend to other rings.
%we do not know how to extend their framework to when the ring $R$ 
%Part \ref{closed secondary stability for dim M or p odd} of the theorem follows from \textcite{MR991102}.
%Although $H_{*}(\text{Conf}_{\partic}(M);\mathbb{F}_{2})$ can be calculated from \textcite{MR991102}
B\"{o}digheimer--Cohen--Taylor do not consider the relative homology groups $H_{*}(\text{Conf}_{\partic}(M),\text{Conf}_{\partic-1}(M);\mathbb{F}_{2})$, although it is conceivable that one could use ideas from their paper to define these groups. They do, however, compute $H_{*}(\text{Conf}_{\partic}(M);\mathbb{F}_{2})$ and this computation suggests the presence of our secondary stability result \cref{closed periodic secondary stability}.
\subsection{Proof Sketch of the Construction of the New Stabilization Map}\label{proof sketch of stabilization map construction}
The hardest part of the paper is constructing the stabilization map for the configuration space of a closed manifold (as opposed to proving that the map is an isomorphism) so we sketch this construction. 

Let $M$ be a manifold of dimension $n$ and fix an open disk $D^{n}\subset M$. There is a map $$\text{Conf}_{j}(M\setminus \bar{D}^{n})\times \text{Conf}_{\partic}(D^{n})\to  \text{Conf}_{j+\partic}(M).$$ If this map were a weak homotopy equivalence, then a stabilization map for $\text{Conf}_{j+\partic}(M)$ could be constructed by using the stabilization map on $\text{Conf}_{\partic}(D^{n})$ since $D^{n}$ is open. Clearly this is not the case. But a similar map is a weak homotopy equivalence and similar ideas will let us construct the desired stabilization map.
%But if this map were a weak equivalence, then for $M=S^{2}$ and $j+k\geq 2$, the homology group $H_{1}(\text{Conf}_{j+k}(S^{2});\mathbb{Z})\cong \mathbb{Z}/( 2(j+k)-2)$ would be isomorphic to
%\begin{align*}
%   H_{1}(\text{Conf}_{j+k}(S^{2});\mathbb{Z})\cong & H_{1}(\text{Conf}_{j}(D^{2})\times \text{Conf}_{\partic}(D^{2});\mathbb{Z})\\
%   \cong & H_{1}(\text{Conf}_{j}(D^{2});\mathbb{Z})\oplus H_{1}(\text{Conf}_{\partic}(D^{2});\mathbb{Z})
%\end{align*}
%which is clearly not the case.
We work with the functor
\begin{align*}
    \text{Conf}\colon \text{Top}&\to   \text{Top}\\
    N& \mapsto \bigsqcup_{\partic=0}^{\infty}\text{Conf}_{\partic}(N)
\end{align*}
instead of the functor $\text{Conf}_{\partic}$ since the former has better algebraic and functorial properties than the latter
%Another issue is that the functor 
%\begin{align*}
%    \text{Conf}_{\partic}\colon \text{Top}&\to   \text{Top}\\
%    N&\mapsto  \text{Conf}_{\partic}(N)
%\end{align*}
%does not have good algebraic or functorial properties. 
(for example, $\text{Conf}(M\sqcup N)\cong \text{Conf}(M)\times \text{Conf}(\smon)$).
%is not homeomorphic to $\text{Conf}_{\partic}(M)\times \text{Conf}_{\partic}(N)$.
%We can avoid this drawback by working with the functor
%\begin{align*}
%    \text{Conf}\colon \text{Top}&\to   \text{Top}\\
%    N&\mapsto  \sqcup_{\partic=0}^{\infty}\text{Conf}_{\partic}(N)
%\end{align*}
%instead. Another advantage of working with the functor $\text{Conf}$ is that 
There are maps
%\footnote{need to switch $M\setminus \bar{D}^{n}$ and $D^{n}$. Also, flip module picture in figure 3} 
\begin{align*}
    \text{Conf}(M\setminus \bar{D}^{n})\times\text{Conf}(S^{n-1}\times \mathbb{R}) &\to  \text{Conf}(M\setminus \bar{D}^{n}),\\
    \text{Conf}(S^{n-1}\times \mathbb{R}) \times \text{Conf}(D^{n})&\to  \text{Conf}(D^{n})
\end{align*}
%$\text{Conf}(M\setminus \bar{D}^{n})\times\text{Conf}(S^{n-1}\times \mathbb{R}) \to  \text{Conf}(M\setminus \bar{D}^{n})$ and $\text{Conf}(S^{n-1}\times \mathbb{R}) \times \text{Conf}(D^{n})\to  \text{Conf}(D^{n})$ 
coming from embeddings 
\begin{align*}
    (M\setminus \bar{D}^{n}) \sqcup (S^{n-1}\times \mathbb{R})&\hookrightarrow M\setminus \bar{D}^{n}, \\
    (S^{n-1}\times \mathbb{R})\sqcup D^{n}&\hookrightarrow D^{n},
\end{align*}
%$$(M\setminus \bar{D}^{n}) \sqcup (S^{n-1}\times \mathbb{R})\hookrightarrow M\setminus \bar{D}^{n}\text{ and } (S^{n-1}\times \mathbb{R})\sqcup D^{n}\hookrightarrow D^{n},$$
%and $(S^{n-1}\times \mathbb{R})\bigsqcup D^{n}$
respectively. The configuration space $\text{Conf}(S^{n-1}\times \mathbb{R})$ is a monoid up to coherent homotopy by stacking cylinders from end to end (see \cref{mon-and-mod-maps-up-to-homotopy}).
%So unlike $\text{Conf}_{j}(M\setminus \bar{D}^{n})$ and $\text{Conf}_{\partic}(D^{n})$, 
In addition, the configuration spaces $\text{Conf}(M\setminus \bar{D}^{n})$ 
and $\text{Conf}(D^{n})$ are modules up to coherent homotopy over $\text{Conf}(S^{n-1}\times \mathbb{R})$ (see \cref{mon-and-mod-maps-up-to-homotopy}).
\begin{figure}
    \centering
    \includegraphics[scale=.7]{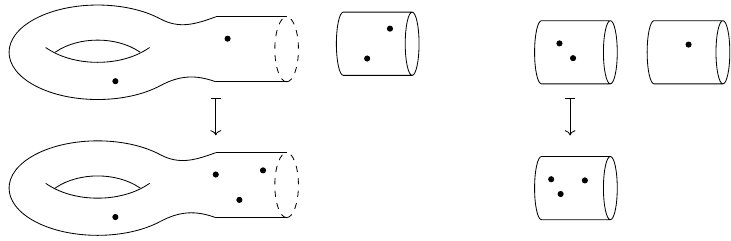}
    \caption{The $\text{Conf}(S^{n-1}\times \mathbb{R})$-module structure of $\text{Conf}(M\setminus \bar{D}^{n})$ and the monoid map for $\text{Conf}(S^{n-1}\times \mathbb{R})$ }
    \label{mon-and-mod-maps-up-to-homotopy}
\end{figure}
%$$ (H_{1}(\text{Conf}_{j}(D^{2});\mathbb{Z})\otimes H_{0}\big (\text{Conf}_{j}(D^{2});\mathbb{Z})\big )\oplus (H_{0}(\text{Conf}_{j}(D^{2});\mathbb{Z})\otimes H_{1}(\text{Conf}_{j}(D^{2});\mathbb{Z}))$$

The map $$\text{Conf}(M\setminus\bar{D}^{n}) \times \text{Conf}(D^{n}) \to  \text{Conf}(M)$$ is still not a weak homotopy equivalence. But, by using the $\text{Conf}(S^{n-1}\times \mathbb{R})$-module maps on $\text{Conf}(M\setminus \bar{D}^{n})$ and $\text{Conf}(D^{n})$ we now have two new maps 
\begin{align*}
    \bigg ( \text{Conf}(M\setminus \bar{D}^{n} )\times \text{Conf}(S^{n-1}\times \mathbb{R})\bigg )\times \text{Conf}(D^{n})& \to  \text{ Conf}(M), \\
    \text{Conf}(M\setminus \bar{D}^{n})\times  \bigg (\text{Conf}(S^{n-1}\times \mathbb{R})\times \text{Conf}(D^{n})\bigg )& \to  \text{ Conf}(M).
\end{align*}
Due to these maps being homotopic (see \cref{fig:excision maps}) and the higher homotopy coherence between these maps, 
%and actually being equal after replacing these configuration spaces with homotopy equivalent spaces with better algebraic properties,
\begin{figure}[htb]
  \includegraphics[scale=.7]{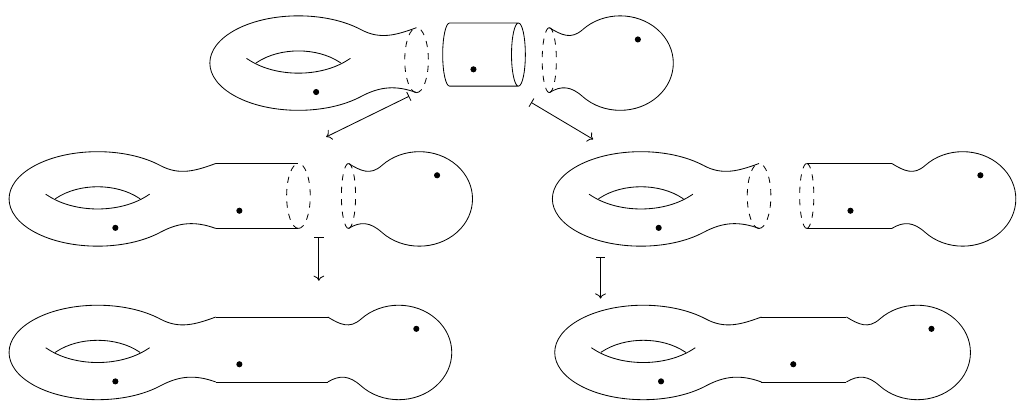}
  %     without .tex extension
  % or use \input{mytikz}
  \caption{$\text{Conf}(M\setminus \bar{D}^{n})\times  \text{Conf}(S^{n-1}\times \mathbb{R})\times \text{Conf}(D^{n})\to  \text{Conf}(M)$ }
  \label{fig:excision maps}
\end{figure}
there is an induced map on the derived tensor product of $\text{Conf}(S^{n-1}\times \mathbb{R})$-modules  $$\text{Conf}(M\setminus\bar{D}^{n}) \bigotimes^{\mathbb{L}}_{\text{Conf}(S^{n-1}\times \mathbb{R})} \text{Conf}(D^{n}) \to  \text{Conf}(M).$$
Our model for the derived tensor product 
%$$\text{Conf}(M \setminus\bar{D}^{n}) \bigotimes^{\mathbb{L}}_{\text{Conf}(S^{n-1}\times \mathbb{R})} \text{Conf}(D^{n})$$ 
is the two-sided bar construction of the modules $\text{Conf}(M\setminus\bar{D}^{n})$ and $\text{Conf}(D^{n})$
%over the Moore loop space model for $\text{Conf}(S^{n-1}\times \mathbb{R})$ 
(see \cref{iterated bar construction model for conf}).
By excision for factorization homology (\textcite[Lemma 3.18]{MR3431668}, or see \cref{map from bar construction to conf is a weak equivalence}), this induced map 
%$\text{Conf}(M\setminus\bar{D}^{n})\bigotimes^{\mathbb{L}}_{\text{Conf}(S^{n-1}\times \mathbb{R})} \text{Conf}(D^{n})\to  \text{Conf}(M)$ 
is a weak homotopy equivalence.
%(see \cref{map from bar construction to conf is a weak equivalence} or excision for factorization homology \textcite[Lemma 3.18]{MR3431668}).  
%From excision for factorization homology \textcite[Lemma 3.18]{MR3431668} (see also \cref{map from bar construction to conf is a weak equivalence}), we have a weak homotopy equivalence $$\text{Conf}(M\setminus\bar{D}^{n})\bigotimes^{\mathbb{L}}_{\text{Conf}(S^{n-1}\times \mathbb{R})} \text{Conf}(D^{n})\to  \text{Conf}(M).$$ 
Therefore, to construct a chain-level stabilization map for $\text{Conf}(M)$ with coefficients in a ring $\fieldc$, it suffices to construct a chain-level stabilization map $$C_{*} (\text{Conf}(D^{n});\fieldc)\to  C_{*} (\text{Conf}(D^{n});\fieldc)$$ that is also a module map over the differential graded ring $C_{*}(\text{Conf}(S^{n-1}\times \mathbb{R});\fieldc)$.
%$\text{Conf}(S^{n-1}\times \mathbb{R})$-module map.

We build such a map by giving a length two resolution of the differential graded module
%$\text{Conf}(S^{n-1}\times \mathbb{R})$-module 
$C_{*}(\text{Conf}(D^{n});\fieldc)$ 
by free modules over the differential graded ring $C_{*} (\text{Conf}(S^{n-1}\times \mathbb{R});\fieldc)$.
%$\text{Conf}(S^{n-1}\times \mathbb{R})$-modules. 
Up to homotopy equivalence, the generator can be taken to be a point in $H_{0}(\text{Conf}_{1}(D^{n});\fieldc)$. The only relation in our presentation comes from the fact that the natural inclusion $\text{Conf}_{1}(S^{n-1}\times \mathbb{R}) \hookrightarrow  \text{Conf}_{1}(D^{n})$ is nullhomotopic (see \cref{explanation for resolution} for a conceptual reason for the existence of such a small resolution).
%This relation corresponds to $H_{n-1}\big (\text{Conf}_{1}(S^{n-1}\times \mathbb{R})\big ) \to  H_{n-1}\big (\text{Conf}_{1}(D^{n})\big )$ being the zero map. 
Using this presentation for the module $C_{*}(\text{Conf}(D^{n});\fieldc)$, we can easily check when a homology class $z\in H_{i}(\text{Conf}_{\partic}(D^{n});\fieldc)$ gives a chain-level stabilization map $$t_{z}\colon C_{*}(\text{Conf}(D^{n});\fieldc)\to  C_{*}(\text{Conf}(D^{n});\fieldc)$$ of $C_{*}(\text{Conf}(S^{n-1}\times \mathbb{R});\fieldc)$-modules: the Browder bracket of the homology class $z$ and the class of a point $e\in H_{0}(\text{Conf}_{1}(D^{n});\fieldc)$ needs to be zero. This condition on the Browder bracket was key to the proofs of stability patterns in 
\textcite{MR3032101} and 
\textcite{MR3556286}, and our construction can be viewed as a categorification of their work.

Using this presentation of $C_{*}(\text{Conf}(D^{n});\fieldc)$, we obtain a stabilization map $$C_{*}(\text{Conf}(D^{n});\fieldc)\to  C_{*}(\text{Conf}(D^{n});\fieldc)$$ which is also a map of $C_{*}(\text{Conf}(S^{n-1}\times \mathbb{R});\fieldc)$-modules. Such a stabilization map %$t_{z}\colon C_{*}(\text{Conf}(D^{n})\to  C_{*}(\text{Conf}(D^{n})$ , 
induces a well-defined chain map $$ C_{*}\big(\text{Conf}(M \setminus\bar{D}^{n}) \bigotimes^{\mathbb{L}}_{\text{Conf}(S^{n-1}\times \mathbb{R})} \text{Conf}(D^{n});\fieldc\big ) \to  C_{*}\big(\text{Conf}(M \setminus\bar{D}^{n}) \bigotimes^{\mathbb{L}}_{\text{Conf}(S^{n-1}\times \mathbb{R})} \text{Conf}(D^{n});\fieldc\big).$$ This chain map on the derived tensor product gives a stabilization map $$C_{*}(\text{Conf}(M);\fieldc)\to C_{*}(\text{Conf}(M);\fieldc)$$ since $\text{Conf}(M \setminus\bar{D}^{n}) \bigotimes^{\mathbb{L}}_{\text{Conf}(S^{n-1}\times \mathbb{R})} \text{Conf}(D^{n})$ is weakly equivalent to $\text{Conf}(M)$. Then, we can specialize to construct stabilization maps  $$ C_{*}(\text{Conf}_{\partic}(M);\mathbb{F}_{p} )\to  C_{*}(\text{Conf}_{\partic+p}(M);\mathbb{F}_{p} )$$ when $n$ is even and $p$ is odd, and $$C_{*}(\text{Conf}_{\partic}(M);\mathbb{F}_{p} )\to  C_{*}(\text{Conf}_{\partic+1}(M);\mathbb{F}_{p} )$$ when $n$ is odd or when $p=2$.
Moreover, the same ideas allow for the construction of secondary and higher stabilization maps.

\subsection{Outline}
%In \cref{sec:notation}, we fix notations and conventions for this paper.
In \cref{sec:base case}, we review homology operations over an $E_{n}$-algebra and use Cohen--Lada--May's calculation of $H_{*}(\text{Conf}(\mathbb{R}^{n});\mathbb{F}_{p})$ to get bounds for homological stability and secondary homological stability. In \cref{sec:stability for open}, we review homological stability and prove secondary homological stability for homology with $\mathbb{F}_{p}$ coefficients of the configuration spaces of an open manifold. 
In \cref{sec:various models of conf}, we construct semi-simplicial models for $\text{Conf}(M)$ and $\text{Conf}(D^{n})$.
%In \cref{models of conf space}, we construct a semi-simplicial model for $\text{Conf}(M)$. In \cref{stabilization map construction}, we construct a semi-simplicial model for $\text{Conf}(D^{n})$. 
%In \cref{sec:constructing stab maps}, we construct stabilization maps for $\nch{\text{Conf}(M)}$.
%In \cref{stab maps on conf(cyl)} and \cref{stab maps on conf(disk)}, we construct stabilization maps for $\nch{\text{Conf}(S^{n-1}\times \mathbb{R})}$ and $\nch{\text{Conf}(D^{n})}$, respectively. 
We construct new
stabilization maps for the configuration spaces of manifolds in \cref{sec:constructing stab maps} and prove periodic secondary homological stability in \cref{sec:puncture resolution}.
\subsection{Acknowledgments}
I am grateful to Jeremy Miller for his guidance and feedback. I thank Alexander Kupers for suggestions on how to construct the stabilization map for the configuration space of a closed manifold. I thank Andrea Bianchi for suggestions on how to compare different stabilization maps for the configuration space of a cylinder in \cref{stab maps on conf(cyl)} and for
describing when classes in $H_{*}(\text{Conf}_{\partic}(\mathbb{R}^{n});\mathbb{F}_{p})$ lift to classes in $H_{*}(\text{Conf}_{\partic}(\mathbb{R}^{n});\mathbb{Z})$. I thank Oscar Randal-Williams for suggestions on how to simplify the proofs of \cref{map from bar construction to conf is a weak equivalence} and for other helpful conversations. I also would like to thank Nicholas Wawrykow and Jennifer Wilson for feedback on an early draft of this paper. I thank the anonymous referee for their detailed feedback and for their suggestion on how to simplify the proof of \cref{res is weakly equivalent to Conf(Disk)}.
%\section{Notation}\label{sec:notation}
\section{Homology Operations and Stability for \texorpdfstring{$\text{Conf}(\mathbb{R}^{n})$}{the Configuration Space of a Disk} }\label{sec:base case}
\begin{comment}
Fix an odd prime $p$. In this section, we prove that stabilization maps $t_{j}\colon\text{Conf}_{\partic}(\mathbb{R}^2)\to  \text{Conf}_{\partic+j}(\mathbb{R}^2)$ induce isomorphisms 
\\ $H_{i}\big (\text{Conf}_{\partic}(\mathbb{R}^2)\big )\to  H_{i+j-1}\big (\text{Conf}_{\partic+j}(\mathbb{R}^2)\big )$ in the range given in Theorem ~\ref{eqtable}. This will imply Theorem ~\ref{eqtable} when $M=\mathbb{R}^2$. Most of the background information follows \textcite{MR3344444} 
%\textcite{A}.
Following an argument of McDuff (page 103 of 
\textcite{MR358766}) for $t_1$, one can show
\begin{theorem}[McDuff]\label{thm-McDuff}
Let $M$ be an open connected manifold. The stabilization maps $t_{1}\colon\text{Conf}_{\partic}(\mathbb{R}^2)\to  \text{Conf}_{\partic}(\mathbb{R}^2)$ are injective on homology.\footnote{I'm not sure where to put this}
\end{theorem}
We need to prove surjectivity of the stabilization maps in the relevant ranges.
\end{comment}
We begin this section by reviewing homology operations and some of their properties for the homology of an algebra over the little $n$-disk operad $E_{n}$ (in this paper, $E_{n}$ means just the little $n$-disk operad, instead of an operad homotopy equivalent to the little $n$-disk operad). Pick a homeomorphism $\mathbb{R}^{n}\to D^{n}$.
Since the unordered configuration space of points in $\mathbb{R}^{n}$, $$\text{Conf}(\mathbb{R}^{n})\colonequals \bigsqcup_{\partic\geq 0}\text{Conf}_{\partic}(\mathbb{R}^{n}),$$ is an example of an $E_{n}$-algebra, there are homology operations and these homology operations are important in Cohen--Lada--May's calculation of $H_{*}(\text{Conf}(\mathbb{R}^{n});\mathbb{F}_{p})$.
%these homology operations are important for proving homological stability of $\text{Conf}(\mathbb{R}^{n})$. 
%When the dimension $n$ is $2$, these homology operations provide a very explicit calculation of the homology $H_{*}(\text{Conf}(\mathbb{R}^{2}))$. 
We use this calculation to obtain the stable ranges for homological stability and secondary homological stability for $\text{Conf}(\mathbb{R}^{n})$. In \cref{sec:stability for open} and \cref{sec:puncture resolution}, we use these stable ranges to prove stability results for the configuration space of a manifold.
%Even though calculating these stable ranges might seem unnecessary
%Finally, we use this calculation to obtain stable ranges for homological stabililty and secondary homological stability for the unordered configuration space of a surface.
%Then we use a calculation of the homology of $\text{Conf}(\mathbb{R}^{2})$ to obtain stable ranges for homological stabililty and secondary homological stability for the unordered configuration space of an open surface. 
%These homology operations appear in the description of the homology of an algebra over the little $n$-disk operad $E_n$ (here, homology means homology with $\mathbb{F}_p $ coefficients).
%Therefore, these homology operations apply in our case because  $\text{Conf}(\mathbb{R}^{n})$ is homotopy equivalent to the free $E_{n}$-algebra on a point.\footnote{Rewrite this sentence}

For simplicity, suppose that a space $A=\bigsqcup_{\partic\geq 0}A_{\partic}$ is an $E_{n}$-algebra. We think of $A$ as being graded over $\mathbb{N}$.
%The main example to keep in mind is $A=\text{Conf}(\mathbb{R}^{n})$ with $A_{\partic}=\text{Conf}_{\partic}(\mathbb{R}^{n})$. 
Since $A$ is an  $E_{n}$-algebra, it has the following operations (for more information about these operations, see \textcite[p.~213--219]{MR0436146} and \textcite[p.~158--164]{galatius2018cellular}).
\hfill\begin{enumerate}
    \item For chains with coefficients in a ring $\fieldc$: a multiplication map
        \begin{align*}
        -\bullet - \colon     C_{q}(A_{j};\fieldc)\otimes C_{r}(A_{\partic};\fieldc)\to  C_{q+r}(A_{j+\partic};\fieldc).
        \end{align*}
        \item For chains with coefficients in a ring $\fieldc$: a Browder bracket  %\footnote{Should I include derivation property and definition of Browder bracket? Maybe include picture of class corresponding to secondary stability} 
        \begin{align*}
            \browd{-}{-} \colon C_{q}(A_{j};\fieldc) \otimes C_{r}(A_{\partic};\fieldc) \to  C_{q+r+n-1}(A_{j+\partic};\fieldc).
        \end{align*} 
        \item For homology with $\mathbb{F}_{p}$ coefficients: when $p$ is odd and $q+n-1$ is even, there are ``top" homology operations
                $$\xi\colon H_{q}(A_{\partic};\mathbb{F}_{p})\to  H_{pq+(n-1)(p-1)}(A_{p\partic};\mathbb{F}_{p}),$$
%        the top homology operation composed with homology Bockstein
                $$\beta\xi\colon H_{q}(A_{\partic};\mathbb{F}_{p})\to  H_{pq+(n-1)(p-1)-1}(A_{p\partic};\mathbb{F}_{p}),$$
        and an associated operation
        $$ \zeta\colon H_{q}(A_{\partic};\mathbb{F}_{p})\to  H_{pq+(n-1)(p-1)-1}(A_{p\partic};\mathbb{F}_{p}).$$
        When $p$ is $2$, the top homology operation $\xi$ is defined for all $q$.
        \item\label{Dyer-Lashof operations} For homology with $\mathbb{F}_{p}$ coefficients:
        when $p$ is odd, there are Dyer-Lashof operations
        %\footnote{Should I mention Bockstein Dyer-Lashof operations?} 
            \begin{align*}
               Q^{s}\colon H_{q}(A_{\partic};\mathbb{F}_{p})\to  H_{q+2s(p-1)}(A_{p\partic};\mathbb{F}_{p})\\
                \text{defined for }s\geq 0 \text{ with } 2s-q < n-1, \text{ and}
            \end{align*}
        % and Dyer-Lashof operations composed with homology Bockstein
        %\footnote{Should I mention Bockstein Dyer-Lashof operations?} 
        %when $p$ is odd
            \begin{align*}
               \beta Q^{s}\colon H_{q}(A_{\partic};\mathbb{F}_{p})\to  H_{q+2s(p-1)-1}(A_{p\partic};\mathbb{F}_{p})\\
                \text{defined for }s\geq 0 \text{ with } 2s-q < n-1.
            \end{align*} 
        When $p$ is $2$, there are Dyer-Lashof operations 
        \begin{align*}
               Q^{s}\colon H_{q}(A_{\partic};\mathbb{F}_{2})\to  H_{q+s}(A_{2\partic};\mathbb{F}_{2})\\
                \text{defined for }s\geq 0 \text{ with } s-q < n-1.
            \end{align*}
        When $p=2$ and $s< q$, the Dyer-Lashof operation $Q^{s}$ is the zero map, and when $p=2$ and $s=q$, the Dyer-Lashof operation $Q^{s}$ is the squaring map $Q^{s}(x)=x^{2}$ (when $p$ is an odd prime number, a similar description of the Dyer-Lashof operations holds for $s\leq 2q$ instead of $s\leq q$).
        %\item\label{Bockstein Dyer-Lashof operations} For homology with $\mathbb{F}_{p}$ coefficients: when $p$ is odd, there are Dyer-Lashof operations composed with homology Bockstein
        %\footnote{Should I mention Bockstein Dyer-Lashof operations?} 
        %when $p$ is odd
        %    \begin{align*}
        %       \beta Q^{s}\colon H_{q}(\text{Conf}_{\partic}(\mathbb{R}^{n});\mathbb{F}_{p})\to  H_{q+2s(p-1)-1}(\text{Conf}_{p\partic}(\mathbb{R}^{n});\mathbb{F}_{p})\\
        %        \text{defined for }s\geq 0 \text{ with } 2s-q < n-1.
        %    \end{align*}
\end{enumerate}
Since we are working with spaces, the operations $\beta \xi$ and $\beta Q^{s}$ are the top homology operation $\xi$ and the Dyer-Lashof operations $Q^{s}$ composed with the Bockstein homomorphism respectively (see \textcite[Remark 16.2]{galatius2018cellular}).
%\footnote{You asked if this sentence matters. I think I use this in \cref{homology classes which give a stabilization map}.}

We write $xy$ for the multiplication $x \bullet y$.
\begin{comment}
Since we are dealing with an $E_2$-algebra, we only have to consider the first four operations.
\end{comment}
Now we explicitly define the multiplication map and Browder bracket on $C_{*}(\text{Conf}(\mathbb{R}^{n});\fieldc)$ and list some of the properties of the Browder bracket because we need these facts 
%to prove secondary stability in dimensions $n>2$ (see \cref{stable ranges for higher dimensions with F2 coefficients}) and
%explicitly define the Browder bracket and list some of its properties because we need these facts in \cref{stabilization map construction} 
to construct a chain-level stabilization map for the configuration space of a closed manifold (see \cref{p-power stabilization} and \cref{homology classes which give a stabilization map}).
\begin{definition}[The multiplication map and the Browder bracket]\label{definition of multiplication and browder maps}
Recall that $E_{n}(2)$ is the space of embeddings $f\colon D^{n}\sqcup D^{n}\to  D^{n}$ such that the restriction of $f$ to each component of $D^{n}\sqcup D^{n}$ is obtained by compositions of dilations and translations. Such an embedding induces a map on configuration spaces $\text{Conf}(D^{n})\times \text{Conf}(D^{n})\to  \text{Conf}(D^{n})$. Therefore, we have a map 
$$\theta \colon E_{n}(2)\times \text{Conf}(D^{n})\times \text{Conf}(D^{n})\to  \text{Conf}(D^{n}).$$
%defined as follows: by definition, an element $f\in E_{n}(2)$ is a rectilinear embedding $f\colon D^{n}\sqcup D^{n}\to  D^{n}$. 
Since $E_{n}(2)$ is homotopy equivalent to $S^{n-1}$, by applying the Eilenberg--Zilber map, the map $\theta$ induces a map on chain complexes
%homology with  coefficients in a ring $\fieldc$:
$$\theta_{*}\colon C_{*}(S^{n-1};\fieldc)\otimes C_{*}(\text{Conf}(D^{n});\fieldc)\otimes C_{*}(\text{Conf}(D^{n});\fieldc)\to  C_{*}(\text{Conf}(D^{n});\fieldc). $$
For $n\geq 2$, there are generators $\sigma_{0}\in H_{0}(S^{n-1};\fieldc)$ and $\sigma_{n-1}\in H_{n-1}(S^{n-1};\fieldc)$, which come from choosing a point of $S^{n-1}$ and an orientation of $S^{n-1}$ respectively.
%gives a generator $\sigma_{n-1}\in H_{n-1}(S^{n-1};\fieldc)$.
The \textbf{multiplication map} $x \bullet y$ of $x$ and $y$ is the chain $$x \bullet y\colonequals  \theta_{*}(\sigma_{0}\otimes x\otimes y).$$
The \textbf{Browder bracket} $\browd{x}{y}$ of $x$ and $y$ is the chain $$\browd{x}{y}\colonequals (-1)^{\text{deg}(x)(n-1)+1} \theta_{*}(\sigma_{n-1}\otimes x\otimes y).$$
\end{definition}
%\begin{remark}\label{geometric description of product map}
%If $e\in H_{*}(\text{Conf}_{1}(D^{n});\fieldc)$ is the class of a point, 
%The multiplication map $x \bullet y$  of $x$ and $y$ has the following geometric description: it is the class obtained by inserting $x$ and $y$ in a disk (see \cref{fig:multiplication map}).\footnote{need to do picture and cite this}
%\begin{figure}[htb]
%  \includegraphics{Browder-bracket-geometric-interpretation.pdf}
  %     without .tex extension
  % or use \input{mytikz}
%  \caption{The Browder bracket $\browd{x}{y}$ }
%  \label{fig:multiplication map}
%\end{figure}
%\end{remark}
\begin{remark}\label{geometric description of Browder bracket}
%If $e\in H_{*}(\text{Conf}_{1}(D^{n});\fieldc)$ is the class of a point, 
The multiplication map $x \bullet y$ and the Browder bracket $\browd{x}{y}$ of $x$ and $y$ in $C_{*}(\text{Conf}(D^{n});\fieldc)$ have the following geometric descriptions: the multiplication of $x$ and $y$ is the chain obtained by inserting $x$ and $y$ in a disk and the Browder bracket of $x$ and $y$ is the chain obtained by rotating $y$ around $x$ about $S^{n-1}$.
\begin{figure}[htb]
    \centering
    \subfloat{\includegraphics{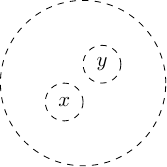}\label{fig:product map}}
    %\begin{subfigure}
    %\caption{The multiplication map $x\bullet y$}
    %\end{subfigure}
    \hspace{2.5cm}
    \subfloat{\includegraphics{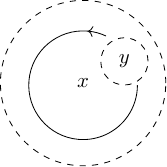}\label{fig:Browder bracket with a point}}
  %     without .tex extension
  % or use \input{mytikz}
  %\caption{The Browder bracket $\browd{x}{y}$}
  \caption{The multiplication map $x\bullet y$ and the Browder bracket $\browd{x}{y}$ on $C_{*}(\text{Conf}(D^{n});\fieldc)$}
\end{figure}
\end{remark}
At the level of homology, the Browder bracket has the following properties (see \textcite[p.~215--218]{MR0436146} and \textcite[p.~176--181]{galatius2018cellular}).
\hfill\begin{enumerate}
    \item\label{Browder brack is linear} The Browder bracket is linear in both entries.
    \item\label{Browder bracket symmetric up to sign} The Browder bracket is symmetric up to a sign:
    $$\browd{x}{y}= (-1)^{\text{1+deg}(x)\text{deg}(y)+(n-1)(\text{deg}(x)+\text{deg}(y)+1)} \browd{y}{x}.$$
    \item\label{Jacobi identity of Browder bracket} The Browder bracket satisfies the Jacobi identity up to sign:
    
    \begin{align*}
        0&=(-1)^{(\text{deg}(x)+n-1)(\text{deg}(z)+n-1)}\browd{x}{\browd{y}{z}}+ (-1)^{(\text{deg}(y)+n-1)(\text{deg}(x)+n-1)}\browd{y}{\browd{z}{x}}\\
        & +(-1)^{(\text{deg}(z)+n-1)(\text{deg}(y)+n-1)}\browd{z}{\browd{x}{y}}.
    \end{align*}
    %For homology with $\mathbb{F}_{p}$ coefficients, the 
    The Browder bracket satisfies $\browd{x}{x}=0$ for homology with $\mathbb{F}_{2}$ coefficients
    %when $p$ is $2$ 
    and it also satisfies $\browd{x}{\browd{x}{x}}=0$.
    \item\label{derivation property of Browder bracket} The Browder bracket is a derivation with respect to multiplication:
    $$\browd{x}{yz}=\browd{x}{y}z + (-1)^{\text{deg}(y)(\text{deg}(x)+n-1)} y\browd{x}{z}.$$
    \item\label{Browder bracket with unit} The Browder bracket with the unit vanishes: $\browd{1}{x}=0$ if $1\in H_{0}(A_{0};\fieldc)$ is the identity element.
    %\footnote{maybe drop this relation}
    \item\label{Browder bracket with Dyer-Lashof} The Browder bracket of a Dyer-Lashof operation vanishes: $\browd{x}{Q^{s}y}=0$.
    \item\label{Bockstein applied to Browder bracket} $\beta \browd{x}{y}=\browd{\beta x}{y}+(-1)^{n-1+\text{deg}(x)}\browd{x}{\beta y}$.
    %for all $a,b \in H_{*}(\text{Conf}(\mathbb{R}^{n});\mathbb{F}_{p})$.
    \item\label{iterated Browder bracket} Let $\text{ad}(x)(y)\colonequals \browd{x}{y}$ and for $i>1$, let $\text{ad}^{i}(x)(y)\colonequals \text{ad}(x)\big (\text{ad}^{i-1}(x)(y)\big )$. The homology operation $\zeta (-) $ can be defined as the map $\beta\xi(-)-\text{ad}^{p-1}(-)(\beta (-) )$. %since we are working with $E_{n}$-algebras in the category of spaces. 
    We have that $$\browd{x}{\zeta y}=0 \text{ and } \browd{x}{\xi y}=\text{ad}^{p}(y)(x).$$
    When $p=2$, we have that $\browd{x}{\xi y}=\browd{y}{\browd{y}{x}}. $
    %For $p>2$ and $a\in H_{i}(\text{Conf}(\mathbb{R}^{n});\mathbb{F}_{p})$ with $i+n$ odd, let $\zeta_{n} (a)\colonequals  \beta \xi_{n}(a) - \text{ad}^{p-1}(a)(\beta a)$.\footnote{I need to check that the top homology class subscript is correct. Also, Soren-Sander-Oscar use a different definition for $\zeta$, see page 159. It might be worthwhile to work with this other definition of $\zeta$.}  We have that $$\browd{\zeta_{n}(a)}{b}=0$$ for all $b \in H_{*}(\text{Conf}(\mathbb{R}^{n});\mathbb{F}_{p})$.
\end{enumerate}

From these operations, Cohen--Lada--May calculated $H_{*}(\text{Conf}(\mathbb{R}^{n});\mathbb{F}_{p})$. Before introducing this calculation, we now describe a generalization of $\text{Conf}(M)$ and a functor which assigns to a based space $(X,*)$ an $E_{n}$-algebra $\mathbf{E_{n}}(X)$ because we will need them to study $H_{*}(\text{Conf}(S^{n-1}\times (0,\infty));\fieldc)$ for $\fieldc= \mathbb{F}_{p}, \mathbb{Q}$ in \cref{different models of conf(cyl) are equivalent}.
%\footnote{motivate this better}
\begin{definition}
We use the notation $\theta$ to denote the composition map $\theta\colon E_{n}(\partic)\times E_{n}(j_{1})\times \cdots \times E_{n}(j_{\partic})\to E_{n}(j)$, with $j=\sum_{i=1}^{\partic}j_{i}$, of the operad $E_{n}$ and we use the notation $*$ to denote the point $E_{n}(0)$ (recall that $E_{n}(0)$ consists of the unique ``embedding'' of the empty set into $D^{n}$).
Let $\text{Top}_{*}$ denote the category of based topological spaces. Let $\mathbf{E_{n}}\colon \text{Top}_{*}\to \text{Top}_{*}$ be the functor defined by sending $(X, *)$ to
$$\mathbf{E_{n}}(X)\colonequals\bigsqcup_{\partic=0}^{\infty}E_{n}(\partic) \times_{\Sigma_{\partic}} X^{\partic}/\sim,$$ where for $(x_{1},\ldots,x_{\partic})\in X^{\partic}$ and $c\in E_{n}(\partic)$,
$$(\theta(c, (\text{id}_{D^{n}}^{i},*,\text{id}_{D^{n}}^{\partic-i-1})), (x_{1},\ldots,x_{\partic}))\sim (c, (x_{1},\ldots,x_{i},*,x_{i+1},\ldots, x_{\partic-1})).$$
\end{definition}
%Note that when $X=S^{0}$, 
The space $\mathbf{E_{n}}(X)$ can be thought of as a model for the following generalization of $\text{Conf}(M)$ when $M$ is $D^{n}$.
%configuration space of points in $D^{n}$ with labels in $X$.
\begin{definition}\label{conf with labels}
Let $M$ be a space and let $(X, x')$ be a based space.
%with basepoint $x_{0}$. 
The \textbf{configuration space of M with labels in X}, denoted by $\text{Conf}(M; X)$, is 
%the space which as a set consists of finite formal sums $\sum m_{a}x_{a}$ with $m_{a}\in M$ distinct and $x_{a}\in X$ subject to the relation
%$$\sum m_{a}x_{a}\sim \sum m_{a}x_{a} + m_{0}x'.$$
%It is topologized as 
the quotient space
$$\bigg(\coprod_{\partic\geq 0} \big(\{(m_1 ,\ldots , m_{\partic })\in M^{\partic} :  m_i \neq m_j \text{ if } i\neq j \} \times X^{\partic}\big)\bigg)/\sim,$$
where $[(m_{1},\ldots,m_{\partic}), (x_{1},\ldots, x_{\partic})]$ is identified $[(m_{\sigma(1)},\ldots,m_{\sigma(\partic)}), (x_{\sigma(1)},\ldots, x_{\sigma(\partic)})]$ for $\sigma\in \Sigma_{\partic}$ and it is also identified with $ [(m_{1},\ldots, \hat{m}_{i}, \ldots, m_{\partic}), (x_{1},\ldots, \hat{x}_{i}, \ldots, x_{\partic})]$ if $x_{i}=x'$.
\end{definition}
Roughly speaking, a point $m_{i}x_{i}$ vanishes if $x_{i}$ converges to the basepoint $x'$ of $X$. A labelled configuration space can be thought of as a generalization of $\text{Conf}(M)$ since $\text{Conf}(M; S^{0})=\text{Conf}(M)$.
The map $E_{n}(\partic)/\Sigma_{\partic}\to \text{Conf}_{\partic}(D^{n})$ that records where an embedding $f\colon \sqcup_{i=0}^{\partic}D^{n}_{i}\to D^{n}$ sends the center of each disk $D^{n}_{i}$ induces a homotopy equivalence $(E_{n}(\partic) \times_{\Sigma_{\partic}} X^{\partic}/\sim) \to \text{Conf}_{\partic}(D^{n}; X)$ and so $\mathbf{E}_{\mathbf{n}}(S^{0})$ is homotopy equivalent to $\text{Conf}(D^{n})$.
%$\mathbf{E}_{\mathbf{n}}(S^{0})\to \text{Conf}(D^{n})$.

\textcite[Theorem 3.1][p.~225, 227]{MR0436146} fully calculate $H_{*}(\mathbf{E_{n}}(X);\fieldc)$ for $\fieldc=\mathbb{F}_{p}, \mathbb{Q}$, which we now summarize.
\begin{definition}
Let $V$ be a vector space.
%over $\mathbb{F}_{p}$. 
An element $v\in V$ is a \textbf{Lie word of weight 1}.
%(Cohen calls this a $\lambda_{n}$-product-see\footnote{cite page 222}). 
Assume that Lie words of weight $i$ have been defined for $i<j$. Then a \textbf{Lie word of weight j} is any $\browd{a}{b}$, where $a$ and $b$ are Lie words such that $\text{weight}(a)+\text{weight}(b)=j$.
%An element $v\in V$ is a \textbf{basic Lie word of weight 1}. Assume that basic Lie words of weight $j$ have been defined and totally ordered amongst themselves for $j<k$. Then a \textbf{basic Lie word of weight k} is any $\browd{a}{b}$ such that
%\begin{enumerate}
%    \item $\browd{a}{b}$ is of weight $k$, and either
%    \item $a< b$, where $a$ and $b$ are basic Lie words, and if $b=\browd{c}{d}$, for $c$ and $d$ basic Lie words, then $c\leq a$    or
%    \item $a=b$ if $p>2$ where $a$ is a basic Lie word
%\end{enumerate}
%$\browd{a}{b}$ is of weight $k$ and $a< b$, where $a$ and $b$ are basic Lie words, an
\end{definition}

Instead of giving a full description of $H_{*}(\mathbf{E_{n}}(X);\fieldc)$ for $\fieldc=\mathbb{F}_{p},\mathbb{Q}$, the following summary will suffice for our purposes. We will give a more precise description of $H_{*}(\mathbf{E_{n}}(S^{0});\fieldc)=H_{*}(\text{Conf}(D^{n});\fieldc)$  when needed--see \cref{thm-Co} and \cref{description of homology of Conf of disk in dim greater than 2}.
%\footnote{cite descriptions of $H_{*}(\text{Conf}(D^{n});\mathbb{F}_{p})$}).
%(we will give more specific 
%For our purposes, we 
%We will use the following description of 
%Instead of giving the full description
%[{\textcite[Theorem 3.1][p.~225--227]{MR0436146}}]
\begin{theorem}[{\textcite[Theorem 3.1][p.~225--227]{MR0436146}}, {\textcite[Theorem 16.4]{galatius2018cellular}}]\label{homology of free En algebra}
Let $(X,*)$ be a based space and suppose that $\fieldc$ is either $\mathbb{F}_{p}$ or $\mathbb{Q}$.
%\footnote{cite Cohen in theorem}
Let $L$ denote the set of Lie words of $\tilde{H}_{*}(X; \fieldc)$.
%Let $L$ be the set of formal symbols constructed by iterated formal applications of the Browder bracket $\browd{-}{-}$ on $\tilde{H}_{*}(X;\mathbb{F}_{p})$ (e.g. if $\sigma$ and $\tau$ are in $\tilde{H}_{*}(X;\mathbb{F}_{p})$, then $\browd{\sigma}{\browd{\sigma}{\tau}}$ is in $L$. We also allow no applications of the Browder bracket to elements of $\tilde{H}_{*}(X;\mathbb{F}_{p})$.
\begin{enumerate}
    \item Suppose that $\fieldc=\mathbb{F}_{p}$. Let $S$ be the set of formal symbols constructed by iterated formal applications of $Q^{s}, \beta Q^{s}, \xi, \beta\xi,$ and $\zeta$ to elements of $L$. We also allow zero applications of these operations to $L$. There is a subset $G_{n}\subset S$ such that $H_{*}(\mathbf{E_{n}}(X);\mathbb{F}_{p})$ is isomorphic as a ring to the free graded commutative algebra on $G_{n}$.
    \item Suppose that $\fieldc=\mathbb{Q}$. Then $H_{*}(\mathbf{E_{n}}(X);\mathbb{Q})$ is isomorphic to free commutative algebra on $L$ modulo the relations in Properties \ref{Browder brack is linear}-\ref{Browder bracket with unit} of the Browder bracket.
    %\footnote{Maybe don't need to be this precise}
\end{enumerate}

%\footnote{figure out if I should modify or get rid of \cref{description of homology of Conf of disk in dim greater than 2}}
%Each element of $G_{n}$ corresponds to an element of $H_{i}(E_{n}(\partic) \times_{\Sigma_{\partic}} X^{\partic}/\sim ;\mathbb{F}_{p})$ where $i$ and $\partic$ are calculated via the formulas involving the homology operations.
\end{theorem}
We will need the following lemma concerning the set $L$ in \cref{homology of free En algebra} when $X=S^{0}$--see \cref{different models of conf(cyl) are equivalent}.
%\footnote{cite}
\begin{lemma}\label{Lie word lemma}
Suppose that $X$ is $S^{0}$, $\fieldc$ is either $\mathbb{F}_{p}$ or $\mathbb{Q}$, and $e\in\tilde{H}_{0}(S^{0};\fieldc)$ is the class of a point. Then we can take the set $L$ in \cref{homology of free En algebra} to be $\{e, \browd{e}{e}\}$.
\end{lemma}
\begin{proof}
By \cref{homology of free En algebra}, elements in $L$ are homology classes in $H_{*}(\mathbf{E_{n}}(S^{0});\fieldc)$. It suffices to show that if $w\in L$ is a Lie word of weight $j>2$, then $w=0$ in $H_{*}(\mathbf{E_{n}}(S^{0});\fieldc)$. In $H_{*}(\mathbf{E_{n}}(S^{0});\fieldc)$, the only Lie word of weight $3$ is $\browd{e}{\browd{e}{e}}$ (since the Browder bracket is symmetric up to a choice of sign by Property \ref{Browder bracket symmetric up to sign}) and this class is zero (again by Property \ref{Browder bracket symmetric up to sign}). Up to sign, the only Lie words of weight $4$ are of the form $\browd{e}{\browd{e}{\browd{e}{e}}}$ and $\browd{\browd{e}{e}}{\browd{e}{e}}$. We have that $\browd{e}{\browd{e}{\browd{e}{e}}}$ is zero in $H_{*}(\mathbf{E_{n}}(S^{0});\fieldc)$, since $\browd{e}{\browd{e}{e}}=0$. By the Jacobi identity of the Browder bracket (Property \ref{Jacobi identity of Browder bracket}), 
$$0=\browd{\browd{e}{e}}{\browd{e}{e}}+\browd{e}{\browd{e}{\browd{e}{e}}}+ \browd{e}{\browd{\browd{e}{e}}{e}}.$$
Since $\browd{e}{\browd{\browd{e}{e}}{e}}=0$ in $H_{*}(\mathbf{E_{n}}(S^{0});\fieldc)$, we have that $\browd{\browd{e}{e}}{\browd{e}{e}}$ is zero as well. Suppose by induction that for $j\geq 5$ and for any Lie word $v$ with $3\leq \text{weight}(v)\leq j$,
%$i=4,\ldots, j$,
%if $v\in L$ is a Lie word of weight $k>4$, 
we have that $v=0$ in $H_{*}(\mathbf{E_{n}}(S^{0});\fieldc)$. If $w\in L$ is a Lie word of weight $j+1$, then by definition $w=\browd{a}{b}$, where $\text{weight}(a)+\text{weight}(a)=j+1$. Since $j\geq 5$, we have that either $3\leq \text{weight}(a)\leq j$ or $3\leq \text{weight}(b)\leq j$. Since the set of all Lie words of weight $3,\ldots, j$ are zero in $H_{*}(\mathbf{E_{n}}(S^{0});\fieldc)$, either $a$ or $b$ is zero in $H_{*}(\mathbf{E_{n}}(S^{0});\fieldc)$, and so $w$ is as well.
%that $w$ is a Lie word of weight $k>4$. Then by definition, $w=\browd{a}{b}$, where 
\end{proof}
%ggggggg

%From this result,\footnote{maybe make a new subsection} Cohen--Lada--May obtained the following explicit computation of $H_{*}(\text{Conf}(\mathbb{R}^{2});\mathbb{F}_{p})$. We omit explicitly describing $H_{*}(\text{Conf}(\mathbb{R}^{n});\mathbb{F}_{p})$ for $n>2$ because we will only need some properties of $H_{*}(\text{Conf}(\mathbb{R}^{n});\mathbb{F}_{2})$ later on in \cref{sec stab in dim n greater than 2} (see \cref{description of homology of Conf of disk in dim greater than 2}).

From \cref{homology of free En algebra}, Cohen--Lada--May calculated $H_{*}(\text{Conf}(\mathbb{R}^{n});\mathbb{F}_{p})$.
%When $n=2$, the homology $\text{Conf}(\mathbb{R}^{n};\mathbb{F}_{p})$ has the following description (
When $n>2$, $H_{*}(\text{Conf}(\mathbb{R}^{n});\mathbb{F}_{p})$ has a somewhat complicated description which we omit because we will only need some properties of $H_{*}(\text{Conf}(\mathbb{R}^{n});\mathbb{F}_{2})$ later on in \cref{sec stab in dim n greater than 2} (see \cref{description of homology of Conf of disk in dim greater than 2}). We have the following explicit computation of $H_{*}(\text{Conf}(\mathbb{R}^{2});\mathbb{F}_{p})$.
%By contrast, the homology $H_{*}(\text{Conf}(\mathbb{R}^{2});\mathbb{F}_{p})$ has the following explicit description.
%but we will only need 
%obtained the following description for the homology of $\text{Conf}(\mathbb{R}^{2})$ with $\mathbb{F}_p$ coefficients.
%maybe cite Callegaro-Salvetti
%Let $Br\colonequals  \bigsqcup_{n\geq 0}\ \text{Br}_n$, where $\text{Br}_n $ denotes the braid group on $n$ strands.
\begin{theorem}[{\textcite[p.~347--348]{MR0436146}}]\label{thm-Co}Let $e\in H_{0}(\text{Conf}_{1}(\mathbb{R}^{2}))$ denote the class of a point.
\hfill\begin{enumerate}
    \item 
    %Let $x_{0}\in H_{0}(\text{Conf}_{1}(\mathbb{R}^{2};\mathbb{F}_{2})$ denote the homology class of a point and
    For $j>0$, let $(a,b)=(2^{j}-1,2^{j})$ and let $x_{j}$ denote the class $\xi\cdots \xi e\in H_{a}(\text{Conf}_{b}(\mathbb{R}^{2});\mathbb{F}_{2})$, where $\xi$ is applied $j$ times. We have that $$H_{*}(\text{Conf}(\mathbb{R}^{2});\mathbb{F}_{2})\cong \mathbb{F}_{2}\left [e, x_{1}, \ldots \right ].$$
    \item Let $p$ be an odd prime. 
    %Let $y_{0}\in H_{0}(\text{Conf}_{1}(\mathbb{R}^{2});\mathbb{F}_{p})$ denote the class of a point and
    Let $z_{0} \in H_{1}(\text{Conf}_{2}(\mathbb{R}^{2});\mathbb{F}_{p})$ denote the class $\browd{e}{e}$.  For $j>0$, let $(c,d)=(2p^{j}-1,2p^{j})$ and let $y_{j}\colonequals  \beta \xi \cdots \xi z_{0}\in H_{c-1}(\text{Conf}_{d}(\mathbb{R}^{2});\mathbb{F}_{p})$ and let $z_{j}$ denote the class $\xi \cdots \xi z_{0}\in H_{c}(\text{Conf}_{d}(\mathbb{R}^{2});\mathbb{F}_{p})$, where $\xi$ is applied $j$ times.
    %\begin{align*}
     %   y_{\mathlarger{\text{i}}}\colonequals  &\beta \xi_{1} \cdots \xi_{1} z_{\mathlarger{0}}\in H_{\mathlarger{2p^{\mathlarger{i}}-2}}(\text{Conf}_{\mathlarger{2p^{\mathlarger{i}}}}(\mathbb{R}^{2});\mathbb{F}_{p})\\
     %   z_{\mathlarger{\text{i}}}\colonequals  &\xi_{1} \cdots \xi_{1} z_{0}\in H_{\mathlarger{2p^{\mathlarger{i}}-1}}(\text{Conf}_{\mathlarger{2p^{\mathlarger{i}}}}(\mathbb{R}^{2});\mathbb{F}_{p})
    %\end{align*}
    %$y_{\mathlarger{\text{i}}} \in H_{\mathlarger{2p^{\mathlarger{i}}-2}}(\text{Conf}_{\mathlarger{2p^{\mathlarger{i}}}}(\mathbb{R}^{2});\mathbb{F}_{p})$ and $z_{\mathlarger{\text{i}}} \in H_{\mathlarger{2p^{\mathlarger{i}}-1}}(\text{Conf}_{\mathlarger{2p^{\mathlarger{i}}}}(\mathbb{R}^{2});\mathbb{F}_{p})$ denote the classes $\beta \xi_{1} \cdots \xi_{1} \browd{y_{0}}{y_{0}}$ and $\xi_{1} \cdots \xi_{1} \browd{y_{0}}{y_{0}}$ respectively, where $\xi_{1}$ is applied $i$ times. 
    Let $\Lambda_{\mathbb{F}_{p}}\left [z_{0},z_{1},\ldots \right ]$ denote the exterior algebra generated by the classes $z_{j}$. We have that 
    $$H_{*}(\text{Conf}(\mathbb{R}^{2} );\mathbb{F}_p)\cong   \mathbb{F}_{p}\left [e, y_{1}, \ldots \right ]\otimes\Lambda_{\mathbb{F}_p}\left [z_{0}, z_{1}, \ldots \right ].  $$
%    
%    and let $y_{0}\in H_{1}(\text{Conf}_{2}(\mathbb{R}^{2});\mathbb{F}_{p})$ denote the class $ \browd{x_{0}}{x_{0}} $. For $i$ greater than zero, let $x_{\mathlarger{\text{i}}} \in H_{\mathlarger{2p^{\mathlarger{i}}-1}}(\text{Conf}_{\mathlarger{2p^{\mathlarger{i}}}}(\mathbb{R}^{2});\mathbb{F}_{p})$ and $y_{\mathlarger{\text{i}}} \in H_{\mathlarger{2p^{\mathlarger{i}}-2}}(\text{Conf}_{\mathlarger{2p^{\mathlarger{i}}}}(\mathbb{R}^{2});\mathbb{F}_{p})$ denote the classes $ \xi_1 \cdots \xi_{1} x_{0} $ and $\beta \xi_{1} \cdots \xi_{1} x_{0}$, respectively, where $\xi_{1} $ is applied $i$ times for both elements. Let $\Lambda_{\mathbb{F}_p}[x_0 ,x_,\ldots ]$ denotes the exterior algebra over the field $\mathbb{F}_p$ generated by $x_1,x_2,\ldots$
    %The element $x_0$ corresponds to $\browd{y_{0}}{y_{0}}$. For $i>0$ , the elements $x_i$ and $y_i$ correspond to 
\end{enumerate}

\end{theorem}
\begin{comment}
Since $\text{Conf}_{n}(\mathbb{R}^2)$ is a classifying space for the braid group $\text{Br}_{n}$, 
\end{comment}
%If we fix a class $x$ in $ H_{*}(\text{Conf}(\mathbb{R}^{n};\fieldc)$, the restriction $-\bullet x \colon H_{*}(\text{Conf}(\mathbb{R}^{n});\fieldc)\to  H_{*}(\text{Conf}(\mathbb{R}^{n});\fieldc)$ extends to a map of chain complexes $-\bullet x \colon C_{*}(\text{Conf}(\mathbb{R}^{n});\fieldc)\to  C_{*}(\text{Conf}(\mathbb{R}^{n});\fieldc)$.
\begin{definition}\label{multiplication by class map when manifold is Rn}
%Let $\fieldc$ be a ring. 
Given a class $x\in H_{a}(\text{Conf}_{b}(\mathbb{R}^{n});\fieldc)$,
let $\text{deg}(x)\colonequals a$ denote the homological degree of $x$ and let $\text{par}(x)\colonequals b$ denote the number of particles of $x$. We introduce this notation to keep track of how homology operations interact with the homological degree and the number of particles of a homology class. Fix a representative cycle $x'\in C_{a}(\text{Conf}_{b}(\mathbb{R}^{n});\fieldc)$ of $x$ and define
%\footnote{drop star notation-use t instead. Also, check I'm consistent with whether I multiply on the left or the right}
\begin{align*}
    t_{x}\colon C_{i}(\text{Conf}_{\partic}(\mathbb{R}^{n});\fieldc)&\to  C_{i+\text{deg}(x)}(\text{Conf}_{\partic+\text{par}(x)}(\mathbb{R}^{n});\fieldc)\\
    \omega &\mapsto \omega x'
\end{align*}
to be the ``multiplication by $x$'' map.
%on $C_{*}(\text{Conf}(\mathbb{R}^{2});\mathbb{F}_{p})$.
By abuse of notation, let $t_{x}$ also denote the induced map on homology. 
%Note that if $x=y_{m}$, then the map $$y_{m}^{*}:H_{i}(\text{Conf}_{\partic}(\mathbb{R}^{2});\mathbb{F}_{p})\to  H_{i+\text{deg}(y_{m})}(\text{Conf}_{\partic+\text{par}(y_{m})}(\mathbb{R}^{2});\mathbb{F}_{p})$$ is injective.
\end{definition}
Now we fix some notation for an iterated mapping cone to formulate higher-order homological stability.
%To formulate higher-order homological stability, we need to fix some notation for an iterated mapping cone.
\begin{definition}\label{iterated mapping cone def}
%Let $\fieldc$ be a ring and let $M$ be a manifold. 
%Suppose that $(t_{1},\ldots , t_{\srange})$ is an ordered $\srange$-tuple of maps 
Let $t_{i}\colon C_{*}(\text{Conf}_{\partic}(M);\fieldc)\to C_{*}(\text{Conf}_{\partic+\partic_{i}}(M);\fieldc)$ for $i=1,\ldots,\srange$ be chain maps that all commute with each other.
%Suppose that
%the maps $t_{i}$ all commute with each other.
%are homotopy coherent, in the sense that for each $a\neq b$, we have a homotopy $H_{a,b}$ from $t_{a}t_{b}$ to $t_{b}t_{a}$, for each $c\neq a,b$, we have a homotopy $H_{a,b,c}$ from $H_{a,b}t_{c}$ to $t_{c}H_{a,b}$, etc.
%
%for all subsets $J=\{j_{1}<\cdots< j_{l}\}\subset \{1\ldots< \srange\}$ and for all $j\in J$, we have a homotopy $H^{\vert J\vert}_{j, J\setminus\{j\}}\colon C_{*}(\text{Conf}(M);\fieldc)\to C_{*+\vert J\vert}(\text{Conf}(M);\fieldc)$ from $t_{j}H^{\vert J\vert}_{j, J\setminus\{j\}}$
%
%
%for all $1\leq i< j\leq\srange$, there is a chain homotopy $H^{1}_{ij}$ from $t_{i}t_{j}$ to $t_{j}t_{i}$. In addition, suppose that each $H^{1}_{ij}$ can be chosen so that there is a chain homotopy $H^{2}_{(ij),k}$ from $t_{k}H^{1}_{ij}$ to $H^{1}_{ij}t_{k}$, and that we can iterate this procedure on all higher homotopies involving compositions of the $t_{i}$'s.
%the maps all commute with each other up to chain homotopy all commute with each other up to chain homotopy. %for $i=1,2,\ldots,\srange $.% and let $\srange$ be a nonnegative integer. 
Recursively define the \textbf{\srange-th iterated mapping cone} $\pariteratedmap{M}{t_{i}}{\srange}{\partic}$ associated to the $\srange$-tuple $(t_{1},\ldots , t_{\srange})$ as follows:
%as follows: set $\iteratedmap{M}{t_{i}}{0}\colonequals C_{*}(\text{Conf}(M);\fieldc)$ and define $\iteratedmap{M}{t_{i}}{\srange}$ to be the mapping cone $$\iteratedmap{M}{t_{i}}{\srange}\colonequals\text{Cone }\big (t_{\srange}\colon\iteratedmap{M}{t_{1},\ldots ,t_{\srange-1}}{\srange-1}\to \iteratedmap{M}{t_{1},\ldots ,t_{\srange-1}}{\srange-1}\big ).$$ Define $\pariteratedmap{M}{t_{i}}{\srange}{\partic}$ recursively 
%set 
%$\text{Cone}_{\partic}^{0}(M)$
%$\pariteratedmap{M}{t_{i}}{0}{\partic}$ to be $C_{*}(\text{Conf}_{\partic}(M);\fieldc)$, 
set $\pariteratedmap{M}{t_{i}}{1}{\partic}$ to be the mapping cone $$\pariteratedmap{M}{t_{i}}{1}{\partic}\colonequals\text{Cone}(t_{1}\colon C_{*}(\text{Conf}_{\partic-k_{1}}(M);\fieldc)\to C_{*}(\text{Conf}_{\partic}(M);\fieldc)),$$ and for $\srange>1$, set $\pariteratedmap{M}{t_{i}}{\srange}{\partic}$  to be the induced mapping cone
$$\text{Cone }\big (t_{\srange}\colon\pariteratedmap{M}{t_{1},\ldots ,t_{\srange-1}}{\srange-1}{\partic-\partic_{\srange}}\to \pariteratedmap{M}{t_{1},\ldots ,t_{\srange-1}}{\srange-1}{\partic}\big ).$$
%(if $\srange=1$, we set $\pariteratedmap{M}{t_{1}}{1}{\partic}$ to be the mapping cone of $t_{1}$).  
Let $\iteratedmap{M}{t_{i}}{\srange}$ denote $\bigsqcup_{\partic=0}^{\infty}\pariteratedmap{M}{t_{i}}{\srange}{\partic}$.
\end{definition}
Suppose that $M$ is an open connected manifold of dimension $n$ and that we have an embedding 
%$\embed\colon M\bigsqcup D^{n}\to M$ as in \cref{stabilization by embedding}.
$\embed\colon M\bigsqcup_{j=1}^{\srange} \mathbb{R}^{n}\to M$ as in \cref{stabilization by embedding}.
%such that the restriction $\embed\vert_{M}\colon M\to M$ is isotopic to the identity (such an embedding exists when $M$ is an open connected manifold--see \cref{stabilization by embedding}).
%(an open manifold is an example of a manifold admitting such an embedding).
%Let $\fieldc$ be a ring.
%ring. 
Given $\omega_{1},\ldots,\omega_{\srange}\in H_{*}(\text{Conf}(\mathbb{R}^{n});\fieldc)$, we describe how to construct stabilization maps $t_{\omega_{j}}\colon C_{*}(\text{Conf}(M);\fieldc)\to C_{*}(\text{Conf}(M);\fieldc)$ for $j=1,\ldots,\srange$ that
commute with each other.
%are homotopy coherent.
%commute with each other up to chain homotopy.
%(i.e. $t_{\omega_{i}}t_{\omega_{j}}=t_{\omega_{j}}t_{\omega_{i}}$ for all $i\neq j$).
%the iterated mapping cone $\iteratedmap{\mathbb{R}^{n}}{t_{\omega_{i}}}{\srange}$ associated to the collection $(t_{\omega_{1}},\ldots , t_{\omega_{\srange}})$. 
For simplicity, we describe the construction when $\srange=2$ (the construction of the mapping cone is similar when $\srange>2$).
%\footnote{need to delete stuff}

Fix an embedding $\embed\colon M\bigsqcup_{j=1}^{2} \mathbb{R}^{n}_{j}\to M$ as in \cref{stabilization by embedding}. Let $\embed_{j}\colon M\sqcup \mathbb{R}^{n}_{j}\to M$ denote the restriction of $\embed$ to $M\sqcup D^{n}_{j}$ and let $\text{Conf}(\embed_{j})\colon  \text{Conf}(M)\times  \text{Conf}(\mathbb{R}^{n}_{j})\to  \text{Conf}(M)$ denote the induced map on configuration spaces by applying $\embed_{j}$ pointwise. The construction of $\embed$ (by a generalization of the proof of \textcite[Lemma 2.4]{MR3344444}) comes from picking two disjoint rays $\rho_{j}\cong [0,\infty)\subset M$ for $j=1,2$, taking a tubular neighborhood $N_{j}$ of $\rho_{j}$ such that $N_{1}$ and $N_{2}$ are disjoint, and inserting $\mathbb{R}^{n}_{j}$ inside $N_{j}$. Outside of $N_{1}\cup N_{2}$, the map $\embed$ is the identity. As a result, by picking the tubular neighborhoods $N_{1}$ and $N_{2}$ to be disjoint, we ensure that
the following diagram commutes.

\begin{center}\begin{tikzcd}
M\sqcup \mathbb{R}_{1}^{n}\sqcup \mathbb{R}_{2}^{n}
\arrow[r, "\embed_{1}"]\arrow[d, "\embed_{2}"] &
M\sqcup \mathbb{R}_{2}^{n}
\arrow[d, "\embed_{2}"]\\
M\sqcup \mathbb{R}_{1}^{n}
\arrow[r,"\embed_{1}"]& 
M
\end{tikzcd}\end{center}
Therefore, we have a commutative diagram
\begin{center}\begin{tikzcd}
\text{Conf}(M)\times \text{Conf}(\mathbb{R}_{1}^{n})\times \text{Conf}(\mathbb{R}_{2}^{n})
\arrow[r, "\text{Conf}(\embed_{1})"]\arrow[d, "\text{Conf}(\embed_{2})"] &
\text{Conf}(M)\times \text{Conf}(\mathbb{R}_{2}^{n})
\arrow[d, "\text{Conf}(\embed_{2})"]\\
\text{Conf}(M)\times \text{Conf}(\mathbb{R}_{1}^{n})
\arrow[r,"\text{Conf}(\embed_{1})"]& 
\text{Conf}(M).
\end{tikzcd}\end{center}

From the map $\text{Conf}(\embed_{j})$ and applying the Eilenberg--Zilber map,
we have an induced map on chain complexes 
$$\text{Conf}(\embed_{j})_{*}\colon C_{*}(\text{Conf}(M);\fieldc)\otimes C_{*}(\text{Conf}(\mathbb{R}^{n}_{j});\fieldc)\to C_{*}(\text{Conf}(M);\fieldc).$$ 
As a result, we have another commutative diagram 
\begin{center}\begin{tikzcd}
C_{*}(\text{Conf}(M);\fieldc)\otimes C_{*}(\text{Conf}(\mathbb{R}_{1}^{n});\fieldc)\otimes C_{*}( \text{Conf}(\mathbb{R}_{2}^{n});\fieldc)
\arrow[r, "\text{Conf}(\embed_{1})_{*}"]\arrow[d, "\text{Conf}(\embed_{2})_{*}"] &
C_{*}(\text{Conf}(M);\fieldc)\otimes C_{*}(\text{Conf}(\mathbb{R}_{2}^{n});\fieldc)
\arrow[d, "\text{Conf}(\embed_{2})_{*}"]\\
C_{*}(\text{Conf}(M);\fieldc)\otimes C_{*}(\text{Conf}(\mathbb{R}_{1}^{n});\fieldc)
\arrow[r,"\text{Conf}(\embed_{1})_{*}"]& 
C_{*}(\text{Conf}(M);\fieldc).
\end{tikzcd}\end{center}
Pick a representative cycle $\omega_{j}'\in C_{*}(\text{Conf}(\mathbb{R}^{n});\fieldc)$ for the class $\omega_{j}\in H_{*}(\text{Conf}(\mathbb{R}^{n});\fieldc)$ and set $t_{\omega_{j}}$ to be the restriction $$t_{\omega_{j}}\colonequals\text{Conf}(\embed_{j})_{*}(-\otimes \omega_{j}')\colon C_{*}(\text{Conf}(M);\fieldc)\to C_{*}(\text{Conf}(M);\fieldc).$$ 
From the previous commutative diagram, we have that
$t_{\omega_{1}}t_{\omega_{2}}=t_{\omega_{2}}t_{\omega_{1}}$.
Now we obtain stable ranges for homological and secondary homological stability for $\text{Conf}_{\partic}(\mathbb{R}^{2})$.
\begin{proposition}\label{stable ranges for the plane}
%Let $p$ be a prime number, $\srange$ a positive integer, and $\partic$ a nonnegative integer. 
Let $e\in H_{0}(\text{Conf}_{1}(\mathbb{R}^{2});\mathbb{F}_{p})$ be the class of a point, let $x_{j}\in H_{*}(\text{Conf}(\mathbb{R}^{2});\mathbb{F}_{2})$ and $y_{j}\in H_{*}(\text{Conf}(\mathbb{R}^{2});\mathbb{F}_{p})$ denote the homology classes from \cref{thm-Co}.
%Let $A(p ,\srange )\colonequals p^{\srange}$, $B(p,\srange)\colonequals p^{\srange}-1$ and let $C(p ,\srange )$
%$C(p ,\srange ) \colonequals \sum_{j=0}^{\srange-1}(p^{\srange}-2p^j)$ 
%denote the constant
%\[
% C(p ,\srange ) \colonequals  
%  \begin{cases} 
%   0       & \text{if } p=2\\
%   \sum_{j=0}^{\srange-1}(p^{\srange}-2p^j) & \text{if } p \text{p is odd,} \\
%  \end{cases}
%\]
Let $\srange$ a positive integer and define $D(p, \srange, \partic)$
%$C(p ,\srange ) \colonequals \sum_{j=0}^{\srange-1}(p^{\srange}-2p^j)$ 
to be the constant
\[
 D(p, \srange, \partic) \colonequals  
  \begin{cases} 
   (2^{\srange}-1)\partic/2^{\srange}       & \text{if p is } 2\\
   \displaystyle\frac{(p^{\srange}-1)\partic -\sum_{j=0}^{\srange-1}p^{\srange}-2p^{j}}{p^{m}} & \text{if p is odd.} \\
  \end{cases}
\]
%\footnote{might need to use different notation for omega}
Let $\omega_{0}$ be the class $e$ and for $\srange >1$, let $w_{\srange -1}$ be the class $x_{\srange-1}$ if $p$ is $2$ and let $w_{\srange-1}$ be the class $y_{\srange-1}$ if $p$ is odd. Let $x$ be a homology class in $H_{i+\text{deg}(w_{\srange -1})}(\text{Conf}_{\partic +\text{par}(w_{\srange -1})}(\mathbb{R}^{2});\mathbb{F}_{p})$ and suppose that $i\leq D(p, \srange, k)$.
\hfill\begin{enumerate}
    \item\label{class of a point case ideal} If $\srange$ is $1$, then $x$ is in the ideal $(e)$.
    %\footnote{I feel breaking up the cases into $\srange=1$ and $\srange>1$ is inelegant, but I don't know how to fix this.}
    \item\label{even prime case ideal} If $\srange$ is greater than $1$ and $p$ is $2$, then the class $x$ is in the ideal $I_{\srange -1}\colonequals (e, x_{1},\ldots, x_{\srange -1})$.
    %let $I_{\srange}$ denote the ideal $(x_{0},\cdots ,x_{\srange})$. If $A(\srange ,2)i \leq B(\srange ,2)k-C(\srange ,2)$, then the class $x$ is in the ideal $I_{\srange -1}$.
    \item\label{odd prime case ideal} If $\srange$ is greater than $1$ and $p$ is odd, then the class $x$ is in the ideal $J_{\srange -1}\colonequals (e, y_{1},\ldots, y_{\srange -1})$.
    %let $J_{\srange}$ denote the ideal $(y_{0},\cdots ,y_{\srange})$. If $A(p ,\srange )i \leq B(p ,\srange )k-C(p ,\srange )$, then the class $x$ is in the ideal $J_{\srange -1}$.
\end{enumerate}
%\hfill\begin{enumerate}
%    \item If $p=2$, let $I_{\srange}$ denote the ideal $(x_{0},\cdots ,x_{\srange})$. If $A(\srange ,2)i \leq B(\srange ,2)k-C(\srange ,2)$, then the class $x$ is in the ideal $I_{\srange -1}$.
%    \item If $p$ is odd, let $J_{\srange}$ denote the ideal $(y_{0},\cdots ,y_{\srange})$. If $A(p ,\srange )i \leq B(p ,\srange )k-C(p ,\srange )$, then the class $x$ is in the ideal $J_{\srange -1}$.
%\end{enumerate}
%If $x$ is a homology class in $H_{i}(\text{Conf}_{\partic}(\mathbb{R}^{2});\mathbb{F}_{p})$ and $A(p ,\srange )i < B(p ,\srange )k-C(p ,\srange )$, then $x$ is in the ideal $I_{\srange -1}\colonequals (y_{0},y_1,\cdots ,y_{\srange-1})$. 
%Let $p$ be an odd prime and suppose that $x\in H_{i}(\text{Conf}_{\partic}(\mathbb{R}^{2});\mathbb{F}_{p})$. There are constants $A(p ,\srange )=p^{\srange}$, $B(p ,\srange )=p^{\srange}-1$, and $C(p ,\srange )=\sum_{j=0}^{\srange-1}(p^{\srange}-2p^j)$ such that if $A(p ,\srange )i < B(p ,\srange )k-C(p ,\srange )$, then the homology class
%$x$ is in the ideal $I_{\srange -1}\colonequals (y_{0},y_1,\cdots ,y_{\srange-1})$.
\begin{comment}
the stabilization map $t_{\partic}$ induces an isomorphism $t_{\partic}^*\colonH_{i}(\text{Conf}_{n}(\mathbb{R}^{2});\mathbb{F}_p)\to  H_{p^{k-1}-1+i}(\text{Conf}_{p^{k}+n}(\mathbb{R}^{2});\mathbb{F}_p)$. 
 NOT SURE WHAT THE stabilization map on homology is
Specifically, $A_{\partic}=p^{k}$, $B_{\partic}=p^{k}-1$, and $C_{\partic}=\Sum_{i=0}^{k-1}\big (p^{k}-2(p^i)\big )$ 
\end{comment}
\end{proposition}
\begin{proof}
\begin{comment}
Let $I_{\partic}=(e,y_1,\ldots ,y_{\partic-1})$.
\\ (PROBABLY DROP THIS PARAGRAPH) For notational simplicity, let $y_0$ denote the homology class $e$. The stabilization map $t_{\partic}$ is the map on homology obtained by multiplying by $y_{\partic}$ (this follows from considering the operadic construction of the homology operations and the explicit description of the $H_*(\text{Conf};\mathbb{F}_p)$.
%REWRITE LAST SENTENCE
Note that multiplication by $y_{\partic}$, $y_{\partic}\bullet \colonH_{*}(\text{Br};\mathbb{F}_p)\to  H_{*}(\text{Br};\mathbb{F}_p)$, is an injective map so the stabilization map is injective (injectivity also follows from \ref{thm-McDuff}).

We just need to prove surjectivity in the relevant range. Equivalently, it suffices to show that for $x \notin I_{\partic}$, $x$ is outside this range. We say a nonzero homology class $x\in H_{i}(\text{Conf}_{j}(\mathbb{R}^2);\mathbb{F}_p)$ class is in the \textit{k-th unstable range} if $A_{\partic}\text{deg}(x) \geq B_{\partic}\text{par}(x)-C_{\partic}$.
\\ PROBABLY REMOVE EVERYTHING MENTIONED UP UNTIL NOW
\end{comment}
%Suppose that a monomial $x\in H_{i}(\text{Conf}_{\partic}(\mathbb{R}^{2});\mathbb{F}_{p})$ is 
We first consider Part \ref{odd prime case ideal}. Suppose that a monomial $x\in H_{i}(\text{Conf}_{\partic}(\mathbb{R}^{2});\mathbb{F}_{p})$ with $p>2$, is in the complement of the ideal  $I_{\srange-1}$. Then $x$ is a product of elements in the set $R_{\srange}\colonequals\{ y_{a}, z_{b} : a\geq \srange ,b\geq 0\}$. We want to show that for such an $x$ we have that $i-\text{deg}(y_{\srange-1})>D(p,\srange,\partic-\text{par}(y_{\srange-1}))$. We have that $i-\text{deg}(y_{\srange-1})= i-(2p^{\srange-1}-2)$ and $$D(p,\srange,\partic-\text{par}(y_{\srange-1}))= \displaystyle\frac{(p^{\srange}-1)(\partic-2p^{\srange-1}) -\sum_{j=0}^{\srange-1}p^{\srange}-2p^{j}}{p^{m}}.$$
Note that 
$$-\text{deg}(y_{\srange-1})>- (p^{\srange}-1)\text{par}(y_{\srange-1})/p^{\srange}$$ since $\text{deg}(y_{\srange-1})p^{\srange}=2(p^{\srange-1}-1)p^{\srange}$ while $(p^{\srange}-1)\text{par}(y_{\srange-1})=2(p^{\srange}-1)p^{\srange}$. Therefore, to prove Part \ref{odd prime case ideal}, it suffices to show that for any $x\in H_{i}(\text{Conf}_{\partic}(\mathbb{R}^{2});\mathbb{F}_{p})$ that is a product of elements in $R_{\srange}$, we have that $i\geq D(p,\srange,\partic)$. Similarly, to prove Part \ref{even prime case ideal}, it suffices to show that for any $x\in H_{i}(\text{Conf}_{\partic}(\mathbb{R}^{2});\mathbb{F}_{p})$ that is a product of elements in the set $\{x_{n} :  n\geq\srange\}$, we have that $i\geq D(2,\srange,\partic)$. We say that a class $x\in H_{i}(\text{Conf}_{\partic}(\mathbb{R}^{2});\mathbb{F}_{p})$ is in the \textit{\srange-th unstable range} if
$i\geq D(p,\srange ,\partic )$.

We first deal with the case $p>2$. Let $A(p ,\srange )=p^{m}$, $B(p,\srange)=p^{\srange}-1$ and let $C(p ,\srange )=\sum_{j=0}^{\srange-1}p^{\srange}-2p^{j}$ so that $ D(p,\srange , \partic )= \frac{B(p,\srange)k-C(p ,\srange )}{A(p ,\srange )}$.
For all $n\geq \srange$, let $\alpha_{n}$ denote $y_{n}$ or $z_{n}$. Since $\text{par}(\alpha_n)=2p^n$ and $\text{deg}(\alpha_n)\geq 2p^n-2$, we have that $\frac{\text{deg}(\alpha_n)}{\text{par}(\alpha_n)}\geq \frac{p^{n}-1}{p^{n}}\geq \frac{B(p ,\srange )}{A(p ,\srange )}=\frac{p^{\srange}-1}{p^{\srange}}$ and $\alpha_n$ is in the \srange-th unstable range. Since  $\text{par}(xy)=\text{par}(x)+\text{par}(y)$ (likewise for $\text{deg}$), it follows that for every $x$ that is a product of elements in the set $\{y_{n},z_{n}: n\geq \srange\}\equalscolon S_{\srange}\subset R_{\srange}$, 
%generated by the set $\{y_{n},z_{n}: n\geq \srange\}$, 
we have that $A(p ,\srange )\text{deg}(x)\geq B(p ,\srange )\text{par}(x)$. We now deal with elements $y$ in $R_{\srange}\setminus S_{\srange}$.
%generated by $xz_{0},\ldots ,z_{\srange-1}$. 
Write $y$ as $y=z_{0}^{\epsilon_0}\cdots z_{\srange-1}^{\epsilon_{\srange-1}}z$, with $\epsilon_i\in\{0,1\}$ and $z\in S_{\srange}$.

To show that $y$ is in the \srange-th unstable range, it suffices to prove that $B(p ,\srange )\text{par}(y)-A(p ,\srange )\text{deg}(y)\leq C(p ,\srange )$. Note that
\begin{center}
    $B(p ,\srange )\text{par}(y)-A(p ,\srange )\text{deg}(y)$
    \\$=B(p ,\srange )\big (\text{par}(z)+\sum_{j=0}^{\srange-1}\text{par}(z_{j}^{\epsilon_j})\big )-A(p ,\srange )\big (\text{deg}(z)+\sum_{j=0}^{\srange-1}\text{deg}(z_{j}^{\epsilon_j})\big )$
    \\ $=B(p ,\srange )\big (\text{par}(z)+\sum_{j=0}^{\srange-1}\epsilon_{j}\text{par}(z_{j})\big )-A(p ,\srange )\big (\text{deg}(z)+\sum_{j=0}^{\srange-1}\epsilon_j\text{deg}(z_{j})\big )$.
\end{center}
Since $z$ is in $S_{\srange}$, we have that $B(p ,\srange )\text{par}(z)-A(p ,\srange )\text{deg}(z)\leq 0$. We now check that 
\begin{center}
    $B(p ,\srange )\big (\sum_{j=0}^{\srange-1}\epsilon_{j}\text{par}(z_{j})\big )-A(p ,\srange )\sum_{j=0}^{\srange-1}\epsilon_j\text{deg}(z_{j})\leq C(p ,\srange )$.
\end{center}
Note that
\begin{center}
    $B(p ,\srange )\big (\sum_{j=0}^{\srange-1}\epsilon_{j}\text{par}(z_{j})\big )-A(p ,\srange )\sum_{j=0}^{\srange-1}\epsilon_{j}\text{deg}(z_{j})$
    \\ $=(p^{\srange}-1)(\sum_{j=0}^{\srange-1}\epsilon_{j}2p^j )-p^{\srange}\sum_{j=0}^{\srange-1}\epsilon_{j}(2p^{j}-1 )$
    \\ $=\sum_{j=0}^{\srange-1}\epsilon_{j}\big ( 2p^{j}(p^\srange -1)-(2p^j -1)p^{\srange}\big )$
    \\ $=\sum_{j=0}^{\srange-1}\epsilon_{j}\big (2p^{j}(p^{\srange} - p^{\srange} -1)+p^{\srange}\big )=\sum_{j=0}^{\srange-1}\epsilon_{j}(p^{\srange} -2p^{j})$
\end{center}
Since $p>2$ and $\srange> j$, $p^{\srange} -2p^j > 0$. In addition, the sum $\sum_{j=0}^{\srange-1}\epsilon_{j}(p^{\srange} -2p^{j})$ achieves its maximum precisely when $\epsilon_{j}=1$ for all $j$, and in this case this sum equals $C(p ,\srange )$.

Now we deal with the case $p=2$. Note that $\text{par}(x_{n})=2^{n}$ and $\text{deg}(x_{n})=2^{n}-1$. For all $n\geq \srange$, we have that $\frac{\text{deg}(x_{n})}{\text{par}(x_{n})}\geq \frac{2^{\srange}-1}{2^{\srange}}$, so $x_{n}$ is in the \srange-th unstable range. By the same argument as the case $p>2$, it follows that if $x$ is a product of elements in the set $\{x_{n} :  n\geq\srange\}$, then $x$ is in the \srange-th unstable range.
\begin{comment}
\\ AHHHHHHHHHHHHH
We just need to prove surjectivity in the relevant range. Equivalently, it suffices to show that for $x \notin I_{\partic}$, $x$ is outside this range. We say a nonzero homology class $x\in H_{i}(\text{Conf}_{j}(\mathbb{R}^2);\mathbb{F}_p)$ class is in the \textit{k-th unstable range} if $A_{\partic}\text{deg}(x) \geq B_{\partic}\text{par}(x)-C_{\partic}$. 
%Given two homology classes $y,z\in H_*(Br;\mathbb{F}_p)$, we say that $y|z$ if $z=xy$ for some $x\in H_*(Br;\mathbb{F}_p)$. 
We define the slope $\text{sl}(x)$ of a homology class $x$ to be $\frac{\text{deg}(x)}{\text{par}(x)}$. The following lemma is a  consequence of the constants we selected and the homology of the braid group.
\end{comment}
\end{proof}
%By taking a suitable iterated mapping on $C_{*}(\text{Conf}(M);\mathbb{F}_{p})$, we can show that 
\begin{corollary}\label{prop:conf-stab}
%Let $p$ be a prime number. 
%Let $e\in H_{0}(\text{Conf}_{1}(\mathbb{R}^{2});\mathbb{F}_{p})$ be the class of a point.
%and let $H_{*}(-)$ denote homology with $\mathbb{F}_{p}$ coefficients.
%Let $x_{j}\in H_{*}(\text{Conf}(\mathbb{R}^{2});\mathbb{F}_{2})$ and $y_{j}\in H_{*}(\text{Conf}(\mathbb{R}^{2});\mathbb{F}_{p})$ denote the homology classes from \cref{thm-Co}. 
Let $e\in H_{0}(\text{Conf}_{1}(\mathbb{R}^{2});\mathbb{F}_{p})$ be the class of a point, let $x_{j}\in H_{*}(\text{Conf}(\mathbb{R}^{2});\mathbb{F}_{2})$ and $y_{j}\in H_{*}(\text{Conf}(\mathbb{R}^{2});\mathbb{F}_{p})$ denote the homology classes from \cref{thm-Co}.
%Suppose that $\partic$ is a nonnegative integer and that $\srange$ is a positive integer. 
Let $D(p,\srange , \partic )$ denote the constant from \cref{stable ranges for the plane}.
%Let $I_{\srange}$ and $J_{\srange}$ denote the ideals defined in \cref{stable ranges for the plane}.
%Let $\omega_{m}\in H_{i}(\text{Conf}_{\partic}(\mathbb{R}^{2});\mathbb{F}_{p})$ denote the homology class $x_{m}$ for $p=2$ and $y_{m}$ for $p$ odd. 
%Consider the map $$\omega_{\srange}^{*}\colonH_{i}(\text{Conf}_{\partic}(\mathbb{R}^{2});\mathbb{F}_{p})\to  H_{i+\text{deg}(\omega_{m})}(\text{Conf}_{\partic +\text{par}(\omega_{m})}(\mathbb{R}^{2});\mathbb{F}_{p})$$ and suppose that $i+\text{deg}(\omega_{\srange})\leq D(p, \srange + 1, \partic +\text{par}\big (\omega_{m})\big )$. 
\hfill\begin{enumerate}
    \item\label{homological stability for plane} The map 
    $$ t_{e}\colon H_{i}(\text{Conf}_{\partic}(\mathbb{R}^{2});\mathbb{F}_{p})\to H_{i}(\text{Conf}_{\partic+1}(\mathbb{R}^{2});\mathbb{F}_{p}) $$ is an isomorphism for
    $i\leq D(p,1,\partic)$.
    %$i< D(p,1,\partic)$ and a surjection for $i=D(p,1,\partic)$.
    \item Suppose that $p$ is $2$. 
    Let $\pariteratedmap{\mathbb{R}^{2}}{t_{x_{j}}}{\srange}{\partic}$ denote the \srange-th iterated mapping cone associated to the collection $(t_{e},  t_{x_{1}}, t_{x_{2}}, \ldots, t_{x_{\srange-1}})$ as defined in \cref{iterated mapping cone def}. The map
    %on homology groups of mapping cones
    $$t_{x_{\srange}}\colon H_{i}(\pariteratedmap{\mathbb{R}^{2}}{t_{x_{j}}}{\srange}{\partic};\mathbb{F}_{2})\to H_{i+2^{\srange}-1}(\pariteratedmap{\mathbb{R}^{2}}{t_{x_{j}}}{\srange}{\partic+2^{\srange}};\mathbb{F}_{2}) $$ is an isomorphism 
    for $i\leq D(2, \srange + 1, \partic )$. 
    %for $i<D(2, \srange + 1, \partic )$ and a surjection for $i= D(2, \srange + 1, \partic )$.
    %Let $I_{\srange-1}(\partic)$ denote the complex $I_{\srange-1}\cap C_{*}(\text{Conf}_{\partic}(\mathbb{R}^{2});\mathbb{F}_{2})$. The map of relative homology groups
%$$x_{\srange}^{*}\colon H_{i}\big (\text{Conf}_{\partic }(\mathbb{R}^{2}),I_{\srange-1}(\partic)\big )\to  H_{i+\text{deg}(x_{\srange})}\big (\text{Conf}_{\partic +\text{par}(x_{\srange})}(\mathbb{R}^{2}),I_{\srange-1} (\partic +\text{par}(x_{\srange}))\big )$$
%is an isomorphism for $i\leq D(2, \srange + 1, \partic )$.
    \item Suppose that $p$ is an odd prime. 
    Let $\pariteratedmap{\mathbb{R}^{2}}{t_{y_{j}}}{\srange}{\partic}$ denote the \srange-th iterated mapping cone associated to the collection $(t_{e},  t_{y_{1}}, t_{y_{2}}, \ldots, t_{y_{\srange-1}})$. The map on homology groups of mapping cones
    $$t_{y_{\srange}}\colon H_{i}(\pariteratedmap{\mathbb{R}^{2}}{t_{y_{j}}}{\srange}{\partic};\mathbb{F}_{p})\to H_{i+2p^{\srange}-2}(\pariteratedmap{\mathbb{R}^{2}}{t_{y_{j}}}{\srange}{\partic+2p^{\srange}};\mathbb{F}_{p}) $$ is an isomorphism 
    for $i\leq D(p, \srange + 1, \partic )$.
    %for $i<D(p, \srange + 1, \partic )$ and a surjection for $i= D(p, \srange + 1, \partic )$.
    %Let $J_{\srange-1}(\partic)$ denote the complex $J_{\srange-1}\cap C_{*}(\text{Conf}_{\partic}(\mathbb{R}^{2});\mathbb{F}_{p})$. The map of relative homology groups
%$$y_{\srange}^{*}\colon H_{i}\big (\text{Conf}_{\partic }(\mathbb{R}^{2}),J_{\srange-1}(\partic)\big )\to  H_{i+\text{deg}(y_{\srange})}\big (\text{Conf}_{\partic +\text{par}(y_{\srange})}(\mathbb{R}^{2}),J_{\srange-1}(\partic +\text{par}(y_{\srange}))\big)$$
%is an isomorphism for $i\leq D(p, \srange + 1, \partic )$.
\end{enumerate}
\end{corollary}
\begin{proof}
We prove this result for the case $p$ is odd and $\srange$ is greater than $1$ since the proofs for the other cases are similar.

To show that the map $$t_{y_{\srange}}\colon C_{i}(\pariteratedmap{\mathbb{R}^{2}}{t_{y_{j}}}{\srange}{\partic};\mathbb{F}_{p})\to C_{i+2p^{\srange}-2}(\pariteratedmap{\mathbb{R}^{2}}{t_{y_{j}}}{\srange}{\partic+2p^{\srange}};\mathbb{F}_{p}) $$ induces an isomorphism on homology in a range, we will show that the map $t_{y_{\srange}}$ induces an injection on homology and that the homology of the $\srange+1$ iterated mapping cone
$$H_{i+2p^{\srange}-2}(\pariteratedmap{\mathbb{R}^{2}}{t_{e},t_{y_{1}},\ldots,t_{y_{\srange}}}{\srange+1}{\partic+2p^{\srange}};\mathbb{F}_{p})$$ vanishes in the same range. Using induction on $\srange$, we first show that the map $$t_{y_{\srange-1}}\colon H_{i}(\pariteratedmap{\mathbb{R}^{2}}{t_{y_{j}}}{\srange}{\partic};\mathbb{F}_{p})\to H_{i+2p^{\srange}-2}(\pariteratedmap{\mathbb{R}^{2}}{t_{y_{j}}}{\srange}{\partic+2p^{\srange}};\mathbb{F}_{p}) $$ is injective and that $H_{*}(\iteratedmap{\mathbb{R}^{2}}{t_{e},t_{y_{1}},\ldots,t_{y_{\srange-1}}}{\srange})$ is isomorphic to $H_{*}(\text{Conf}(\mathbb{R}^{2});\mathbb{F}_{p})/(e, y_{1},\ldots, y_{\srange-1})$.
%We first show that the map $y_{\srange}^{*}$ is injective. 

Suppose that $\srange=1$. The map $t_{y_{0}}\colonequals t_{e}\colon H_{*}(\text{Conf}(\mathbb{R}^{2});\mathbb{F}_{p})\to  H_{*}(\text{Conf}(\mathbb{R}^{2});\mathbb{F}_{p})$ is injective. We have a long exact sequence of homology groups
$$\cdots\rightarrow H_{*-1}\big( \pariteratedmap{R^{2}}{t_{e}}{1}{\partic}\big)
\xrightarrow{\partial_{*}}  H_{*}(\text{Conf}(\mathbb{R}^{2}))\xrightarrow{t_{e}} 
H_{*}(\text{Conf}(\mathbb{R}^{2}))\rightarrow
H_{*}\big(\iteratedmap{\mathbb{R}^{2}}{t_{e}}{1}\big)\rightarrow  \cdots.$$
Since the map $t_{e}\colon H_{*}(\text{Conf}(\mathbb{R}^{2});\mathbb{F}_{p})\to  H_{*}(\text{Conf}(\mathbb{R}^{2});\mathbb{F}_{p})$ is injective, the boundary homomorphism $$\partial_{*}\colon H_{*-1}(\pariteratedmap{R^{2}}{t_{e}}{1}{\partic};\mathbb{F}_{p})\to H_{*}(\text{Conf}(\mathbb{R}^{2});\mathbb{F}_{p})$$ is the zero map and we have a short exact sequence
$$ 0\rightarrow  H_{*}(\text{Conf}(\mathbb{R}^{2});\mathbb{F}_{p})\xrightarrow{t_{e}} 
H_{*}(\text{Conf}(\mathbb{R}^{2});\mathbb{F}_{p})\rightarrow
H_{*}(\iteratedmap{\mathbb{R}^{2}}{t_{e}}{1};\mathbb{F}_{p})\rightarrow  0.$$ As a result, we have that $H_{*}(\iteratedmap{\mathbb{R}^{2}}{t_{e}}{1};\mathbb{F}_{p})$ is isomorphic to $H_{*}(\text{Conf}(\mathbb{R}^{2});\mathbb{F}_{p})/(e)$.

Suppose now that $\srange>1$ and that for $l=1,\ldots,\srange-1$, $H_{*}(\iteratedmap{\mathbb{R}^{2}}{t_{e},t_{y_{1}},\ldots,t_{y_{l-1}}}{l};\mathbb{F}_{p})$ is isomorphic to $H_{*}(\text{Conf}(\mathbb{R}^{2});\mathbb{F}_{p})/(e, y_{1},\ldots, y_{l-1})$. Since $$H_{*}(\text{Conf}(\mathbb{R}^{2});\mathbb{F}_{p})/(e, y_{1},\ldots, y_{\srange-2})\cong \mathbb{F}_{p}\left [y_{\srange-1}, y_{\srange}, \ldots \right ]\otimes\Lambda_{\mathbb{F}_p}\left [z_{0}, z_{1}, \ldots \right ],$$ the map $$t_{y_{\srange-1}}\colon H_{*}(\iteratedmap{\mathbb{R}^{2}}{t_{e},t_{y_{1}},\ldots,t_{y_{\srange-2}}}{\srange-1};\mathbb{F}_{p})\to H_{*}(\iteratedmap{\mathbb{R}^{2}}{t_{e},t_{y_{1}},\ldots,t_{y_{\srange-2}}}{\srange-1};\mathbb{F}_{p})$$ is injective. From the same argument for the case $\srange=1$, we have a short exact sequence
$$ 0\rightarrow  H_{*}(\iteratedmap{\mathbb{R}^{2}}{t_{y_{j}}}{\srange-1};\mathbb{F}_{p})\xrightarrow{t_{y_{\srange-1}}}
H_{*}(\iteratedmap{\mathbb{R}^{2}}{t_{y_{j}}}{\srange-1};\mathbb{F}_{p})\rightarrow
H_{*}(\iteratedmap{\mathbb{R}^{2}}{t_{y_{j}}}{\srange};\mathbb{F}_{p})\rightarrow  0$$ and so $H_{*}(\iteratedmap{\mathbb{R}^{2}}{t_{y_{j}}}{\srange};\mathbb{F}_{p})$ is isomorphic to $H_{*}(\text{Conf}(\mathbb{R}^{2});\mathbb{F}_{p})/(e, y_{1},\ldots, y_{\srange-1})$. 

We now show that $H_{i+2p^{\srange}-2}(\pariteratedmap{\mathbb{R}^{2}}{t_{e},t_{y_{1}},\ldots,t_{y_{\srange}}}{\srange+1}{\partic+2p^{\srange}};\mathbb{F}_{p})$ vanishes for $i\leq D(p, \srange+1, \partic)$. 
%Since $H_{i+2p^{\srange}-2}(\pariteratedmap{\mathbb{R}^{2}}{t_{e},t_{y_{1}},\ldots,t_{y_{\srange}}}{\srange+1}{\partic+2p^{\srange}};\mathbb{F}_{p})$ is isomorphic to $H_{*}(\text{Conf}(\mathbb{R}^{2});\mathbb{F}_{p})/(e, y_{1},\ldots, y_{\srange})$, 
%Showing that $H_{i+2p^{\srange}-2}(\pariteratedmap{\mathbb{R}^{2}}{t_{e},t_{y_{1}},\ldots,t_{y_{\srange}}}{\srange+1}{\partic+2p^{\srange}};\mathbb{F}_{p})$ vanishes for $i\leq D(p, \srange+1, \partic)$ 
This is equivalent to showing that if $x$ is a class in $ H_{i+2p^{\srange}-2}(\text{Conf}_{\partic+2p^{\srange}}(\mathbb{R}^{2});\mathbb{F}_{p})$ with $i\leq D(p, \srange+1, \partic)$, then $x$ is in the ideal $(e, y_{1}, \ldots, y_{\srange})$, which follows immediately from \cref{stable ranges for the plane}.
\end{proof}
\begin{comment}
\begin{remark}
From the previous proposition, we can calculate stable ranges for homological stability and higher-order stability for $H_{*}\big (\text{Conf}(\mathbb{R}^{2})\big )$. For example, 
$\text{Conf}(\mathbb{R}^{2})$ having homological stability means that the multiplication map $y_{0}^{*}\colon  H_{i}\big (\text{Conf}_{n}(\mathbb{R}^{2})\big )\to  H_{i}\big (\text{Conf}_{n+1}(\mathbb{R}^{2})\big )$, coming from the homology class $y_{0}$, is an isomorphism in a range. The statement that the map $y_{0}^{*}$ is a surjection in a range is equivalent to the homology groups $H_{i}\big (\text{Conf}_{n}(\mathbb{R}^{2})\big ) $ being in the ideal $(y_{0})$ in that range (the map $y_{0}^{*}$ is an injection because $y_{0}$ is in the free polynomial algebra $\mathbb{F}_{p}\left [y_{0},y_{1},\ldots \right ]$).\footnote{I don't know if I should say anything more, for example why the proposition allows one to obtain ranges for higher-order stability.}
\end{remark}
\end{comment}
\section{Stability for Configuration Spaces of an Open Manifold}\label{sec:stability for open}
In this section, we prove secondary homological stability for the unordered configuration space of an open connected manifold $M$. Our proof mostly follows from an argument from \textcite[Section 4]{MR3344444}. In \cite[Section 4]{MR3344444}, Kupers--Miller streamline the argument behind the proof of \textcite[Proposition A.1]{MR533892} in order to prove homological stability for the unordered configurations space of an open connected manifold of dimension greater than 2.
%(the argument streamlines the argument behind the proof of \textcite[Proposition A.1]{MR533892}). 
In \cref{partial review of homological stability}, we review the setup from \textcite{MR3344444}. In \cref{second stability}, we prove secondary homological stability when the dimension of $M$ is $2$. In \cref{sec stab in dim n greater than 2}, we prove secondary homological stability when the dimension of $M$ is greater than $2$. 
%In this section, (co)homology will always mean (co)homology with $\mathbb{F}_{p}$ coefficients except when explicitly stated otherwise. 
For simplicity, we always assume that the dimension of $M$ is even (or equal to one) and orientable--these conditions are equivalent to $\text{Conf}(M)$ being orientable. If $M$ is odd dimensional or not orientable, secondary homological stability still holds by an argument at the end of Appendix A of \textcite[p.~71]{MR533892}.
%In this section, we prove secondary homological stability for the unordered configuration space of an open connected surface $M$ for homology with $\mathbb{F}_{p} $ coefficients. In this section, (co)homology will always mean (co)homology with $\mathbb{F}_{p}$ coefficients except when explicitly stated otherwise.
%of characteristic $p$, for $p$ an odd prime. 
%Our proof follows from an argument using compactly supported cohomology in \textcite[Section 4]{MR3344444} that proves homological stability for the unordered configurations space of an open connected manifold $M$ of dimension greater than 2. 
%Their proof works verbatim to prove homological stability in the dimension 2 case.
%Before proving secondary homological stability, we will summarize the setup in \textcite[Section 4]{MR3344444}.
\begin{comment}
Throughout this section, cohomology and homology will be over the field $\mathbb{F}_p$, for a fixed prime $p$, unless explicitly stated otherwise (SHOULD PROBABLY REMOVE THIS SENTENCE). 
\end{comment}
\subsection{Homological Stability for Configuration Spaces of an Open Manifold}\label{partial review of homological stability}
Stability for the configuration space of an open manifold depends crucially on the following lemma.
\begin{lemma}[{\textcite[Lemma 2.4]{MR3344444}}]\label{stabilization by embedding}
Let $M$ be an open connected manifold of dimension $n$. For any $\srange\in\mathbb{N}$,  there exists an embedding $\embed\colon M\bigsqcup(\bigsqcup_{i=1}^{\srange} \mathbb{R}^{n})\to M$ such that $\embed\vert_{M}$ is isotopic to $\text{id}_{M}$.
%\footnote{Can/should I replace $\embed\colon M\sqcup \mathbb{R}^{n}\to  M$ with $\embed\colon M\bigsqcup(\bigsqcup_{j=1}^{\srange}\mathbb{R}^{n})\to M$?} 

Furthermore, the embedding $\embed$ can be chosen so that $M$ has an exhaustion by the interiors of compact manifolds $\bar{M}_{j}$ admitting a finite handle decomposition with a single $0$-handle and $\embed$ restricts to an embedding $\embed_{j}\colon M_{j}\bigsqcup(\bigsqcup_{i=1}^{\srange} \mathbb{R}^{n})\to  M_{j}$ such that $\embed_{j}\vert_{M_{j}}$ is isotopic to $\text{id}_{M_{j}}$.
\end{lemma}
\begin{comment}
For an open manifold $M$ of dimension $n$, \textcite[ Lemma 2.4]{MR3344444} states that there is an embedding $e_1 \colon  M\sqcup \mathbb{R}^{n} \hookrightarrow M$ such that $e_{1}\colon  M\hookrightarrow M$ is isotopic to the identity.
\end{comment}
\begin{comment}
Without loss of generality, we can choose $e_1$ so that the image of $\mathbb{R}^n$ lies near the boundary $\partictial M$.
\\MAYBE REWRITE/JUSTIFY THE PREVIOUS SENTENCE?
\end{comment}
Kupers--Miller only considered the case $\srange=1$, but the general case follows immediately since an embedding $\bigsqcup_{j=1}^{\srange}\mathbb{R}^{n}\to \mathbb{R}^{n}$ exists. An embedding $\embed\colon  \bigsqcup M_{j}\to  M$ and a collection of nonnegative integers $k_{j}$ with $\sum k_{j}=k$ induce a map of configuration spaces $$\text{Conf}(\embed)\colon \prod \text{Conf}_{\partic_{j}}(M_{j})\to  \text{Conf}_{\partic}(M)$$ by applying the embedding $\embed$ to each point of the configuration.
%Therefore, we obtain the following stabilization map $s_{1}$ (this stabilization map is sometimes called ``the map obtained by adding a point from infinity" in the literature).
\begin{definition}\label{stabilization map first def}
Suppose that $M$ is an open manifold. Fix an embedding $\embed\colon  M\sqcup \mathbb{R}^{n}\to  M$ as in \cref{stabilization by embedding} and define $$\text{Conf}(\embed)\colon  \text{Conf}_{\partic}(M)\times  \text{Conf}_{l}(\mathbb{R}^{n})\to  \text{Conf}_{\partic +l}(M)$$ to be the induced map on configuration spaces. Let $$\text{Conf}(\embed)_{*}\colon \nch{\text{Conf}_{\partic}(M)}\otimes \nch{\text{Conf}_{l}(\mathbb{R}^{n})}\to \nch{\text{Conf}_{\partic}(M)}$$ be the composition of the Eilenberg--Zilber map $$EZ\colon \nch{\text{Conf}_{\partic}(M)}\otimes \nch{\text{Conf}_{l}(\mathbb{R}^{n})}\to \nch{\text{Conf}_{\partic}(M)\times  \text{Conf}_{l}(\mathbb{R}^{n})}$$ with the map on chain complexes induced from $\text{Conf}(\embed)$.
%$\nch{\text{Conf}_{\partic}(M)}\otimes \nch{\text{Conf}_{l}(\mathbb{R}^{n})}\to \nch{\text{Conf}_{\partic}(M)\times  \text{Conf}_{l}(\mathbb{R}^{n})}\to \nch{\text{Conf}_{\partic}(M)}$ where the first map is the Eilenberg--Zilber map and the second map is the map on chain complexes induced from $\text{Conf}(\embed)$.

%Let $\fieldc$ be a ring. 
Given a class $z\in H_{a}(\text{Conf}_{b}(\mathbb{R}^{n});\fieldc)$, fix a representative cycle $z'\in C_{a}(\text{Conf}_{b}(\mathbb{R}^{n});\fieldc)$ of $z$ and define 
$$t_{z}\colon C_{*}(\text{Conf}_{\partic}(M);\fieldc)\to  C_{*+a}(\text{Conf}_{\partic +b}(M);\fieldc)$$ to be the map sending $\omega\in C_{*}(\text{Conf}_{\partic}(M);\fieldc)$ to $\text{Conf}(\embed)_{*}(\omega\otimes z')$.
%(we can view the chain $\omega\otimes z_{0}$ as a chain in $C_{*+a}\big (\text{Conf}(M)\times \text{Conf}(\mathbb{R}^{n});\fieldc\big )$ by applying the Eilenberg--Zilber map to it).
%$$z_{*}\colon C_{i}\big (\text{Conf}_{\partic}(M)\big )\to  C_{i}\big (\text{Conf}_{\partic}(M)\big )\otimes C_{a}\big (\text{Conf}_{b}(\mathbb{R}^{n})\big ) $$ be the map sending $\omega\in C_{i}\big (\text{Conf}_{\partic}(M)\big )$ to $\omega \otimes z$. Viewing $C_{i}\big (\text{Conf}_{\partic}(M)\big )\otimes C_{a}\big (\text{Conf}_{b}(\mathbb{R}^{n})\big )$ as a subgroup of $C_{i+a}\big (\text{Conf}_{\partic}(M)\times \text{Conf}_{b}(\mathbb{R}^{n})\big )$ via the K\"unneth theorem, let $$t_{z}\colon C_{i}\big (\text{Conf}_{\partic}(M)\big )\to  C_{i+a}\big (\text{Conf}_{\partic +b}(M)\big )$$ denote the composition $\text{Conf}(\embed)_{*}\circ z_{*}$. 
Let $t_{z}$ also denote the induced maps on (co)homology (it will be clear from context which map we mean).
\end{definition}
%If $M$ is $\mathbb{R}^{n}$, the map $t_{z}\colon C_{*}(\text{Conf}_{\partic}(M))\to  C_{*+a}(\text{Conf}_{\partic +b}(M))$ is the 
%\begin{remark}
%If $M$ is $\mathbb{R}^{n}$ and $z$ is a class in $H_{*}(\text{Conf}(\mathbb{R}^{n}))$, then the map $t_{z}\colon C_{*}(\text{Conf}(\mathbb{R}^{n}))\to C_{*}(\text{Conf}(\mathbb{R}^{n}))$ is homotopic to the map $z^{*}$ in \cref{multiplication by class map when manifold is Rn}.
%\end{remark}
%\begin{definition}
%Fix an embedding $\embed\colon  M\sqcup \mathbb{R}^{n}\to  M$ as in \cref{stabilization by embedding}. Let $t_{1}\colon  \text{Conf}_{\partic}(M)\times\mathbb{R}^{n}\to  C_{\partic+1}(M)$ be the map $\text{Conf}(\embed)$ precomposed with the natural identification of $\mathbb{R}^{n}$ with $\text{Conf}_{1}(\mathbb{R}^{n})$. Let $s_{1}\colon \text{Conf}_{\partic}(M)\to  \text{Conf}_{\partic+1}(M)$ be given by the formula $s_{1}(\{m_{1},\ldots ,m_{\partic} \})=t_{1}(\{m_{1},\ldots ,m_{\partic} \},0)$ where $0\in\mathbb{R}^{n}$ is the origin.\footnote{fix notation}
%\end{definition}
\begin{comment}
Such an embedding yields a stabilization map $s_{1}\colon  \text{Conf}_{\partic}(M)\to  \text{Conf}_{\partic+1}(M)$ obtained by ``adding a point from infinity" (for more information, see  \textcite[Section 2]{MR3344444}). 
\end{comment}
\begin{remark}
Our reason for having that the restriction $\embed\vert_{M}$ is isotopic to $\text{id}_{M}$ is the following:
%, the map $t_{z}$ on homology satisfies the following property:
if $z\in H_{0}(\text{Conf}_{0}(\mathbb{R}^{n});R)$ is the class of the empty configuration, then the map $t_{z}\colon H_{*}(\text{Conf}_{\partic}(M);R)\to H_{*}(\text{Conf}_{\partic}(M);R)$ is the identity map instead of a non-identity automorphism.
\end{remark}
\begin{definition}\label{def of primary stabilization map}
%Let $\fieldc$ be a ring. 
If $e\in H_{0}(\text{Conf}_{1}(\mathbb{R}^{n});\fieldc)$ is the class of a point, the map $$t_{1}\colonequals t_{e}\colon C_{*}(\text{Conf}_{\partic}(M);\fieldc)\to  C_{*}(\text{Conf}_{\partic +1}(M);\fieldc)$$ is called the \textit{stabilization map}.
\end{definition}
%\begin{remark}
%In the literature, the stabilization map
%$$H_{*}(\text{Conf}_{\partic}(M))\to  H_{*} (\text{Conf}_{\partic +1}(M))$$ is defined to be induced from the following map of spaces
%\begin{align*}
%    t_{1}'\colon \text{Conf}_{\partic}(M)&\to  \text{Conf}_{\partic +1}(M)\\
%    \{m_{1},\ldots , m_{\partic}\}&\mapsto \text{Conf}(\embed)(\{m_{1},\ldots , m_{\partic}\},0)
%\end{align*}
%with $0\in \mathbb{R}^{n}$ the origin. 
%%From the construction of the multiplication map for $ H_{*}(\text{Conf}_{\partic}(M))$, 
%The maps $(t_{1}')_{*}\colon C_{*}(\text{Conf}_{\partic}(M))\to C_{*}(\text{Conf}_{\partic+1}(M))$ and $t_{1}$ induce the same maps on homology. We choose to define the stabilization map as a map of chain complexes instead of spaces because we need to work with maps of chain complexes to construct stabilization maps for secondary homological stability and for the configuration spaces of a closed manifold.
%\end{remark}
McDuff showed that the induced map on integral homology $t_{1} \colon H_{i}(\text{Conf}_{\partic}(M);\mathbb{Z})\to  H_{i} (\text{Conf}_{\partic +1}(M);\mathbb{Z})$ is an isomorphism for all $i\ll \partic$ if $M$ is a smooth manifold of finite type (\textcite[p.~104]{MR358766}). Isomorphisms of the form $H_{i}(\text{Conf}_{\partic}(M);\fieldc)\cong H_{i}(\text{Conf}_{\partic +1}(M);\fieldc)$ for all $i\ll \partic$ are called \textit{homological stability}. To prove homological stability for $\text{Conf}_{\partic}(M)$, it helps to work with compactly supported cohomology. When $X$ is a locally compact Hausdorff space, the reduced cohomology $\tilde{H}^{*}(X^{+}, X^{+}\setminus X; \fieldc)$ of the point-compactification $X^{+}$ of $X$ is isomorphic to $H^{*}_{c}(X; \fieldc)$.  One advantage of working with compactly supported cohomology is the following well-known long exact sequence.
%[\textcite{MR1481706}, Theorem \MakeUppercase{\romannumeral 3}.1.1, \footnote{I'm not sure if I should include the reference to Bredon. The reference from Bredon basically states that sheaf cohomology and singular cohomology with compact support agree for locally compact spaces. I assume this result is relatively well-known} \textcite{MR842190},\MakeUppercase{\romannumeral 3}.7.6]
%\textcite{MR1481706}, Theorem \MakeUppercase{\romannumeral 3}.1.1
\begin{proposition}[e.g. \textcite{MR842190}, \MakeUppercase{\romannumeral 3}.7.6]\label{long exact sequence in csc}
Let $R$ be an abelian group, W be a locally compact and locally path-connected Hausdorff space, and $A\subset W$ a closed subspace that is also locally path-connected. Let $U=W\setminus A$ denote its complement. There is a long exact sequence in compactly supported cohomology
\begin{center}
    $\cdots \to  H^{*}_{c}(U;\fieldc)
    \to  H^{*}_{c}(W;\fieldc)
    \to  H^{*}_{c}(A;\fieldc)
    \to  H^{*+1}_{c}(U;\fieldc)\to \cdots$
\end{center}
The group $R$ can also be replaced by a twisted coefficient system on $W$.
\end{proposition}
 %, where $X^{+}$ denotes the one-point compactification  
%its compactly supported cohomology agrees with the reduced cohomology of its one-point compactification $\tilde{H}^{*}_{c}(X^{+}, X^{+}\setminus X; \fieldc)$
The maps $ H^{*}_{c}(U;\fieldc) \to  H^{*}_{c}(W;\fieldc)$ and $H^{*}_{c}(W;\fieldc)\to  H^{*}_{c}(A;\fieldc)$ in \cref{long exact sequence in csc} come from the extension by zero map $ C^{*}_{c}(U;\fieldc) \to  C^{*}_{c}(W;\fieldc)$ and restriction map $C^{*}_{c}(W;\fieldc)\to  C^{*}_{c}(A;\fieldc)$ respectively.
%We briefly describe when we have a map of long exact sequences in compactly supported cohomology because we will need this later on in the proof of \cref{filtered case}.
%Suppose that  $A\subset W$ satis 
For $i=1,2$, let $W_{i}$ and $A_{i}$ satisfy the hypothesis of $W$ and $A$, respectively, in \cref{long exact sequence in csc} and let $U_{i}=W_{i}\setminus A_{i}$.
%be a locally compact and locally path-connected Hausdorff space, $A_{i}\subset W_{i}$ a closed subspace that is also locally path-connected, and $U_{i}=W_{i}\setminus A_{i}$. 
From the functoriality of extension by zero and restriction, given a map $C^{*}_{c}(W_{1};\fieldc)\to C^{*}_{c}(W_{2};\fieldc)$ that restricts to maps $C^{*}_{c}(A_{1};\fieldc)\to C^{*}_{c}(A_{2};\fieldc)$ and $C^{*}_{c}(U_{1};\fieldc)\to C^{*}_{c}(U_{2};\fieldc)$, we have an induced map of long exact sequences in compactly supported cohomology
\begin{center}
\begin{tikzcd}
\cdots\arrow[r]& H^{*}_{c}(U_{1};\fieldc)\arrow[r]\arrow[d]& H^{*}_{c}(W_{1};\fieldc)\arrow[r]\arrow[d]& H^{*}_{c}(A_{1};\fieldc)\arrow[r]\arrow[d]& H^{*+1}_{c}(U_{1};\fieldc)\arrow[r]\arrow[d]& \cdots\\
\cdots\arrow[r]& H^{*}_{c}(U_{2};\fieldc)\arrow[r]& H^{*}_{c}(W_{2};\fieldc)\arrow[r]& H^{*}_{c}(A_{2};\fieldc)\arrow[r]& H^{*+1}_{c}(U_{2};\fieldc)\arrow[r]& \cdots .
\end{tikzcd}
\end{center}
%The stabilization map $t_{1}$ does not induce a map on compactly supported cohomology. 

We now construct stabilization maps on $C_{c}^{*}(\text{Conf}(M);\fieldc)$.
%\footnote{maybe need to replace $\mathbb{R}^{n}$ with $D^{n}$} 
Given $z\in H^{*}_{c}(\text{Conf}_{l}(D^{n});\fieldc)$, fix a representative compactly suppported cocycle $z'\in C^{*}_{c}(\text{Conf}_{l}(D^{n});\fieldc)$ to get a restriction
$K(-\otimes z')$ of the K\"unneth map $$K\colon C^{*}_{c}(\text{Conf}_{\partic}(M);\fieldc)\otimes C^{*}_{c}(\text{Conf}_{l}(D^{n});\fieldc)\to C^{*}_{c}(\text{Conf}_{\partic}(M)\times \text{Conf}_{l}(D^{n});\fieldc).$$
The map 
$\text{Conf}(\embed)\colon  \text{Conf}_{\partic}(M) \times \text{Conf}_{l}(D^{n}) \to  \text{Conf}_{\partic+l}(M)$ is an open embedding for all nonnegative $\partic$ and $l$. Therefore, we have an induced map $$\text{Conf}(\embed)^{*}_{c}\colon  C^{*}_{c}(\text{Conf}_{\partic}(M) \times \text{Conf}_{l}(D^{n});\fieldc) \to  C^{*}_{c}(\text{Conf}_{\partic+l}(M);\fieldc)$$ via extension by zero. Define $$t_{z}\colon C^{*}_{c}(\text{Conf}_{\partic}(M);\fieldc)\to C^{*}_{c}(\text{Conf}_{\partic+l}(M);\fieldc)$$
to be the composition
$\text{Conf}(\embed)^{*}_{c}\circ  K(-\otimes z')$.

Specializing to the case $l=1$, when $z\in H^{n}_{c}(\text{Conf}_{1}(D^{n});\fieldc)\cong H_{0}(\text{Conf}_{1}(D^{n});\fieldc)$ corresponds to the class of a point (under Poincar\'e duality), we get a map $t_{1}\colon H_{c}^{*}(\text{Conf}_{\partic}(M);\fieldc) \to  H_{c}^{*+n}(\text{Conf}_{\partic +1 }(M);\fieldc)$.
%When $z\in H^{n}_{c}(\text{Conf}_{1}(D^{n});\fieldc)$ is the class of a point, we get a map $t_{1}\colon H_{c}^{*}(\text{Conf}_{\partic+1}(M);\fieldc) \to  H_{c}^{*}(\text{Conf}_{\partic }(M);\fieldc)$. 
%Under the Poincar\'e duality $H_{*}(D^{n};\fieldc)\cong H^{n-*}_{c}(D^{n};\fieldc)$, the map $t_{1}$ can be identified with the map $H_{c}^{*}(\text{Conf}_{\partic}(M);\fieldc) \to  H_{c}^{*+n}(\text{Conf}_{\partic +1}(M);\fieldc)$ coming from  multiplication by a generator of $H_{c}^{n}(D^{n};\fieldc)$. 
By our assumption that $\text{Conf}(M)$ is orientable and Poincar\'e duality, the map $t_{1}$ being an isomorphism in a range is equivalent to the stabilization map $t_{1}\colon  H_{*} (\text{Conf}_{\partic}(M);\fieldc)\to  H_{*} (\text{Conf}_{\partic +1}(M);\fieldc)$ being an isomorphism in the dual range. In particular, one can show:
%Therefore, $\text{Conf}(\embed)$ induces
%a map $$\text{Conf}(\embed)_{*}\colon  H_{c}^{*}(\text{Conf}_{\partic}(M) \times \text{Conf}_{l}(\mathbb{R}^{n});\fieldc) \to  H_{c}^{*}(\text{Conf}_{\partic +l}(M);\fieldc)$$ on compactly supported cohomology for all nonnegative $k$ and $l$ via extension by zero. Specializing to the case $l=1$, we obtain a map $t_{1}\colon H_{c}^{*}(\text{Conf}_{\partic}(M) \times\mathbb{R}^{n};\fieldc) \to  H_{c}^{*}(\text{Conf}_{\partic +1}(M);\fieldc)$ on compactly supported cohomology. The map $t_{1}$ can be identified with the map $H_{c}^{*}(\text{Conf}_{\partic}(M);\fieldc) \to  H_{c}^{*+n}(\text{Conf}_{\partic +1}(M);\fieldc)$ coming from  multiplication by a generator of $H_{c}^{n}(\mathbb{R}^{n};\fieldc)$. By our assumption that $\text{Conf}(M)$ is orientable and Poincar\'e duality, the map $t_{1}$ being an isomorphism in a range is equivalent to the stabilization map $t_{1}\colon  H_{*} (\text{Conf}_{\partic}(M);\fieldc)\to  H_{*} (\text{Conf}_{\partic +1}(M);\fieldc)$ being an isomorphism in the dual range. In particular, one can show:
\begin{proposition}\label{homological stability in dimension 2}
Let $M$ be an open connected surface and let $D(p,\srange , \partic )$ denote the constant from \cref{stable ranges for the plane}.
%and let $p$ be a prime number. 
%Let \\$t_{1}\colon C_{*}(\text{Conf}_{\partic}(M);\mathbb{F}_{p})\to  C_{*}(\text{Conf}_{\partic +1}(M);\mathbb{F}_{p})$ denote the map defined in \cref{def of primary stabilization map}. Let $D(p,\srange,\partic)$ denote the constant from \cref{stable ranges for the plane}. 
%\[
%D(k,p)= 
%    \begin{cases}
%      k/2 & \text{if p is 2}\\
%      \frac{k(p-1)-(p-2)}{p} & \text{if p is odd.}\\
%   \end{cases}
%   \]
The stabilization map $$t_{1}\colon  H_{i}(\text{Conf}_{\partic}(M);\mathbb{F}_{p})\to  H_{i}(\text{Conf}_{\partic+1}(M);\mathbb{F}_{p})$$ from \cref{def of primary stabilization map} is an isomorphism for $i < D(p,1,\partic )$ and a surjection for $i=D(p,1,\partic )$.
%=\displaystyle\frac{(p-1)k-(p-2)}{p}
%\hfill\begin{enumerate}
%    \item Suppose $p$ is $2$. The map $s_{1}\colon  H_{i}(\text{Conf}_{\partic}(M);\mathbb{F}_{2})\to  H_{i}(\text{Conf}_{\partic+1}(M);\mathbb{F}_{2})$ is an isomorphism for $2i \leq k$.
%    \item Suppose $p$ is odd. The map $s_{1}\colon  H_{i}(\text{Conf}_{\partic}(M);\mathbb{F}_{p})\to  H_{i}(\text{Conf}_{\partic+1}(M);\mathbb{F}_{p})$ is an isomorphism for $ip \leq k(p-1)-(p-2)$. It follows that the map $s_1 \colon  H_{i}(\text{Conf}_{\partic}(M);\mathbb{Z}\left [1/2\right])\to  H_{i}(\text{Conf}_{\partic+1}(M);\mathbb{Z}\left [1/2\right])$ is an isomorphism for $3i < 2k-1$ and a surjection for $3i = 2k-1.$
%\end{enumerate}
As a consequence, the map $t_{1} \colon  H_{i}(\text{Conf}_{\partic}(M);\mathbb{Z}\left [1/2\right ])\to  H_{i}(\text{Conf}_{\partic+1}(M);\mathbb{Z}\left [1/2 \right])$ is an isomorphism for $3i < 2k-1$ and a surjection for $3i = 2k-1.$ These ranges for homological stability are optimal.
\end{proposition}
\begin{proof}
The proposition follows by applying the argument of \textcite[Sections 3 and 4]{MR3344444} and using Part \ref{homological stability for plane} of \cref{prop:conf-stab}.
\end{proof}
\subsection{Secondary Homological Stability for Configuration Spaces of an Open Surface}\label{second stability}
We now describe secondary homological stability for $H_{*}(\text{Conf}_{\partic}(M);\mathbb{F}_{p})$ when $M$ is an open connected surface. For simplicity, we assume that $p$ is an odd prime number (the case $p$ is $2$ is similar). 

Let $y_{1}\in H_{2p-2}(\text{Conf}_{2p}(\mathbb{R}^{2});\mathbb{F}_{p})$ denote the class defined in \cref{thm-Co} and consider the map $$t_{2}\colonequals t_{\displaystyle y_{1}} \colon C_{i}(\text{Conf}_{\partic}(M);\mathbb{F}_{p})\to  C_{i +2p-2}(\text{Conf}_{\partic +2p}(M);\mathbb{F}_{p})$$ (if $p$ were $2$, we would consider the map $t_{\displaystyle x_{1}}$ corresponding to $x_{1}\in H_{1}(\text{Conf}_{2}(\mathbb{R}^{2});\mathbb{F}_{2})$ defined in \cref{thm-Co}). The map $t_{2}$ does not induce an isomorphism of homology groups. But, we can consider
the relative chain complex
%\begin{align*}
%    C_{*} (\text{Conf}_{\partic}(M),\text{Conf}_{\partic -1}(M);\mathbb{F}_{p})&\colonequals \text{Cone}\big (t_{1}\colon C_{*} (\text{Conf}_{\partic-1}(M))\to C_{*} (\text{Conf}_{\partic}(M))\big)\\
%    &=C_{*}(\iteratedmap{M}{t_{1}}{1};\mathbb{F}_{p}).
%\end{align*}
$$C_{*} (\text{Conf}_{\partic}(M),\text{Conf}_{\partic -1}(M);\mathbb{F}_{p})\colonequals \text{Cone}\big (t_{1}\colon C_{*} (\text{Conf}_{\partic-1}(M))\to C_{*} (\text{Conf}_{\partic}(M))\big)=C_{*}(\pariteratedmap{M}{t_{1}}{1}{\partic};\mathbb{F}_{p}).$$
%Recall that in the construction of the iterated mapping cone $\iteratedmap{D^{n}}{t_{\omega_{1}}}{1}$ following \cref{iterated mapping cone def}, we defined a map $t_{\omega_{2}}\colon \iteratedmap{D^{n}}{t_{\omega_{1}}}{1}\to \iteratedmap{D^{n}}{t_{\omega_{1}}}{1}$. By a similar argument, we can define a
Since the maps $t_{1}$ and $t_{2}$ can be constructed to commute with each other (see the discussion following \cref{iterated mapping cone def}), 
%we constructed stabilization that commute with each other. Therefore
%the maps $t_{1}$ and $t_{2}$ commute, so we have 
%and the
we have a 
map of relative chain complexes 
$$t_{2} \colon C_{i}\big(\text{Conf}_{\partic}(M),\text{Conf}_{\partic -1}(M);\mathbb{F}_{p}\big)\to  C_{i+2p-2}\big(\text{Conf}_{\partic+2p}(M),\text{Conf}_{\partic+2p -1}(M);\mathbb{F}_{p}\big)$$
%(here, we use that the maps $t_{1}$ and $t_{2}$ can be constructed to commute with each other ---see the discussion following \cref{iterated mapping cone def}).
%(the only change to the earlier argument is that the manifolds $M_{i}$ from the earlier argument should be homeomorphic to $M$ instead of $D^{n}$).
%by the same argument for defining a map $t_{\omega_{2}}\colon \iteratedmap{D^{n}}{t_{\omega_{1}}}{1}\to \iteratedmap{D^{n}}{t_{\omega_{1}}}{1}$ in the construction of the iterated mapping cone $\iteratedmap{D^{n}}{t_{\omega_{1}}}{1}$ in\footnote{cref} 
%and the map $t_{2}$ of relative chain complexes
%$$t_{2} \colon C_{i}(\text{Conf}_{\partic}(M),\text{Conf}_{\partic -1}(M))\to  C_{i+2p-2}(\text{Conf}_{\partic +2p}(M),\text{Conf}_{\partic +2p -1}(M)).$$
%where the relative chain complex $C_{*}\big (\text{Conf}_{\partic}(M),\text{Conf}_{\partic -1}(M)\big )$ denotes the quotient chain complex $C_{*}\big (\text{Conf}_{\partic}(M)\big )/t_{1}(C_{*}(\text{Conf}_{\partic -1}(M)))$. 
This map of relative chain complexes is called the \textit{secondary stabilization map}.
Similar to the isomorphism $$t_{y_{1}} \colon H_{i} (\text{Conf}_{\partic}(\mathbb{R}^{2}),\text{Conf}_{\partic -1}(\mathbb{R}^{2});\mathbb{F}_{p})\to  H_{i+2p-2} (\text{Conf}_{\partic +2p}(\mathbb{R}^{2}),\text{Conf}_{\partic +2p -1}(\mathbb{R}^{2});\mathbb{F}_{p})$$ 
%defined 
from \cref{prop:conf-stab}, the secondary stabilization map $t_{2}$ induces an isomorphism of relative homology groups. This phenomenon is called \textit{secondary homological stability}. %\footnote{On the next page, \cref{filt for Conf(M)} messes up the page's line spacing if the definition is longer than one line. I don't know how to fix this.}

To prove secondary homological stability, we pass to compactly supported cohomology. Let $y_{1}^{*}\in H^{2p+2}_{c}(\text{Conf}_{2p}(\mathbb{R}^{2});\mathbb{F}_{p})$ denote the compactly supported cohomology class corresponding to $y_{1}$ under Poincar\'e duality. We have a map $$t_{2}\colonequals t_{y_{1}^{*}}\colon C^{*}_{c}(\text{Conf}_{\partic}(M);\mathbb{F}_{p})\to C^{*+2p+2}_{c}(\text{Conf}_{\partic+2p}(M);\mathbb{F}_{p})$$
and a relative cochain complex
$$C^{*}_{c}(\text{Conf}_{\partic}(M),\text{Conf}_{\partic -1}(M);\mathbb{F}_{p})\colonequals \text{Cone}\big (t_{1}\colon C^{*-n}_{c} (\text{Conf}_{\partic-1}(M);\mathbb{F}_{p})\to C^{*}_{c} (\text{Conf}_{\partic}(M);\mathbb{F}_{p})\big).$$ The maps $t_{1}$ and $t_{2}$ can be constructed so that they commute with each other, so we have a secondary stabilization map for compactly supported cochains
$$t_{2} \colon C^{i}_{c}\big(\text{Conf}_{\partic}(M),\text{Conf}_{\partic -1}(M);\mathbb{F}_{p}\big)\to  C^{i+2p+2}_{c}\big(\text{Conf}_{\partic+2p}(M),\text{Conf}_{\partic+2p -1}(M);\mathbb{F}_{p}\big).$$ By Poincar\'e duality, the following proposition is equivalent to secondary homological stability.
%for $H_{*}(\text{Conf}_{\partic}(M),\text{Conf}_{\partic -1}(M);\mathbb{F}_{p})$.\footnote{write down version for compactly supported cohomology}
\begin{proposition}
\label{compactly supported cohomology version}
%Let $p$ be a prime number and 
%let $H_{*}(-)$ denote homology with $\mathbb{F}_{p}$ coefficients. 
Let $M$ be an open connected surface and let $D(p, \srange, \partic )$ denote the constant from \cref{stable ranges for the plane}.
\hfill\begin{enumerate}
    \item Suppose that $p$ is $2$. The secondary stabilization map induces an isomorphism
$$ t_{2} \colon H^{i}_{c}(\text{Conf}_{\partic}(M),
    \text{Conf}_{\partic -1}(M);\mathbb{F}_{2})\to 
    H^{i+3}_{c}(\text{Conf}_{\partic +2}(M),\text{Conf}_{\partic +1}(M);\mathbb{F}_{2})$$
for $i> 2\partic - D(2, 2, \partic )$ and a surjection for $i= 2\partic - D(2, 2, \partic )$.
    \item Suppose that $p$ is an odd prime and that $M$ is orientable. The secondary stabilization map induces an isomorphism
$$ t_{2} \colon H^{i}_{c}(\text{Conf}_{\partic}(M),
    \text{Conf}_{\partic -1}(M);\mathbb{F}_{p})\to 
    H^{i+2p+2}_{c}(\text{Conf}_{\partic +2p}(M),\text{Conf}_{\partic +2p-1}(M);\mathbb{F}_{p})$$
for $i> 2\partic - D(p, 2, \partic )$ and a surjection for $i= 2\partic - D(p, 2, \partic )$.
%for $p^{2}i>  (p^{2}+1)(k+1)+2p^{2}-2p-2$ and 
%a surjection for $p^{2}i=  (p^{2}+1)(k+1)+2p^{2}-2p-2$.
\end{enumerate}
%The map
%$$ s_{2} \colon H^{i}_{c}(\text{Conf}_{\partic}(M), 
%    \text{Conf}_{\partic -1}(M))\to 
%    H^{i+2p+2}_{c}(\text{Conf}_{\partic+2p}(M),\text{Conf}_{\partic -1+2p}(M))$$
%is an isomorphsism for $p^{2}i>  (p^{2}+1)(k+1)+2p^{2}-2p-2$ and 
%a surjection for $p^{2}i=  (p^{2}+1)(k+1)+2p^{2}-2p-2$.
%and a surjection for $i= 2(2p+k+1)+\frac{2p^2 -2p -2 }{p^2}-\frac{p^2 -1}{p^2}(k+1)$.
\end{proposition}
%By Lemma 2.1 , every non-compact connected manifold has an exhaustion by compact manifolds admitting a finite handle decomposition with a single 0-handle.
By the following lemma and \cref{stabilization by embedding}, it suffices to prove this proposition when $M$ is the interior of a compact surface admitting a finite handle decomposition with a single 0-handle. 
\begin{lemma}[e.g. {\textcite[Lemma 2.1]{MR3344444}}]\label{open manifold has exhaustion by compact manifold with finite handle decomposition}
Every open connected manifold has an exhaustion by compact manifolds admitting a finite handle decomposition with a single 0-handle.
\end{lemma}
As a result, we can prove stability for a general connected open manifold $M$ from a colimit argument involving an exhaustion of $M$ (see \textcite[p.~10]{MR3344444}).
%Applying a colimit argument to stability for a connected open manifold 
The main reasons for assuming $M$ is the interior of such a compact manifold are the following well-known results.
\begin{lemma}[e.g. {\textcite[Lemma 2.2]{MR3344444}}]\label{handle-decomp}
Let
%\footnote{You suggested I combined this lemma with \cref{open manifold has exhaustion by compact manifold with finite handle decomposition}. I didn't do this because I think combining them might be kind of awkward. Also, I use these results at different points in the paper and the reader might be confused about which part of the lemma I would be referring to.} 
$M$ be the interior of an $n$-dimensional manifold admitting a finite handle decomposition with a single 0-handle. Then there is closed subspace $X$ of $M$ homeomorphic to an open subset of a finite complex of dimension at most $n-1$, such that $M \setminus X$ is homeomorphic to $\mathbb{R}^{n}$.
\end{lemma}
\begin{lemma}[e.g. {\textcite[Lemma 4.1]{MR3344444}}]\label{homology of conf is finitely generated}
Let $\fieldc$ be a ring and $M$ the interior of an $n$-dimensional manifold $\bar{M}$ admitting a finite handle decomposition. The homology $H_{i}(\text{Conf}_{\partic}(M);\fieldc)$ is a finitely-generated $\fieldc$-module for all $i\geq 0$.
\end{lemma}
To prove \cref{compactly supported cohomology version} we use a five lemma argument due to   \textcite[Appendix A]{MR533892} involving \cref{long exact sequence in csc} and the following filtration on $\text{Conf}(M)$.
\begin{definition}\label{filt for Conf(M)}
Suppose that $M$ is the interior of an $n$-dimensional manifold admitting a finite handle decomposition with a single 0-handle. Pick a closed subspace $X\subset M$ with the properties described in \cref{handle-decomp}.
Define $\conffilt^{j}_{\partic}(M)$ to be the subspace of $\text{Conf}_{\partic}(M)$ with at least $j$ points in $X$.
\end{definition}
Note that $\conffilt^{j}_{\partic}(M)$ is empty for $j>\partic$, $\conffilt^{j+1}_{\partic}(M)\subset \conffilt^{j}_{\partic}(M)$, the complement $\conffilt^{j}_{\partic}(M)\setminus \conffilt^{j+1}_{\partic}(M)$ is homeomorphic to $\text{Conf}_{\partic-j}(M \setminus X )\times \text{Conf}_{j}(X)$, and $\conffilt^{j}_{\partic}(M)$ is closed for all $j, \partic \geq 0$.

Recall that the construction of the secondary stabilization map depends on choosing 
an embedding
%embeddings $f\colon D^{n}_{1}\sqcup D^{n}_{2}\to D^{n}$ and 
$\embed\colon M\sqcup D^{n}_{1}\sqcup D^{n}_{2}\to M$ such that the restriction of $\embed$ to $M$ is isotopic to the identity. It is implicit in the proofs of \cref{stabilization by embedding} and \cref{handle-decomp} and in the definition of the stabilization map that we can choose $X\subset M$ and $\embed$ such that $\embed$ restricts to an embedding $\embed\colon  (M\setminus X)\sqcup D^{n}_{1}\sqcup D^{n}_{2}\to M\setminus X$ with the restriction of $\embed$ to $M\setminus X$ isotopic to the identity. Therefore, the map 
$$\text{Conf}(\embed)\colon \text{Conf}_{\partic}(M)\times \text{Conf}_{l_{1}}(D^{1})\times \text{Conf}_{l_{2}}(D^{2})\to \text{Conf}_{\partic+ l_{1} + l_{2}}(M)$$
restricts to a map $\text{Conf}(\embed)\colon\conffilt^{j}_{\partic}(M)\times \text{Conf}_{l_{1}}(D^{1})\times \text{Conf}_{l_{2}}(D^{2})\to \conffilt^{j}_{\partic+l_{1}+l_{2}}(M)$ for all $j$. As a result, by restriction we have maps $t_{1}\colon C^{i-n}_{c}(\conffilt^{j}_{\partic-1}(M);\fieldc)\to C^{i}_{c}(\conffilt^{j}_{\partic}(M);\fieldc)$ and $t_{2}\colon C^{i}_{c}(\conffilt^{j}_{\partic}(M);\fieldc)\to  C^{i+3}_{c}(\conffilt^{j}_{\partic +2}(M);\fieldc)$
%\begin{align*}
%    t_{1}\colon C^{i-n}_{c}(\conffilt^{j}_{\partic-1};\fieldc)&\to C^{i}_{c}(\conffilt^{j}_{\partic};\fieldc)\text{ and}\\
%    t_{2}\colon 
%    C^{i}_{c}(\conffilt^{j}_{\partic}(M);\fieldc)&\to  C^{i+3}_{c}(\conffilt^{j}_{\partic +2}(M);\fieldc)
%\end{align*}
which commute. So we have an induced map
%$t_{1}\colon \conffilt^{j}_{\partic-1}\to \conffilt^{j}_{\partic}$
%\footnote{I probably need to give more details. I probably can cut out a page}
$$t_{2} \colon  C^{i}_{c}(\conffilt^{j}_{\partic}(M),\conffilt^{j}_{\partic -1}(M);\fieldc)\to  C^{i+3}_{c}(\conffilt^{j}_{\partic +2}(M), \conffilt^{j}_{\partic +1}(M);\fieldc).$$

\cref{compactly supported cohomology version} is a consequence of the following proposition.
\begin{proposition}\label{filtered case}
Let $M$ be the interior of a connected surface admitting a finite handle decomposition with a single $0$-handle. Pick $X\subset M$ as in \cref{handle-decomp}. 
%Let $i$ and $\partic $  be nonnegative integers and let $p$ be a prime number. 
Let $D(p, \srange, \partic )$ denote the constant from \cref{stable ranges for the plane}.
\hfill\begin{enumerate}
    \item Suppose that $p$ is $2$. The map
%\begin{align*}
%    s_{2} \colon  H^{i}_{c}(\conffilt^{j}_{\partic}(M),\conffilt^{j}_{\partic -1}(M);\mathbb{F}_{p})\to  H^{i+2p+2}_{c}(\conffilt^{j}_{\partic+2p}(M), \conffilt^{j}_{\partic -1+2p}(M);\mathbb{F}_{p})
%\end{align*}
%\\ \-\hspace{.6cm}
$$t_{2} \colon  H^{i}_{c}(\conffilt^{j}_{\partic}(M),\conffilt^{j}_{\partic -1}(M);\mathbb{F}_{2})\to  H^{i+3}_{c}(\conffilt^{j}_{\partic +2}(M), \conffilt^{j}_{\partic +1}(M);\mathbb{F}_{2})$$
is an isomorphism for 
$i>  2k- D(2, 2 , \partic )$ and a surjection for $i=2k- D(2, 2 , \partic )$.
\item Suppose that $p$ is an odd prime and that $M$ is orientable. The map
%\\\-\hspace{.6cm}
$$ t_{2} \colon  H^{i}_{c}(\conffilt^{j}_{\partic}(M),\conffilt^{j}_{\partic -1}(M);\mathbb{F}_{p})\to  H^{i+2p+2}_{c}(\conffilt^{j}_{\partic +2p}(M), \conffilt^{j}_{\partic -1+2p}(M);\mathbb{F}_{p})$$
is an isomorphism for 
$i>  2k- D(p, 2 , \partic )$ and a surjection for $i=2k- D(p, 2 , \partic )$.
\end{enumerate}
\end{proposition}
%Suppose $p$ is odd. The map $$s_{2} \colon  H^{i}_{c}(\conffilt^{j}_{\partic}(M),\conffilt^{j}_{\partic -1}(M);\mathbb{F}_{p})\to  H^{i+2p+2}_{c}(\conffilt^{j}_{\partic+2p}(M), \conffilt^{j}_{\partic -1+2p}(M);\mathbb{F}_{p})$$
%\begin{align*}
%is an isomorphism for 
%$i>  (\frac{p^{2}+1}{p^{2}})(k+1)+\frac{2p^{2}-2p-2}{p^{2}}$ and a surjection for $i=(\frac{p^{2}+1}{p^{2}})(k+1)+\frac{2p^{2}-2p-2}{p^{2}}$.
\begin{proof}
We prove this proposition for the case $p$ is odd since the proof for the case $p$ is $2$ is similar. The proof is essentially the same as \textcite[Propositionn A.1]{MR533892} and \textcite[Proposition 4.4]{MR3344444}, but with slightly different maps and ranges. Let $\conffilt^{j}_{\partic}$ denote $\conffilt^{j}_{\partic}(M)$.
%and let $H^{i}_{c}(-)$ denote $H^{i}_{c}(- ;\mathbb{F}_{p})$.

Consider the long exact sequence of \cref{long exact sequence in csc} associated to the inclusion of pairs $(\conffilt^{j+1}_{\partic},\conffilt^{j+1}_{\partic -1})\subset (\conffilt^{j}_{\partic},\conffilt^{j}_{\partic-1})$ coming from the inclusion $\conffilt^{j+1}_{\partic}\subset \conffilt^{j}_{\partic}$
$$\cdots\to H^{i}_{c}(\conffilt^{j}_{\partic}\setminus \conffilt^{j+1}_{\partic}, \conffilt^{j}_{\partic-1}\setminus \conffilt^{j+1}_{\partic-1};\mathbb{F}_{p})\to H^{i}_{c}(\conffilt^{j}_{\partic}, \conffilt^{j}_{\partic-1};\mathbb{F}_{p})\to H^{i}_{c}(
\conffilt^{j+1}_{\partic}, \conffilt^{j+1}_{\partic-1};\mathbb{F}_{p})\to \cdots .$$
%Since $\conffilt^{j}_{\partic}\setminus \conffilt^{j+1}_{\partic}=\text{Conf}_{k-j}(M\setminus X)\times \text{Conf}_{j}(X)$ 
%$t_{2}$ maps $C^{*}_{c}(\conffilt^{j}_{\partic};\fieldc)$ to  $C^{*}_{c}(\conffilt^{j}_{\partic+2p};\fieldc)$ for all $j$ and 
The restriction of $t_{2}\colon  C_{*}(\conffilt^{j}_{\partic},\conffilt^{j}_{\partic-1};\mathbb{F}_{p})\to C_{*+1}(\conffilt^{j}_{\partic+1},\conffilt^{j}_{\partic+2};\mathbb{F}_{p})$ to $$C_{*}(\conffilt^{j}_{\partic}\setminus \conffilt^{j+1}_{\partic}, \conffilt^{j}_{\partic-1}\setminus \conffilt^{j+1}_{\partic-1};\mathbb{F}_{p})=C_{*}(\text{Conf}_{k-j}(M\setminus X)\times \text{Conf}_{j}(X),\text{Conf}_{k-1-j}(M\setminus X)\times \text{Conf}_{j}(X);\mathbb{F}_{p})$$ is the map $t_{2}\times \text{id}_{*}\colonequals t_{2}\times (\text{id}_{\text{Conf}_{j}(X)})_{*}$.
%
%Given $(W_{i}, W_{i}$
%
%relative version of the long exact sequence of \cref{long exact sequence in csc}
%\begin{center}
%    $\cdots \to  H^{*}_{c}(U, U^{'};\fieldc)
%    \to  H^{*}_{c}(W, W^{'};\fieldc)
%    \to  H^{*}_{c}(C, C^{'};\fieldc)
%    \to  H^{*+1}_{c}(U, U^{'};\fieldc)\to \cdots$
%\end{center}
%is natural in maps 
%\begin{tikzcd}
%[cells={nodes={text height=2ex,text depth=0.75ex}}]
%   H^{i}_{c}(\text{Conf}_{k-j}(M\setminus X)\times \text{Conf}_{j}(X), \text{Conf}_{k-1-j}(M\setminus X)\times \text{Conf}_{j}(X);\mathbb{F}_{p}) \arrow{r} & \cdots \\
  % H^k(M) \arrow{r} & H^k(U) \oplus %H^k(V) \arrow{r}
  %\arrow[draw=none]{u}[name=Y, shape=coordinate]{}
  %\arrow[draw=none]{d}[name=Z,shape=coordinate]{}
%  & H^k(U \cap V)\\ %\arrow[curarrow=Y]{ull}{} \\
  %& & & \cdots \arrow{r}{j^*} & H^{k-1}(U \cap V)\\
  %\arrow[curarrow=Z]{ull}{} \\
%\end{tikzcd}
%$$H^{i}_{c}(\text{Conf}_{k-j}(M\setminus X)\times \text{Conf}_{j}(X), \text{Conf}_{k-1-j}(M\setminus X)\times \text{Conf}_{j}(X);\mathbb{F}_{p}) = H^{i}_{c}(\conffilt^{j}_{\partic}\setminus \conffilt^{j+1}_{\partic}, \conffilt^{j}_{\partic-1}\setminus \conffilt^{j+1}_{\partic-1};\mathbb{F}_{p})\to H^{i}_{c}(\conffilt^{j}_{\partic}, \conffilt^{j}_{\partic-1};\mathbb{F}_{p})\to H^{i}_{c}(\conffilt^{j+1}_{\partic}, \conffilt^{j+1}_{\partic-1};\mathbb{F}_{p})$$
From the discussion following \cref{long exact sequence in csc}, the secondary stabilization map $t_{2}$ induces a map between long exact sequences:
\begin{center}\begin{tikzcd}
%\centering
\vdots\arrow[d] & \vdots\arrow[d]\\
H^{i}_{c}(\conffilt^{j}_{\partic}\setminus \conffilt^{j+1}_{\partic}, \conffilt^{j}_{\partic-1}\setminus \conffilt^{j+1}_{\partic-1};\mathbb{F}_{p})
\arrow[r, "t_2\times\text{id}_{*}"]\arrow[d] &
H^{i+2p+2}_{c}(\conffilt^{j}_{\partic+2p}\setminus \conffilt^{j+1}_{\partic+2p}, \conffilt^{j}_{\partic-1+2p}\setminus \conffilt^{j+1}_{\partic-1+2p};\mathbb{F}_{p})
\arrow[d]\\
H^{i}_{c}(\conffilt^{j}_{\partic}, \conffilt^{j}_{\partic-1};\mathbb{F}_{p})
\arrow[r,"t_{2}"]\arrow[d] & 
H^{i+2p+2}_{c}(\conffilt^{j}_{\partic+2p},\conffilt^{j}_{\partic-1+2p};\mathbb{F}_{p})
\arrow[d]\\
H^{i}_{c}(
\conffilt^{j+1}_{\partic}, \conffilt^{j+1}_{\partic-1};\mathbb{F}_{p})
\arrow[r, "t_{2}"]\arrow[d]
& 
H^{i+2p+2}_{c}(\conffilt^{j+1}_{\partic+2p},\conffilt^{j+1}_{\partic-1+2p};\mathbb{F}_{p})
\arrow[d]
\\ \vdots & \vdots
\end{tikzcd}\end{center}

We want to show that the middle horizontal map
\begin{equation}\tag*{$(*)_{\partic}$}\label{secondary stability iso on filtration of Conf(M)}
t_{2} \colon  H^{i}_{c}(\conffilt^{j}_{\partic}, \conffilt^{j}_{\partic -1};\mathbb{F}_{p})
\to  H^{i+2p+2}_{c}(\conffilt^{j}_{\partic +2p},\conffilt^{j}_{\partic -1+2p};\mathbb{F}_{p})
\end{equation}
is an isomorphism for $i>  2k- D(p, 2 , \partic )$ and a surjection for $i=  2k- D(p, 2 , \partic )$. 

For $j>\partic +2p$, the pairs $(\conffilt^{j}_{\partic}, \conffilt^{j}_{\partic -1})$ and $(\conffilt^{j}_{\partic +2p}, \conffilt^{j}_{\partic -1+2p})$ are both empty, so \ref{secondary stability iso on filtration of Conf(M)} holds for $j>\partic +2p$.
%the middle horizontal map $t_{2}$ is an isomorphism. 
Proceed by downward induction on $j$ and suppose that \ref{secondary stability iso on filtration of Conf(M)}
%the proposition 
holds for $j+1$. Then by our induction hypothesis, the lower horizontal map $$t_{2} \colon H^{i}_{c}(
\conffilt^{j+1}_{\partic}, \conffilt^{j+1}_{\partic-1};\mathbb{F}_{p})
\to 
H^{i+2p+2}_{c}(\conffilt^{j+1}_{\partic+2p},\conffilt^{j+1}_{\partic-1+2p};\mathbb{F}_{p})$$ is an isomorphism for $i>  2k- D(p, 2 , \partic )$ and a surjection for $i=2k- D(p, 2 , \partic )$. To prove that \ref{secondary stability iso on filtration of Conf(M)} holds, by the five lemma it suffices to show:
%\begin{align}\tag{e}\label{isomorphism on filtration of Conf}
%the map
\begin{equation}\tag*{$(\dagger)_{\partic, j}$}\label{isomorphism on filtration of Conf}
    t_{2}\times \text{id}_{*} \colon  H^{i}_{c}(\conffilt^{j}_{\partic}\setminus \conffilt^{j+1}_{\partic}, \conffilt^{j}_{\partic-1}\setminus \conffilt^{j+1}_{\partic-1};\mathbb{F}_{p}) \to  H^{i+2p+2}_{c}(\conffilt^{j}_{\partic +2p}\setminus \conffilt^{j+1}_{\partic+2p}, \conffilt^{j}_{\partic-1+2p}\setminus \conffilt^{j+1}_{\partic-1+2p};\mathbb{F}_{p})
\end{equation}
 is an isomorphism for $i> 2k- D(p, 2 , \partic -j)$ and a surjection for $i= 2k- D(p, 2 , \partic -j).$
%\end{align}
%\begin{equation}\tag{e}\label{isomorphism on filtration of Conf}
%    \mbox{that the map }
%     t_{2}\times \text{id}_{*} \colon  H^{i}_{c}(\conffilt^{j}_{\partic}\setminus \conffilt^{j+1}_{\partic}, \conffilt^{j}_{\partic}\setminus \conffilt^{j+1}_{\partic}) \to  H^{i+2p+2}_{c}(\conffilt^{j}_{\partic +2p}\setminus \conffilt^{j+1}_{\partic+2p}, \conffilt^{j}_{\partic+2p}\setminus \conffilt^{j+1}_{\partic+2p})\\
%     \mbox{ is an isomorphism for }i> 2k- D(p, 2 , \partic -j) \mbox{ and a surjection for }i= 2k- D(p, 2 , \partic -j).
%\end{equation}
%\text{that the map} $$t_{2}\times \text{id}_{*} \colon  H^{i}_{c}(\conffilt^{j}_{\partic}\setminus \conffilt^{j+1}_{\partic}, \conffilt^{j}_{\partic}\setminus \conffilt^{j+1}_{\partic}) \to  H^{i+2p+2}_{c}(\conffilt^{j}_{\partic +2p}\setminus \conffilt^{j+1}_{\partic+2p}, \conffilt^{j}_{\partic+2p}\setminus \conffilt^{j+1}_{\partic+2p})$$\text{ is an isomorphism for }$i> 2k- D(p, 2 , \partic -j)$ \text{ and a surjection for }$i= 2k- D(p, 2 , \partic -j).$

Recall that since $\conffilt^{j}_{\partic}\setminus \conffilt^{j+1}_{\partic}=\text{Conf}_{\partic-j}(M\setminus X)\times \text{Conf}_{j}(X) $ and $M\setminus X$ is homeomorphic to $\mathbb{R}^{2}$, the pair $(\conffilt^{j}_{\partic}\setminus \conffilt^{j+1}_{\partic},\conffilt^{j}_{\partic-1}\setminus \conffilt^{j+1}_{\partic-1})$ is isomorphic to $\big  ( \text{Conf}_{\partic-j}(\mathbb{R}^{2})\times \text{Conf}_{j}(X),\text{Conf}_{\partic-1-j}(\mathbb{R}^{2})\times \text{Conf}_{j}(X)\big )$. By \cref{prop:conf-stab} and Poincar\'e duality,
%we have:
the secondary stabilization map
%\begin{equation}\tag*{$(*)_{\partic}$}\label{secondary stability iso of Conf(M)}
$$t_{2}\colon H^{i}_{c}( \text{Conf}_{\partic -j}(\mathbb{R}^{2}), \text{Conf}_{\partic -1-j} (\mathbb{R}^{2});\mathbb{F}_{p})\to  H^{i+2p+2}_{c}( \text{Conf}_{\partic -j+2p}(\mathbb{R}^{2}), \text{Conf}_{\partic -1-j+2p}(\mathbb{R}^{2});\mathbb{F}_{p})$$
%\end{equation}
is an isomorphism for $i> 2k- D(p, 2 , \partic -j)$ and a surjection for $i= 2k- D(p, 2 , \partic -j)$.

Since we are working with cohomology with coefficients in a field, the K\"unneth theorem holds. Suppose that $x\in H^{a}_{c}(\text{Conf}_{\partic-j+2p}(M\setminus X),\text{Conf}_{\partic-j-1+2p}(M\setminus X);\mathbb{F}_{p})$ and $y\in H^{b}_{c} (\text{Conf}_{j}(X),\text{Conf}_{j-1}(X);\mathbb{F}_{p})$ ,
%are classes in $H^{a}_{c}(\text{Conf}_{\partic-j+2p}(M\setminus X),\text{Conf}_{\partic-j-1+2p}(M\setminus X);\mathbb{F}_{p})$ and $H^{b}_{c} (\text{Conf}_{j}(X),\text{Conf}_{j-1}(X);\mathbb{F}_{p})$, respectively, 
and that $x\otimes y$ is not in the image of $t_{2}\times \text{id}_{*}$. We have that $\text{deg}(x)<2k- D(p, 2 , \partic -j)$. Since the dimension of $X$ is at most $n-1=1$, $H^{i}_{c}\big (\text{Conf}_{j}(X);\mathbb{F}_{p}\big )=0$ for $i>j$. Since $y\neq 0$, we have $b\leq j$. Since $\text{deg}(x\otimes y)=\text{deg}(x)+\text{deg}(y)$, the homological degree of $x\otimes y$ satisfies
\begin{align*}
\text{deg}(x\otimes y)&<\big( 2(k-j) - D(p, 2 ,\partic -j)\big) +j\\
&=2(k-j) - \frac{(p^{2}-1)(\partic -j)-(2p^{2}-2p-2)}{p^{2}}+j\\
%&=(\frac{p^{2}+1}{p^{2}})(\partic -j)+\frac{2p^{2}-2p-2}{p^{2}} +j\\
&=(\frac{p^{2}+1}{p^{2}})\partic +\frac{2p^{2}-2p-2}{p^{2}} -\frac{j}{p^{2}}\\
&< (\frac{p^{2}+1}{p^{2}})\partic +\frac{2p^{2}-2p-2}{p^{2}}\\%=2\partic -D(p, 2 ,\partic ).\\
&=2\partic - \frac{(p^{2}-1)\partic -(2p^{2}-2p-2)}{p^{2}}=2\partic -D(p, 2 ,\partic ). 
\end{align*}
%Here, we are taking the minimum over all $j=0,\ldots , k$ since $\displaystyle\frac{1}{p^2}-2 <0$ (we include $j=0$ and $j=k$ since $\text{Conf}_{0}$ of a space is a point).
Therefore, $t_{2}\times \text{id}_{*}$ is surjective for $i\geq 2\partic -D(p, 2 ,\partic )$. 
%The secondary stabilization map is always injective by \cref{thm-McDuff}.
In addition, note that since $M\setminus X$ is homeomorphic to $\mathbb{R}^{2}$, the secondary stabilization map is always injective on $H^{a}_{c}(\text{Conf}_{\partic-j+2p}(M\setminus X),\text{Conf}_{\partic-1-j+2p}(M\setminus X);\mathbb{F}_{p})$ by \cref{prop:conf-stab}. Since $H^{b}_{c}(\text{Conf}_{j}(X),\text{Conf}_{j-1}(X); \mathbb{F}_p)$ is a flat $\mathbb{F}_p$-module, it follows that $t_2 \times \text{id}_*$ is injective and so \ref{isomorphism on filtration of Conf} holds.
\end{proof}
We also have the following analogue of \cref{prop:conf-stab}.
%\begin{definition}\label{notation for higher stabilization maps}
%Let $p$ be a prime number and let $e\in H_{0}(\text{Conf}_{1}(\mathbb{R}^{2})$ denote the class of a point. Let $x_{j}\in H_{*}(\text{Conf}(\mathbb{R}^{2});\mathbb{F}_{2})$ and $y_{j}\in H_{*}(\text{Conf}(\mathbb{R}^{2});\mathbb{F}_{p})$ denote the homology classes from \cref{thm-Co}. For $j$ at least $1$, let $t_{j}\colon C_{*}(\text{Conf}(\mathbb{R}^{2});\mathbb{F}_{p})\to C_{*}(\text{Conf}(\mathbb{R}^{2});\mathbb{F}_{p})$ denote the map $t_{x_{j-1}}$ if $p$ is $2$ and let it denote the map $t_{x_{j-1}}$ if $p$ is odd.
%\end{definition}
\begin{corollary}\label{higher stab open surface}
Suppose that $M$ is an open connected surface.
%Let $p$ be a prime number and let $\srange$ be a positive integer. 
Let $x_{\srange}\in H_{*}(\text{Conf}(\mathbb{R}^{2});\mathbb{F}_{2})$ and $y_{\srange}\in H_{*}(\text{Conf}(\mathbb{R}^{2});\mathbb{F}_{p})$ denote the homology classes from \cref{thm-Co}. 
%Suppose that $\partic$ and $\srange$ are positive integers and 
Let $D(p,\srange , \partic )$ denote the constant defined in \cref{stable ranges for the plane}.
\hfill\begin{enumerate}
    \item Suppose that $p$ is $2$.  Let $\pariteratedmap{M}{t_{\displaystyle x_{j}}}{\srange}{\partic}$ denote the \srange-th iterated mapping cone associated to the collection $(t_{1},  t_{\displaystyle x_{1}}, \ldots, t_{\displaystyle x_{\srange-1}})$ as defined in \cref{iterated mapping cone def}.
    %\footnote{Is the notation $t_{j}$ for the higher order stabilization maps better $t_{\displaystyle x_{j}}$? I think the former notation might be more confusing, but I'm not sure}
    %Let $I_{\srange}\colonequals \bigcup_{j=0}^{\srange}\text{im}(t_{x_{j}})\subset C_{*}(\text{Conf}(M);\mathbb{F}_{2})$ denote the subcomplex generated by the images of $x_{0},\ldots , x_{\srange}.$ Let $I_{\srange}(\partic )$ denote the subcomplex of $I_{\srange}$ of chains in $C_{*}(\text{Conf}_{\partic}(M);\mathbb{F}_{2})$.
    The map% of relative homology groups 
    $$ t_{\displaystyle x_{\srange}}\colon H_{i}(\pariteratedmap{M}{t_{\displaystyle x_{j}}}{\srange}{\partic};\mathbb{F}_{2})\to  H_{i+2^{\srange}-1}(\pariteratedmap{M}{t_{\displaystyle x_{j}}}{\srange}{\partic+2^{\srange}};\mathbb{F}_{2})$$
    %$$ t_{x_{\srange}}\colon H_{i}(\text{Conf}_{\partic }(M),I_{\srange -1}(\partic);\mathbb{F}_{2})\to  H_{i+2^{\srange}-1}(\text{Conf}_{\partic +2^{\srange}}(M),I_{\srange-1}(\partic +2^{\srange});\mathbb{F}_{2})$$
    is an isomorphism for $i< D(2, \srange+1 , \partic )$ and a surjection for $i= D(2, \srange+1 , \partic ).$
    \item Suppose that $p$ is an odd prime and that $M$ is orientable. Let $\pariteratedmap{M}{t_{y_{j}}}{\srange}{\partic}$ denote the \srange-th iterated mapping cone associated to the collection $(t_{1},  t_{\displaystyle y_{1}},\ldots, t_{\displaystyle y_{\srange -1}})$.
    %Let $J_{\srange}\colonequals \bigcup_{j=0}^{\srange}\text{im}(t_{y_{j}})\subset C_{*}(\text{Conf}(M);\mathbb{F}_{p})$ denote the subcomplex generated by the images of $y_{0},\ldots , y_{\srange}.$ Let $J_{\srange}(\partic )$ denote the subcomplex of $J_{\srange}$ of chains in $C_{*}(\text{Conf}_{\partic}(M);\mathbb{F}_{p})$. 
    The map% of relative homology groups 
     $$ t_{\displaystyle y_{\srange}}\colon H_{i}(\pariteratedmap{M}{t_{\displaystyle y_{j}}}{\srange}{\partic};\mathbb{F}_{p})\to  H_{i+2p^{\srange}-2}(\pariteratedmap{M}{t_{\displaystyle y_{j}}}{\srange}{\partic+2p^{\srange}};\mathbb{F}_{p})$$
    %$$t_{\displaystyle y_{\srange}}\colon H_{i}(\text{Conf}_{\partic }(M),J_{\srange-1}(\partic);\mathbb{F}_{p})\to  H_{i+2p^{\srange}-2}(\text{Conf}_{\partic +2p^{\srange}}(M),J_{\srange-1}(\partic +2p^{\srange});\mathbb{F}_{p})$$
    is an isomorphism for $i< D(p, \srange+1 , \partic )$ and a surjection for $i= D(p, \srange+1, \partic ).$ 
\end{enumerate}
\end{corollary}
\begin{proof}
Apply the argument of \cref{second stability} by replacing $x_{1}$ and $y_{1}$ with $x_{\srange}$ and $y_{\srange}$, respectively.
\end{proof}
The case $\srange=1$ in \cref{higher stab open surface} proves \cref{secondary stability open case} in the case that $M$ is a surface and the homology is with $\mathbb{F}_{p}$ coefficients. From \cref{higher stab open surface},
%secondary homological stability for $H_{*}(\text{Conf}_{\partic}(M),\text{Conf}_{\partic-1}(M);\mathbb{F}_{p})$, 
we can also prove \cref{secondary stability open case} in the case that $M$ is a surface and the homology is integral homology.
\begin{proposition}\label{integral secondary stability for surface}
Suppose that $M$ is an open connected surface. %Let $E(M,\partic)$ denote the constant
%\[
%E(M,\partic)= \begin{cases}
%      D(3,1,\partic-1) & \text{if }M \text{ is non-orientable}\\
%      \min(D(3,1,\partic-1)+1, \partic-1) & \text{if }M \text{ is orientable.}
%   \end{cases}
%   \]
The map 
$$t_{\displaystyle x_{1}}\colon C_{i}(\text{Conf}_{\partic}(M),\text{Conf}_{\partic-1}(M);\mathbb{F}_{2})\to C_{i+1}(\text{Conf}_{\partic+2}(M),\text{Conf}_{\partic+1}(M);\mathbb{F}_{2})$$ lifts to a map of singular chain complexes
$$ t_{2}\colon C_{i}(\text{Conf}_{\partic}(M),\text{Conf}_{\partic-1}(M);\mathbb{Z})\to C_{i+1}(\text{Conf}_{\partic+2}(M),\text{Conf}_{\partic+1}(M);\mathbb{Z})$$
and induces an isomorphism on integral homology for $i< D(3,1,\partic-1)+1=(2/3)k$ and a surjection for $i= D(3,1,\partic-1)+1$.
\end{proposition}
\begin{proof}
The proof is essentially the same as the proof of \textcite[Theorem 1.4]{MR3344444}. Suppose that $M$ is orientable (if $M$ is not orientable or odd dimensional, one modifies the argument by working with compactly supported cohomology with coefficients in the orientation systems of the configuration spaces instead---see \textcite[p.~71]{MR533892}).
%We may also assume that $M$ is orientable---the proof is similar for the non-orientable case. 
The class $x_{1}\in H_{1}(\text{Conf}_{2}(\mathbb{R}^{2});\mathbb{F}_{2})\cong H_{1}(\mathbb{R}P^{1};\mathbb{F}_{2})$ comes from the class of a loop $\omega\in H_{1}(\mathbb{R}P^{1};\mathbb{Z})$ in $\mathbb{R}P^{1}$. Therefore, we have a map $$t_{2}\colonequals t_{\omega}\colon C_{i}(\pariteratedmap{M}{t_{1}}{1}{\partic};\mathbb{Z}) \to C_{i+1}(\pariteratedmap{M}{t_{1}}{1}{\partic+2};\mathbb{Z})$$
%C_{i}(\text{Conf}_{\partic}(M),\text{Conf}_{\partic-1}(M);\mathbb{Z})\to C_{i+1}(\text{Conf}_{\partic+2}(M),\text{Conf}_{\partic+1}(M);\mathbb{Z})$$
and we just need to show that this map induces an isomorphism on integral homology for $i< D(3,1,\partic-1)+1$ and a surjection for $i= D(3,1,\partic-1)+1$. 

It suffices to show that the homology $H_{i+1}(\pariteratedmap{M}{t_{1},t_{2}}{2}{\partic+2};\mathbb{Z})$ vanishes for $i\leq D(3,1,\partic-1)+1$. We first show that it suffices to show that the homology $H_{i+1}(\pariteratedmap{M}{t_{1},t_{2}}{2}{\partic+2};\mathbb{F}_{p})$ vanishes in the same range for all primes $p$ by an argument in the proof of \textcite[Theorem 18.2]{galatius2018cellular} (see Reduction 2 and Reduction 3) which we now recall. 

Let $\text{Cone}^{2}_{k+2}$ denote $\pariteratedmap{M}{t_{1},t_{2}}{2}{\partic+2}$. By \cref{stabilization by embedding} and \cref{open manifold has exhaustion by compact manifold with finite handle decomposition}, we can assume that $M$ is the interior of a compact manifold with a finite handle decomposition with only one $0$-handle. Since $M$ is the interior of a compact manifold with a finite handle decomposition, by \cref{homology of conf is finitely generated}, $H_{i}(\text{Conf}_{\partic}(M);\mathbb{Z})$ is a finitely generated abelian group for all $i,\partic\geq 0$. Since $H_{i+1}(\text{Cone}^{2}_{k+2};\mathbb{Z})$ is a quotient of $H_{i+1}(\text{Conf}_{\partic+2}(M);\mathbb{Z})$, $H_{i+1}(\text{Cone}^{2}_{k+2};\mathbb{Z})$ is also a finitely generated abelian group.

By the universal coefficient theorem, we have the following short exact sequence:
%$$0\to H_{i+1}(\pariteratedmap{M}{t_{i}}{2}{\partic+2};\mathbb{Z})\otimes \mathbb{F}_{p}\to H_{i+1}(\pariteratedmap{M}{t_{1},t_{2}}{2}{\partic+2};\mathbb{F}_{p})\to \text{Tor}^{\mathbb{Z}}_{1}(H_{i+1}(\pariteratedmap{M}{t_{i}}{2}{\partic+2};\mathbb{Z}),\mathbb{F}_{p})\to 0$$
$$0\to H_{i+1}(\text{Cone}^{2}_{k+2};\mathbb{Z})\otimes \mathbb{F}_{p}\to H_{i+1}(\text{Cone}^{2}_{k+2};\mathbb{F}_{p})\to \text{Tor}^{\mathbb{Z}}_{1}(H_{i}(\text{Cone}^{2}_{k+2};\mathbb{Z}),\mathbb{F}_{p})\to 0.$$
If $H_{i+1}(\text{Cone}^{2}_{k+2};\mathbb{Z})\neq 0$, then $H_{i+1}(\text{Cone}^{2}_{k+2};\mathbb{Z})\otimes \mathbb{F}_{p}$ is also non-zero for some prime $p$ (here, we implicitly use that $H_{i+1}(\text{Cone}^{2}_{k+2};\mathbb{Z})$ is a finitely generated abelian group). From the short exact sequence, $H_{i+1}(\text{Cone}^{2}_{k+2};\mathbb{Z})\otimes \mathbb{F}_{p}$ injects into $H_{i+1}(\text{Cone}^{2}_{k+2}; \mathbb{F}_{p})$, so $H_{i+1}(\text{Cone}^{2}_{k+2}; \mathbb{F}_{p})$ must also be non-zero for some prime $p$. On the other hand, if $H_{i+1}(\text{Cone}^{2}_{k+2}; \mathbb{F}_{p})=0$ for all primes $p$ in a range then $H_{i+1}(\text{Cone}^{2}_{k+2}; \mathbb{Z})$ also vanishes in the same range (here, we again use that $H_{i+1}(\text{Cone}^{2}_{k+2};\mathbb{Z})$ is a finitely generated abelian group).

We now show that the homology $H_{i+1}(\text{Cone}^{2}_{k+2};\mathbb{F}_{p})$ vanishes for $i\leq D(3,1,\partic-1)+1$ and for all primes $p$. We first deal with the case that $p$ is an odd prime.
By \cref{homological stability in dimension 2}, we have that $H_{i}(\pariteratedmap{M}{t_{1}}{1}{\partic};\mathbb{F}_{p})$ and $H_{i+1}(\pariteratedmap{M}{t_{1}}{1}{\partic+2};\mathbb{F}_{p})$ vanish for $i\leq D(p,1,\partic-1)$ and $i\leq D(p,1,\partic+1)-1$, respectively. Since $D(3,1,\partic-1)\leq D(p,1,\partic-1)< D(p,1,\partic +1)-1$, the homology $H_{i+1}(\text{Cone}^{2}_{k+2};\mathbb{F}_{p})$ trivially vanishes for $i\leq D(3,1,\partic-1)+1$ when $p$ is an odd prime. When $p=2$, the homology $H_{i+1}(\text{Cone}^{2}_{k+2};\mathbb{F}_{2})$ vanishes for $i\leq D(2, 2 , \partic )$ by \cref{higher stab open surface}. Since $D(3,1,\partic-1 )+1\leq D(2, 2 , \partic )$, we have that the homology $H_{i+1}(\text{Cone}^{2}_{k+2};\mathbb{F}_{2})$ vanishes for $i\leq D(3,1,\partic-1)+1$.
\end{proof}
\subsection{Secondary Homological Stability for Configuration Spaces of Higher Dimensional Manifolds}\label{sec stab in dim n greater than 2}
If $M$ is an open connected manifold of dimension $n\geq 3$, Kupers--Miller showed that the relative homology $H_{*}(\text{Conf}_{\partic+1}(M),\text{Conf}_{\partic}(M);\mathbb{Z})$ is a finite abelian $2$-group for $*\leq \partic$ (see \cref{best known stable ranges}). Thus in this range, all secondary stability (and higher-order stability) results only concern $2$-torsion.
%the only possible non-trivial secondary homological stability for $\text{Conf}(M)$ is $2$-torsion by work of Church \textcite{MR2909770}, Randal-Williams \textcite{MR3032101}, Kupers--Miller \textcite{MR3344444} (see \cref{best known stable ranges}).
%For an open connected manifold $M$ of dimension $n\geq 3$, the only possible non-trivial secondary homological stability for $\text{Conf}(M)$ is $2$-torsion\footnote{Nick was confused by this comment (``the only possible... is 2-torsion''). Would something like ``the only possible torsion of $H_{*}(\text{Conf}_{\partic}(M),\text{Conf}_{\partic-1}(M);\mathbb{Z})$ in the range $*\leq\partic$ is of order $2^{r}$ for some $r$'' be better/clearer?} by work of Church \textcite{MR2909770}, Randal-Williams \textcite{MR3032101}, Kupers--Miller \textcite{MR3344444} (see \cref{best known stable ranges}). 
Consider the case $M$ is $\mathbb{R}^{n}$. The relative homology groups $H_{i}(\text{Conf}_{\partic}(\mathbb{R}^{n}), \text{Conf}_{\partic -1}(\mathbb{R}^{n});\mathbb{F}_{2})$ vanish only for $i\leq k/2$, so there might be non-trivial secondary homological stability for $\text{Conf}(\mathbb{R}^{n})$. To figure out whether there actually is secondary homological stability, we need the following description of $H_{*}(\text{Conf}(\mathbb{R}^{n});\mathbb{F}_{2})$, which follows from combining \cref{homology of free En algebra}, \cref{Lie word lemma}, and the fact that $\browd{e}{e}=0$ when working with $\mathbb{F}_{2}$ coefficients.
\begin{theorem}\label{description of homology of Conf of disk in dim greater than 2}
%Let
%\footnote{You asked if this result should be in \cref{sec:base case}. I chose not to because I just wanted to focus on stability for $\text{Conf}(\mathbb{R}^{2})$ in \cref{sec:base case}. I referenced this result in \cref{sec:base case} right before \cref{thm-Co}.} 
Suppose that $n\geq 3$. Let $e\in H_{0}(\text{Conf}_{1}(\mathbb{R}^{n});\mathbb{F}_{2})$ be the class of a point. Let $S$ be the set of formal symbols constructed by iterated formal applications of the Dyer-Lashof operations $Q^{s}$ and the top homology operation $\xi $ to $e$ (e.g. $Q^{2}Q^{1}e$ is an element of $S$). We also allow zero applications of $Q^{s}$ and $\xi $ to $e$. Let $D\subset S$ denote the subset of $S$ containing elements of the form $e$ and $Q^{2^{j}}Q^{2^{j-1}}\ldots Q^{2}Q^{1}e$ for all non-negative $j$.
There is a subset $G_{n}\subset S$ containing $D$ such that $H_{*}(\text{Conf}(\mathbb{R}^{n});\mathbb{F}_{2})$ is isomorphic to the free graded commutative algebra on $G_{n}$.
\end{theorem}
%Cohen--Lada--May explicitly described $G_{n}$ (see \textcite[p.~227]{MR0436146}), but such a description is not necessary for our purposes. From Cohen--Lada--May's result, 
We have the following analogue of \cref{stable ranges for the plane}.
\begin{proposition}\label{stable ranges for higher dimensions with F2 coefficients}
%Let $n$ be greater than $2$ and 
Let $e\in H_{0}(\text{Conf}_{1}(\mathbb{R}^{n});\mathbb{F}_{2})$ be the class of a point, with $n>2$. Let $j$ be a positive integer and let $A(j)\colonequals  2^{j}-1$. Define a class $\omega_{j}\in H_{A(j)}(\text{Conf}_{A(j)+1}(\mathbb{R}^{n});\mathbb{F}_{2})$ recursively by setting $\omega_{1}\colonequals Q^{1}e$ and $\omega_{j}\colonequals Q^{A(j-1)+1}\omega_{j-1}$. Let $D(p, \srange, \partic )$ denote the constant from \cref{stable ranges for the plane}.
%Let $\srange$ be a positive integer and l
%Let $D(p, \srange, \partic )$ denote the constant from \cref{stable ranges for the plane}.
%Let $E(\srange , \partic)\colonequals A(\srange )k/(A(\srange)+1)$

 Let $x$ be a homology class in $H_{i+A(\srange -1)}(\text{Conf}_{\partic +A(\srange -1)+1}(\mathbb{R}^{n});\mathbb{F}_{2})$ and suppose that $i\leq D(2, \srange, \partic )=A(\srange )k/(A(\srange)+1)$. Then $x$ is in the ideal $(e, \omega_{1},\ldots , \omega_{\srange -1})$.
\end{proposition}
\begin{proof}
The proof is similar to that of \cref{stable ranges for the plane}. We say that a class $x\in H_{i}(\text{Conf}_{\partic}(\mathbb{R}^{n});\mathbb{F}_{2})$ is in the \textit{\srange-th unstable range} if $i\geq D(2, \srange, \partic )$. Clearly, for $j$ in the range $0\leq j \leq \srange-1$, the classes $\omega_{j}$ are not in the \srange-th unstable range. Let $R_{\srange}\subset H_{*}(\text{Conf}(\mathbb{R}^{n});\mathbb{F}_{2})$ denote the complement of the ideal $(e, \omega_{1},\ldots , \omega_{\srange -1})$. It suffices to show that if $x$ is in $R_{\srange}$, then $x$ is in the \srange-th unstable range. 

We need the following results on how homology operations affect the homological degree and number of particles of a class.
\hfill\begin{enumerate}[label=\roman*)]
    \item\label{higher omegas are unstable} For all $j\geq \srange$, the class $\omega_{j}$ is in the \srange-th unstable range because $A(j)\geq A(\srange)$ for all $j\geq m$. 
    \item\label{top operation on omega is unstable} For all $j$, the class $\xi\omega_{j}\in H_{2A(j)+n-1}(\text{Conf}_{2(A(j)+1)}(\mathbb{R}^{n});\mathbb{F}_{2})$ is in the \srange-th unstable range. Since $A(j)=2^{j}-1$ and $n-1$ is at least $2$, we have that $2A(j)+n-1\geq 2(A(j)+1)$.
    \item\label{Dyer-Lashof and top applied to unstable is unstable} If $y\in H_{a}(\text{Conf}_{b}(\mathbb{R}^{n});\mathbb{F}_{2})$ is in the \srange-th unstable range, then the classes $\xi y$ and  $Q^{s}y$, for all $s$ in the range $a\leq s< a+ n-1$, are also in the \srange-th unstable range. This is because $\xi y$ and $Q^{s}y$ are classes of degree $2a+n-1$ and $a+s$, respectively, in  $\text{Conf}_{2b}(\mathbb{R}^{n})$ and $2a+n-1, a+s\geq 2a\geq D(2, \srange , 2b)$.
    %\item If $y\in H_{a}(\text{Conf}_{b}(\mathbb{R}^{n});\mathbb{F}_{2})$ is in the \srange-th unstable range, then the class $\xi y$ is in the \srange-th unstable range. This is because $\xi y$ is a class in $H_{2a+n-1}(\text{Conf}_{2b}(\mathbb{R}^{n});\mathbb{F}_{2})$ and $2a+n-1 > 2a\geq E(\srange , 2b)$.
    %Since the homological degrees of $Q^{s}y$ and $\xi y$ is at least
    \item\label{product of unstables is unstable} If $y\in H_{a}(\text{Conf}_{u}(\mathbb{R}^{n});\mathbb{F}_{2})$ and $z\in H_{b}(\text{Conf}_{v}(\mathbb{R}^{n});\mathbb{F}_{2})$ is in the \srange-th unstable range, then the class $yz$ is also in the \srange-th unstable range. This is because if $y$ and $z$ are in the \srange-th unstable range, then $a\geq D(2, \srange, u)$ and $b\geq D(2, \srange, v)$. Since $yz$ is a class in $H_{a+b}(\text{Conf}_{u+v}(\mathbb{R}^{n});\mathbb{F}_{2})$, we have that $a+b\geq D(2, \srange, u+v)$.
    \item\label{Dyer-Lashof on stable omega is unstable}
    If $j$ is in the range $1\leq j\leq \srange -1$ and $s$ is in the range $A(j)+1 <s < A(j)+ n-1$, then the class $Q^{s}\omega_{j}$ is in the \srange-th unstable range. This is because $Q^{s}\omega_{j}$ is a class in $H_{s+A(j)}(\text{Conf}_{2(A(j)+1)}(\mathbb{R}^{n});\mathbb{F}_{2})$ and $s+A(j)$ is at least $A(j+1)+1=2(A(j)+1)$ in this range. 
    %We omit the range $s\leq A(j)+1$ because $Q^{s}\omega_{j}$ is zero by Property\footnote{need to write down this property} and we omit $s>A(j)+n-1$  because $Q^{s}\omega_{j}$ is not defined in this range.
    \item\label{Dyer-Lashof and top on point class is unstable} The classes $\xi e$ and $Q^{s}e$ for all $s$ in the range $1<s< n-1$ are in the \srange-th unstable range. This is because $\xi e$ and $Q^{s}e$ are classes in $H_{n-1}(\text{Conf}_{2}(\mathbb{R}^{n});\mathbb{F}_{2})$ and $H_{s}(\text{Conf}_{2}(\mathbb{R}^{n});\mathbb{F}_{2})$ respectively, and $n-1>s\geq 2$.
\end{enumerate}
By \ref{higher omegas are unstable}, for all $j\geq \srange$ the classes $\omega_{j}$ are in the \srange-th unstable range so by \ref{Dyer-Lashof and top applied to unstable is unstable} and \ref{product of unstables is unstable}, applications of homology operations and products of these classes are also in the \srange-th unstable range as well. Therefore, the only classes that are possibly not in the \srange-th unstable range come from applications of homology operations to $e$ or $\omega_{j}$ for $j< \srange$. By \ref{Dyer-Lashof and top on point class is unstable} and \ref{top operation on omega is unstable}, the classes $\xi e$ and $\xi\omega_{j}$ are always in the \srange-th unstable range. Note that $Q^{0}e=e^2$, $Q^{1}e=\omega_{1}$, and $Q^{s}e$, for all $s$ in the range $1<s< n-1$, is in the \srange-th unstable range by \ref{Dyer-Lashof and top on point class is unstable} (we ignore $Q^{s}e$ for $s$ negative because $Q^{s}e=0$ in this range). Similarly, $Q^{s}\omega_{j}$ is equal to $0$ if $s<A(j)$ and it is equal to $\omega_{j}^{2}$ if $s=A(j)$. If $s=A(j)+1$, then $Q^{s}\omega_{j}$ is equal to $\omega_{j+1}$, and for all $s$ in the remaining range $A(j)+1 <s \leq A(j)+ n-1$, the class $Q^{s}\omega_{j}$ is in the \srange-th unstable range by \ref{Dyer-Lashof on stable omega is unstable}. Therefore, for $j \leq \srange -1$, the class $Q^{s}\omega_{j}$ is in the \srange-th unstable range only when $s$ is in the range $A(j)+1 <s \leq A(j)+ n-1$. Applications of homology operations and products of these $Q^{s}e$ and $Q^{s}\omega_{j}$ are also in the \srange-th unstable range by \ref{Dyer-Lashof and top applied to unstable is unstable} and \ref{product of unstables is unstable}. It follows that the classes that are not in the \srange-th unstable must be in the ideal $(e, \omega_{1},\ldots , \omega_{\srange -1})$.
\end{proof}
From \cref{stable ranges for higher dimensions with F2 coefficients}, one can show:
\begin{proposition}\label{higher stab open higher dimensional manifold}
Suppose that $M$ is an open connected  manifold of dimension $n>2$. 
%Let $\srange$ be a positive integer and $\partic$ be a positive integer. 
Let $\omega_{j}\in H_{*}(\text{Conf}(\mathbb{R}^{n});\mathbb{F}_{2})$ denote the homology classes defined in \cref{stable ranges for higher dimensions with F2 coefficients}. Let $D(p, \srange, \partic )$ denote the constant from \cref{stable ranges for the plane}. Let $\pariteratedmap{M}{t_{\displaystyle \omega_{j}}}{\srange}{\partic}$ denote the \srange-th iterated mapping cone associated to the collection $(t_{1},  t_{\displaystyle \omega_{1}}, \ldots, t_{\displaystyle \omega_{\srange-1}})$ as defined in \cref{iterated mapping cone def}.
\begin{enumerate}
    \item The map
    $t_{1}\colon H_{i}(\text{Conf}_{\partic}(M);\mathbb{F}_{2})\to H_{i}(\text{Conf}_{\partic+1}(M);\mathbb{F}_{2})$ is an isomorphism for $i< k/2$ and a surjection for $i= k/2$.
    \item The map% of relative homology groups 
    $$ t_{\displaystyle \omega_{\srange}}\colon H_{i}(\pariteratedmap{M}{t_{\displaystyle \omega_{j}}}{\srange}{\partic};\mathbb{F}_{2})\to  H_{i+2^{\srange}-1}(\pariteratedmap{M}{t_{\displaystyle \omega_{j}}}{\srange}{\partic+2^{\srange}};\mathbb{F}_{2})$$
    is an isomorphism for $i< D(2, \srange+1 , \partic )$ and a surjection for $i= D(2, \srange+1 , \partic ).$
\end{enumerate}
%Let $K_{\srange}\colonequals \bigcup_{j=0}^{\srange}\text{im}(t_{\omega_{j}})\subset C_{*}(\text{Conf}(M);\mathbb{F}_{2})$ denote the subcomplex generated by the images of $\omega_{0},\ldots , \omega_{\srange}.$ Let $K_{\srange}(\partic )$ denote the subcomplex of $K_{\srange}$ of chains in $C_{*}(\text{Conf}_{\partic}(M);\mathbb{F}_{2})$. The map of relative homology groups 
%    $$ t_{\omega_{\srange}}\colon H_{i}(\text{Conf}_{\partic }(M),K_{\srange -1}(\partic);\mathbb{F}_{2})\to  H_{i+2^{\srange}-1}(\text{Conf}_{\partic +2^{\srange}}(M),K_{\srange-1}(\partic +2^{\srange});\mathbb{F}_{2})$$
%    is an isomorphism for $i< D(2, \srange+1, \partic )$ and a surjection for $i= D(2, \srange+1 , \partic ).$
\end{proposition}
\begin{proof}
This immediately follows by applying the arguments of \cref{second stability} and the range is obtained from \cref{stable ranges for higher dimensions with F2 coefficients} just like in \cref{prop:conf-stab}.
\end{proof}
The case $\srange=1$ in \cref{higher stab open higher dimensional manifold} proves \cref{secondary stability open case} in the case that the dimension of $M$ is greater than $2$ and the homology is with $\mathbb{F}_{2}$ coefficients. From \cref{higher stab open higher dimensional manifold},
%secondary homological stability for $H_{*}(\text{Conf}_{\partic}(M),\text{Conf}_{\partic-1}(M);\mathbb{F}_{p})$, 
we can also prove \cref{secondary stability open case} in the case that the dimension of $M$ is greater than $2$ and the homology is integral homology.  We need the following result on homological stability of $H_{*}(\text{Conf}_{\partic}(M);\fieldc)$ when the dimension of $M$ is greater than $2$.
%First, we need the following result involving homological stability for $H_{*}(\text{Conf}_{\partic}(M);\fieldc)$ for homology with coefficients in a ring $\fieldc$
%when the dimension of $M$ is greater than $2$.
%This proposition also allows us to obtain a secondary homological stability result for the integral relative homology $H_{*}(\text{Conf}_{\partic}(M),\text{Conf}_{\partic-1}(M);\mathbb{Z})$. First, we need the following result involving homological stability for $H_{*}(\text{Conf}_{\partic}(M);\fieldc)$ for homology with coefficients in a ring $\fieldc$.
\begin{theorem}[{\textcite[Proposition A.1]{MR533892}}, {\textcite[ Theorem 1.4]{MR3344444}}]\label{best known stable ranges}
%\begin{theorem}[{\textcite[Proposition A.1]{MR533892}}, {\textcite[Corollary 3]{MR2909770}}, {\textcite[Theorem B]{MR3032101}}, {\textcite[ Theorem 1.4]{MR3344444}}, {\textcite[Theorem 1.3]{MR3704255}}]\label{best known stable ranges}
%\label{best known stable ranges}
%Let $M$ be an open connected manifold of dimension $n\geq 2$. The stabilization map induces an isomorphism
%$$t_{1}\colon H_{*}(\text{Conf}_{\partic}(M);\mathbb{Q})\to H_{*}(\text{Conf}_{\partic+1}(M);\mathbb{Q})$$ for $i\leq \partic$ if $M$ is orientable and for $i\leq \partic-1$ if $M$ is not orientable.
Let $M$ be an open connected manifold of dimension $n> 2$ and let $\fieldc$ be a ring such that $2$ is a unit.
%inverting a  Let $S\subset \text{Spec }\mathbb{Z}$ be a (possibly empty) set of nonzero prime ideals of $\mathbb{Z}$ and let $\mathbb{Z}[\frac{1}{S}]$ denote the ring obtained by inverting each of the prime ideals of $S$.
%The stabilization map $t_{1}\colon C_{*}(\text{Conf}_{\partic}(M);\fieldc)\to C_{*}(\text{Conf}_{\partic+1}(M);\fieldc)$ 
%induces an isomorphism
The stabilization map
$$t_{1}\colon H_{i}(\text{Conf}_{\partic}(M);\fieldc)\to H_{i}(\text{Conf}_{\partic +1}(M);\fieldc)$$ is an isomorphism for $i<f(\fieldc,\partic)$ and a surjection for $i=f(\fieldc,\partic)$, 
where
%\footnote{should I drop the cases $M$ is a surface and $\fieldc$ is a field of characteristic zero? I don't think I use these cases in the paper }
%\[
%f(\partic,M,R)= \begin{cases}
%      \partic & \text{if }2\text{ is a unit in } R\text{ and dim(M)}\geq 3\\
%      \partic & \text{if }R\text{ is a field of characteristic zero and M is a non-orientable surface}\\
%      \partic-1 & \text{if }R\text{ is a field of characteristic zero and M is an orientable surface}\\
%      \partic/2 & \text{otherwise.}
%   \end{cases}
%   \]
\[
f(\fieldc,\partic)= \begin{cases}
      \partic & \text{if }2\text{ is a unit in } R\\
      \partic/2 & \text{otherwise.}
   \end{cases}
   \]   
\end{theorem}
%\footnote{this result might be wrong because it might not be possible to construct integral homology classes from $\omega_{\srange}$} We first will show that the classes $\omega_{j}$ come from integral classes $\tilde{\omega}_{j}\in H_{*}(\text{Conf}(\mathbb{R}^{n});\mathbb{Z})$.
%\begin{proposition}
%Let $\omega_{j}\in H_{*}(\text{Conf}(\mathbb{R}^{n});\mathbb{F}_{2})$ denote the homology classes defined in \cref{stable ranges for higher dimensions with F2 coefficients}. The class $\omega_{j}$ is in the image of the ``reduction $(\text{mod }2)$ map'' 
%\end{proposition}
\begin{proposition}\label{integral secondary stab open manifold}
Suppose that $M$ is an open connected manifold of dimension $n>2$. The map 
$$t_{\displaystyle \omega_{1}}\colon C_{i}(\text{Conf}_{\partic}(M),\text{Conf}_{\partic-1}(M);\mathbb{F}_{2})\to C_{i+1}(\text{Conf}_{\partic +2}(M),\text{Conf}_{\partic+1}(M);\mathbb{F}_{2})$$ lifts to a map of singular chain complexes
$$ t_{\displaystyle \omega_{1}}\colon C_{i}(\text{Conf}_{\partic}(M),\text{Conf}_{\partic-1}(M);\mathbb{Z})\to C_{i+1}(\text{Conf}_{\partic+2}(M),\text{Conf}_{\partic+1}(M);\mathbb{Z})$$
and induces an isomorphism on integral homology for $i< D(2, 2, \partic )$ and a surjection for $i= D(2, 2, \partic )$.
%The map of relative homology groups 
%    The map% of relative homology groups 
%    $$ t_{\displaystyle \omega_{\srange}}\colon H_{i}(\pariteratedmap{M}{t_{\displaystyle \omega_{j}}}{\srange}{\partic};\mathbb{Z})\to  H_{i+2^{\srange}-1}(\pariteratedmap{M}{t_{\displaystyle \omega_{j}}}{\srange}{\partic+2^{\srange}};\mathbb{Z})$$
%    is an isomorphism for $i< D(2, \srange+1 , \partic )$ and a surjection for $i= D(2, \srange+1 , \partic )$.
%    $$ t_{\omega_{\srange}}\colon H_{i}(\text{Conf}_{\partic }(M),K_{\srange -1}(\partic);\mathbb{Z})\to  H_{i+2^{\srange}-1}(\text{Conf}_{\partic +2^{\srange}}(M),K_{\srange-1}(\partic +2^{\srange});\mathbb{Z})$$
%    is an isomorphism for $i< D(2, \srange+1, \partic )$ and a surjection for $i= D(2, \srange+1 , \partic )$.
\end{proposition}
\begin{proof}
The proof is the same as the proof of \cref{integral secondary stability for surface}. The only difference is that we have to apply \cref{higher stab open higher dimensional manifold} and the range $f(\fieldc,\partic)$ from \cref{best known stable ranges} when proving that the $t_{\displaystyle \omega_{1}}$ induces an isomorphism on relative integral homology groups in a range.
\end{proof}
%then the stabilization map $$t_{1} \colon  C_{i}(\text{Conf}_{\partic}(M);\mathbb{F}_{p})\to  C_{i}(\text{Conf}_{\partic +1}(M);\mathbb{F}_{p})$$ induces an isomorphism for $i<k$, so there is not
%the relative chain complex $C_{*}(\text{Conf}_{\partic}(M), \text{Conf}_{\partic -1}(M);\mathbb{F}_{p})$ does not have any 
%non-trivial secondary homological stability in this case.
\begin{remark}
%By using the Bockstein homomorphism $\beta \colon H_{*}(\text{Conf}(\mathbb{R}^{n});\mathbb{F}_{2})\to H_{*-1}(\text{Conf}(\mathbb{R}^{n});\mathbb{F}_{2})$, 
The classes $\omega_{j}\in H_{*}(\text{Conf}(\mathbb{R}^{n});\mathbb{F}_{2})$ for $j>1$ do not lift to classes in $H_{*}(\text{Conf}(\mathbb{R}^{n});\mathbb{Z})$. To see this, consider the mod-$2$ Bockstein spectral sequence on $H_{*}(\text{Conf}(\mathbb{R}^{n});\mathbb{F}_{2})$. If $\omega_{j}$ lifted to a class in $H_{*}(\text{Conf}(\mathbb{R}^{n});\mathbb{Z})$, then it would survive to the $E^{\infty}$-page of this spectral sequence. The $d^{1}$-differential of this spectral sequence is the Bockstein homomorphism $\beta \colon H_{*}(\text{Conf}(\mathbb{R}^{n});\mathbb{F}_{2})\to H_{*-1}(\text{Conf}(\mathbb{R}^{n});\mathbb{F}_{2})$ coming from the short exact sequence $0\to \mathbb{Z}/2 \to \mathbb{Z}/4 \to \mathbb{Z}/2\to 0$. We now show that $\beta(\omega_{j})$ is non-zero for $j>1$, and so $\omega_{j}$ does survive to the $E^{\infty}$-page.

Recall that $\beta$ is dual to the Steenrod square operation $\text{Sq}^{1}$. Therefore, by the Nishida relation for Dyer-Lashof operations (see Property 6 in \textcite[214]{MR0436146}), for any $s>0$, we have that $\beta Q^{s}=(s-1)Q^{s-1}$. It follows that for $j>1$,
$$\beta (\omega_{j})=\beta (Q^{2^{j-1}}\omega_{j-1})= Q^{2^{j-1}-1}\omega_{j-1}.$$
Since $\omega_{j-1}$ is a class of degree $2^{j-1}-1$, by Property 2 in \textcite[213]{MR0436146}), $Q^{2^{j-1}-1}\omega_{j-1}=\omega_{j-1}^{2}$, which is non-zero for $j>1$.
%Let $\beta \colon H_{*}(\text{Conf}(\mathbb{R}^{n});\mathbb{F}_{2})\to H_{*-1}(\text{Conf}(\mathbb{R}^{n});\mathbb{F}_{2})$ denote the Bockstein homomorphism coming from the short exact sequence $0\to \mathbb{Z}/2 \to \mathbb{Z}/4 \to \mathbb{Z}/2\to 0$. If $\omega_{j}$ lifted to a class in $H_{*}(\text{Conf}(\mathbb{R}^{n});\mathbb{Z})$, then 
%According to Proposition 4.2.iv on page 103 of Araki-Kudo, $\beta Q_{i}(x)=(i+\text{deg}(x)+1)Q_{i-1}(x)$
%\footnote{check if $\beta Q^{s}=(s-1)Q^{s-1}$ is a relation in May's ``A general algebraic approach to Steenrod operations'' or maybe page 6 of Browder's homology operations paper. Actually, this is shown in the remark on page 51 of Dyer--Lashof's paper.} Therefore, it is unclear if there is any form of higher-order stability for integral homology.
%We have that for homology with $\mathbb{F}_{2}$ coefficients, I think we have that $\beta Q^{s}=(s-1)Q^{s-1}$ (see, for example, page 58 on  Homology Operations for $H_{\infty}$ and $H_{n}$ Ring Spectra by Mark Steinberger. If $\omega_{j}$ lifts to a class in $H_{*}(\text{Conf}(\mathbb{R}^{n});\mathbb{Z})$, then we need that $\omega_{j}$ lies on the $E^{\infty}$ page of the Bockstein spectral sequence. The $d_{1}$ differential of the Bockstein spectral sequence corresponds to the Bockstein differential associated to the short exact sequence $1\to \mathbb{Z}/p \to \mathbb{Z}/p^{2} \to \mathbb{Z}/p $. Since $\beta \omega_{j}= \omega_{j}^{2}\neq 0$, the class doesn't live on the $E^{\infty}$ page.
\end{remark}
We now prove \cref{secondary stability open case}.
\begin{proof}[Proof of \cref{secondary stability open case}]
\cref{secondary stability open case} immediately follows from \cref{higher stab open surface}, \cref{integral secondary stability for surface}, \cref{higher stab open higher dimensional manifold}, and \cref{integral secondary stab open manifold} in the case $\srange=1$.
%The part of  homology with $\mathbb{F}_{2}$ coefficients
%\cref{secondary stability open case} involving homology follows from 
%Secondary stability for homology with $\mathbb{F}_{2}$ coefficients follows from 
\end{proof}

\section{Models for the Configuration Spaces of a Manifold and for the Configuration Spaces of a Disk}\label{sec:various models of conf}
In this section, we give various models of configuration spaces. 
In \cref{models of conf space}, we begin by introducing a topological monoid and a topological module over it 
%in order to model the configuration space of a space of the form $N\times (0,\infty)$ and of a manifold with non-empty boundary, respectively,
and we use these spaces to produce a model of $\text{Conf}(M)$. Then in \cref{stabilization map construction}, we work with $\text{Conf}(D^{n})$, now viewed as a module over $\text{Conf}(S^{n-1}\times (0,\infty))$.
%We construct a model for $\text{Conf}(D^{n})$ different from our model for $\text{Conf}(M)$.  
Roughly speaking, we use free modules over $\text{Conf}(S^{n-1}\times (0,\infty))$ to construct a model for $\text{Conf}(D^{n})$. This model for $\text{Conf}(D^{n})$ has the additional structure of being a module over $\text{Conf}(S^{n-1}\times (0,\infty))$ and is different from our model for $\text{Conf}(M)$.
%
%construct a different model for $\text{Conf}(D^{n})$, having 
%viewed as a module (up to homotopy) over $\text{Conf}(S^{n-1}\times (0,\infty))$. Our model for $\text{Conf}(D^{n})$ will 
%
%one that reflects the structure of $\text{Conf}(D^{n})$ as a module (up to homotopy) over $\text{Conf}(S^{n-1}\times (0,\infty))$, 
%having the structure 
%viewed as a module (up to homotopy) over $\text{Conf}(S^{n-1}\times (0,\infty))$, using free modules over $\text{Conf}(S^{n-1}\times (0,\infty))$, roughly speaking.
%, viewed now as a module (up to homotopy) over $\text{Conf}(S^{n-1}\times (0,\infty))$, roughly speaking
%We introduce this model for $\text{Conf}(D^{n})$ because 
Ultimately, the stabilization map we construct on $\nch{\text{Conf}(M)}$  comes from a stabilization map on $\nch{\text{Conf}(D^{n})}$ and we will define a stabilization map on the latter by using this other model for $\text{Conf}(D^{n})$.
%for $\text{Conf}(D^{n})$ to define a stabilization map on we will use it to define a stabilization map on $\nch{\text{Conf}(D^{n})}$
%the stabilization map  con on $\text{Conf}(M)$ depends 
%A Model of \texorpdfstring{$\text{Conf}(M)$}
\subsection{Models for the Configuration Spaces of a Manifold}\label{models of conf space}
In this subsection, we describe two related models of configuration spaces for two different contexts.
%We develop these models in order to formulate secondary and higher order homological stability in\footnote{cref} and to construct stabilization maps for the configuration space of a closed manifold in\footnote{cref}. 
The first model associates to a space $\smon$ a topological monoid $\mon$ whose underlying space is homotopy equivalent to $\text{Conf}(\smon\times (0,\infty ))$. The second model associates to a manifold $M_{1}$ with boundary and an ordered $\srange$-tuple $(\smon_{1},\ldots, \smon_{\srange})$ of disjoint subspaces contained in $\partial M_{1}$, a space $\rimodc[\sqcup_{j}\smon_{j}]{M_{1}}{\srange}$ whose underlying space is homotopy equivalent to $\text{Conf}(M_{1})$ and which is a module over $\prod_{j=1}^{\srange}\mon[\smon_{j}]$.
%For the third model, given two manifolds $M_{1}$ and $M_{2}$ with boundary, a space $\smon$ contained in $\partial M_{1}$ and $\partial M_{2}$, and a space $M_{0}\colonequals M_{1}\cup_{\smon} M_{2}$, we associate a ``floppy configuration space'' $\fconf$ which is weakly homotopy equivalent to $\text{Conf}(M_{0})$.
We also construct a semi-simplicial space
%simplicial space 
$\iterxbarc{M_{1}}{\smon}{M_{2}}{\bullet}{\srange}$
%$\xbarc{M_{1}}{\smon}{M_{2}}{\bullet}$ 
which is
%. This simplicial space is 
built from the spaces $\prod_{j=1}^{\srange}\mon[\smon_{j}]$, $\rimodc[\sqcup_{j}\smon_{j}]{M_{1}}{\srange}$, and $\rimodc[\sqcup_{j}\smon_{j}]{M_{2}}{\srange}$. 
%$\colonequals B_{\bullet}\big (, \rimodc{M_{2}}{\srange}\big)$
%$\mon,$ $\modc{M_{1}},$ and $\modc{M_{2}}$. 
Its geometric realization is a space which is weakly equivalent to $\text{Conf}(M_{1}\cup_{\sqcup_{j}\smon_{j}}M_2)$ and we will use this semi-simplicial space to define a stabilization map in \cref{new stab maps for conf(M)}.
\begin{comment}
\begin{definition}
Suppose that $\smon$ is a space. For $t\in \mathbb{R}_{\geq 0}$ or $t=\infty$, let $\smon_{t}\colonequals {\smon}\times \left [0,t)$.
\end{definition}
\end{comment}
\begin{definition}\label{shift map}
Suppose that $\smon$ is a space and that $a\in \mathbb{R}_{\geq 0}$. Let $\text{sh}_{a} \colon \text{Conf}(\smon\times (-\infty,\infty ))\to  \text{Conf}(\smon\times (-\infty,\infty ))$ be the ``shift by $a$'' map
%map on $\text{Conf}(\smon\times (-\infty,\infty ))$  to be
\begin{align*}
\text{sh}_{a} \colon \text{Conf}(\smon\times (-\infty,\infty ))&\to  \text{Conf} ({\smon}\times (-\infty,\infty))
\\ \{(q_{1},u_{1}),\ldots , (q_{k},u_{k})\}&\mapsto \{(q_{1},a + u_{1}),\ldots , (q_{k},a + u_{k})\}.
\end{align*}
%\begin{align*}
%\text{sh}_{a} \colon \text{Conf}(\smon\times \left [0,\infty \right ))&\to  \text{Conf}(\smon\times \left [0,\infty \right))
%\\ \{(n_{1},\neck_{1}),\ldots , (n_{k},\neck_{k})\}&\mapsto \{(n_{1},a+\neck_{1}),\ldots , (n_{k},a+\neck_{k})\}.
%\end{align*}
\begin{comment}
$$\text{sh}_{a} \colon \text{Conf}(\smon_{\infty})\to  \text{Conf}(\smon_{\infty})$$ to be the map which sends $\{(q_{1},u_{1}),\ldots (q_{k},u_{k})\}\in\text{Conf}(\smon_{\infty})$ to $\{(q_{1},a+u_{1}),\ldots (q_{k},a+u_{k})\}$.
\end{comment}
\end{definition}
\begin{definition}
Suppose that $\smon$ is a space and  $a\in \mathbb{R}\setminus\{0\}$. Let $\text{sc}_{a} \colon \text{Conf}(\smon\times (-\infty,\infty))\to  \text{Conf}(\smon\times (-\infty,\infty))$ be 
%the Define 
the ``scale by $a$'' map
%on $\text{Conf}(\smon\times (-\infty,\infty))$  to be
\begin{align*}
\text{sc}_{a} \colon \text{Conf}(\smon\times (-\infty,\infty ))&\to  \text{Conf} ({\smon}\times (-\infty,\infty))
\\ \{(q_{1},u_{1}),\ldots , (q_{k},u_{k})\}&\mapsto \{(q_{1},a\cdot u_{1}),\ldots , (q_{k},a\cdot u_{k})\}.
\end{align*}
\end{definition}
\begin{definition}\label{Moore monoid}
Let $\smon$ be a space. Define the \textbf{strict monoid replacement} $\mon$ for the configuration space $\text{Conf}(\smon\times (0,\infty))$ to be the subspace of $\text{Conf}(\smon\times (0,\infty))\times \left [0,\infty \right )$ consisting of pairs $(\moconfigtwo , \neck)$ with $\moconfigtwo \in \text{Conf}(\smon\times (0,\neck ))\subset \text{Conf}(\smon\times \mathbb{R})$ for $\neck>0$ and $(\{\emptyset \},0)$ for $\neck=0$.
%Let $\monp \subset \mon$ denote the subspace of $\mon$ consisting of pairs $(\moconfigtwo , \neck)$ with $\neck>0$.

Addition of elements in $\mon$ is given by
\begin{align*}
-\cdot -  \colon \mon\times \mon&\to  \mon
\\\big ( (\moconfigtwo_{1},\neck_{1}), (\moconfigtwo_{2}, \neck_{2})\big)&\mapsto 
%(\moconfigtwo_{1},\neck_{1})\cdot (\moconfigone_{2},\neck_{2})\colonequals 
(\moconfigtwo_{1}\cup \text{sh}_{\neck_{1}}(\moconfigtwo_{2}) ,\neck_{1}+\neck_{2}).
\end{align*}
\begin{comment}
where $sh_a  \colon  \mon\to  \mon$ comes from the ``shift by $a$'' map on $Y$ that sends a point $(p,t)\in Y\times \left [0,\infty )$ to $(p,t+a)$, with the map naturally extending to $\mon$.
\end{comment}
The space $\mon$ is a topological associative monoid with unit $(\{\emptyset \},0)$. Given an $\srange$-tuple of spaces $(\smon_{1},\ldots,\smon_{\srange})$, we use the notation $(\vec{\moconfigone},\vec{\neck})$ to denote an element of $\prod_{j=1}^{\srange}\mon[\smon_{j}]$, with $$\vec{\moconfigone}\colonequals(\vec{\moconfigone}(1),\ldots,\vec{\moconfigone}(\srange))\in\prod_{j=1}^{\srange}\text{Conf}(\smon_{j}\times (0,\infty))\text{ and }\vec{\neck}\colonequals(\vec{\neck}(1),\ldots,\vec{\neck}(\srange))\in \left[0,\infty\right)^{\srange}$$ satisfying the condition that $(\vec{\moconfigone}(j),\vec{\neck}(j))$ is an element of $\mon[\smon_{j}]$ for all $j=1,\ldots,\srange$.
%and $\monp$ is a topological semigroup.
\end{definition}
\begin{remark}
The motivation behind \cref{Moore monoid} is a construction, due to Moore, which turns the loop space $\Omega N$ of a space $\smon$ into a strict monoid instead of a monoid up to homotopy (see \textcite[p.~529--530]{MR1361898}
for the definition of the Moore loop space).
See \textcite[p.~9]{hatcher2014short} for applications of the strict monoid replacement $\mon$ to the Barrat--Priddy--Quillen theorem.
%The construction of \cref{Moore monoid} has appeared before--for example \textcite[p.~9]{hatcher2014short}.
 %To our knowledge, \cref{Moore monoid} first appears in \textcite[p.~9]{hatcher2014short}.\footnote{Should I drop this sentence. It seems kind of awkward.} 
 %\footnote{I think a Moore monoid model for configuration spaces first appears in Jeremy-Sander-Inbar's paper, but I can't find where. Should I mention their construction?}
\end{remark}
\begin{definition}\label{rev notation for iterated configuration space model}
Suppose that $M$ is a manifold with boundary and that  $(\smon_{1},\ldots, \smon_{\srange})$ is an $\srange$-tuple of disjoint subspaces $\smon_{i}\subset\partial M$. Let $\smon$ denote $\bigsqcup\limits_{j=1}^{\srange}\smon_{j}$.
%\footnote{I probably should write some remark explaining that $\rimodc[S^{n-1}]{\bar{D}^{n}}{2}$ and $\rimodc[D^{n-1}\cup \bar{D}^{n-1}]{\bar{D}^{n}}{2}$ are the same spaces, but the latter has right two module structures} 
%If $\srange=1$, we drop the parentheses and set $\smon\colonequals \smon_{1}$.\footnote{maybe drop this sentence} 
Define the \textbf{strict module replacement} $\rimodc[\smon]{M}{\srange}$ over the strict monoid replacement $\prod\limits_{j=1}^{\srange}\mon[\smon_{j}]$ for the configuration space $\text{Conf}(M)$ to be the subspace of 
$\text{Conf}\Big(M\cup_{\smon}\big({\smon}\times \left[0,\infty\right)\big)\Big)\times \left[0,\infty\right)^{\srange}$ consisting of elements 
$(\moconfigone , \vec{\neck})$, with $\moconfigone\in \text{Conf}\Big(M\cup_{\smon}\big({\smon}\times \left[0,\infty\right)\big)\Big)$ and $\vec{\neck}=(\vec{\neck}(1),\ldots, \vec{\neck}(\srange))\in [0,\infty)^{\srange}$ satisfying the condition that $$\moconfigone\in \text{Conf}\Big(M\cup_{\smon}\big(\bigsqcup\limits_{j=1,\ldots,\srange, \text{ with } \vec{\neck}(j)\neq 0}^{\srange} \smon_{j}\times [0,\vec{\neck}(j))\big)\setminus (\bigcup\limits_{k=1\ldots,\srange, \text{ with } \vec{\neck}(k)=0}^{\srange} \smon_{k})\Big).$$ If $\srange=1$, 
%we drop the superscript $\srange$ from $\rimodc[\smon]{M}{\srange}$ (instead of writing $\rimodc{M}{1}$) and 
we denote an element of $\modc{M}$ by $(\moconfigone , \neck)$. 
%(instead writing of $(\moconfigone , \vec{\neck})$).\footnote{probably need to redo notation}
%We will use the notation $\vec{\neck}(i)$ to denote the $i$-th component of $\vec{\neck}$ (in this case, $\vec{\neck}(i)=\neck_{i}$).
%for $(\neck_{1},\ldots, \neck_{\srange})\neq (0,\ldots,0)$ and $\moconfigone\in\text{Conf}(\smon\setminus \bar{E}_{\srange})$ for $(\neck_{1},\ldots, \neck_{\srange})= (0,\ldots,0)$.

%Let $\embed_{j}\colon S^{n-1}$

Let $$\embed_{j} \colon \smon_{j}\times (0,\infty)\hookrightarrow \smon_{j}\times (0,\infty)\subset M\cup_{\smon}( {\smon}\times \left [0,\infty\right))$$ be the inclusion under the identity map and let $$\text{Conf}(\embed_{j}) \colon \text{Conf}(\smon_{j}\times (0,\infty))\to \text{Conf}\big(M\cup_{\smon}( {\smon}\times \left [0,\infty\right))\big)$$ denote the induced map of configuration spaces.  
%\footnote{I might want to put the map i into its own notation section. I need to figure out what the notation for the induced map on configuration spaces should be-maybe Conf(\embed) }
The map%\footnote{should I use different notation for a configuration of $\mon$ and a configuration of $\modc{M}$?}
\begin{align*}
m_{r} \colon  \rimodc[\smon]{M}{\srange}\times \prod_{j=1}^{\srange}\mon[\smon_{j}] &\to  \rimodc[\smon]{M}{\srange}
\\\big( (\moconfigone_{1},\vec{\lambda}_{1} ), (\vec{\moconfigone}_{2},\vec{\lambda}_{2} )\big)&\mapsto \big(\bigsqcup_{j=1}^{\srange} \text{Conf}(\embed_{j})\big (\text{sh}_{\displaystyle\vec{\neck}_{1}(j)}(\vec{\moconfigtwo}_{2}(j))\big )\cup\moconfigone_{1}, \vec{\lambda}_{1}+\vec{\lambda}_{2}\big)
\end{align*}
gives the space $\rimodc[\smon]{M}{\srange}$ the structure of a right module over the monoid $\prod_{j=1}^{\srange}\mon[\smon_{j}]$ (see \cref{fig:iterated mod-conf}).
\begin{figure}
  \centering
  \includegraphics[scale=.68]{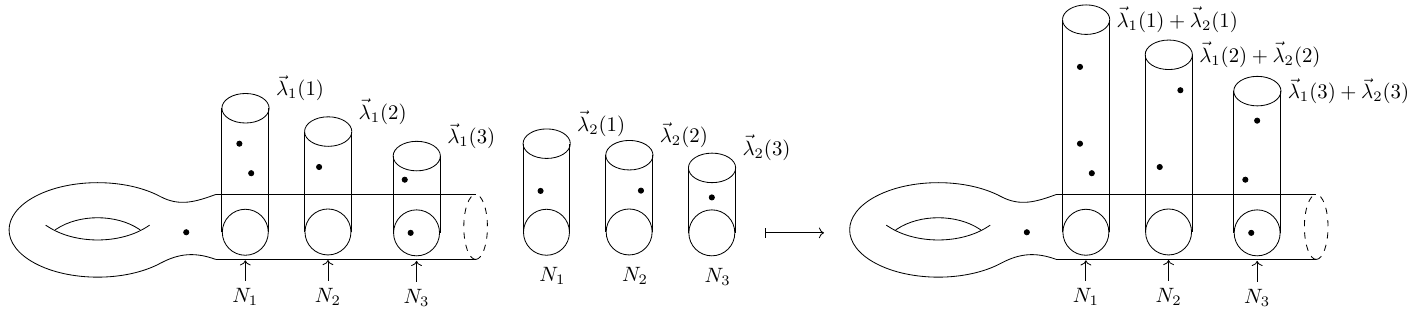}
  \caption{The right $\prod_{j=1}^{\srange}\mon[\smon_{j}]$-module structure of $\rimodc[\sqcup_{j}\smon_{j}]{M}{\srange}$}
  \label{fig:iterated mod-conf}
\end{figure}
Similarly, the map%\footnote{I need to fix swap the cylinder and the circle in the diagram}
\begin{align*}
m_{l} \colon  \big(\prod_{j=1}^{\srange}\mon[\smon_{j}]\big)\times\rimodc[\smon]{M}{\srange} &\to  \rimodc[\smon]{M}{\srange}
\end{align*}
%\begin{align*}
$$\big( (\vec{\moconfigone}_{1},\vec{\lambda}_{1} ), (\moconfigone_{2},\vec{\lambda}_{2} )\big)\mapsto \big(\bigsqcup_{j=1}^{\srange} \text{Conf}(\embed_{j})\big (\text{sh}_{\displaystyle\vec{\neck}_{1}(j) + \vec{\neck}_{2}(j) }\text{sc}_{-1}(\vec{\moconfigtwo}_{1}(j) )\big )\cup\moconfigone_{2}, \vec{\lambda}_{1}+\vec{\lambda}_{2}\big)$$
gives the space $\rimodc[\smon]{M}{\srange}$ the structure of a left module over the monoid $\prod_{j=1}^{\srange}\mon[\smon_{j}]$. When $\srange=1$, 
%we will drop the subscript for $\embed$ and 
we will write $\embed$ instead of $\embed_{1}$.
\end{definition}
\begin{definition}
%Suppose that $(\moconfigtwo ,\neck )$ is an element of $\mon$ or $\modc{M}$. 
Given an element $(\moconfigtwo,\neck)$ of $\mon$ or an element $(\moconfigtwo,\vec{\neck})$ of $\rimodc[\smon]{M}{\srange}$, we define a \emph{point} of $(\moconfigtwo ,\neck)$ or $(\moconfigtwo,\vec{\neck})$ to be any point of the configuration $\moconfigtwo$.
%Suppose that $\mon$ is a strict monoid replacement and $(\moconfigtwo ,\neck )$ is an element of $\mon$. Define a \emph{point} of $(\moconfigtwo ,\neck)$ to be any point of the configuration $\moconfigtwo\in\text{Conf}(\smon\times (0,\neck))$.
%\footnote{Carefully explain the notation of a point $(q,u)$}
%Define the \emph{wall} of $(\moconfigtwo ,\neck)$ to be the set ${\smon}\times \{\neck\}$ in ${\smon}\times (0,\infty)$ (see \cref{fig:wall-of-replacement}). 
%If  $(\moconfigtwo ,\neck)$ is an element of $\modc{M}$
%a strict module replacement $\modc{M}$ over $\mon$ and a point $(\moconfigtwo ,\neck)$, 
%a point of $(\moconfigtwo ,\neck )$ is defined similarly. 
We will denote a point of $(\moconfigtwo,\neck)$ or $(\moconfigone,\vec{\neck})$ by $(q,u)$, with $u\in \left[0,\neck\right)$ and $q\in \smon$ for $u>0$ and $q\in M$ for $u=0$.
%\footnote{should I put this sentence outside the definition environment}
%\footnote{should be wall of $(\moconfigtwo,\neck_{1})$ in picture , not wall of the configuration}
\begin{figure}[htb]
  \centering
  \includegraphics{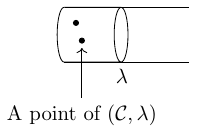}
  %\caption{A point of $(\moconfigtwo ,\neck)$}
  \label{fig:wall-of-replacement}
\end{figure}
\begin{comment}
Fix\footnote{I need to rewrite this paragraph because I don't really need to deal with $X_{p}$} an element $\xi = ((y_{0},t_{0}),(y_{1},t_{1}),\ldots ,(y_{p},t_{p}),(y_{p+1},t_{p+1}))\in B_{p}(\text{MConf}(M\setminus\bar{D}), \text{Conf}(Y),\text{MConf}(D))$, where $(y_{0},t_{0})\in \text{Conf}(M_{1})$, $(y_{i},t_{i})\in \text{Conf}(Y)$ and $(y_{p+1},t_{p+1})\in \text{Conf}(M_{2})$. The \textit{walls} of the configuration $\xi$ will refer to the boundary  ${\smon}\times \{t_{i}\}\subset \smon_{t_{i}}$ of $\text{Conf}((M_{2})_{t_{0}}), \text{Conf}(\smon_{t_{i}}),$ and $\text{Conf}((M_{2})_{t_{p+1}})$ for $i=0,\ldots , p+1$ and a \textit{point} of the configuration $\xi$ will refer to any point of the configurations $y_{0},\ldots , y_{p+1}$ (for example, if $(y_{1},t_{1})=(\{ (p_{1}, s_{1}),\ldots , (p_{\partic}, s_{\partic})\},t_{1})\in \mon$, then any $(p_{i}, s_{i})$ is considered a point of the configuration $\xi$).
\end{comment}
\end{definition}
\begin{lemma}\label{replacements are homotopy equivalent to config spaces}
%Suppose that $M$ is a manifold $M$ with boundary and that $\smon$ is contained in $\partial M$. 
Suppose that $M$ is a manifold with boundary and that $(\smon_{1},\ldots, \smon_{\srange})$ is an $\srange$-tuple of disjoint subspaces $\smon_{i}\subset\partial M$. Let $\smon$ denote $\bigsqcup_{j=1}^{\srange}\smon_{j}$.
%\footnote{probably need to take $\bigcup$ instead of $\bigsqcup$ here and elsewhere} 
The spaces $\mon$ and $\rimodc[\smon]{M}{\srange}$ are homotopy equivalent to $\text{Conf}(\smon\times (0,1))$ and $\text{Conf} \big (M\cup_{\smon} ({\smon}\times \left [0,1\right ))\big)$, respectively.
\end{lemma}
\begin{proof}
There is a deformation retract of the Moore loop space to the loop space (see, for example, the proof of \textcite[Proposition 5.1.1]{MR1361898}) and a similar deformation retract of $\mon$ onto $\text{Conf}(\smon\times (0,1))\times\{1\}$ works in our setting.
%that deformation retract can work in our setting. We give a similar deformation retract of $\mon$ onto $\text{Conf}(\smon\times (0,1))\times \{1\}$.
%and then give a homotopy equivalence $\text{Conf}(\smon\times \left [0,1))\times \{1\}\to  \text{Conf}(\smon\times (0,1))\times \{1\}$. 
%The argument for $\modc{M}$ is similar so we omit it.
The argument for $\rimodc[\smon]{M}{\srange}$ is similar so we omit it.

Let $\text{MConf}^{\geq 1}(\smon)\subset \mon$ denote the subspace of pairs $(\moconfigtwo ,\neck)\in\mon$ with $\neck \geq 1$. A deformation retract $H \colon \left [0,1\right ]\times\mon\to  \mon$ to $\text{MConf}^{\geq 1}(\smon)$ is given by the following formula:
\[
H\big (s,(\moconfigtwo ,\neck )\big )=\begin{cases} 
      (\moconfigtwo , s+\neck ) &s+\neck\leq 1 \\
      (\moconfigtwo ,1) & \neck\leq 1 \text{ and } s+\neck\geq 1\\ 
      (\moconfigtwo ,\neck) & \neck\geq 1.
   \end{cases}
   \]
%Let $\text{MConf}^{\geq 1}_{+}(\smon)\subset \text{MConf}^{\geq 1}(\smon)$ denote the subspace $\text{MConf}^{\geq 1}(\smon)\cap (\text{Conf}(\smon\times (0,\infty ))\times [1,\infty ))$, which as a set consists of pairs $(y,t)\in \text{MConf}^{\geq 1}(\smon)$ with $y\in \text{Conf}(\smon\times (0,t))$.
%By pushing points in a configuration away from the boundary ${\smon}\times \{0\}$, the inclusion $\text{MConf}^{\geq 1}_{+}(\smon)\hookrightarrow \text{MConf}^{\geq 1}(\smon)$ is a homotopy equivalence.

Let $f \colon \left [0,1\right ]\times \left [1,\infty \right)\to  \left [0,\infty\right )$ be the function sending $(s,\neck )$ to $\displaystyle\frac{(1-s)\neck+s}{\neck}$. The map 
\begin{align*}
G\colon \left [0,1 \right ]\times \text{MConf}^{\geq 1}(\smon)& \to  \text{MConf}^{\geq 1}(\smon)
\\(s,(\moconfigtwo ,\neck))&\mapsto (\text{sc}_{f(s,\neck )}(\moconfigtwo ), (1-s)\neck+s)
\end{align*}
is a deformation retract of $\text{MConf}^{\geq 1}(\smon)$ to $\text{Conf}(\smon\times (0,1))\times \{1\}$.
%$G(s,(y,t))= (\text{sc}_{f(s,t)}(y), (1-s)t+t)$ where $f(s,t)\colonequals \frac{(1-s)t+s}{t}$.
\end{proof}
\begin{definition}[Bar Construction]\label{def: bar construction}
Suppose that $V$ is a monoid, $U$ is a right $V$-module, and $W$ is a left $V$-module. The \textbf{bar construction} $B_{\bullet}(U,V,W)$ is the semi-simplicial space whose $p$-simplices $B_{p}(U,V,W)$ are $U\times V^{p}\times W$. Given $u\in U$, $v_{1},\ldots , v_{p}\in V$, and $w\in W$, the face maps $d_{i} \colon B_{p}(U,V,W)\to  B_{p-1}(U,V,W)$ for $0\leq i \leq p$ send $(u,v_{1},\ldots , v_{p},w)$ to 
%are (for $e\in E, (y_{1},\ldots , y_{p})\in Y^{p}, f\in F$)\footnote{I don't know if I should drop the stuff in parentheses}
\begin{align*}
  d_{0}(u,v_{1},\ldots , v_{p},w)&=(m_{r}( u, v_{1}),v_{2},\ldots , v_{p},w)\\
  d_{i}(u,v_{1},\ldots , v_{p},w)&=(u, v_{1},\ldots ,v_{i-1},v_{i}\cdot v_{i+1},v_{i+2},\ldots , v_{p},w) \text{ for } 0<i<p \\ 
  d_{p}(u,v_{1},\ldots , v_{p},w)&=(u, v_{1},\ldots , v_{p-1}, m_{l}(v_{p}, w)),  
\end{align*}
where $m_{r}\colon U\times V\to U$ is the right module map of $U$ and $m_{l}\colon V\times W\to W$ is the left module map of $W$.
%If\footnote{maybe drop this sentence} 

%If in addition $V$ is a monoid with unit $1$, then $B_{\bullet}(U,V,W)$ has the structure of a simplicial space with degeneracy maps $s_{i} \colon B_{p}(U,V,W)\to  B_{p+1}(U,V,W)$ given by
%\begin{align*}
%    s_{0}(u,v_{1},\ldots , v_{p},w) &=(u,1,v_{1},\ldots, v_{p},w)\\
%    s_{i}(u,v_{1},\ldots , v_{p},w) &=(u,v_{1},\ldots ,v_{i},1,v_{i+1},\ldots , v_{p},w) \text{ for } 0< i\leq p .
%\end{align*}
\end{definition}
Suppose that $M_{1}$ and $M_{1}$ are manifolds with boundary and that $(\smon_{1},\ldots, \smon_{\srange})$ is an $\srange$-tuple of disjoint subspaces $\smon_{i}$ contained in $\partial M_{1}$ and $\partial M_{2}$. Let $\smon\colonequals \bigsqcup_{j=1}^{\srange}\smon_{j}$.
Since $\rimodc[\smon]{M_{1}}{\srange}$ is a left module over the monoid $\prod_{j=1}^{\srange}\mon[\smon_{j}]$, and $\rimodc[\smon]{M_{2}}{\srange}$) is a right module over it, we can take the bar construction $$B_{\bullet}(\rimodc[\smon]{M_{1}}{\srange}, \prod_{j=1}^{\srange}\mon[\smon_{j}], \rimodc[\smon]{M_{2}}{\srange}).$$
\begin{definition}\label{iterated bar construction model for conf}
%Let $M$ be a manifold of dimension $n$ and let $\srange$ be a positive integer. Fix $\srange$ disjoint closed disks $\bar{D}^{n}_{1},\ldots,\bar{D}^{n}_{\srange}$ in $M$ and let $E_{\srange}\colonequals \bigcup_{j=1}^{\srange}\bar{D}^{n}_{j}$. 
Suppose that $M_{1}$ and $M_{2}$ are two manifolds with boundary and that $(\smon_{1},\ldots, \smon_{\srange})$ is an $\srange$-tuple of disjoint subspaces contained in $\partial M_{1}$ and $\partial M_{2}$. Let $\smon$ denote $\bigsqcup_{j=1}^{\srange}\smon_{j}$. Define $\iterxbarc{M_{1}}{\smon}{M_{2}}{\bullet}{\srange}$ to be the semi-simplicial space
$$\iterxbarc{M_{1}}{\smon}{M_{2}}{\bullet}{\srange}\colonequals B_{\bullet}\big (\rimodc[\smon]{M_{1}}{\srange},\prod_{j=1}^{\srange}\mon[\smon_{j}], \rimodc[\smon]{M_{2}}{\srange}\big).$$
Let $\piterxbarc{M_{1}}{\smon}{M_{2}}{\bullet}{\srange}{\partic}\subset \iterxbarc{M_{1}}{\smon}{M_{2}}{\bullet}{\srange}$ (respectively, $\iterxpbarc{M_{1}}{\smon}{M_{2}}{\bullet}{\srange} \subset \iterxbarc{M_{1}}{\smon}{M_{2}}{\bullet}{\srange}$) denote the semi-simplicial space whose $p$-simplices consist of elements
$\big ((\moconfigtwo_{0},\vec{\neck}_{0}),(\vec{\moconfigtwo}_{1},\vec{\neck}_{1}),\ldots,(\vec{\moconfigtwo}_{p},\vec{\neck}_{p}), (\moconfigtwo_{p+1},\vec{\neck}_{p+1})\big )\in\iterxbarc{M_{1}}{\smon}{M_{2}}{p}{\srange}$ such that 
%the cardinality of all the points of $(\moconfigtwo_{0},\vec{\neck}_{0}),(\vec{\moconfigtwo}_{1},\vec{\neck}_{1}),\ldots,(\vec{\moconfigtwo}_{p},\vec{\neck}_{p}),$ and $(\moconfigtwo_{p+1},\vec{\neck}_{p+1})$ is equal to $\partic$ 
$\vert \moconfigtwo_{0}\vert + \vert \vec{\moconfigtwo}_{1}\vert+\ldots+ \vert \vec{\moconfigtwo}_{p}\vert+\vert \moconfigtwo_{p+1}\vert=\partic$
(respectively, $\sum_{j=0}^{p+1}\vec{\neck}_{j}(i)=1$ for all $i=1,\ldots,\srange$).
%Let $\iterxpbarc{M_{1}}{\smon}{M_{2}}{\bullet}{\srange}\subset \iterxbarc{M_{1}}{\smon}{M_{2}}{\bullet}{\srange}$ denote the (semi-)simplicial space whose $p$-simplices consist of elements
%$\big ((\moconfigtwo_{0},\vec{\neck}_{0}),(\vec{\moconfigtwo}_{1},\vec{\neck}_{1}),\ldots,(\vec{\moconfigtwo}_{p},\vec{\neck}_{p}), (\moconfigtwo_{p+1},\vec{\neck}_{p+1})\big )\in\iterxbarc{M_{1}}{\smon}{M_{2}}{p}{\srange}$ such that the sum $\sum_{j=0}^{p+1}\vec{\neck}_{j}(i)$ is equal to $1$ for all $i=1,\ldots,\srange$.
%%$\big ((\moconfigtwo_{0},\neck_{0}),\ldots , (\moconfigtwo_{p+1},\neck_{p+1})\big )\in\xbarc{M_{1}}{\smon}{M_{2}}{p}$ such that the sum $\sum_{i=0}^{p+1}\neck_{i}$ is equal to $1$.
%$((\moconfigtwo_{1,}' ,\neck_{1}'),\ldots,(\moconfigtwo_{\srange}',\neck_{\srange}')),\ldots, $
%$\big ((\moconfigtwo_{0},\neck_{0}),\ldots , (\moconfigtwo_{p+1},\neck_{p+1})\big )\in\xbarc{M_{1}}{\smon}{M_{2}}{p}$ such that the sum $\sum_{i=0}^{p+1}\neck_{i}$ is equal to $1$.
\end{definition}
%\begin{definition}\label{bar construction notation}
%Suppose that $M_{1}$ and $M_{1}$ are manifolds with boundary and that $(\smon_{1},\ldots, \smon_{\srange})$ is an ordered collection of disjoint subspaces $\smon_{i}$ contained in $\partial M_{1}$ and $\partial M_{2}$. Let $\smon\colonequals \bigsqcup_{j=1}^{\srange}\smon_{j}$. 
%Suppose that $M_{1}$ and $M_{2}$ are manifolds with boundary and that $\smon$ is contained in $\partial M_{1}$ and $\partial M_{2}$. 
%Let $M$ be the space $M_{1}\cup_{\smon}M_{2}$. Let 
%\footnote{You suggested maybe writing $X_{\bullet}=B_{\bullet}$ and $X=|B_{\bullet}=|=B$. I didn't write $X_{\bullet}$ because I thought this notation could sometimes be ambiguous. I didn't abbreviate the geometric realization because I switch between the thick and thin geometric realization} 
%Let $\xbarc{M_{1}}{\smon}{M_{2}}{\bullet}$ denote the (semi-)simplicial space $$\xbarc{M_{1}}{\smon}{M_{2}}{\bullet}\colonequals\barc{M_{1}}{\smon}{M_{2}}.$$
%Let $\xpbarc{M_{1}}{\smon}{M_{2}}{\bullet}\subset \xbarc{M_{1}}{\smon}{M_{2}}{\bullet}$ denote the (semi-)simplicial space whose $p$-simplices consist of elements $\big ((\moconfigtwo_{0},\neck_{0}),\ldots , (\moconfigtwo_{p+1},\neck_{p+1})\big )\in\xbarc{M_{1}}{\smon}{M_{2}}{p}$ such that the sum $\sum_{i=0}^{p+1}\neck_{i}$ is equal to $1$.
%$B_{\bullet}(\modc{M_{1}},\monp,\modc{M_{2}})$.
%\end{definition}
We want to compare the geometric realizations $\Vert \iterxpbarc{M_{1}}{\smon}{M_{2}}{\bullet}{\srange}\Vert$ and $\Vert \iterxbarc{M_{1}}{\smon}{M_{2}}{\bullet}{\srange}\Vert$. We will need the following well-known fact concerning the geometric realization of a semi-simplicial space: 
%[{e.g. \textcite[Proposition A.1.ii]{MR353298}}]
\begin{theorem}[{e.g. \textcite[Theorem 2.2]{MR3995026}}]\label{levelwise weak homotopy equivalence}
Let $g_{\bullet} \colon Y_{\bullet}\to  Z_{\bullet}$ be a map of semi-simplicial spaces which is a levelwise weak homotopy equivalence, i.e. $g_{p} \colon Y_{p}\to  Z_{p}$ is a weak homotopy equivalence for all $p$. Then $\Vert g_{\bullet}\Vert  \colon  \Vert Y_{\bullet}\Vert\to  \Vert Z_{\bullet}\Vert$ is a weak homotopy equivalence.
%\footnote{Do I need to worry about the case that the fibers of $f$ are empty? Technically, the empty set is not weakly contractible. I can slightly modify the definition of $\modc{M}$ so that the fibers of $f$ are not empty, but I'm worried there is a chance this risks messing up the proof that $f$ is a microfibration.}
\end{theorem}
\begin{lemma}\label{postive mconf is equivalent to mconf}
Let $g_{\bullet}\colon \iterxpbarc{M_{1}}{\smon}{M_{2}}{\bullet}{\srange}\to \iterxbarc{M_{1}}{\smon}{M_{2}}{\bullet}{\srange}$ be the map of semi-simplicial spaces coming from the natural inclusion $\iterxpbarc{M_{1}}{\smon}{M_{2}}{p}{\srange}\hookrightarrow\iterxbarc{M_{1}}{\smon}{M_{2}}{p}{\srange}$. The map on geometric realizations $$\Vert g_{\bullet}\Vert\colon \Vert\iterxpbarc{M_{1}}{\smon}{M_{2}}{\bullet}{\srange}\Vert\to \Vert\iterxbarc{M_{1}}{\smon}{M_{2}}{\bullet}{\srange}\Vert$$ is a weak homotopy equivalence.
\end{lemma}
\begin{proof}
%Let $g\colon \left[0,1\right]$
For simplicity, we prove this result only for the case $\srange=1$. Given $\big((\moconfigtwo_{0} ,\neck_{0}),\ldots, (\moconfigtwo_{p+1} ,\neck_{p+1})\big)\in \iterxbarc{M_{1}}{\smon}{M_{2}}{p}{1}$, let $\neck\colonequals \sum_{j=0}^{p+1}\neck_{j}$.  
Let $f_{p}\colon \iterxbarc{M_{1}}{\smon}{M_{2}}{p}{1}\to \iterxpbarc{M_{1}}{\smon}{M_{2}}{p}{1}$ denote the map
\begin{align*}
    f_{p}\colon \iterxbarc{M_{1}}{\smon}{M_{2}}{p}{1}&\to \iterxpbarc{M_{1}}{\smon}{M_{2}}{p}{1}\\
    \big((\moconfigtwo_{0} ,\neck_{0}),\ldots, (\moconfigtwo_{p+1} ,\neck_{p+1})\big)&\mapsto \big(\big(\text{sc}_{1/(1+\neck)}(\moconfigtwo_{0}),\frac{1+\neck_{0}}{1+\neck}\big),\big(\text{sc}_{1/(1+\neck)}(\moconfigtwo_{1}),\frac{\neck_{1}}{1+\neck}\big),\ldots,\big(\text{sc}_{1/(1+\neck)}(\moconfigtwo_{p+1}),\frac{\neck_{p+1}}{1+\neck}\big)\big).
    %\big((\moconfigtwo_{0} ,\neck_{0}),\ldots, (\moconfigtwo_{p+1} ,\neck_{p+1})\big)&\mapsto \big((\text{sc}_{g(1,\neck)}(\moconfigtwo_{0}),g(1,\neck)(1+\neck_{0})),(\text{sc}_{g(1,\neck)}(\moconfigtwo_{1}),g(1,\neck)(\neck_{1})),\ldots,(\text{sc}_{g(1,\neck)}(\moconfigtwo_{p+1}),g(1,\neck)(\neck_{p+1}))\big)
\end{align*}
and let $g_{p}\colon \iterxpbarc{M_{1}}{\smon}{M_{2}}{p}{1}\hookrightarrow \iterxbarc{M_{1}}{\smon}{M_{2}}{p}{1}$ be the inclusion map. The maps $f_{p}$ and $g_{p}$ assemble to define maps of semi-simplicial spaces $f_{\bullet}\colon \iterxbarc{M_{1}}{\smon}{M_{2}}{\bullet}{1}\to \iterxpbarc{M_{1}}{\smon}{M_{2}}{\bullet}{1}$ and $g_{\bullet}\colon \iterxpbarc{M_{1}}{\smon}{M_{2}}{\bullet}{1}\to \iterxbarc{M_{1}}{\smon}{M_{2}}{\bullet}{1}$.

We now check that the map $f_{p}\colon \iterxbarc{M_{1}}{\smon}{M_{2}}{p}{1}\to \iterxpbarc{M_{1}}{\smon}{M_{2}}{p}{1}$ is a homotopy equivalence. Let $r\colon \left[0,1\right]\times\mathbb{R}_{>0}\to \mathbb{R}_{>0}$ be the function sending $(t,a)$ to $1/(1+at)$. The map
\begin{align*}
    H\colon \iterxbarc{M_{1}}{\smon}{M_{2}}{p}{1}\times\left[0,1\right]&\to \iterxbarc{M_{1}}{\smon}{M_{2}}{p}{1}\\
    \big(((\moconfigtwo_{0} ,\neck_{0}),\ldots, (\moconfigtwo_{p+1} ,\neck_{p+1})),t\big)&\mapsto \big(\big(\text{sc}_{r(t,\neck)}(\moconfigtwo_{0}),\frac{t+\neck_{0}}{1+t\neck}\big),\big(\text{sc}_{r(t,\neck)}(\moconfigtwo_{1}),\frac{\neck_{1}}{1+t\neck}\big),\ldots,\big(\text{sc}_{r(t,\neck)}(\moconfigtwo_{p+1}),\frac{\neck_{p+1}}{1+t\neck}\big)\big)
    %\big((\moconfigtwo_{0} ,\neck_{0}),\ldots, (\moconfigtwo_{p+1} ,\neck_{p+1})\big)&\mapsto \big((\text{sc}_{g(1,\neck)}(\moconfigtwo_{0}),g(1,\neck)(1+\neck_{0})),(\text{sc}_{g(1,\neck)}(\moconfigtwo_{1}),g(1,\neck)(\neck_{1})),\ldots,(\text{sc}_{g(1,\neck)}(\moconfigtwo_{p+1}),g(1,\neck)(\neck_{p+1}))\big)
\end{align*}
defines a homotopy equivalence between $\iterxbarc{M_{1}}{\smon}{M_{2}}{p}{1}$ and $\iterxpbarc{M_{1}}{\smon}{M_{2}}{p}{1}$. Since $f_{p}\colon \iterxbarc{M_{1}}{\smon}{M_{2}}{p}{1}\to \iterxpbarc{M_{1}}{\smon}{M_{2}}{p}{1}$ is a homotopy equivalence for all $p$, by \cref{levelwise weak homotopy equivalence}, the map $$\Vert f_{\bullet}\Vert\colon \Vert\iterxbarc{M_{1}}{\smon}{M_{2}}{\bullet}{1}\Vert\to \Vert\iterxpbarc{M_{1}}{\smon}{M_{2}}{\bullet}{1}\Vert$$ is a weak homotopy equivalence.

 The proof for $\srange>1$ is similar to the case $\srange=1$--one just needs to rescale additional terms.
\end{proof}
Suppose that $M_{1}$ and $M_{2}$ are manifolds with boundary and that $(\smon_{1},\ldots, \smon_{\srange})$ is an $\srange$-tuple of disjoint subspaces contained in $\partial M_{1}$ and $\partial M_{2}$. Let $\smon\colonequals \bigsqcup_{j=1}^{\srange}\smon_{j}$. 
%and that $\smon$ is contained in $\partial M_{1}$ and $\partial M_{2}$. 
%Let $M$ be the space $M_{1}\cup_{\smon}M_{2}$.
Let $M$ be the space 
$M_{1}\cup_{\smon}({\smon}\times [0,1]^{\srange})\cup_{\smon}M_{2}$, where $M_{1}$ and $M_{2}$ attached along ${\smon}\times \{(0,0,\ldots,0)\}$ and ${\smon}\times \{(1,1,\ldots,1)\}$, respectively. 
%$M_{1}\cup_{\smon}(\bigsqcup_{j=1}^{\srange}(\smon_{j}\times [0,1]))\cup_{\smon}M_{2}$, where $M_{1}$ and $M_{2}$ attached along $\bigsqcup_{j=1}^{\srange}(\smon_{j}\times \{0\})$ and $\bigsqcup_{j=1}^{\srange}(\smon_{j}\times \{1\})$, respectively. 
We now introduce a semi-simplicial model for $\text{Conf}(M)$.
%\footnote{the flow is kind of awkward}
\begin{definition}
Suppose that $M_{1}$ and $M_{2}$ are two manifolds with boundary and that $(\smon_{1},\ldots, \smon_{\srange})$ is an $\srange$-tuple of disjoint subspaces contained in $\partial M_{1}$ and $\partial M_{2}$. Let $\smon$ denote $\bigsqcup_{j=1}^{\srange}\smon_{j}$.
%\footnote{fix condition regarding N. Also, maybe want to rewrite definition of $B_p (M,\smon)$ so that curly braces are two lines long} 
%Let $M$ be the space $M_{1}\cup_{\smon}({\smon}\times [0,1])\cup_{\smon}M_{2}$, where $M_{1}$ and $M_{2}$ attached along ${\smon}\times\{0\}$ and ${\smon}\times\{1\}$, respectively.
Let $M$ be the space $M_{1}\cup_{\smon}({\smon\times\{0\}}\times [0,1])\cup_{\smon\times\{1\}}M_{2}$, where $M_{1}$ and $M_{2}$ attached along ${\smon}\times \{0\}$ and ${\smon}\times \{1\}$, respectively. 
%Let $M$ be the space $M_{1}\cup_{\smon}M_{2}$. Fix an embedding ${\smon}\times \left[0,1\right]\hookrightarrow M$ such that ${\smon}\times\left[0,1/2\right]$ is contained in $M_{1}$ and ${\smon}\times\left[1/2,1\right]$ is contained in $M_{2}$. 
Let $B_{\bullet}(M,\smon)$ denote the semi-simplicial space with
\begin{equation*}
    B_{p}(M,\smon)\colonequals \left\{ (\moconfigone; \vec{\neck}_{0},\ldots,\vec{\neck}_{p})\in \text{Conf}(M)\times (0,1)^{\srange(p+1)}:\begin{aligned} &  \vec{\neck}_{0}(j)<\cdots<\vec{\neck}_{p+1}(j)\text{ and } \moconfigone\cap \smon_{j}\times\{\vec{\neck}_{i}(j)\}=\emptyset \\&\text{ for all } i=1,\ldots, p, \text{ and } j=1,\ldots,\srange\end{aligned}\right\}
\end{equation*}
%$$B_{p}(M,\smon)\colonequals \{ (\moconfigone; \vec{\neck}_{0},\ldots,\vec{\neck}_{p})\in \text{Conf}(M)\times (0,1)^{\srange(p+1)}:  \vec{\neck}_{0}(j)<\cdots<\vec{\neck}_{p+1}(j)\text{ and } \moconfigone\cap \smon_{j}\times\{\vec{\neck}_{i}(j)\}=\emptyset \text{ for all } i \text{ and } j=1,\ldots,\srange\}$$
and with $d_{i}\colon B_{p}(M,\smon)\to B_{p-1}(M,\smon)$ the map that forgets $\vec{\neck_{i}}$.
\end{definition}
%We now recall the notion of an augmentation map and a (semi-)simplicial resolution (see \cref{augmentation map defintion} and \cref{resolution defintion}) in order to relate $\vert \xbarc{M_{1}}{\smon}{M_{2}}{\bullet}\vert$ to $\fconf$.
%We will show that the simplicial space $\vert \xbarc{M_{1}}{\smon}{M_{2}}{\bullet}\vert $ is weakly homotopy equivalent to $\fconf$.
%a resolution for the space $\fconf$ (we will define a resolution later-see \cref{resolution defintion}).
\begin{definition}\label{augmentation map defintion}
An \textbf{augmented semi-simplicial space} is a triple $(Z_{\bullet},Z_{-1},\epsilon )$ consisting of a semi-simplicial space $Z_{\bullet}$, a space $Z_{-1}$, and a map $\epsilon  \colon  Z_{0}\to  Z_{-1}$ such that $\epsilon d_{0}=\epsilon d_{1}$. The map $\epsilon  \colon  Z_{0}\to  Z_{-1}$ is called the \textbf{augmentation map}.
\end{definition}
An augmented semi-simplicial space $(Z_{\bullet},Z_{-1},\epsilon )$ induces a map
$\Vert\epsilon\Vert\colon\Vert Z_{\bullet}\Vert\to Z_{-1}$.
%from both the thick and thin geometric realizations of a simplicial space $Z_{\bullet}$ to $Z_{-1}$.
\begin{definition}\label{resolution defintion}
A \textbf{semi-simplicial resolution} of a space $Z_{-1}$ is an augmented semi-simplicial space $(Z_{\bullet},Z_{-1},\epsilon )$ such that the induced map $\Vert\epsilon\Vert\colon\Vert Z_{\bullet}\Vert \to  Z_{-1}$ is a weak homotopy equivalence.
%A \textbf{simplicial resolution} of a space $Z_{-1}$ is an augmented simplicial space $(Z_{\bullet},Z_{-1},\epsilon )$ such that the induced map $\vert\epsilon\vert\colon\vert Z_{\bullet}\vert \to  Z_{-1}$ is a weak homotopy equivalence.
%(which is $\Vert Z_{\bullet}\Vert \to  Z_{-1}$ if $Z_{\bullet}$ is a semi-simplicial space and 
%which is $\vert Z_{\bullet}\vert \to  Z_{-1}$ if $Z_{\bullet}$ is a simplicial space)
%is a weak homotopy equivalence.
\end{definition}
Let $\epsilon\colon B_{0}(M,\smon)\to\text{Conf}(M)$ be the map which forgets $\vec{\neck}_{0}$. The triple $(B_{\bullet}(M,\smon),\text{Conf}(M),\epsilon)$ is an augmented semi-simplicial space. To show that $(B_{\bullet}(M,\smon),\text{Conf}(M),\epsilon)$ is a semi-simplicial resolution of $\text{Conf}(M)$, we now introduce a technique of \textcite[Section 6.2]{MR3207759}.
\begin{definition}
An \textbf{augmented topological flag complex} is an augmented semi-simplicial space $(Z_{\bullet},Z_{-1},\epsilon)$ such that
\hfill\begin{enumerate}
    \item The map $Z_{p}\to Z_{0}\times_{Z_{-1}}Z_{0}\times_{Z_{-1}}\cdots\times_{Z_{-1}}Z_{0}$ to the $(p+1)$-fold product--which takes a $p$-simplex to its $(p+1)$ vertices--is a homeomorphism onto its image, which is an open subset, and
    \item A tuple $(v_{0},\ldots,v_{p+1})\in Z_{0}\times_{Z_{-1}}Z_{0}\times_{Z_{-1}}\cdots\times_{Z_{-1}}Z_{0}$ lies in $Z_{p+1}$ if and only if $(z_{i},z_{j})\in Z_{1}$ for all $i<j$.
\end{enumerate}
\end{definition}
%The augmented semi-simplicial space $(B_{\bullet}(M,\smon),\text{Conf}(M),\epsilon)$ is an augmented topological flag complex. Therefore, we can use the following result.
\begin{theorem}[{\textcite[Theorem 6.2]{MR3207759}}]\label{flag complex}
Let $(Z_{\bullet},Z_{-1},\epsilon)$ be an augmented topological flag complex. Suppose that
\hfill\begin{enumerate}
    \item\label{local setion} The map $\epsilon\colon Z_{0}\to Z_{-1}$ has local sections (in the sense that given any $z\in Z_{0}$, there is a neighborhood $U\subset Z_{-1}$ of $\epsilon (z)$ and a section $s\colon U\to Z_{0}$ with $s(\epsilon(z))=z$),
    \item\label{map from zero simplices is onto} The map $\epsilon\colon Z_{0}\to Z_{-1}$ is surjective, and
    \item\label{flag can be inserted} For any $x\in Z_{-1}$ and any (non-empty) finite set $\{z_{1},\ldots,z_{n}\}\subset \epsilon^{-1}(x)$, there exists a $z\in \epsilon^{-1}(x)$ with $(z_{i},z)\in Z_{1}$ for all $i$.
\end{enumerate}
Then $\Vert\epsilon\Vert\colon\Vert Z_{\bullet}\Vert\to Z_{-1}$ is a weak homotopy equivalence.
\end{theorem}
\begin{proposition}\label{augmented flag version of Conf(M)}
The map $\Vert\epsilon\Vert\colon\Vert B_{\bullet}(M,\smon)\Vert\to\text{Conf}(M)$ is a weak homotopy equivalence.
\end{proposition}
\begin{proof}
Our argument is the same as the proof of \textcite[Proposition 2.1]{randal2013topological}. We want to apply \cref{flag complex}. The augmented semi-simplicial space $(B_{\bullet}(M,\smon),\text{Conf}(M),\epsilon)$ is an augmented topological flag complex. Conditions \ref{local setion} and \ref{map from zero simplices is onto} are clear. For Condition \ref{flag can be inserted}, given a configuration $\moconfigone\in \text{Conf}(M)$, the set 
$$\epsilon^{-1}(\moconfigone)\cong\{\vec{\neck}\in (0,1)^{\srange}: \moconfigone\cap \smon_{j}\times\{\vec{\neck}(j)\}=\emptyset\text{ for all }j=1,\ldots,\srange\}$$ is infinite since $\moconfigone$ has finitely many points. Therefore, for any finite subset $\{\vec{\neck}_{1},\ldots,\vec{\neck}_{n}\}\subset \{\vec{\neck}\in (0,1)^{\srange}: \moconfigone\cap \smon_{j}\times\{\vec{\neck}(j)\}=\emptyset\text{ for all }j=1,\ldots,\srange\}$ we can pick an element $\vec{\neck}\in \epsilon^{-1}(\moconfigone)$ such that $\vec{\neck}_{i}(j)<\vec{\neck}(j)$ for all $i=1,\ldots, n$ and $j=1,\ldots,\srange$, so Condition \ref{flag can be inserted} holds. By \cref{flag complex}, the map $\Vert\epsilon\Vert\colon\Vert B_{\bullet}(M,\smon)\Vert\to\text{Conf}(M)$ is a weak homotopy equivalence.
\end{proof}
\begin{lemma}\label{map from xpbarc to conf is weak homotopy equivalence}
%There is a weak homotopy equivalence $\Vert \xbarc{M_{1}}{\smon}{M_{2}}{\bullet}\Vert\to \text{Conf}(M)$
The space $\Vert \iterxbarc{M_{1}}{\smon}{M_{2}}{\bullet}{\srange}\Vert$is weakly homotopy equivalent to  $\Vert B_{\bullet}(M,\smon)\Vert$.
\end{lemma}
\begin{proof}
Let $\wpbarc{M_{1}}{\smon}{M_{2}}{\bullet}\subset \iterxpbarc{M_{1}}{\smon}{M_{2}}{\bullet}{\srange}$ denote the semi-simplicial space whose $p$-simplices consist of elements $\big ((\moconfigtwo_{0},\vec{\neck}_{0}),(\vec{\moconfigtwo}_{1},\vec{\neck}_{1}),\ldots,(\vec{\moconfigtwo}_{p},\vec{\neck}_{p}), (\moconfigtwo_{p+1},\vec{\neck}_{p+1})\big )\in\iterxpbarc{M_{1}}{\smon}{M_{2}}{p}{\srange}$ such that $\vec{\neck}_{i}(j)>0$ for all $i=1,\ldots,p$ and all $j$. We have a weak homotopy equivalence
%\footnote{should I prove this, or is it obvious?} 
$$\Vert\iterxpbarc{M_{1}}{\smon}{M_{2}}{\bullet}{\srange}\Vert\to \Vert\wpbarc{M_{1}}{\smon}{M_{2}}{\bullet}\Vert.$$ 
Given $\big ((\moconfigtwo_{0},\vec{\neck}_{0}),(\vec{\moconfigtwo}_{1},\vec{\neck}_{1}),\ldots,(\vec{\moconfigtwo}_{p},\vec{\neck}_{p}), (\moconfigtwo_{p+1},\vec{\neck}_{p+1})\big )\in \wpbarc{M_{1}}{\smon}{M_{2}}{p}$, let $\vec{\neck}_{i}'=\sum_{k=0}^{i}\vec{\neck}_{k}$. We can view $\bigcup_{i=0}^{p+1}\moconfigtwo_{i}$ as an element of $\text{Conf}(M)$ under the identifications of
\begin{enumerate}
    \item $\sqcup_{j=1}^{\srange}\smon_{j}\times [0,\vec{\neck}_{0}(j)]$ with $\sqcup_{j=1}^{\srange}\smon_{j}\times [0,\vec{\neck}_{0}(j)]\subset M$ via the identity map,
    \item $\sqcup_{j=1}^{\srange}\smon_{j}\times [0,\vec{\neck}_{i}(j)]$ with $\sqcup_{j=1}^{\srange}\smon_{j}\times [\vec{\neck}_{i-1}'(j),\vec{\neck}_{i}'(j)]\subset M$,
    %${\smon}\times [\neck^{'}_{i-1},\neck^{'}_{i}]\subset M$, 
    for $i=1,\ldots,p$, via the map $(n,t)\mapsto (n,\vec{\neck}_{i-1}'(j) +t)$, and
    \item $\sqcup_{j=1}^{\srange}\smon_{j}\times [0,\vec{\neck}_{p+1}(j)]$ with $\sqcup_{j=1}^{\srange}\smon_{j}\times [\vec{\neck}_{p+1}'(j),1]\subset M$ via the map $(n,t)\mapsto (n,1-t)$ (see \cref{fig:another picture of augmentation map}).
\end{enumerate}
\begin{figure}
    \centering
    \includegraphics{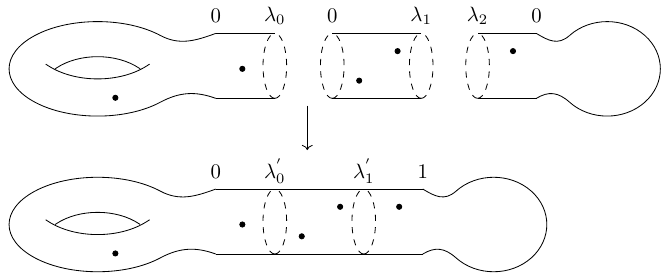}
    \caption{The map $f_{\bullet}\colon \wpbarc{M_{1}}{\smon}{M_{2}}{\bullet}\to B_{\bullet}(M,\smon)$}
    \label{fig:another picture of augmentation map}
\end{figure}
%${\smon}\times \left[0,\neck_{i}\right]$ with ${\smon}\times \left[\neck_{i-1},\neck_{i-1}+\neck_{i}\right]\subset M$ for $i=1,\ldots,p$
There are maps 
\begin{align*}
    f_{p}\colon \wpbarc{M_{1}}{\smon}{M_{2}}{p}&\to B_{p}(M,\smon)\\
    \big ((\moconfigtwo_{0},\vec{\neck}_{0}),(\vec{\moconfigtwo}_{1},\vec{\neck}_{1}),\ldots,(\vec{\moconfigtwo}_{p},\vec{\neck}_{p}), (\moconfigtwo_{p+1},\vec{\neck}_{p+1})\big )&\mapsto (\bigcup_{i=0}^{p+1}\moconfigtwo_{i}; \vec{\neck}'_{0}, \vec{\neck}'_{1},\ldots,\vec{\neck}'_{p})
\end{align*}
that assemble to define a map of semi-simplicial spaces $f_{\bullet}\colon \wpbarc{M_{1}}{\smon}{M_{2}}{\bullet}\to B_{\bullet}(M,\smon)$. Since the map $f_{p}$ is a homeomorphism for all $p$, the induced map $\Vert f_{\bullet}\Vert\colon \Vert \wpbarc{M_{1}}{\smon}{M_{2}}{\bullet}\Vert\to \Vert B_{\bullet}(M,\smon)\Vert$ is a weak homotopy equivalence by \cref{levelwise weak homotopy equivalence}. Since the composition of weak homotopy equivalences is a weak homotopy equivalence, the map $\Vert \iterxpbarc{M_{1}}{\smon}{M_{2}}{\bullet}{\srange}\Vert\to \Vert B_{\bullet}(M,\smon)\Vert$ is a weak homotopy equivalence.
%$$f_{p}\colon \wpbarc{M_{1}}{\smon}{M_{2}}{p}\to B_{p}(M,\smon) $$ given by the map
%identifications of
%\begin{enumerate}
%    \item M\setminus (M_{2}\cup ({\smon}\times\left[t_{0},1\right]))
%\end{enumerate}
%We have a levelwise homemorphism $\xpbarc{M_{1}}{\smon}{M_{2}}{p}\to B_{p}(M,\smon)$ coming 
\end{proof}
\begin{corollary}
\label{map from bar construction to conf is a weak equivalence}
There is a weak homotopy equivalence $\Vert \iterxbarc{M_{1}}{\smon}{M_{2}}{\bullet}{\srange}\Vert\to\text{Conf}(M)$.
%There is a weak homotopy equivalence $\Vert \xbarc{M_{1}}{\smon}{M_{2}}{\bullet}\Vert\to \text{Conf}(M)$.
\end{corollary}
\begin{proof}
We have the following composition
$$\Vert \iterxbarc{M_{1}}{\smon}{M_{2}}{\bullet}{\srange}\Vert\to\Vert\iterxpbarc{M_{1}}{\smon}{M_{2}}{\bullet}{\srange}\Vert\to \Vert B_{\bullet}(M,\smon)\Vert\to \text{Conf}(M)$$
of weak homotopy equivalences by \cref{postive mconf is equivalent to mconf}, \cref{map from xpbarc to conf is weak homotopy equivalence}, and \cref{augmented flag version of Conf(M)}. Since the composition of weak homotopy equivalences is a weak homotopy equivalence, the map $\Vert \iterxbarc{M_{1}}{\smon}{M_{2}}{\bullet}{\srange}\Vert\to \text{Conf}(M)$ is a weak homotopy equivalence.
%The first map $\Vert \xbarc{M_{1}}{\smon}{M_{2}}{\bullet}\Vert\to\Vert\xpbarc{M_{1}}{\smon}{M_{2}}{\bullet}\Vert$
%by \cref{postive mconf is equivalent to mconf}
%Since the composition of weak homotopy equivalences is a weak homotopy equivalence, the composition $$\Vert\xpbarc{M_{1}}{\smon}{M_{2}}{\bullet}\Vert\to \Vert\wpbarc{M_{1}}{\smon}{M_{2}}{\bullet}\Vert\to \Vert B_{\bullet}(M,\smon)\Vert$$ is a weak homotopy equivalence. Since $\Vert B_{\bullet}(M,\smon)\Vert$ is weakly homotopy equivalent to $\text{Conf}(M)$ by \cref{augmented flag version of Conf(M)}, $\Vert\xpbarc{M_{1}}{\smon}{M_{2}}{\bullet}\Vert$ is weakly homotopy equivalent to $\text{Conf}(M)$.
\end{proof}
\begin{remark}
The weak homotopy equivalence $$\Vert \iterxbarc{M_{1}}{\smon}{M_{2}}{\bullet}{\srange}\Vert \to \text{Conf}(M)$$ also follows from the theory of factorization homology. Factorization homology, also known as topological chiral homology or manifoldic homology, is a gadget which takes as inputs an $n$-manifold $M$ and an algebra $A$ over the framed little $n$-disk operad $fE_{n}$ and outputs an object $\int_{M}A$ in a category (see \textcite{MR3431668} or \textcite[Section 2]{MR3286505} for more information about factorization homology). When the $fE_{n}$-algebra is $\text{Conf}(\mathbb{R}^{n})$, the factorization homology $\int_{M}A$ is $\text{Conf}(M)$. The weak homotopy equivalence $$\Vert \iterxbarc{M_{1}}{\smon}{M_{2}}{\bullet}{\srange}\Vert \to \text{Conf}(M)$$ is a special case of excision for factorization homology due to Ayala--Francis (see \textcite[ Lemma 3.18]{MR3431668}). One could prove this weak homotopy equivalence by quoting their result, but it is easier to prove this weak homotopy equivalence directly than to compare our framework with theirs. 
%Another reason for proving this weak homotopy equivalence directly is that microfibration arguments will appear in other parts of the paper and our work serves as a warm-up for those arguments.
%than to put our model for the derived tensor product\footnote{should I just drop mentioning the derived tensor product, since I don't mention it in the remark?} $$\modc{M_{1}}\bigotimes^{\mathbb{L}}_{\mon}\modc{M_{1}}$$ (the two-sided bar-construction $\vert \barc{M_{1}}{\smon}{M_{2}}\vert$) in the framework of factorization homology. 
Recent work of An--Drummond-Cole--Knudsen has also obtained a similar excision result for configuration spaces in finite CW complexes (see \textcite[Theorem 3.20]{MR4033833}).
\begin{comment}
In our case, the factorization homology of $M$ with coefficients in $A$ is the space $\int_{M}A \colonequals  \Vert B_{\bullet}(\text{M}, \text{Mo}(E_{n}),\text{A})\Vert$, where $\text{Mo}(E_{n})$ is the monad associated to the operad $E_{n}$.\footnote{Should I elaborate on this definition? For example, should I explain what the monad is associated to an operad? I'm trying to be intentionally vague for the factorization homology stuff because I don't think it's worth getting too deep in the details.}

When $A$ is $\text{Conf}(\mathbb{R}^n)$, the free $E_{n}$-algebra on a point, the factorization homology $\int_{M}A$ is $\text{Conf}(M)$. Excision for factorization homology states that, given a collar-gluing for $M$, i.e. an open cover $\{M',M^{''}\}$ of $M$ satisfying $M'\cap M^{''}=\emptyset$, $M_{0}=\partictial M'=\partictial M^{''}$,  and an inclusion $M_{0}\times\mathbb{R}$ into $M$  (see \textcite[Definition 3.13]{MR3431668} for a proper definition of a collar-gluing), there is a natural homotopy equivalence $\Vert B(\int_{M'}A,\int_{M_0 \times\mathbb{R}}A,\int_{M^{''}}A)\Vert\to  \int_{M}A$. In our case, $M'$ is an open collar neighborhood of $M_{1}$, $M^{''}$ is a collar neighborhood of $M_{2}, M_{0}\times\mathbb{R} = Y$, and $A$ is a collar neighborhood of $\text{Conf}(\mathbb{R}^n)$.\footnote{Should I say more about factorization homology? I'm not sure how to flatter John Francis here}
\end{comment}
\end{remark}
\subsection{A Model of the Configuration Space of the Disk}\label{stabilization map construction}
%By the results of \cref{map from bar construction to conf is a weak equivalence} and \cref{floppy conf is equivalent to conf}, 
%By \cref{chain model for conf(M)}, we have a weak homotopy equivalence $$\Vert \cbarc{M_{1}}{\smon}{M_{2}}{\bullet}{\fieldc}\Vert\to C_{*}(\text{Conf}(M);\fieldc).$$
%In this subsection, we will construct a stabilization map for $H_{*}(\text{Conf}(M);\fieldc)$. 
By \cref{map from bar construction to conf is a weak equivalence}, there is a weak homotopy equivalence 
%$$\Vert Z_{\bullet}\Vert\colonequals\Vert \cbarc{\modc[S^{n-1}]{\bar{D}^{n}}}{\mon[S^{n-1}]}{\modc[S^{n-1}]{M\setminus D^{n}}}{\bullet}{\fieldc}\Vert\to C_{*}(\text{Conf}(M);\fieldc),$$ 
$$\Vert \iterxbarc{\bigsqcup_{j=1}^{\srange}\bar{D}^{n}}{\bigsqcup_{j=1}^{\srange} S^{n-1}}{M\setminus \bigsqcup_{j=1}^{\srange}D^{n}}{\bullet}{\srange}\Vert \to \text{Conf}(M).$$ We will use $\iterxbarc{\bigsqcup_{j=1}^{\srange}\bar{D}^{n}}{\bigsqcup_{j=1}^{\srange} S^{n-1}}{M\setminus \bigsqcup_{j=1}^{\srange}D^{n}}{\bullet}{\srange}$ to construct a semi-simplicial model for $\nch{\text{Conf}(M)}$ in \cref{new stab maps for conf(M)}. As a consequence of this model, to define a stabilization map on $C_{*}(\text{Conf}(M);\fieldc)$, it will suffice to construct a stabilization map on $C_{*}(\modc[S^{n-1}]{\bar{D}};\fieldc)$ that is a map of $C_{*}(\mon[S^{n-1}];\fieldc)$-modules. Our first step in constructing such a stabilization map on $C_{*}(\modc[S^{n-1}]{\bar{D}};\fieldc)$ is to build a model for $\modc[S^{n-1}]{\bar{D}}$ using free $\mon[S^{n-1}]$-modules and we will do that in this subsection. 
\begin{definition}
Define $\bdisk\subset\modc[S^{n-1}]{\bar{D}^{n}}$ to be the submodule consisting of elements $(\moconfigtwo ,\neck)\in\modc[S^{n-1}]{\bar{D}^{n}}$ such that at most one point of $(\moconfigtwo ,\neck)$ is in $\bar{D}^{n}$.
%(equivalently, if  $(q_{i}, 0)$ and $(q_{j}, 0)$ are two points of $(\moconfigtwo ,\neck)$, then we must have that $i=j$).
%$i,j=1,\ldots,\partic$, 
%$\moconfigtwo=\{(q_{1},u_{1}),\ldots,(q_{k},u_{k})\}$ and we have that
%such that the configuration $\moconfigtwo$ has at most one point $(q,u)$ satisfying the condition $u=0$.
%with $\moconfigtwo=\{(q_{1},u_{1}),\ldots,(q_{k},u_{k})\}$, such that $u_{i}=0$ for at most one point $(q_{i},u_{i})$ of $(\moconfigtwo ,\neck)$ (i.e. at most one point of $(\moconfigtwo ,\neck)$ is in $\bar{D}^{n}$; equivalently, if $(q_{i},u_{i})=(q_{i},0)$ and $(q_{j},u_{j})=(q_{j},0)$ for some $i,j=1,\ldots,\partic$, then we must have that $i=j$).
\end{definition}
%\begin{definition}
%Let $\text{Conf}(\mathbb{R}^{n})^{\leq 1}\subset \text{Conf}(\mathbb{R}^{n})$ denote the subspace of configurations $$\moconfigtwo\colonequals\{(q_{1},u_{1}),\ldots,(q_{\partic},u_{\partic})\}\in \text{Conf}(\mathbb{R}^{n})$$ such that at most one point of $\moconfigtwo$ is in the closed disk $\bar{B}(0,1)$ of radius $1$ centered at the origin.
%\end{definition}
\begin{lemma}\label{bdisk and mod-conf(disk) are equivalent}
The inclusion $\bdisk\hookrightarrow\modc[S^{n-1}]{\bar{D}^{n}}$ is a map of $\mon[S^{n-1}]$-modules and a homotopy equivalence.
\end{lemma}
\begin{proof}
It is clear that $\bdisk\hookrightarrow\modc[S^{n-1}]{\bar{D}^{n}}$ is a map of $\mon[S^{n-1}]$-modules. Let $\text{Conf}(\mathbb{R}^{n})^{\leq 1}\subset \text{Conf}(\mathbb{R}^{n})$ denote the subspace of configurations $$\moconfigtwo\colonequals\{(q_{1},u_{1}),\ldots,(q_{\partic},u_{\partic})\}\in \text{Conf}(\mathbb{R}^{n})$$ such that at most one point of $\moconfigtwo$ is in the closed disk $\bar{B}(0,1)$ of radius $1$ centered at the origin.
%Consider $\text{Conf}(\mathbb{R}^{n})^{\leq 1}$. Here, 
We think of $\mathbb{R}^{n}$ as $\bar{B}(0,1)\cup_{S^{n-1}}S^{n-1}\times\left[0,\infty\right)$. From the inclusion $$\bar{B}(0,1)\cup_{S^{n-1}}S^{n-1}\times\left[0,\neck\right)\to \bar{B}(0,1)\cup_{S^{n-1}}S^{n-1}\times\left[0,\infty\right),$$
we have a map
\begin{align*}
    \modc[S^{n-1}]{\bar{B}(0,1)}&\to \text{Conf}(\mathbb{R}^{n})\\
    (\moconfigone,\neck)&\mapsto \moconfigone
\end{align*}
%$\modc{\bar{B}(0,1)}\to \text{Conf}(\mathbb{R}^{n})$
which restricts to a map $\modc[S^{n-1}]{\bar{B}(0,1)}^{\leq 1}\to \text{Conf}(\mathbb{R}^{n})^{\leq 1}$. We also have the following commutative diagram

\begin{center}\begin{tikzcd}
\centering
\modc[S^{n-1}]{\bar{B}(0,1)}^{\leq 1} \arrow[r]\arrow[hookrightarrow,d] & \text{Conf}(\mathbb{R}^{n})^{\leq 1}\arrow[hookrightarrow,d]\\
\modc[S^{n-1}]{\bar{B}(0,1)} \arrow[r] & \text{Conf}(\mathbb{R}^{n}).
\end{tikzcd}\end{center}
%There is a map
%$\modc[S^{n-1}]{\bar{B}(0,1)}\to \text{Conf}(\mathbb{R}^{n})$ coming from the inclusion $$\bar{B}(0,1)\cup_{S^{n-1}}S^{n-1}\times\left[0,\neck\right)\to \bar{B}(0,1)\cup_{S^{n-1}}S^{n-1}\times\left[0,\infty\right).$$
By \cref{replacements are homotopy equivalent to config spaces}, $\modc[S^{n-1}]{\bar{D}^{n}}$ is homotopy equivalent to $\text{Conf}(\mathbb{R}^{n})$.
By the argument in the proof of \cref{replacements are homotopy equivalent to config spaces}, $\text{Conf}(\mathbb{R}^{n})^{\leq 1}$ is homotopy equivalent to $\bdisk$. Therefore,
to show that $\modc[S^{n-1}]{\bar{B}(0,1)}^{\leq 1}$ is homotopy equivalent to $\modc[S^{n-1}]{\bar{B}(0,1)}$
it suffices to show that the right vertical map $\text{Conf}(\mathbb{R}^{n})^{\leq 1}\hookrightarrow \text{Conf}(\mathbb{R}^{n})$ is a homotopy equivalence.

We prove that $\text{Conf}(\mathbb{R}^{n})^{\leq 1}$ is homotopy equivalent to $\text{Conf}(\mathbb{R}^{n})$ by pushing points in a configuration $\moconfigone\in\text{Conf}(\mathbb{R}^{n})$ off the ball $B(0,1)$. We can do this by scaling a configuration $\moconfigone$ by a function $f$ depending on the configuration's second closest point to the origin, which we now define.
%Given $\{q_{1},\ldots,q_{\partic}\}\in \text{Conf}(\mathbb{R}^{n})$,
%define $m\colon \text{Conf}(\mathbb{R}^{n})\to (0,\infty)$ to be the function  $m(\{q_{1},\ldots,q_{\partic}\})\colonequals \min \vert q_{j}\vert$ 
%Given a configuration $\moconfigone\in\text{Conf}(\mathbb{R}^{n})$, we can push points in $\moconfigone$ so that at most one point of $\moconfigone$ remains in $B(0,1)$  by scaling by a function $f$ depending on the second closest point to the origin
%We want to construct a map $\text{Conf}(\mathbb{R}^{n})^{\leq 1}\to\text{Conf}(\mathbb{R}^{n})$
%Given a configuration $\moconfigone\in\text{Conf}(\mathbb{R}^{n})$, we want 
%and define $f\colon \text{Conf}(\mathbb{R}^{n})\to (0,\infty)$ to be the ``second closest point'' function
Given $\{q_{1},\ldots,q_{\partic}\}\in \text{Conf}(\mathbb{R}^{n})$, reorder $q_{1},\ldots,q_{\partic}$ so that $\lvert q_{1}\rvert\leq \lvert q_{2}\rvert\ldots\leq \lvert q_{\partic}\rvert$ and let $f\colon \text{Conf}(\mathbb{R}^{n})\to (0,\infty)$ be 
%a continuous function such that $\moconfigone$ intersects the ball of radius $f(\moconfigone)$
the ``second closest point'' function 
\[
f(\{ q_{1},\ldots ,q_{\partic}\} )=\begin{cases} 
      1 & \text{if }\partic= 0,1, \\
      2/\lvert q_{2}\rvert & \text{otherwise}. 
      %\partic >2\text{ and } \vert q_{1} \vert=\cdots =\vert q_{k} \vert\\ 
      %2/\min_{j\in\{1,\ldots, \Hat{i} ,\ldots,\partic\}; \vert x_{i}\vert=m(\{q_{1},\ldots,q_{\partic}\})} \vert x_{j}\vert &\text{otherwise}
      %\text{the second smallest } \vert q_{i} \vert/2 &\text{otherwise}.
   \end{cases}
   \]

Define $F\colon \text{Conf}(\mathbb{R}^{n})\to \text{Conf}(\mathbb{R}^{n})^{\leq 1}$ to be 
the scaling function
\begin{align*}
    F\colon \text{Conf}(\mathbb{R}^{n})&\to \text{Conf}(\mathbb{R}^{n})^{\leq 1}\\
    \moconfigone\colonequals\{q_{1},\ldots,q_{\partic}\}&\mapsto \text{sc}_{f(\moconfigone)}(\moconfigone)=\{f(\moconfigone)\cdot q_{1},\ldots , f(\moconfigone)\cdot q_{\partic}\}
\end{align*}
%the map which sends a configuration $\moconfigone\colonequals\{q_{1},\ldots,q_{\partic}\}$ to $\text{sc}_{f(\moconfigone)}(\moconfigone)=\{f(\moconfigone)\cdot q_{1},\ldots , f(\moconfigone)\cdot q_{\partic}\} $. 
Let $G\colon \text{Conf}(\mathbb{R}^{n})^{\leq 1}\to \text{Conf}(\mathbb{R}^{n})$ be the inclusion map. There is a homotopy
\begin{align*}
    H\colon \text{Conf}(\mathbb{R}^{n})\times \left[0,1\right]&\to \text{Conf}(\mathbb{R}^{n})\\
    (\moconfigone, t)&\mapsto \text{sc}_{f(\moconfigone)(1-t)+t}(\moconfigone)
\end{align*}
from $G\circ F$ to $\text{id}$. 
%\footnote{Should I explain why $H$ restricts to a homotopy on $\text{Conf}(\mathbb{R}^{n})^{\leq 1}$? It's not hard to show, but it's a little bit technical.}
By restriction of the map $H$ to $\text{Conf}(\mathbb{R}^{n})^{\leq 1}\times \left[0,1\right]$, the map $H$ also gives a homotopy from $F\circ G$ to $\text{id}$. Therefore, $\text{Conf}(\mathbb{R}^{n})^{\leq 1}$ is homotopy equivalent to $\text{Conf}(\mathbb{R}^{n})$.
%\begin{align*}
%    f\colon \text{Mod-Conf}^{\geq 1}_{S^{n-1}}(D^{n}) &\to (0,\infty)\\
%    (\moconfigtwo ,\neck)\equalscolon (\{(q_{1},u_{1}),\ldots,(q_{k},u_{k})\},\neck)
%\end{align*}
%
%which sends $\moconfigtwo$ to the second smallest number $\vert q_{i} \vert/2$, where, by convention if $\vert q_{1} \vert=\cdots \vert q_{k} \vert$
%
%Let $\text{Mod-Conf}^{\geq 1}_{S^{n-1}}(D^{n})\subset \modc[S^{n-1}]{\bar{D}^{n}}$ denote the subspace of pairs $(\moconfigtwo ,\neck)\in\modc[S^{n-1}]{\bar{D}^{n}}$ with $\neck \geq 1$ and let $\bdisk^{\geq 1}$ denote the intersection $\bdisk\cap\text{Mod-Conf}^{\geq 1}_{S^{n-1}}(D^{n})$. By the deformation retract given in \cref{replacements are homotopy equivalent to config spaces}, $\text{Mod-Conf}^{\geq 1}_{S^{n-1}}(D^{n})$ and $\bdisk^{\geq 1}$ are homotopy equivalent to $\modc[S^{n-1}]{\bar{D}^{n}}$ and $\bdisk$, respectively. Therefore, it suffices to show that $\text{Mod-Conf}^{\geq 1}_{S^{n-1}}(D^{n})$ is homotopy equivalent to $\bdisk^{\geq 1}$.
%
%Given a
%
%Define $f\colon \text{Mod-Conf}^{\geq 1}_{S^{n-1}}(D^{n}) \to (0,\infty)$ to be the function 
%\begin{align*}
%    f\colon \text{Mod-Conf}^{\geq 1}_{S^{n-1}}(D^{n}) &\to (0,\infty)\\
%    (\moconfigtwo ,\neck)\equalscolon (\{(q_{1},u_{1}),\ldots,(q_{k},u_{k})\},\neck)
%\end{align*}
\end{proof}
We now introduce a semi-simplicial model for $\modc[S^{n-1}]{\bar{D}^{n}}$.
\begin{definition}
Suppose that $\smon$ is a space. Let $\bmon{k}\subset\mon$ denote the subspace of elements $(\moconfigtwo ,\neck )\in\mon$ such that the cardinality of the configuration $\moconfigtwo$ is equal to $k$.
%and for each point $(q,u)$  of $(\moconfigtwo ,\neck )$, we have that $u>0$. 
If $M$ is a manifold with boundary and $\smon$ is contained in $\partial M$, define $\bmodc{M}{\partic}$ similarly and let $\bmodc{M}{\leq k}$ denote the subspace of elements $(\moconfigtwo ,\neck)\in\modc{M}$ such that the cardinality of the configuration $\moconfigtwo$ is at most $k$.
%Let $\bmon{\leq k}$ denote the subspace of elements $(\moconfigtwo ,\neck)\in\mon$ such that the cardinality of the configuration $\moconfigtwo$ is at most $k$. If $M$ is a manifold with boundary and $\smon$ is contained $\partial M$, define $\bmodc{M}{k}$ and $\bmodc{M}{\leq k}$ similarly.
\end{definition}
%\begin{definition}
%Let\footnote{I think this notation is redundant-see strict module replacement def on page 15} $$i \colon  S^{n-1}\times \left [0,\infty )\hookrightarrow D^{n}\cup_{S^{n-1}}(S^{n-1}\times \left [0,\infty )) $$ be the inclusion sending $S^{n-1}\times \left [0,\infty )$ to $S^{n-1}\times \left [0,\infty ) \subset D^{n}\cup_{S^{n-1}}(S^{n-1}\times \left [0,\infty ))$ via the identity map and let $$i^{'} \colon  \text{Conf}(S^{n-1}\times \left [0,\infty ))\hookrightarrow \text{Conf}(D^{n}\cup_{S^{n-1}}(S^{n-1}\times \left [0,\infty )))$$ denote the corresponding map on configuration spaces.
%\end{definition}
%\begin{definition}
%Let $\incl\colon S^{n-1}\times (0,\infty)\to \mathbb{R}^{n}\setminus \{0\}$ denote the homeomorphism
%\begin{align*}
%    \incl\colon S^{n-1}\times (0,\infty)&\to \mathbb{R}^{n}\setminus \{0\}\\
%    (q,u)&\mapsto u\cdot q.
%\end{align*}
%By viewing $\mathbb{R}^{n}\setminus \{0\}$ as a subspace of $\mathbb{R}^{n}$ and fixing a homeomorphism $\mathbb{R}^{n}\to D^{n}$, the map $\incl$ induces a natural map of spaces
%\begin{align*}
%    M(\incl)\colon \bmon{1}&\to \bmodc{D^{n}}{1}\\
%    (\moconfigone,\neck)&\mapsto (\incl(\moconfigone),\min(\neck-1,0)).
%\end{align*}
%\end{definition}
\begin{definition}\label{resolution for Conf(Disk)}
%Given $\smon\subseteq S^{n-1}$, 
Let $\dmodel_{\bullet}$ denote the semi-simplicial space whose $p$-simplices are
\[
\dmodel_{p}=\begin{cases}
     \bmodc[S^{n-1}]{\bar{D}^{n}}{\leq 1}\times\mon[S^{n-1}] & p=0\\
     \bmon[S^{n-1}]{1}\times\mon[S^{n-1}] & p=1\\
     \emptyset &\text{for } p>1.
   \end{cases}
   \]
%$Y_{0}\colonequals \mon[S^{n-1}]\times\bmodc{\bar{D}^{n}}{\leq 1}$, $Y_{1}\colonequals \mon[S^{n-1}]\times \bmon{1}$ and $Y_{p}\colonequals \emptyset$ for all $p>1$. 
The face maps $d_{0},d_{1} \colon \dmodel_{1}\to  \dmodel_{0}$ are defined as follows (see \cref{fig:face maps of semisimplicial model for Conf(D)}): given $\xi=\big ((\moconfigone_{0},\neck_{0}), (\moconfigone_{1}, \neck_{1})\big )\in \bmon[S^{n-1}]{1}\times\mon[S^{n-1}]$, let
%\footnote{maybe switch $\moconfigone$ to $\moconfigtwo$ }
\begin{align*}
    d_0 (\xi )\colonequals  & \big ((\emptyset,0), (\moconfigone_{0}, \neck_{0})\cdot (\moconfigone_{1},\neck_{1})\big )=\big ((\emptyset ,0), (\text{sh}_{\neck_0}(\moconfigone_{1})\cup \moconfigone_{0},\neck_{0} + \neck_{1} )\big ),\text{ and}\\
    d_1 (\xi )\colonequals  & \big ((\text{Conf}(\embed)(\moconfigone_{0}),\neck_{0}),(\moconfigone_{1},\neck_{1})\big ),
    %d_1 (\xi )\colonequals  & \big ((    M(\incl)(\moconfigone_{0},\neck_{0})),(\moconfigone_{1},\neck_{1})\big )= \big ((    \incl(\moconfigone_{0}),\neck_{0}),(\moconfigone_{1},\neck_{1})\big).
    %d_1 (\xi )\colonequals  & \big ((    \incl(\moconfigone_{0}),\neck_{0}),(\moconfigone_{1},\neck_{1})\big ).
\end{align*}
%$$d_0 (\xi )= ((y_{0}, t_{0})) \cdot (y_{1},t_{1}) ,(\{\emptyset\},0))=((\text{sh}_{t_1}(y_{0})\cup y_{1},t_0 +t_1),(\{\emptyset\},0))$$ and 
%$$d_1 (\xi ) = ((y_{0},t_{0}),(i^{'}(y_{1}),t_{1})).$$
where $\text{sh}_{\neck_{0}}$ is ``the shift by $\neck_{0}$'' map from \cref{shift map}, and $\embed\colon S^{n-1}\times (0,\infty)\hookrightarrow \bar{D}^{n}\cup_{S^{n-1}}(S^{n-1}\times \left [0,\infty\right))$ is the inclusion under the identity map from \cref{rev notation for iterated configuration space model}.
Let $\dmodel$ denote the geometric realization $\Vert \dmodel_{\bullet}\Vert$.
%and let $Y$ denote $\dmodel[S^{n-1}]$.
%\footnote{need to put picture in tikz-pdfs}
\begin{comment}
The space $\dmodel$ can be interpreted as the geometric realization of the semi-simplicial space $\dmodel_{\bullet}$, with $\dmodel_{0}\colonequals \mon[S^{n-1}]\times\bmodc{\bar{D}^{n}}{\leq 1}$, and $Y_{1}\colonequals \mon[S^{n-1}]\times \bmon{1}$ for $n>0$ and $Y_{n}\colonequals \emptyset $ for $n>1$.
Let
\\ \begin{center}\begin{tikzcd}
Y\colonequals \text{hocoeq}(\mon[S^{n-1}]\times \bmon{1} \ar[r,shift left=.75ex,"d_0"]
  \ar[r,shift right=.75ex,swap,"d_1"]
&
\mon[S^{n-1}]\times\bmodc{\bar{D}^{n}}{\leq 1}),
\end{tikzcd}\end{center}
\\ where for $\xi=\big ((\xi^{'},t_0), (p, t_1)\big )\in \mon[S^{n-1}]\times \bmon{1},$
%\\ $d_0 (\xi )= \big ((\text{sh}_{t_1}(\xi^{'})\cup \text{p},t_0 +t_1),(\{\emptyset\},0)\big )$ and 
\\$d_1 (\xi ) = \big ((\xi^{'},t_0), (p,t_{1})\big )$

$\xi=\big ((y_{0},t_{0}), (y_{1}, t_{1})\big )\in \mon[S^{n-1}]\times \bmon{1},$
%\\ $d_0 (\xi )= ((y_{1}, t_{1})) \cdot (y_{0},t_{0}) ,(\{\emptyset\},0))=((\text{sh}_{t_1}(y_{0})\cup y_{1},t_0 +t_1),(\{\emptyset\},0))$ and 
\\$d_1 (\xi ) = ((y_{0},t_{0}),i_{*}((y_{1},t_{1})))$
\end{comment}
\begin{figure}
    \centering
    \includegraphics{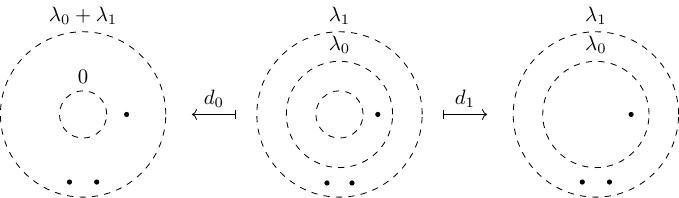}
  %\caption{I finished my secon tikz diagram. It took over 1 hours. less soy vey}
  \caption{The face maps $d_{0}, d_{1} \colon Y_{1}\to  Y_{0}$}
    \label{fig:face maps of semisimplicial model for Conf(D)}
\end{figure}
%(for the map $d_1$, we are implicitly using the fact that a point in $\text{Conf}(Y)$ can be treated as a configuration in the disk by an inclusion of the cylinder $Y$ into the disk $D$).
\end{definition}
The space $\dmodel$ can be described as a homotopy coequalizer $\text{hocoeq}(d_{0},d_{1} \colon \dmodel_{1}\to  \dmodel_{0})$. Note that up to homotopy, $\dmodel_{0}$ and $\dmodel_{1}$ can be viewed as free right $\mon[S^{n-1}]$-modules over $S^{0}$ and $S^{n-1}$, respectively. We have natural maps of spaces
\begin{align*}
    \dmodel_{p}\times \mon[S^{n-1}]&\to \dmodel_{p}\\
    \big ((\moconfigone_{0},\neck_{0}), (\moconfigone_{1}, \neck_{1})\big )\times (\moconfigone_{2}, \neck_{2})&\mapsto \big ((\moconfigone_{0},\neck_{0}), (\moconfigone_{1}, \neck_{1})\cdot (\moconfigone_{2}, \neck_{2})\big ).
\end{align*}
which are compatible with the face maps of $\dmodel_{\bullet}$. As a result, we have an induced map $\dmodel \times \mon[S^{n-1}]\to \dmodel$ that gives $\dmodel$ the structure of a right $\mon[S^{n-1}]$-module.

%We have a natural map $$\alpha_{0} \colon \bmodc{\bar{D}^{n}}{\leq 1}\times\mon[S^{n-1}] \to \bdisk$$ of $\mon[S^{n-1}]$-modules coming from the left $\mon[S^{n-1}]$-module map $$m_{r} \colon  \modc[S^{n-1}]{\bar{D}^{n}}\times \mon[S^{n-1}] \to \modc[S^{n-1}]{\bar{D}^{n}}.$$ Since $\alpha_{0}\circ d_0 =\alpha_{0}\circ d_1$, we have that $(Y_{\bullet},\bdisk,\alpha_{0} )$ is an augmented semi-simplicial space and there is an induced map $$\alpha  \colon Y \to  \bdisk$$ of $\mon[S^{n-1}]$-modules. We want to show that $\alpha \colon Y \to  \bdisk$ is a weak homotopy equivalence by showing it is a Serre microfibration with weakly contractible fibers. We cannot directly copy the argument of \cref{bar construction map is a microfibration} because $Y_{\bullet}$ is a semi-simplicial space. But we can replace $Y_{\bullet}$ with a simplicial space $EY_{\bullet}$ whose thick geometric realization $\Vert EY_{\bullet}\Vert$ is weakly equivalent to its thin geometric realization $\vert Y_{\bullet}\vert$.
%\bdisk
We have a natural map $$\alpha_{0} \colon \bmodc[S^{n-1}]{\bar{D}^{n}}{\leq 1}\times\mon[S^{n-1}] \to \modc[S^{n-1}]{\bar{D}^{n}}$$ of $\mon[S^{n-1}]$-modules coming from the right $\mon[S^{n-1}]$-module map $$m_{r} \colon  \modc[S^{n-1}]{\bar{D}^{n}}\times \mon[S^{n-1}] \to \modc[S^{n-1}]{\bar{D}^{n}}.$$ Since $\alpha_{0}\circ d_0 =\alpha_{0}\circ d_1$, we have that $(\dmodel_{\bullet},\modc[S^{n-1}]{\bar{D}^{n}},\alpha_{0} )$ is an augmented semi-simplicial space and there is an induced map $$\Vert\alpha_{0}\Vert\colon \dmodel \to  \modc[S^{n-1}]{\bar{D}^{n}}$$ of $\mon[S^{n-1}]$-modules. We want to show that $\Vert\alpha_{0}\Vert \colon \dmodel \to  \modc[S^{n-1}]{\bar{D}^{n}}$ is a weak homotopy equivalence. The map $\Vert\alpha_{0}\Vert \colon \dmodel \to  \modc[S^{n-1}]{\bar{D}^{n}}$ factors through $\bdisk$. Since $\bdisk$ is homotopy equivalent to $\modc[S^{n-1}]{\bar{D}^{n}}$, 
% by \cref{bdisk and mod-conf(disk) are equivalent}
showing that the map $\Vert\alpha_{0}\Vert \colon \dmodel\to \bdisk$ is a weak homotopy equivalence implies that the map $\Vert\alpha_{0}\Vert \colon \dmodel \to  \modc[S^{n-1}]{\bar{D}^{n}}$ is a weak homotopy equivalence.
%Since the inclusion $\bdisk\to \modc[S^{n-1}]{\bar{D}^{n}}$ is a homotopy equivalence by \cref{bdisk and mod-conf(disk) are equivalent}, it suffices to show that the map $\alpha \colon Y \to  \bdisk$ is a weak homotopy equivalence.
%We will show that the map $\Vert\alpha_{0}\Vert \colon \dmodel \to  \bdisk$ is a weak homotopy equivalence by using the following technique.
We will show that the map $\Vert\alpha_{0}\Vert \colon \dmodel \to  \bdisk$ is a weak homotopy equivalence by using Gray's theorem.
\begin{theorem}[{Gray's theorem, e.g. \textcite[Proposition 2.2]{MR2045835}}]\label{Gray's thm}
    Let $f\colon X_{1}\to X_{2}$ be a map of spaces and let $U$ and $V$ form an open cover of $X_{2}$. Suppose that the induced maps $$f^{-1}(U)\to U, \qquad f^{-1}(V)\to V, \quad  \text{ and } \quad f^{-1}(U\cap V)\to U\cap V$$ are all weak homotopy equivalences. Then $f\colon X_{1}\to X_{2}$ is also a weak homotopy equivalence.
\end{theorem}
\begin{proposition}\label{res is weakly equivalent to Conf(Disk)}
%Let $Y_{\bullet}$ be the semi-simplicial $\mon[S^{n-1}]$-module defined in \cref{resolution for Conf(Disk)}.
%%Let $Y_{0}\colonequals \mon[S^{n-1}]\times \bmodc{\bar{D}^{n}}{\leq 1}$ and $Y_{1}\colonequals \mon[S^{n-1}]\times \modc[S^{n-1}]{\bar{D}^{n}}$. 
The map $$\Vert\alpha_{0}\Vert\colon \Vert \dmodel_{\bullet}\Vert \to  \modc[S^{n-1}]{\bar{D}^{n}}$$ is a weak homotopy equivalence of right $\mon[S^{n-1}]$-modules.
\end{proposition}
\begin{proof}
The image of
the map $\Vert\alpha_{0}\Vert\colon \Vert \dmodel_{\bullet}\Vert \to  \modc[S^{n-1}]{\bar{D}^{n}}$ is contained in $\bdisk$.
%factors through the composition $\Vert Y_{\bullet}\Vert \to \bdisk\hookrightarrow \modc[S^{n-1}]{\bar{D}^{n}}$. 
Since the inclusion $\bdisk\hookrightarrow \modc[S^{n-1}]{\bar{D}^{n}}$ is a homotopy equivalence by \cref{bdisk and mod-conf(disk) are equivalent}, it suffices to show that the map $\Vert\alpha_{0}\Vert\colon \Vert \dmodel_{\bullet}\Vert \to\bdisk$ is a weak homotopy equivalence. 

Let $U\subset \bdisk$ denote the subspace of configurations with no point in the closed disk $\bar{B}(0,1)$ of radius $1$ centered at the origin. Let $V\subset \bdisk$ denote the subspace of configurations where there is a unique closest point to the origin. The map $\Vert\alpha_{0}\Vert^{-1}(U)\to U$ is a weak homotopy equivalence because $U$ is homotopy equivalent to $\text{Conf}(S^{n-1}\times (0,\infty))$ and $\Vert\alpha_{0}\Vert^{-1}(U)$ deformation retracts onto the subspace 
$$\{((\{\emptyset\}, 0), (\moconfigone,\neck))\}\cong \mon[S^{n-1}]\subset Y_{0},$$ 
which is homotopy equivalent to $\text{Conf}(S^{n-1}\times (0,\infty))$ by \cref{replacements are homotopy equivalent to config spaces}.

One can show that $\Vert\alpha_{0}\Vert^{-1}(V)\to V$ is a weak homotopy equivalence by showing that the map is a fibration with contractible fibers as follows. Given $(\moconfigone, \neck)\in V$, we write $\moconfigone$ as $\moconfigone=\{(q_{1}, u_{1}),\ldots, (q_{\partic}, u_{\partic})\}$ such that $u_{1}<u_{2}\leq\ldots u_{k}$. Given any map $f\colon D^{r}\times [0,1]\to V$, we have that $f(x,t)\equalscolon(\moconfigone(x,t), \neck(x,t))$ satisfies that $u_{1}(x,t)< u_{i}(x,t)$ for all $i=2,\ldots, \partic$ and $(x,t)\in D^{r}\times[0,1]$. In addition, since $D^{r}\times [0,1]$ is compact, there is some constant $K>0$ such that $u_{i}(x,t)-u_{1}(x,t)\geq K$ for all $i=2,\ldots, \partic$ and $(x,t)\in D^{r}\times [0,1]$. As a result, given any commutative diagram, 
\begin{center}\begin{tikzcd}
 D^{r} \times \{ 0 \}\arrow[r,"f_{0}"]\arrow[d] & \Vert\alpha_{0}\Vert^{-1}(V)\arrow[d,"\Vert\alpha_{0}\Vert"]\\
  D^{r}\times [0, 1]\arrow[r,"f"] & V
\end{tikzcd}\end{center}
we can lift $f$ to a map $$\tilde{f}(x,t)\colonequals((\moconfigone_{0}(x,t), \neck_{0}(x,t)),(\moconfigone_{1}(x,t), \neck_{1}(x,t))\colon D^{r}\times [0, 1]\to  \Vert\alpha_{0}\Vert^{-1}(V).$$ Roughly speaking, using the data of $f$ and the fact that $u_{i}(x,t)-u_{1}(x,t)\geq K$ for all $i>1$ and $(x,t)\in D^{r}\times [0,1]$, we can continuously pick $\neck_{0}(x,t)$ to be an element in $\{0\}\cup (u_{1}(x,t), u_{2}(x,t))$. We can use the data of the underlying configuration $\moconfigone(x,t)$ in $V$ to continuously assign underlying configurations $\moconfigone_{0}(x,t)$ and $ \moconfigone_{1}(x,t)$ in $\Vert\alpha_{0}\Vert^{-1}(V)$. Finally, the data of $\neck(x,t)$ in $V$ and $\neck_{0}(x,t)$ uniquely determines $\neck_{1}(x,t)$.

We now show that the fiber of the map $\Vert\alpha_{0}\Vert^{-1}(V) \to  V$ is contractible. Suppose that $\zeta\colonequals (\moconfigtwo,\neck)$ is an element of $V$, with $\moconfigtwo \colonequals\{ (q_{1},u_{1}),\ldots , (q_{k},u_{k})\}$ and $u_{1}< u_{2}\leq \ldots \leq u_{k}$. 
The fiber $\Vert\alpha_{0}\Vert^{-1}(\zeta)$ is given by the geometric realization of the sub-semi-simplicial space of $\dmodel_{\bullet}$ whose image under $\alpha_{0}$ is $\zeta$.
%\footnote{might want to fix this description} 
Denote this sub-semi-simplicial space by $\dmodel_{\bullet}(\zeta)$. We have two cases to consider: $u_{1}=0$ and $u_{1}>0$.

Suppose $u_{1}=0$. Then the space $\dmodel_{1}(\zeta)$ is empty, so $\Vert \dmodel_{\bullet}(\zeta)\Vert=\dmodel_{0}(\zeta)$. We have that $\dmodel_{0}(\zeta)$ is contractible because it is homeomorphic to an open interval or a half-open interval, depending on whether $q_{1}$ is in the interior of $\bar{D}^{n}$ or the boundary of it:
\begin{align*}
    \dmodel_{0}(\zeta)= 
    \begin{cases}
    \big\{\big(((q_{1},0),\neck_{0}),\text{sh}_{-\neck_{0}}(\moconfigtwo \setminus (q_{1},0) ,\neck - \neck_{0})\big):\neck_{0}\in \left[0, u_{2}\right)\big\}&\text{if } q_{1}\in D^{n},\\
        \big\{\big(((q_{1},0),\neck_{0}),\text{sh}_{-\neck_{0}}(\moconfigtwo \setminus (q_{1},0) ,\neck - \neck_{0})\big):\neck_{0}\in \left(0, u_{2}\right)\big\}:&\text{if } q_{1}\in \partial \bar{D}^{n}.
    \end{cases}
    %\dmodel_{0}(\zeta)&=\big\{\big(\big((q_{1},0),\neck_{2}\big ), \big (\text{sh}_{-\neck_{2}} (\moconfigtwo \setminus (q_{1},0) ,\neck - \neck_{2}\big )\big ):\neck_{2}\in [0, u_{2})\big\}\quad\text{if } q_{1}\in D^{n},\\
    %\dmodel_{0}(\zeta)&=\big\{\big(\big((q_{1},0),\neck_{2}\big ), \big (\text{sh}_{-\neck_{2}} (\moconfigtwo \setminus (q_{1},0) ,\neck - \neck_{2}\big )\big ):\neck_{2}\in (0, u_{2})\big\}\quad\text{if } q_{1}\in \partial \bar{D}^{n}.\footnote{do cases}
\end{align*}
Suppose $u_{1}>0$. We have that $\dmodel_{0}(\zeta)=F_{0}\cup F_{1}$, where
\begin{align*}
    F_{0}&=\big\{\big((\{\emptyset\},\neck_{0}), \text{sh}_{-\neck_{0}}(\moconfigtwo,\neck - \neck_{0})\big) :
    \neck_{0}\in \left[0, u_{1}\right)\big\}\cong \left[0, u_{1}\right), \text{ and}
    \\  
    F_{1}&=\big\{\big(((q_{1}, u_{1}), \neck_{0}), \text{sh}_{-\neck_{0}} (\moconfigtwo\setminus (q_{1},u_{1}),\neck - \neck_{0})\big):\neck_{0}\in \left(u_{1}, u_{2}\right)\big\}\cong  (u_{1}, u_{2}),
\end{align*}
and $\dmodel_{1}(\zeta)$ is equal to $F_{1}$.  The face maps $d_{0}, d_{1} \colon \dmodel_{1}(\zeta)\to  \dmodel_{0}(\zeta)$ correspond to the constant map $(u_{1}, u_{2})\to  0\in [0, u_{1})$ 
%%sends $\dmodel_{1}(\zeta)$ to $((y,t),(\{\emptyset\},0))\in F_{1}(h)$ and
and the identity map $(u_{1}, u_{2})\to  (u_{1}, u_{2})$, respectively.
%%$d_{1}$ sends $\dmodel_{1}(\zeta)$ to $F_{2}(h,h_{1})\subset \dmodel_{0}(\zeta)$ by the identity map.
The geometric realization is homeomorphic to a half-open interval and so it is contractible.

Finally, by a similar argument, the map $\Vert\alpha_{0}\Vert^{-1}(U\cap V)\to U\cap V$ is a weak homotopy equivalence. Therefore, by \cref{Gray's thm}, the map $\Vert\alpha_{0}\Vert\colon \Vert \dmodel_{\bullet}\Vert \to\bdisk$ is a weak homotopy equivalence.
\end{proof}
\begin{remark}\label{explanation for resolution}
%We have given a resolution of $\modc[S^{n-1}]{\bar{D}^{n}}$, but we have not described the origins of the resolution. One might wonder if there is a conceptual explanation behind the resolution or if it is just an inspired guess. 
We have given a resolution of $\modc[S^{n-1}]{\bar{D}^{n}}$ without any motivation. We will now show that there is a completely conceptual explanation for how to come up with a guess for a resolution.
%when $N=S^{n-1}$.
%We now show that there is a completely conceptual explanation for how to
%at least coming up with a guess for what a resolution of $\modc[S^{n-1}]{\bar{D}^{n}}$ should be.
%We will now give a completely conceptual explanation for at least guessing what a resolution of $\modc[S^{n-1}]{\bar{D}^{n}}$ should be.
%We have given a resolution of $\modc[S^{n-1}]{\bar{D}^{n}}$, but we have not described the origins of the resolution. 
%We will now show that there is a completely conceptual explanation for at least guessing what a resolution of $\modc[S^{n-1}]{\bar{D}^{n}}$ should be.
%One might ask if there is a conceptual explanation for this resolution of $\modc[S^{n-1}]{\bar{D}^{n}}$.
%E we have found a resolution  of $\modc[S^{n-1}]{\bar{D}^{n}}$ by free $\mon[S^{n-1}]$-modules, it is probably unclear how to 
%how we even came up with it.
%an explanation of how such a resolution 

%A free resolution $F_{\bullet}\to M$ of a  $S$-module $M$ 
%A resolution $Z_{\bullet}\to $
%Out of convenience, we will work at the level of chain complexes of spaces rather than spaces, even though we can arrive at the same guess by working directly with spaces--see the last paragraph of this remark.
%work with the chain complexes $S\colonequals C_{*}(\mon[S^{n-1}];\mathbb{Z})$ and $C_{*}(\modc[S^{n-1}]{\bar{D}^{n}};\mathbb{Z})$ instead of the spaces $\mon[S^{n-1}]$ and $\modc[S^{n-1}]{\bar{D}^{n}}$, even though we can obtain the same guess by working directly with these spaces--see the last paragraph of this remark. 
We have that $S\colonequals C_{*}(\mon[S^{n-1}];\mathbb{Z})$ is a differential graded ring and $C_{*}(\modc[S^{n-1}]{\bar{D}^{n}};\mathbb{Z})$ is a differential graded module over $S$. We want to find a resolution of  $C_{*}(\modc[S^{n-1}]{\bar{D}^{n}};\mathbb{Z})$ by free $S$-modules. 
%In our setting, the differential graded module $C_{*}(\modc[S^{n-1}]{\bar{D}^{n}};\mathbb{Z})$ over the differential graded ring $S\colonequals C_{*}(\mon[S^{n-1}];\mathbb{Z})$ is the analogue of the space $X$ and free $S$-modules correspond to cells. 
%A resolution of a module should be viewed as analogous to a CW approximation of a space. Recall that a CW approximation of a space $X$ is a CW complex $Z$ and a weak homotopy equivalence $Z\to X$. A lower bound on the number of cells of dimension $d$ needed to build a CW approximation to the space $X$ can be calculated from the singular homology group $H_{d}(X;\mathbb{Z})$.
%and the torsion subgroup of $H_{d-1}(X;\mathbb{Z})$. 
%In our setting, the differential graded module $C_{*}(\modc[S^{n-1}]{\bar{D}^{n}};\mathbb{Z})$ over the differential graded ring $S\colonequals C_{*}(\mon[S^{n-1}];\mathbb{Z})$ is the analogue of the space $X$ and free $S$-modules correspond to cells. 
The ring $S$ is augmented over $\mathbb{Z}$ via the map which sends a chain $\sum_{i}n_{i}\sigma_{i}$ to $\sum_{i}n_{i}$ (here, $n_{i}\in \mathbb{Z}$ and $\sigma_{i}\colon \bar{D}^{r}\to \modc[S^{n-1}]{\bar{D}^{n}}$ is a singular $r$-simplex). Therefore, the hyper Tor groups $$\text{Tor}^{S}_{*}(C_{*}(\modc[S^{n-1}]{\bar{D}^{n}};\mathbb{Z}),\mathbb{Z})$$ are defined. These hyper Tor groups provide lower bounds on the ranks of the free $S$-modules in a free resolution of $C_{*}(\modc[S^{n-1}]{\bar{D}^{n}};\mathbb{Z})$.
%These hyper Tor groups play a role similar to the singular homology $H_{*}(X;\mathbb{Z})$ because they provide lower bounds on the ranks of the free $S$-modules in a free resolution of $C_{*}(\modc[S^{n-1}]{\bar{D}^{n}};\mathbb{Z})$.
We can calculate these hyper Tor groups by calculating the hyper homology of the following bar construction
$$B_{\bullet}(C_{*}(\modc[S^{n-1}]{\bar{D}^{n}};\mathbb{Z}),C_{*}(\mon[S^{n-1}];\mathbb{Z}), \mathbb{Z})$$
which is equivalent to $ C_{*}(B_{\bullet}(\modc[S^{n-1}]{\bar{D}^{n}},\mon[S^{n-1}],\text{pt}) ;\mathbb{Z})$.
%and is a free resolution of $C_{*}(\modc[S^{n-1}]{\bar{D}^{n}};\mathbb{Z})$. 
The homology of this double complex is isomorphic to the singular homology $ H_{*}(\Vert B_{\bullet}(\modc[S^{n-1}]{\bar{D}^{n}},\mon[S^{n-1}],\text{pt})\Vert ;\mathbb{Z})$.

We can understand the space $\Vert B_{\bullet}(\modc[S^{n-1}]{\bar{D}^{n}},\mon[S^{n-1}],\text{pt})\Vert$ better by working with relative configuration spaces. 
%\footnote{Should I put the definition of relative configuration spaces in a definition environment?}
Let $L\subset X$ be a closed subspace of a space $X$. The relative configuration space $\text{Conf}(X,L)$ is the quotient space of $\text{Conf}(X)$ by the equivalence relation which identifies two finite subsets $s$ and $s_{1}$ of $X$ if $s\cap (X\setminus L)=s_{1}\cap (X\setminus L)$. The relative configuration space $\text{Conf}(D^{n},\text{pt})$ is homotopy equivalent to a point, so
$\modc[S^{n-1}]{\bar{D}^{n},\text{pt}}$ is homotopy equivalent to a point by a relative version of  \cref{replacements are homotopy equivalent to config spaces}. By a relative version of \cref{map from bar construction to conf is a weak equivalence}, the space $$\Vert B_{\bullet}(\modc[S^{n-1}]{\bar{D}^{n}},\mon[S^{n-1}],\modc[S^{n-1}]{\bar{D}^{n},\text{pt}})\Vert$$ is weakly equivalent to the relative configuration space $\text{Conf}(S^{n},x_{0})$. By pushing all but one point of a configuration in $\text{Conf}(S^{n},x_{0})$ to $x_{0}$, we have that $\text{Conf}(S^{n},x_{0})$ is homotopy equivalent to $S^{n}$ (this is a classic argument which happens in the construction of the scanning map--see \textcite[p.~92--93]{MR358766}). 

Since the homology $H_{*}(S^{n})$ is only nonzero in degrees $0$ and $n$, this suggests we might be able to give a resolution of $C_{*}(\modc[S^{n-1}]{\bar{D}^{n}};\mathbb{Z})$ from two free $S$-modules based on a suitable model of $S^{n}$. One such model for $S^{n}$ is given by 
%\footnote{I'm not sure which way of describing $S^{n}$ is better. I think the first one (as a homotopy coequslizer is} 
%i)
the homotopy coequalizer $\text{hocoeq}(S^{n-1}\rightrightarrows S^{0})$.
%ii) the pushout
%\begin{center}\begin{tikzcd}
%S^{n-1} \arrow[r]\arrow[d] & *\arrow[d]\\
%\bar{D}^{n} \arrow[r] & S^{n}.
%\end{tikzcd}\end{center}
This model for $S^{n}$ can be viewed as a length two \textit{semi-simplicial resolution} $$S^{n-1}\to S^{0}\rightrightarrows S^{n}$$ of $S^{n}$ (the description mentioned in \cref{proof sketch of stabilization map construction} involves a different model for $S^{n}$, the cone on $S^{n-1}\to\text{pt}$, which we will not use). As a result, it is natural for us to guess that a resolution of $C_{*}(\modc[S^{n-1}]{\bar{D}^{n}};\mathbb{Z})$ of length two can be built from free $S$-modules using $C_{*}(S^{n-1};\mathbb{Z})$ and $C_{*}(S^{0};\mathbb{Z})$.
\end{remark}
\section{Constructing Stabilization Maps}\label{sec:constructing stab maps}
In this section, we construct various stabilization maps: first on $\mon[S^{n-1}]$ in \cref{stab maps on conf(cyl)}, then on $\modc[S^{n-1}]{\bar{D}^{n}}$ in \cref{stab maps on conf(disk)}. Finally, in \cref{new stab maps for conf(M)} we define a stabilization map $\newstab{z}$ on $\text{Conf}(M)$ when $\browd{z}{e}=0$ using a stabilization map on $\modc[S^{n-1}]{\bar{D}^{n}}$ from \cref{stab maps on conf(disk)}. When $M$ is an open manifold, we also relate $\newstab{z}$ to $t_{z}$ at the end of \cref{new stab maps for conf(M)} in \cref{new stabilization on open manifold}. Due to technicalities that only appear in the proof of this proposition, we will need to work with various stabilization maps on both $\mon[S^{n-1}]$ and $\modc[S^{n-1}]{\bar{D}^{n}}$.
%as well as models for $\nch{\mon[S^{n-1}]}$ and $\nch{\modc[S^{n-1}]{\bar{D}^{n}}}$.
%We start by defining stabilization on $C_{*}(\mon[S^{n-1}];\fieldc)$ in \label{stab maps on conf(cyl)} 
\subsection{Stabilization Maps for the Configuration Space of a Cylinder}\label{stab maps on conf(cyl)}
In this subsection, we construct a stabilization map $t_{z}^{l}$ on $C_{*}(\mon[S^{n-1}];\fieldc)$.  In \cref{new stab maps for conf(M)}, we will use $t_{z}^{l}$ to define a new stabilization map on $C_{*}(\text{Conf}(M);\fieldc)$ when $\browd{z}{e}=0$. In order to obtain stability results for the configuration spaces of a closed manifold, we will need to show that when $M$ is open, the new stabilization map and $t_{z}$ are chain homotopic in a certain sense (see \cref{new stabilization on open manifold} for a more precise statement). Our proof that they are chain homotopic will essentially involve showing that various stabilization maps on $C_{*}(\mon[S^{n-1}];\fieldc)$ are chain homotopic. For this reason, we begin this section by introducing a model for $C_{*}(\mon[S^{n-1}];\fieldc)$. 
%with a chain complex $\fieldc[x] \otimes_{\fieldc}C_{*}(\text{Conf}(D^{n});\fieldc)$ that is chain homotopy equivalent to  
Then we construct stabilization maps $s_{z}^{l}$ and $s_{z}^{r}$ on
this model
%$\fieldc[x] \otimes_{\fieldc}C_{*}(\text{Conf}(D^{n});\fieldc)$. 
and relate them to $t_{z}^{l}$ (see \cref{explanation for why we have three different stabilizations} for more of an explanation of why we introduce $s_{z}^{l}$ and $s_{z}^{r}$ instead of working only with $t_{z}^{l}$).
%construct stabilization maps $s_{z}^{l}, s_{z}^{r}$ on another chain complex $\fieldc[x] \otimes_{\fieldc}C_{*}(\text{Conf}(D^{n});\fieldc)$ and show that this chain complex is chain homotopy equivalent to $C_{*}(\mon[S^{n-1}];\fieldc)$.
%
%In this section, we work with $\fieldc[x] \otimes_{\fieldc}C_{*}(\text{Conf}(D^{n});\fieldc)$ and show that it is chain homotopy equivalent to $C_{*}(\mon[S^{n-1}];\fieldc)$.
%consider the 
%introduce a model for $C_{*}(\mon[S^{n-1}];\fieldc)$
%give a chain homotopy equivalence  $\fieldc[x] \otimes_{\fieldc}C_{*}(\text{Conf}(D^{n});\fieldc)\to C_{*}(\mon[S^{n-1}];\fieldc)$. 
%Then we construct a stabilization map $t_{z}^{l}$ on $C_{*}(\mon[S^{n-1}];\fieldc)$ and stabilization maps $s_{z}^{l}, s_{z}^{r}$ on $\fieldc[x] \otimes_{\fieldc}C_{*}(\text{Conf}(D^{n});\fieldc)$ and compare these maps (see \cref{explanation for why we have three different stabilizations} for an explanation of why we introduce $s_{z}^{l}$ and $s_{z}^{r}$ instead of working only with $t_{z}^{l}$). 
In \cref{stab maps on conf(disk)}, we will use these stabilization maps to construct stabilization maps on various models of $C_{*}(\modc[S^{n-1}]{\bar{D}^{n}};\fieldc)$.
\begin{definition}
Given spaces $W$ and $X$, let
$EZ\colon C_{*}(W;\fieldc)\otimes C_{*}(X;\fieldc) \to C_{*}(W\times X;\fieldc)$
denote the Eilenberg--Zilber map.
%\footnote{maybe move this notation to section 2. Actually, it doesn't make much sense to move this to section 2. I only use it there to define the stabilization map on open manifolds, and I barely use it there}
%\footnote{I don't think I need the Alexander--Whitney map}
\end{definition}
The Eilenberg--Zilber map is a lax monoidal functor. As a result, if $M$ is a topological monoid with monoid map $(-\cdot -)\colon M\times M \to M$ then $\nch{M}$ is a (differential-graded) ring under the map $(-\cdot -)_{*}\circ EZ\colon \nch{M}\otimes \nch{M}\to \nch{M}$. In addition, if $Z$ is a right module over $M$ with module map $m_{r}\colon M \times Z \to Z$, then $\nch{Z}$ is a (differential graded) right module over $\nch{M}$ under the map $(m_{r})_{*}\circ EZ\colon \nch{M}\otimes \nch{Z}\to \nch{Z}$.

View $\fieldc[x]$ as a subring of $H_{*}(\mon[S^{n-1}];\fieldc)$ by identifying $x$ with a generator of $H_{n-1}((S^{n-1}\times\{1/2\})\times\{1\};\fieldc)\cong H_{n-1}(\bmon[S^{n-1}]{1};\fieldc)$ and extending to $x^{j}$ by using the ring structure on $H_{*}(\mon[S^{n-1}];\fieldc)$ 
(the ring structure on $H_{*}(\mon[S^{n-1}];\fieldc)$ comes from the ring structure on $\nch{\mon[S^{n-1}]}$).
%(the multiplication map on $H_{*}(\mon[S^{n-1}];\fieldc)$ is the map $(-\cdot-)_{*}\circ EZ$, where $(-\cdot-)\colon \mon[S^{n-1}]\times \mon[S^{n-1}]\to \mon[S^{n-1}]$ is the monoid map). 
The differential $\partial$ on $\fieldc[x] \otimes_{\fieldc}C_{*}(\modc[S^{n-1}]{\bar{D}^{n}};\fieldc)$ sends $x^{j}\otimes y$ to $x^{j}\otimes \partial y$.
\begin{definition}
%Let $N^{'}$ denote either $\bar{D}^{n-1}$ or $S^{n-1}$. 
Fix an embedding $f'\colon D^{n}\to D^{n-1}\times (0,1)\subset S^{n-1}\times (0,1)$. Let $\text{MConf}(f')$ 
%and $\text{Mod-Conf}(f')$ 
denote the map
%Define
%\begin{align*}
%    \text{MConf}(f')\colon \text{Conf}(D^{n})&\to \mon[N^{'}]\text{ and}\\
%    \text{Mod-Conf}(f')\colon \modc[N^{'}]{\bar{D}^{n}}\times\text{MConf}(D^{n})&\to \modc[N^{'}]{\bar{D}^{n}}
%\end{align*}
%$\text{MConf}(f')\colon \text{Conf}(D^{n})\to \mon[N^{'}]$ and $\text{Mod-Conf}(f')\colon \modc[N^{'}]{\bar{D}^{n}}\times\text{MConf}(D^{n})\to \modc[N^{'}]{\bar{D}^{n}}$ to be the maps
%The embedding $f'$ gives rise to maps
\begin{align*}
    \text{MConf}(f')\colon \text{Conf}(D^{n})&\to \mon[S^{n-1}].\\
    \moconfigone&\mapsto (\text{Conf}(f')(\moconfigone),1)
\end{align*}
%and\footnote{probably don't need the following map}
%\begin{align*}
%    \text{Mod-Conf}(f')\colon \modc[S^{n-1}]{\bar{D}^{n}}\times\text{Conf}(D^{n})&\to \modc[S^{n-1}]{\bar{D}^{n}}\\
%    \big((\moconfigone,\neck), \moconfigone_{1}\big)&\mapsto m_{r}\big((\moconfigone,\neck), \text{MConf}(f')(\moconfigone_{1})\big).
%\end{align*}
\end{definition}
\begin{definition}\label{map between different models of Conf(cylinder)}
%by identifying $x$ with a generator of %$H_{n-1}((S^{n-1}\times\{1/2\})\times\{1\};\fieldc)\cong H_{n-1}(\bmon[S^{n-1}]{1};\fieldc)$
%$H_{n-1}(\bmon[S^{n-1}]{1};\fieldc)\cong H_{n-1}(S^{n-1};\fieldc)$
%and using the ring structure on $H_{*}(\mon[S^{n-1}];\fieldc)$.
%Fix a generator $\nu\in H_{n-1}(\bmon[S^{n-1}]{1};\fieldc)\cong H_{n-1}(S^{n-1};\fieldc)$.
%\footnote{Should I just say something like "View $\fieldc[x]$ as a subring of $H_{*}(\mon[S^{n-1}];\fieldc)$ by identifying $x$ with a generator of $H_{n-1}(\bmon[S^{n-1}]{1};\fieldc)$ and using the ring structure on $H_{*}(\mon[S^{n-1}];\fieldc)$."?}
%\footnote{use different notation for a generator of $ H_{n-1}(\bmon[S^{n-1}]{1};\fieldc)$ } 
%The monoid map on $\mon[S^{n-1}]$ gives $H_{*}(\mon[S^{n-1}];\fieldc)$ the structure of an associative ring.
%\begin{align*}
%    \mu\colon H_{*}(\mon[S^{n-1}];\fieldc)\otimes H_{*}(\mon[S^{n-1}];\fieldc)&\to H_{*}(\mon[S^{n-1}];\fieldc).\\
%    x\otimes y &\mapsto x\cdot y\colonequals (-\cdot-)\circ EZ(x\otimes y)
%\end{align*}
%Let $G\colon R[x]\to H_{*}(\mon[S^{n-1}];\fieldc)$ denote the map of rings obtained by sending $x^{i}$ to $\nu^{i}$ and extending linearly.
%\begin{align*}
%G\colon \mathbb{Z}[x]&\to H_{*}(\mon[S^{n-1}];\fieldc)\\
%x^{i}&\mapsto \sigma^{i}=\sigma\cdot \sigma\cdots \cdot \sigma\text{ } (i \text{ times})\\
%1&\mapsto \{\emptyset,0\}
%\end{align*}
Let $\alpha\colon \fieldc[x] \otimes_{\fieldc}C_{*}(\modc[S^{n-1}]{\bar{D}^{n}};\fieldc)\to C_{*}(\mon[S^{n-1}];\fieldc)$ denote the map (see \cref{alpha map on cylinder})
%\footnote{include a picture}
%The map $ \text{MConf}(f')\colon \text{Conf}(D^{n})\to \mon[S^{n-1}]$ induces a map on chain complexes, and so we have a map\footnote{include a picture}
\begin{align*}
   \alpha\colon \fieldc[x] \otimes_{\fieldc}C_{*}(\modc[S^{n-1}]{\bar{D}^{n}};\fieldc)&\to C_{*}(\mon[S^{n-1}];\fieldc)\\
   x^{j}\otimes y &\mapsto (-\cdot-)_{*}\circ EZ(x^{j}\otimes \text{MConf}(f')_{*}(y)).
\end{align*}
\begin{figure}[htb]
  \centering
  \includegraphics{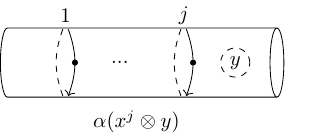}
  \caption{}
  \label{alpha map on cylinder}
\end{figure}
\end{definition}
\begin{lemma}\label{different models of conf(cyl) are equivalent}
The map $\alpha\colon \fieldc[x] \otimes_{\fieldc}C_{*}(\modc[S^{n-1}]{\bar{D}^{n}};\fieldc)\to C_{*}(\mon[S^{n-1}];\fieldc)$ is a chain homotopy equivalence.
%The weighed chain complex $\fieldc[x] \otimes_{\fieldc} C_{*}(\text{Conf}(D^{n});\fieldc) $ is chain homotopy equivalent to $C_{*}(\mon[S^{n-1}];\fieldc)$.
\end{lemma}
\begin{proof}
Since $\fieldc[x] \otimes_{\fieldc} C_{*}(\modc[S^{n-1}]{\bar{D}^{n}};\fieldc)$ and $C_{*}(\mon[S^{n-1}];\fieldc)$ are bounded below chain complexes of free $\fieldc$-modules, it suffices to show that $\alpha$ is a quasi-isomorphism. 
%Let $$D_{*}\colonequals \text{Cone}(\alpha\colon \fieldc[x] \otimes_{\fieldc}C_{*}(\modc[S^{n-1}]{\bar{D}^{n}};\fieldc)\to C_{*}(\mon[S^{n-1}];\fieldc))$$ be the mapping cone of $\alpha$. We want to show that $\tilde{H}_{*}(D_{*};\fieldc)$ vanishes. 
By the universal coefficient theorem, it suffices to prove this for $\fieldc=\mathbb{F}_{p}$ and $\mathbb{Q}$. 
%By the same argument as in the proof of \cref{integral secondary stability for surface}, it suffices to prove that $\tilde{H}_{*}(D_{*};\mathbb{F}_{p})$ vanishes for all primes $p$. We can do this by calculating $H_{*}(\mon[S^{n-1}];\mathbb{F}_{p})\cong H_{*}(\text{Conf}(S^{n-1}\times\mathbb{R});\mathbb{F}_{p})$. 

Note that 
$$\text{Conf}(S^{n-1}\times\mathbb{R})\simeq \text{Conf}(\mathbb{R}^{n}\setminus \text{pt})\simeq \text{Conf}(\mathbb{R}^{n}\setminus \text{pt})\times \mathbb{R}^{n}.$$
Let $\text{Conf}_{\partic,l}(\mathbb{R}^{n})$ denote the configuration space of $\mathbb{R}^{n}$ consisting of $\partic +l$ unordered distinct points in $\mathbb{R}^{n}$ with $\partic$ of the points labelled red and $l$ points labelled blue. 
The space $\text{Conf}_{\partic}(\mathbb{R}^{n}\setminus \text{pt})\times \mathbb{R}^{n}$ is the same as $\text{Conf}_{\partic,1}(\mathbb{R}^{n})$.
%The space $\text{Conf}_{\partic}(\mathbb{R}^{n}\setminus \text{pt})\times \mathbb{R}^{n}$ is the same as the unordered labelled configuration space $\text{Conf}_{\partic,1}(\mathbb{R}^{n})$ consisting of $\partic +1$ unordered distinct points in $\mathbb{R}^{n}$ with $\partic$ of the points labelled red and one point labelled blue. 
We can think of $\bigsqcup_{\partic\geq 0}\text{Conf}_{\partic,1}(\mathbb{R}^{n})$ as a subspace of the labelled configuration space $\text{Conf}(\mathbb{R}^{n};S^{0}\bigvee S^{0})$ (see \cref{conf with labels}). The space $\text{Conf}(\mathbb{R}^{n};S^{0}\bigvee S^{0})$ is homotopy equivalent to $\mathbf{E_{n}}(S^{0}\bigvee S^{0})$. 

Fix generators $b, r\in \tilde{H}_{0}(S^{0}\bigvee S^{0};\fieldc)$ corresponding to the two non-basepoint elements of $S^{0}\bigvee S^{0}$. Here, we think of $b$ as the class of a blue point and $r$ as the class of a red point. By \cref{homology of free En algebra}, $H_{*}(\mathbf{E_{n}}(S^{0}\bigvee S^{0});\fieldc)$ can be calculated by first forming the set of Lie words on $b$ and $r$. The free commutative algebra on the set of Lie words generates $H_{*}(\mathbf{E_{n}}(S^{0}\bigvee S^{0});\mathbb{Q})$, and for $\fieldc=\mathbb{F}_{p}$, we obtain $H_{*}(\mathbf{E_{n}}(S^{0}\bigvee S^{0});\mathbb{F}_{p})$ by formally applying $Q^{s}, \beta Q^{s}, \xi, \beta\xi,$ and $\zeta$. Note that by a generalization of the geometric interpretation of the Browder bracket (see \cref{geometric description of Browder bracket}), we have that a Lie word containing more than one $b$ (for example, $\browd{b}{\browd{r}{b}}$) corresponds to a homology class of $\text{Conf}(\mathbb{R}^{n};S^{0}\bigvee S^{0})$ with at least two blue points. As a result, instead of considering the set of all Lie words on $b$ and $r$, we only need to consider the set of Lie words containing at most one $b$ to calculate $H_{*}(\bigsqcup_{\partic\geq 0}\text{Conf}_{\partic,1}(\mathbb{R}^{n});\fieldc)$. If $w$ is  a Lie word containing only $r$'s, then by \cref{Lie word lemma} $w$ is either $r$ or $\browd{r}{r}$, so we only need to figure out the set of all Lie words containing exactly one $b$.

Suppose that $w$ is a Lie word containing exactly one $b$. We claim that we can assume that $w$ is of the form $w=\browd{r}{\browd{r}{\ldots\browd{r}{b}}\ldots}$. We prove this claim by induction on $\text{weight}(w)$. If $y$ is a Lie word of weight $1$ and contains $b$, then by definition $w=b$. Suppose that the claim holds for all Lie words of weight $1,\ldots, m$. Let $w$ be a word of weight $m+1$. Then $w=\browd{x}{y}$ with $m+1=\text{weight}(x)+\text{weight}(y)$ and $b$ contained exactly once in $x$ or $y$. Since the Browder bracket is symmetric up to a sign in $H_{*}(\mathbf{E_{n}}(S^{0}\bigvee S^{0});\fieldc)$ (Property \ref{Browder bracket symmetric up to sign} of the Browder bracket), without loss of generality we can assume that $x$ is a Lie word containing only $r$'s and $y$ is a Lie word containing exactly one $b$. Since $y$ is a Lie word of weight at most $m$ it is of the form $y=\browd{r}{\browd{r}{\ldots\browd{r}{b}}\ldots}$. For simplicity, we write $y$ as $y=\browd{r}{z}$. Since $x$ is a Lie word containing only $r$'s, by \cref{Lie word lemma}, we can assume that $x$ is either $r$ or $\browd{r}{r}$. If $x$ is $r$, then the claim holds. Otherwise, by the Jacobi identity for the Browder bracket (Property \ref{Jacobi identity of Browder bracket}), we have that
\begin{align*}
    0 =& \browd{\browd{r}{r}}{\browd{r}{z}}\pm\browd{r}{\browd{z}{\browd{r}{r}}}\pm \browd{z}{\browd{\browd{r}{r}}{r}}\\
      =& w \pm\browd{r}{\browd{z}{\browd{r}{r}}}\pm \browd{z}{\browd{\browd{r}{r}}{r}}.
\end{align*}
%$$0=\browd{\browd{r}{r}}{\browd{r}{z}}\pm\browd{r}{\browd{z}{\browd{r}{r}}}\pm \browd{z}{\browd{\browd{r}{r}}{r}}.$$ 
By induction, since $\browd{z}{\browd{r}{r}}$ is a Lie word of weight at most $m$, it is of the form $\browd{r}{\browd{r}{\ldots\browd{r}{b}}\ldots}$. Since $\browd{\browd{r}{r}}{r}=0$  (Property \ref{Jacobi identity of Browder bracket}), we have that $w$ is of the desired form.

We now describe how homology operations affect Lie words containing only $b$'s and $r$'s. We can think of $\mathbf{E_{n}}(S^{0}\bigvee S^{0})$ as an $E_{n}$-algebra graded over $\mathbb{Z}_{\geq 0}\times \mathbb{Z}_{\geq 0}$, i.e. $$\mathbf{E_{n}}(S^{0}\bigvee S^{0})=\bigsqcup_{(\partic,l)\in \mathbb{Z}_{\geq 0}\times \mathbb{Z}_{\geq 0}}A_{\partic,l},$$
where $A_{\partic,l}$ is homotopy equivalent
%with $(\partic,l)\in \mathbb{Z}_{\geq 0}\times \mathbb{Z}_{\geq 0}$ corresponding 
to $\text{Conf}_{\partic,l}(\mathbb{R}^{n})$. Similar to how the homology operations $Q^{s}, \beta Q^{s}, \xi, \beta\xi,$ and $\zeta$ act on the homology an $E_{n}$-algebra graded over $\mathbb{N}$, these homology operations send classes in $H_{*}(\text{Conf}_{\partic,l}(\mathbb{R}^{n});\mathbb{F}_{p})$ to classes in $H_{*}(\text{Conf}_{p\partic, pl}(\mathbb{R}^{n});\mathbb{F}_{p})$. As a result, the only way that applying one of the homology operations $Q^{s}, \beta Q^{s}, \xi, \beta\xi,$ or $\zeta$ to a Lie word containing at most one $b$ can be a homology class in $\bigsqcup_{\partic\geq 0}\text{Conf}_{\partic,1}(\mathbb{R}^{n})$ is if the Lie word contains only $r$'s. 

In Cohen's calculation of $H_{*}(\mathbf{E_{n}}(X);\mathbb{F}_{p})$ (see page 227 of \textcite{MR0436146}), the constraint on 
which formal applications of the homology operations $Q^{s}, \beta Q^{s}, \xi, \beta\xi,$ and $\zeta$ are allowed to act a Lie word $w$
depends only on $n$ and the homological degrees of $w$ and $p$. Therefore, we have an explicit isomorphism from $\fieldc[x]\otimes H_{*}(\text{Conf}(\mathbb{R}^{n});\fieldc)$ to $H_{*}(\bigsqcup_{\partic\geq 0}\text{Conf}_{\partic,1}(\mathbb{R}^{n});\fieldc)$ given by the map sending $x^{i}\otimes \sigma$ to $\browd{r}{\browd{r}{\ldots\browd{r}{b}}\ldots}\sigma$, where $r$ appears $i$ times in $\browd{r}{\browd{r}{\ldots\browd{r}{b}}\ldots}$, and extending linearly. Since this isomorphism corresponds to the map on homology induced by $\alpha$,
$$\alpha_{*}\colon \fieldc[x] \otimes_{\fieldc} H_{*}(\modc[S^{n-1}]{\bar{D}^{n}};\fieldc)\to H_{*}(\mon[S^{n-1}];\fieldc),$$ we have that $\alpha$ is a quasi-isomorphism.
\end{proof}
%Since $\mon[S^{n-1}]$ is a topological monoid, we have a map
%%\footnote{I'm confusing old notation for Mon with new notation. I need to figure out what notation I should use (here and elsewhere)}
%$$ (-\cdot-)_{*}\circ EZ\colon C_{*}(\mon[S^{n-1}];\fieldc)\otimes C_{*}(\mon[S^{n-1}];\fieldc)\to  C_{*}(\mon[S^{n-1}];\fieldc).$$ Fixing a homology class in $z\in H_{*}(\mon[S^{n-1}];\fieldc)$ yields two restriction maps by applying the Eilenberg--Zilber
%%\footnote{need to decide on notation for Eilenberg--Zilber map} 
%map to  $z\otimes -$ or $-\otimes z$.
%%$$(-\cdot-)_{*}\circ EZ, (-\cdot z)_{*}\circ EZ\colon C_{*}(\mon[S^{n-1}];\fieldc)\to  C_{*}(\mon[S^{n-1}];\fieldc).$$
Now we construct stabilization maps and compare them to each other.
%$$s^{l}_{z}, s^{r}_{z}\colon \fieldc[x]\otimes C_{*}(\modc[S^{n-1}]{\bar{D}^{n}};\fieldc)\to \fieldc[x]\otimes C_{*}(\modc[S^{n-1}]{\bar{D}^{n}};\fieldc)$$ and compare $s^{l}_{z}$ to the stabilization map $t^{l}_{z}\colon C_{*}(\mon[S^{n-1}];\fieldc)\to C_{*}(\mon[S^{n-1}];\fieldc)$.
\begin{definition}\label{left stabilization on mconf(sphere)}
%%Let $\fieldc$ be a ring. 
Given a class $z\in H_{*}(\text{Conf}(D^{n});\fieldc)$, let $$t_{z}^{l}\colon C_{*}(\mon[S^{n-1}];\fieldc)\to  C_{*}(\mon[S^{n-1}];\fieldc) $$ denote the map $(-\cdot-)_{*}\circ EZ \big( \text{MConf}(f')_{*}(z)\otimes -\big)$.
%Given a class $z\in H_{*}(\text{Conf}(D^{n});\fieldc)$, let $$t_{z}^{l}, t_{z}^{r}\colon C_{*}(\mon[S^{n-1}];\fieldc)\to  C_{*}(\mon[S^{n-1}];\fieldc) $$ denote the maps $(-\cdot-)_{*}\circ EZ \big( \text{MConf}(f')_{*}(z)\otimes -\big)$ and $(-\cdot-)_{*}\circ EZ \big(- \otimes  \text{MConf}(f')_{*}(z)\big)$ respectively.
%by abuse of notation,
%\footnote{maybe change notation for $t_{z}$} 
%define 
%\footnote{maybe denote this map by $t_{z}'$ instead of $t_{z}$} 
%$$t_{z}\colon C_{*}(\modc[S^{n-1}]{\bar{D}^{n}};\fieldc)\to C_{*}(\modc[S^{n-1}]{\bar{D}^{n}};\fieldc)$$ to be the restriction $\text{Mod-Conf}(f')_{*}\circ E(-\otimes z)$.
%the stabilization map $$t_{z}\colon C_{*}(\modc[S^{n-1}]{\bar{D}^{n}};\fieldc)\to C_{*}(\modc[S^{n-1}]{\bar{D}^{n}};\fieldc)$$ from \cref{stabilization map first def} is defined because $\text{Conf}(\mathbb{R}^{n})$ is homotopy equivalent to $\modc[S^{n-1}]{\bar{D}^{n}}$ by \cref{replacements are homotopy equivalent to config spaces}.
%By abuse of notation, let $$g^{'}_{z} \colon C_{*}(\bmodc[\bar{D}^{n-1}]{\bar{D}^{n}}{\leq 1} \times \mon[\bar{D}^{n-1}];\fieldc) \to  C_{*}(\modc[\bar{D}^{n-1}]{\bar{D}^{n}};\fieldc)$$ denote the restriction of $g_{z}'$ to $C_{*}(\dmodel[\bar{D}^{n-1}]_{0};\fieldc)\subset C_{*}(\dmodel[S^{n-1}]_{0};\fieldc)$.
%$(m_{r})_{*}\circ (t_{z} \otimes(\text{id}_{\mon[S^{n-1}]})_{*} )$.
%By abuse of notation, let $g_{z}$ also denote the induced map on homology.
\end{definition}
%Intuitively, we can think of $t^{l}_{z}$ and $t^{r}_{z}$ as maps which puts $z$ on the left-hand side of the cylinder.
%%\footnote{explain the difference between these two maps} 
\begin{definition}\label{stabilization maps on other model of Conf(cylinder)}
Given a positive integer $i$, let $\lambda^{i}(e,z)$ denote $\browd{e}{\browd{e}{\ldots\browd{e}{z}}\ldots}$, where $e$ appears in the iterated Browder bracket $i$ times, and let $\lambda^{0}(e,z)$ denote $z$.
Let $s^{l}_{z}, s^{r}_{z}\colon \fieldc[x]\otimes C_{*}(\modc[S^{n-1}]{\bar{D}^{n}};\fieldc)\to \fieldc[x]\otimes C_{*}(\modc[S^{n-1}]{\bar{D}^{n}};\fieldc)$ denote the maps
\begin{align*}
    s^{l}_{z}\colon \fieldc[x]\otimes C_{*}(\modc[S^{n-1}]{\bar{D}^{n}};\fieldc)&\to \fieldc[x]\otimes C_{*}(\modc[S^{n-1}]{\bar{D}^{n}};\fieldc),\\
     x^{j}\otimes y&\mapsto \sum_{i=0}^{j}\binom{j}{i}  x^{j-i}\otimes (m_{l})_{*}\circ EZ (\mon[f']_{*}(\lambda^{i}(e,z))\otimes y)\\
     s^{r}_{z}\colon \fieldc[x]\otimes C_{*}(\modc[S^{n-1}]{\bar{D}^{n}};\fieldc)&\to \fieldc[x]\otimes C_{*}(\modc[S^{n-1}]{\bar{D}^{n}};\fieldc).\\
    x^{j}\otimes y&\mapsto x^{j}\otimes (m_{r})_{*}\circ EZ (y\otimes \mon[f']_{*}(z))
    %x^{j}\otimes y&\mapsto x^{j}\otimes (m_{l})_{*}\circ EZ ( \mon[f']_{*}(z)\otimes y)\\
\end{align*}
%Define $s^{l}_{z}, s^{r}_{z}\colon \fieldc[x]\otimes C_{*}(\modc[S^{n-1}]{\bar{D}^{n}};\fieldc)\to \fieldc[x]\otimes C_{*}(\modc[S^{n-1}]{\bar{D}^{n}};\fieldc)$ to be the maps sending $x^{j}\otimes y$ to 
%Let $s^{l}_{z}, s^{r}_{z}\colon \fieldc[x]\otimes C_{*}(\modc[S^{n-1}]{\bar{D}^{n}};\fieldc)\to \fieldc[x]\otimes C_{*}(\modc[S^{n-1}]{\bar{D}^{n}};\fieldc)$
\end{definition}
\begin{remark}\label{explanation for why we have three different stabilizations}
We think of $s_{z}^{l}$ and $s_{z}^{r}$ as representing stabilizing $\nch{\mon[S^{n-1}]}$ by putting $z$ on the ``left'' and ``right'' side of $\nch{\mon[S^{n-1}]}$, respectively. For technical reasons, in the proof of \cref{new stabilization on open manifold} we will need that these two ways of stabilizing $\nch{\mon[S^{n-1}]}$ are chain homotopic when $\browd{z}{e}=0$. 
In that proof, we will also need something stronger: roughly speaking, the homotopy will also need to be a map of right $\nch{\mon[\bar{D}^{n-1}]}$-modules ($\nch{\mon[S^{n-1}]}$ is a right $\nch{\mon[\bar{D}^{n-1}]}$-module since $\nch{\mon[S^{n-1}]}$ is a right module over itself and $\bar{D}^{n-1}\subset S^{n-1}$).

An advantage of working with $\fieldc[x]\otimes \nch{\modc[S^{n-1}]{\bar{D}^{n}}}$ is that we can construct a chain homotopy between $s_{z}^{l}$ and $s_{z}^{r}$ when $\browd{z}{e}=0$ \textit{explicitly} and algebraically. This will ensure that the homotopy that we need in the proof of \cref{new stabilization on open manifold} will also be a map of right $\nch{\mon[\bar{D}^{n-1}]}$-modules. 
We still need to work with $t_{z}^{l}$ because we use it to define a stabilization map on $C_{*}(\modc[S^{n-1}]{\bar{D}^{n}};\fieldc)$ that is also a map of right $C_{*}(\mon[S^{n-1}];\fieldc)$-modules (we cannot use $s_{z}^{l}$ and $s_{z}^{r}$ because they are not maps of right $C_{*}(\mon[S^{n-1}];\fieldc)$-modules).
%
%Our reason for introducing $s_{z}^{l}$ and $s_{z}^{r}$ is that eventually we will need to construct or use \textit{explicit} chain homotopies between various stabilization maps (see \cref{induced gamma maps are homotopic} and \cref{new stabilization on open manifold}, in particular). From an algebraic standpoint, it is easier to do this by working with $s_{z}^{l}$ and $s_{z}^{r}$ than with $t_{z}^{l}$.
%We still need to work with $t_{z}^{l}$ because we use it to define a stabilization map on $C_{*}(\modc[S^{n-1}]{\bar{D}^{n}};\fieldc)$ that is also a map of right $C_{*}(\mon[S^{n-1}];\fieldc)$-modules (although we use $s_{z}^{l}$ and $s_{z}^{r}$ to define stabilization maps on a model of $C_{*}(\modc[S^{n-1}]{\bar{D}^{n}};\fieldc)$, these maps are not right $C_{*}(\mon[S^{n-1}];\fieldc)$-module maps).
\end{remark}
\begin{lemma}\label{comparing stabilization maps on different conf(cyl) models}
%The map $t_{z}^{r}\circ \alpha$ is chain homotopic to $\alpha\circ s_{z}^{r}$ and  
The maps $t_{z}^{l} \circ\alpha$ and $\alpha\circ s_{z}^{l}$ are chain homotopic.
%\footnote{I think I only need the second half of the lemma. I'm confused by which result I actually need}
%The diagrams\footnote{maybe get rid of these diagrams because they take up a lot of space} 
%\begin{center}\begin{tikzcd}
%\centering
%\fieldc[x]\otimes C_{*}(\modc[S^{n-1}]{\bar{D}^{n}};\fieldc) \arrow[r,"\alpha"]\arrow[d,"s^{l}_{z}"] & C_{*}(\mon[S^{n-1}];\fieldc)\arrow[d,"t^{l}_{z}"]\\
%\fieldc[x]\otimes C_{*}(\modc[S^{n-1}]{\bar{D}^{n}};\fieldc) \arrow[r,"\alpha"] & C_{*}(\mon[S^{n-1}];\fieldc).
%\end{tikzcd}\end{center}
%and
%\begin{center}\begin{tikzcd}
%\centering
%\fieldc[x]\otimes C_{*}(\modc[S^{n-1}]{\bar{D}^{n}};\fieldc) \arrow[r,"\alpha"]\arrow[d,"s^{r}_{z}"] & C_{*}(\mon[S^{n-1}];\fieldc)\arrow[d,"t^{r}_{z}"]\\
%\fieldc[x]\otimes C_{*}(\modc[S^{n-1}]{\bar{D}^{n}};\fieldc) \arrow[r,"\alpha"] & C_{*}(\mon[S^{n-1}];\fieldc).
%\end{tikzcd}\end{center}
%commute up to chain homotopy.
\end{lemma}
\begin{proof}
We have that
\begin{align*}
    \alpha\circ s^{l}_{z}(x^{j}\otimes y)&=\alpha \big(\sum_{i=0}^{j}\binom{j}{i}  x^{j-i}\otimes\big( (m_{l})_{*}\circ EZ (\mon[f']_{*}(\lambda^{i}(e,z))\otimes y)\big)\big)\\
    &=(-\cdot-)_{*}\circ EZ\big(\sum_{i=0}^{j}\binom{j}{i}  x^{j-i}\otimes \mon[f']_{*}\big( (m_{l})_{*}\circ EZ (\mon[f']_{*}(\lambda^{i}(e,z))\otimes y)\big)\big)
\end{align*}
and
\begin{align*}
    t^{l}_{z}\circ\alpha (x^{j}\otimes y)&=t^{l}_{z} \big((-\cdot-)_{*}\circ EZ \big(x^{j}\otimes \mon[f']_{*}(y)\big)\big)\\
%    &= (-\cdot-)_{*}\circ EZ\big((-\cdot-)_{*}\circ EZ\big(x^{j}\otimes \mon[f']_{*}(y)\big)\otimes \mon[f']_{*}(z)\big)\\
%    &= (-\cdot-)_{*}\circ EZ\big(x^{j}\otimes (-\cdot-)_{*}\circ EZ\big(\mon[f']_{*}(y)\otimes \mon[f']_{*}(z)\big)\big).
    &=(-\cdot-)_{*}\circ EZ\big(\mon[f']_{*}(z)\otimes \big( (-\cdot-)_{*}\circ EZ (x^{j}\otimes \mon[f']_{*}(y))\big)\big).
\end{align*}
We will prove that they are chain homotopic by recursion on $j$. %We first consider the case when $j=0$.
We have that $\alpha\circ s^{l}_{z}(1\otimes -)$ and $ t^{l}_{z}\circ\alpha (1\otimes -)$ are chain homotopic because the difference between these two maps corresponds to a different choice of orientation-preserving embeddings $\mathbb{R}^{n}\sqcup \mathbb{R}^{n}\hookrightarrow S^{n-1}\times (0,\infty)$ and the space of such embedding is path-connected.
%\footnote{include figures?}
%We now consider the case when $j=1$. We can represent $\alpha \circ s^{l}_{z}(x\otimes y)$ and $t^{l}_{z}\circ \alpha (x\otimes y)$ by the following figures.\footnote{include diagrams}
%Up to a chain homotopy, we have that the difference of $t^{l}_{z}\circ \alpha (x\otimes -)$ and $$(-\cdot-)_{*}\circ EZ\big (x\otimes \mon[f']_{*}( (m_{l})_{*}\circ EZ (\mon[f']_{*}(z)\otimes -)\big)\big)$$ is the same as $$(-\cdot-)_{*}\circ EZ\big (G(1)\otimes \mon[f']_{*}( (m_{l})_{*}\circ EZ (\mon[f']_{*}(\browd{e}{z})\otimes -)\big)\big).$$ So we can construct a chain homotopy in this case from $t^{l}_{z}\circ \alpha (x\otimes -)$ to $\alpha \circ s^{l}_{z}(x\otimes -)$.
%$$(-\cdot-)_{*}\circ EZ\big( \mon[f']_{*}(z) ( (-\cdot-)_{*}\circ EZ(x\otimes \mon[f']_{*}( (m_{l})_{*}\circ EZ (\mon[f']_{*}(-))\big)\big)$$

Suppose now that $j>0$.
%and that for all $i=1,\ldots, j-1$ we have that $t_{z}^{l}\circ \alpha (x^{i}\otimes -)$ is homotopic to $\alpha \circ  s_{z}^{l}(x^{i}\otimes -)$ for all classes $z$. 
We have that
%$$\alpha \circ s^{l}_{z}(x^{j}\otimes -)= (-\cdot-)_{*}EZ ((x\otimes (\emptyset,0))\cdot(\sum_{i=0}^{j-1}]\binom{j-1}{i}x^{j-1-i}\otimes(\mon[f']_{*}(m_{l})_{*} EZ (\lambda^{i}(e,z)\otimes -))))$$
\begin{align*}
    s^{l}_{z}(x^{j}\otimes -) =& (-\cdot-)_{*}\Big((x\otimes (\emptyset,0))\otimes\Big(\sum_{i=0}^{j-1}\binom{j-1}{i}x^{j-1-i}\otimes\big((m_{l})_{*}\circ EZ (\mon[f']_{*}(\lambda^{i}(e,z))\otimes -)\big)\Big)\Big)\\
    & + \sum_{i=0}^{j-1}\binom{j-1}{i}x^{j-1-i}\otimes ((m_{l})_{*}\circ EZ (\mon[f']_{*}(\lambda^{i}(e,\browd{e}{z}))\otimes -))\\
    =& (-\cdot-)_{*}\big((x\otimes (\emptyset,0))\otimes s^{l}_{z}(x^{j-1}\otimes -)\big) + s^{l}_{\browd{e}{z}}(x^{j-1}\otimes -)
\end{align*}
%As a result, we have that $\alpha \circ s^{l}_{z}(x^{j}\otimes -)=$

since $\binom{j-1}{i}+\binom{j-1}{i-1}=\binom{j}{i}$. Therefore,
\begin{align*}
    \alpha\circ s^{l}_{z}(x^{j}\otimes -)&= \alpha \big((-\cdot-)_{*}\big((x\otimes (\emptyset,0))\otimes s^{l}_{z}(x^{j-1}\otimes -)\big) + s^{l}_{\browd{e}{z}}(x^{j-1}\otimes -)\big)\\
    &=(-\cdot-)_{*}\circ EZ (x\otimes (\alpha\circ  s^{l}_{z}(x^{j-1}\otimes -))) + (\alpha\circ s^{l}_{\browd{e}{z}})(x^{j-1}\otimes -)
\end{align*}
%$$\alphas^{l}_{z}(x^{j}\otimes -)=$$
%It follows that we can recursively construct a chain homotopy on $x^{j}\otimes y$.

Similarly, by \cref{stab maps on conf(cylinder)}
%\footnote{make figure. Also, this might not be true} 
we have that $$t_{z}^{l}\circ \alpha (x^{j}\otimes -) \simeq (-\cdot -)_{*}\circ EZ( x\otimes (t_{z}^{l}\circ \alpha (x^{j-1}\otimes -)))+ (t_{\browd{e}{z}}^{l}\circ \alpha) (x^{j-1}\otimes -).$$
\begin{figure}[htb]
  \centering
  \includegraphics[scale=.93]{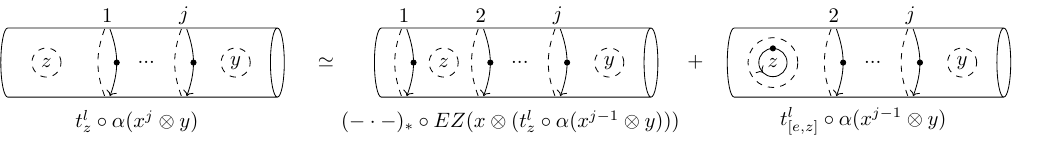}
  \caption{}
  \label{stab maps on conf(cylinder)}
\end{figure}
By recursion on $j$, it follows that 
%Since $t_{z}^{l}\circ \alpha (x^{j-1}\otimes -)$ is homotopic to $\alpha \circ s_{z}^{l}(x^{j-1}\otimes -)$ by our induction hypothesis, we have that 
$t_{z}^{l}\circ \alpha (x^{j}\otimes -)$ is chain homotopic to $\alpha \circ  s_{z}^{l}(x^{j}\otimes -)$.
\end{proof}
\begin{proposition}\label{comparing stabilization maps on the cylinder}
Suppose that $z\in H_{*}(\text{Conf}(D^{n});\fieldc)$ is a homology class such that the Browder bracket $\browd{z}{e}=0$. Then the maps $s^{l}_{z}$ and $s^{r}_{z}$ are chain homotopic.
\end{proposition}
\begin{proof}
%By \cref{comparing stabilization maps on different conf(cyl) models}, we have that $t^{l}_{z}\circ \alpha$ is homotopic to $\alpha\circ s^{l}_{z}$ and similarly $t^{r}_{z}\circ \alpha$ is homotopic to $\alpha\circ s^{r}_{z}$. Since $\alpha$ is a homotopy equivalence by \cref{different models of conf(cyl) are equivalent}, it suffices to show that $s^{l}_{z}$ is homotopic to $s^{r}_{z}$. 
We have that 
$$s^{l}_{z}(x^{j}\otimes y) = \sum_{i=0}^{j}\binom{j}{i}  x^{j-i}\otimes (m_{l})_{*}\circ EZ (\mon[f']_{*}(\lambda^{i}(e,z))\otimes y).$$ Since the map  $(x^{j}\otimes y)\mapsto x^{j}\otimes (m_{l})_{*}\circ EZ (\mon[f']_{*}(z)\otimes y)$ is homotopic to $s^{r}_{z}(x^{j}\otimes y)$, it suffices to show that the map $$F(x^{j}\otimes y)\colonequals \sum_{i=1}^{j}\binom{j}{i}  x^{j-i}\otimes (m_{l})_{*}\circ EZ (\mon[f']_{*}(\lambda^{i}(e,z))\otimes y)$$ is null-homotopic. Since $i$ is positive and $\lambda^{i}(e,z)=\lambda^{i-1}(e,\browd{e}{z})$, we have that 
$$F(x^{j}\otimes y)= \sum_{i=1}^{j}\binom{j}{i}  x^{j-i}\otimes (m_{l})_{*}\circ EZ (\mon[f']_{*}(\lambda^{i-1}(e,\browd{e}{z}))\otimes y).$$
Fix a chain $\omega\in C_{*}(\text{Conf}(D^{n});\fieldc)$ such that $\partial \omega=\browd{e}{z}$.
Consider the map 
$$H(x^{j}\otimes y)= (-1)^{\text{deg}(y)}\sum_{i=1}^{j}\binom{j}{i}  x^{j-i}\otimes (m_{l})_{*}\circ EZ (\mon[f']_{*}(\lambda^{i-1}(e,\omega))\otimes y).$$
We now check that $F=\partial H + H\partial$. We have that $\partial H(x^{j}\otimes y)$ is equal to 
%$$(-1)^{\text{deg}(y)}\sum_{i=1}^{j}\binom{j}{i}  x^{j-i}\otimes (m_{l})_{*}\circ EZ \big(\big(\mon[f']_{*}( \lambda^{i-1}(e,\omega))\otimes \partial y\big)+(-1)^{\text{deg}(y)}\mon[f']_{*}(\partial \lambda^{i-1}(e,\omega))\otimes y\big).$$
$$(-1)^{\text{deg}(y)}\sum_{i=1}^{j}\binom{j}{i}  x^{j-i}\otimes (m_{l})_{*}\circ EZ \big( \big((-1)^{\text{deg}(y)}\mon[f']_{*}(\partial \lambda^{i-1}(e,\omega))\otimes y\big)+ \mon[f']_{*}( \lambda^{i-1}(e,\omega))\otimes \partial y\big).$$ We have that $\partial \lambda(e,\sigma )= \lambda(e, \partial\sigma)$ for any chain $\sigma\in C_{*}(\text{Conf}(D^{n});\fieldc)$ because $e$ is a class of even degree. Therefore, we have that $\partial \lambda^{i-1}(e,\omega)=\lambda^{i-1}(e,\browd{e}{z})=\lambda^{i}(e,z)$ and so
$$\partial H(x^{j}\otimes y)=F(x^{j}\otimes y) + (-1)^{\text{deg}(y)}\sum_{i=1}^{j}\binom{j}{i}  x^{j-i}\otimes (m_{l})_{*}\circ EZ \big(  \mon[f']_{*}( \lambda^{i-1}(e,\omega))\otimes \partial y\big). $$
%$\partial \lambda^{i-1}(e,\omega)=\lambda^{i-1}(e,\browd{e}{z})=\lambda^{i}(e,z)$.
Since $H\partial (x^{j}\otimes y)=H(x^{j}\otimes \partial y)=(-1)^{\text{deg}(y)-1}\sum_{i=1}^{j}\binom{j}{i}  x^{j-i}\otimes (m_{l})_{*}\circ EZ \big(  \mon[f']_{*}( \lambda^{i-1}(e,\omega))\otimes \partial y\big)$, we have that $F=\partial H + H\partial$.
%\begin{align*}
%   \partial H(x^{j}\otimes y)&= (-1)^{\text{deg}(y)}\sum_{i=1}^{j}\binom{j}{i}  x^{j-i}\otimes \partial (m_{l})_{*}\circ EZ (\mon[f']_{*}(\lambda^{i-1}(e,\omega))\otimes y)\\
%   &= (-1)^{\text{deg}(y)}\sum_{i=1}^{j}\binom{j}{i}  x^{j-i}\otimes (m_{l})_{*}\circ EZ (\mon[f']_{*}(\lambda^{i-1}(e,\omega))\otimes \partial y)\\
%   & + (-1)^{\text{deg}(y)}\sum_{i=1}^{j}\binom{j}{i}  x^{j-i}\otimes (m_{l})_{*}\circ EZ (\mon[f']_{*}((-1)^{\text{deg}(y)}\partial\lambda^{i-1}(e,\omega))\otimes y)
%\end{align*}
%$$\partial H(x^{j}\otimes y)= (-1)^{\text{deg}(y)}\sum_{i=1}^{j}\binom{j}{i}  x^{j-i}\otimes (m_{l})_{*}\circ EZ (\mon[f']_{*}(\lambda^{i-1}(e,\omega))\otimes y)$$
%Since $\partial \lambda(e,\sigma )= \lambda(e, \partial\sigma)$ for any chain $\sigma\in C_{*}(\text{Conf}(D^{n});\fieldc)$ (because $e$ is a class of even degree), we have that $\partial H = F$.
\end{proof}
\subsection{Stabilization Maps for the Configuration Space of a Disk}\label{stab maps on conf(disk)}
In \cref{stabilization map construction} we constructed a semi-simplicial space $Y_{\bullet}$ whose geometric realization is weakly equivalent to $\modc[S^{n-1}]{\bar{D^{n}}}$.
%showed that $\Vert Y_{\bullet}\Vert$ is weakly homotopy equivalent to $\modc[S^{n-1}]{\bar{D^{n}}}$. 
In this subsection, we use $Y_{\bullet}$ to do something analogous on the chain level: we construct semi-simplicial chain complexes $U_{\bullet}$ and $V_{\bullet}$ whose totalizations $T(U)_{*}$ and $T(V)_{*}$, respectively, are chain complexes quasi-isomorphic to $C_{*}(\modc[S^{n-1}]{\bar{D^{n}}};\fieldc)$. The chain complex $T(V)_{*}$ has the structure of a right $\nch{\mon[S^{n-1}]}$-module. Then we construct stabilization maps on $T(U)_{*}$ and $T(V)_{*}$. We use the stabilization map
$t_{z}^{l}$ on $\nch{\mon[S^{n-1}]}$ from \cref{left stabilization on mconf(sphere)} to define a stabilization map $g_{z}^{l}$ on $T(V)_{*}$ that is also a map of right $\nch{\mon[S^{n-1}]}$-modules. Similarly, we use the stabilization maps $s_{z}^{l}$ and $s_{z}^{r}$ on $\fieldc[x]\otimes C_{*}(\modc[S^{n-1}]{\bar{D}^{n}};\fieldc)$ from \cref{stabilization maps on other model of Conf(cylinder)} to define stabilization maps $\gamma_{z}^{l}$ and $\gamma_{z}^{r}$ on $T(U)_{*}$. We conclude this subsection by showing that $\gamma_{z}^{l}$ and $\gamma_{z}^{r}$ are chain homotopic--we will need this in the proof of \cref{new stabilization on open manifold}.
\begin{definition}
%Let $V_{\bullet}'$ denote the semi-simplicial chain complex with $V_{0}'$
Let $V_{\bullet}''$ denote the semi-simplicial chain complex whose $p$-simplices are
\[
V_{p}''=\begin{cases}
     \nch{\bmodc[S^{n-1}]{\bar{D}^{n}}{\leq 1}}\otimes_{\fieldc} \nch{\mon[S^{n-1}]} & p=0\\
     \nch{\bmon[S^{n-1}]{1}}\otimes_{\fieldc} \nch{\mon[S^{n-1}]} & p=1\\
     \emptyset & p>1.
   \end{cases}
   \]
%$Y_{0}\colonequals \mon[S^{n-1}]\times\bmodc{\bar{D}^{n}}{\leq 1}$, $Y_{1}\colonequals \mon[S^{n-1}]\times \bmon{1}$ and $Y_{p}\colonequals \emptyset$ for all $p>1$. 
The face maps $d_{0},d_{1} \colon V_{1}''\to  V_{0}''$ are defined as follows: given $\sigma\otimes\tau\in V_{1}''$,
%$\xi=\big ((\moconfigone_{0},\neck_{0}), (\moconfigone_{1}, \neck_{1})\big )\in \bmon{1}\times\mon$, 
let
%\footnote{maybe switch $\moconfigone$ to $\moconfigtwo$ }
\begin{align*}
    d_0 (\sigma\otimes\tau)\colonequals  & (\emptyset,0)\otimes (-\cdot-)_{*}\circ EZ (\sigma \otimes \tau),\text{ and}\\
    d_1 (\sigma\otimes\tau)\colonequals  & \text{Conf}(\embed)_{*}(\sigma)\otimes \tau.
    %d_1 (\xi )\colonequals  & \big ((    M(\incl)(\moconfigone_{0},\neck_{0})),(\moconfigone_{1},\neck_{1})\big )= \big ((    \incl(\moconfigone_{0}),\neck_{0}),(\moconfigone_{1},\neck_{1})\big).
    %d_1 (\xi )\colonequals  & \big ((    \incl(\moconfigone_{0}),\neck_{0}),(\moconfigone_{1},\neck_{1})\big ).
\end{align*}
Let $V_{\bullet}'\subset V_{\bullet}''$ denote the sub-semi-simplicial chain complex with 
$V_{0}'=V_{0}''$ and $$V_{1}'= \nch{(S^{n-1}\times\{1/2\})\times\{1\}}\otimes_{\fieldc} \nch{\mon[S^{n-1}]}.$$
%$\nch{S^{n-1}\times\{1/2\})\times\{1\}}$
\end{definition}
\begin{definition}
Given a CW complex $X$,
%and $\fieldc$ a ring. 
let $C_{*}^{\text{cell}}(X;\fieldc)$ denote the cellular chain complex of $X$ with coefficients in $\fieldc$.
\end{definition}
\begin{definition}
%Given a space $X$, let $\nch{X}$ denote the normalized singular chain complex on $X$.
%\footnote{comment here if I define normalized chains earlier in the paper}
Given a semi-simplicial chain conplex $Z_{\bullet}$, with $Z_{p}=C_{p,*}$, let $(Z_{\bullet,*}, \partial_{H}, \partial_{V})$ denote the double complex $Z_{p,q}\colonequals C_{p,q}$, with the differential $\partial_{V}\colon Z_{\bullet ,q}\to Z_{\bullet ,q-1} $ coming from the differential on the chain complex $Z_{p}$ and the other differential $\partial_{H}\colon Z_{p,*}\to Z_{p-1,*} $ coming from the alternating sum of the face maps $d_{i} \colon Z_{p}\to  Z_{p-1}$.
%$\nch{Z_{\bullet}}_{*,*}$ denote the double complex on $$\nch{Z_{\bullet}}_{p,q}\colonequals C_{q}(Z_{p};\fieldc)$$ 
%with one differential coming from the differential on chains and the other coming from the alternating sum of the face maps $d_{i} \colon Z_{q}\to  Z_{q-1}$. 
By abuse of notation, let $\tot{Z_{\bullet}}$ denote the total complex of the double complex $(Z_{\bullet,*}, \partial_{H}, \partial_{V})$.
\end{definition}
We will need the following results in \cref{first zig-zag between models of chains(conf(disk))}.
\begin{theorem}[see e.g. {\textcite[Theorem 4.18]{MR1802847}}]\label{singular chains versus cellular chains}
Let $X$ be a CW complex.
There is a quasi-isomorphism $ C_{*}^{\text{cell}}(X;\fieldc)\to C_{*}(X;\fieldc)$.
\end{theorem}
\begin{proposition}[see e.g. {\textcite[Proposition 1.2]{MR1010881}}]\label{total complex is quasi-isomorphic to geometric realization}
%\label{MR1010881}
%Let $\fieldc$ be either a ring and let $Z_{\bullet}$ be a semi-simplicial space. Let $Z_{*,*}$ be the double complex with $Z_{q,r}=C_{r}(Z_{q};\fieldc )$ with one differential coming from by the differential on chains and the other coming from the alternating sum of the face maps $d_{i} \colon Z_{q}\to  Z_{q-1}$. Let $\Tot{Z}$ denote the total complex of the double complex $Z_{*,*}$.
Let $X_{\bullet}$ be a semi-simplicial space, and let $Z_{\bullet}$ be the semi-simplicial chain complex with $Z_{p}\colonequals \nch{X_{p}}$ and face maps on $Z_{\bullet}$ induced from the face maps of $X_{\bullet}$. Then there is a quasi-isomorphism $\tot{Z_{\bullet}}\to \nch{\Vert Z_{\bullet}\Vert}$.
\end{proposition}
Let $\mathbb{S}^{n-1}$ denote $S^{n-1}$ equipped with a CW-complex structure. 
%We identify $\mathbb{S}^{n-1}$ with $(S^{n-1}\times\{1/2\})\times\{1\}$ via the inclusion $x\mapsto ((x, 1/2),1)$. 
%Since we are taking normalized chains, we have an inclusion $C^{\text{cell}}_{*}(\mathbb{S}^{n-1};\fieldc)\hookrightarrow \nch{S^{n-1}}$.
\begin{definition}\label{main model for chains on Conf(D^n)}
Fix a CW structure on $\mathbb{S}^{n-1}$ using a single $0$-cell and a single $(n-1)$-cell only.
%Identify $\mathbb{S}^{n-1}$ with $ (S^{n-1}\times\{1/2\})\times\{1\}\subset \bmon[S^{n-1}]{1}$ via the map $x\mapsto ((x, 1/2),1)$. 
We map the $0$-cell to a point $p\in (S^{n-1}\times\{1/2\})\times\{1\}$ and we map the $(n-1)$-cell to $(S^{n-1}\times\{1/2\})\times\{1\}$ via the map induced by sending its boundary to $p$, so that we may identify $\mathbb{S}^{n-1}$ with $ (S^{n-1}\times\{1/2\})\times\{1\}\subset \bmon[S^{n-1}]{1}$.
%$C_{*}^{\text{cell}}(S^{n-1};\fieldc)$ is generated by a cell in degree $0$ and degree $n-1$. 
Let $V_{\bullet}\subset V_{\bullet}'$ denote the sub-semi-simplicial chain complex with 
$V_{0}=V_{0}'$ and $$V_{1}= C^{\text{cell}}_{*}(\mathbb{S}^{n-1};\fieldc)\otimes_{\fieldc} \nch{\mon[S^{n-1}]}$$ 
Let $T(V)_{*}$ denote $\tot{V_{\bullet}}$.
%The face maps $d_{0},d_{1} \colon V_{1}''\to  V_{0}''$ are defined as follows: given $\sigma\otimes\tau\in V_{1}''$,
%$\xi=\big ((\moconfigone_{0},\neck_{0}), (\moconfigone_{1}, \neck_{1})\big )\in \bmon{1}\times\mon$, 
%let
%\footnote{maybe switch $\moconfigone$ to $\moconfigtwo$ }
%\begin{align*}
%    d_0 (\sigma\otimes\tau)\colonequals  & (\emptyset,0)\otimes (-\cdot-)_{*}\circ EZ (\sigma \otimes \tau),\text{ and}\\
%    d_1 (\sigma\otimes\tau)\colonequals  & \text{Conf}(\embed)_{*}(\sigma)\otimes \tau.
    %d_1 (\xi )\colonequals  & \big ((    M(\incl)(\moconfigone_{0},\neck_{0})),(\moconfigone_{1},\neck_{1})\big )= \big ((    \incl(\moconfigone_{0}),\neck_{0}),(\moconfigone_{1},\neck_{1})\big).
    %d_1 (\xi )\colonequals  & \big ((    \incl(\moconfigone_{0}),\neck_{0}),(\moconfigone_{1},\neck_{1})\big ).
%\end{align*} 
\end{definition}
\begin{definition}
Let $U_{\bullet}$ denote the semi-simplicial chain complex with 
\[
U_{p}=\begin{cases}
     \nch{\bmodc[S^{n-1}]{\bar{D}^{n}}{\leq 1}} \otimes_{\fieldc} \fieldc[x]\otimes_{\fieldc} \nch{\modc[S^{n-1}]{\bar{D}^{n}}} & p=0\\
     C^{\text{cell}}_{*}(\mathbb{S}^{n-1};\fieldc)\otimes_{\fieldc} \fieldc[x]\otimes_{\fieldc}  \nch{\modc[S^{n-1}]{\bar{D}^{n}}} & p=1\\
     0 & p>1.
   \end{cases}
   \]
The face maps $d_{0},d_{1} \colon U_{1}\to  U_{0}$ are defined as follows: given $\sigma\otimes x^{j}\otimes\tau\in U_{1}$, let
%$\xi=\big ((\moconfigone_{0},\neck_{0}), (\moconfigone_{1}, \neck_{1})\big )\in \bmon{1}\times\mon$, 
%\footnote{maybe switch $\moconfigone$ to $\moconfigtwo$ }
\[
d_{0}(\sigma\otimes x^{j}\otimes\tau)=\begin{cases}
     (\emptyset,0)\otimes x^{j+1}\otimes\tau & \text{deg}(\sigma)=n-1\\
     (\emptyset,0)\otimes x^{j}\otimes\tau & \text{deg}(\sigma)=0 \\
   \end{cases}
   \]
and \[
d_{1}(\sigma\otimes x^{j}\otimes\tau)=\begin{cases}
     \text{Conf}(\embed)_{*}(\sigma)\otimes x^{j}\otimes\tau & \text{deg}(\sigma)=n-1\\
     (\emptyset,0)\otimes x^{j}\otimes\tau & \text{deg}(\sigma)=0 .\\
   \end{cases}
   \]
Let $T(U)_{*}$ denote $\tot{U_{\bullet}}$.   
\end{definition}
%Let $\fieldc$ be a ring and consider the chain complexes $C_{*}(\xbarc{M_{1}}{\smon}{M_{2}}{\bullet};\fieldc)$ and $C_{*}(\text{Conf}(M);\fieldc)$. Since $\xbarc{M_{1}}{\smon}{M_{2}}{\bullet}$ is a simplicial space, we can apply the following well-known result (see \textcite[Proposition 1.2]{MR1010881} for a similar result):
%\footnote{This proposition isn't the same as the proposition in the reference-that proposition proves this result for singular cochains and it assumes we're working with a simplicial space. Should I try to find another reference?}
\begin{lemma}\label{first zig-zag between models of chains(conf(disk))}
There is a quasi-isomorphism of right $\nch{\mon[S^{n-1}]}$-modules
$$T(V)_{*} \to \nch{\Vert Y_{\bullet}\Vert}.$$
\end{lemma}
\begin{proof}
%\text{Tot}(Z_{\bullet})
We have that $V_{\bullet}''$ has a right $\nch{\mon[S^{n-1}]}$-module structure coming from the right module structure of $\nch{\mon[S^{n-1}]}$. 
%Let $Z_{\bullet}$ denote the semi-simplicial chain complex $\nch{Y_{\bullet}}$. 
The Eilenberg--Zilber map gives a levelwise quasi-isomorphism of semi-simplicial chain complexes $V_{\bullet}''\to \nch{Y_{\bullet}}$. It is also a map of right $\nch{\mon[S^{n-1}]}$-modules since the Eilenberg--Zilber map is a lax monoidal functor and the right $\nch{\mon[S^{n-1}]}$-module structure on $V_{\bullet}''$ and $Y_{\bullet}$ comes from the monoid map on $\mon[S^{n-1}]$. We also have levelwise quasi-isomorphisms of semi-simplicial chain complexes $V_{\bullet}'\to V_{\bullet}''$ and $V_{\bullet}\to V_{\bullet}'$, since $S^{n-1}$ is homotopy equivalent to $\bmon[S^{n-1}]{1}$ and $C^{\text{cell}}_{*}(\mathbb{S}^{n-1};\fieldc)$ is quasi-isomorphic to $\nch{S^{n-1}}$ by \cref{singular chains versus cellular chains}. By taking the totalization of these semi-simplicial chain complexes, we have a quasi-isomorphism of right $\nch{\mon[S^{n-1}]}$-modules
$T(V)_{*}\to \tot{\nch{Y_{\bullet}}}$. Finally the map $\tot{\nch{Y_{\bullet}}}\to \nch{\Vert Y_{\bullet}\Vert}$ is a quasi-isomorphism by \cref{total complex is quasi-isomorphic to geometric realization}, and it is a map of right $\nch{\mon[S^{n-1}]}$-modules.
\end{proof}
\begin{lemma}\label{alpha map gives a semi-simplcial map}
The map $$(\text{id}\otimes \alpha)_{\bullet}\colon U_{\bullet}\to V_{\bullet},$$ where $\alpha\colon \fieldc [x]\otimes_{\fieldc} \nch{\modc[S^{n-1}]{\bar{D}^{n}}}\to \nch{\mon[S^{n-1}]}$ is the map from \cref{map between different models of Conf(cylinder)}, is a map of semi-simplicial chain complexes and it induces a quasi-isomorphism $((\text{id}\otimes \alpha)_{\bullet})_{*}\colon T(U)_{*} \to T(V)_{*}$.
\end{lemma}
\begin{proof}
It is clear that $(\text{id}\otimes \alpha)_{\bullet}$ is a map of semi-simplicial chain complexes. Since $\alpha$ is a homotopy equivalence by \cref{different models of conf(cyl) are equivalent}, the map $(\text{id}\otimes \alpha)_{\bullet}$ induces a quasi-isomorphism $((\text{id}\otimes \alpha)_{\bullet})_{*}$. 
%claim follows
%Do this later.\footnote{Do this later.}
\end{proof}
\begin{lemma}
There is a composition of quasi-isomorphisms.
$$T(U)_{*} \to T(V)_{*}\to \nch{Y}.$$
\end{lemma}
\begin{proof}
%By \cref{alpha map gives a semi-simplcial map}, we have a map of semi-simplicial chain complexes $\text{id}\otimes \alpha_{\bullet}\colon U_{\bullet}\to V_{\bullet}$. Since $\alpha$ is a homotopy equivalence 
By \cref{alpha map gives a semi-simplcial map}, $(\text{id}\otimes \alpha_{\bullet})_{*}\colon T(U)_{*}\to T(V)_{*}$ is a quasi-isomorphism. Since there is a quasi-isomorphism $T(V)_{*}\to \nch{Y}$ by \cref{first zig-zag between models of chains(conf(disk))}, and the composition of quasi-isomorphisms is a quasi-isomorphism, we have that $T(U)_{*}$ is quasi-isomorphic to $\nch{Y}$.
\end{proof}
%\subsection{Stabilization maps for Conf(Disk)}
Now we define stabilization maps
\begin{align*}
    \gamma_{z}^{l}, \gamma_{z}^{r}\colon T(U)_{*}&\to\nch{\modc[S^{n-1}]{\bar{D}^{n}}}\\
    g_{z}^{l}\colon T(V)_{*} &\to\nch{\modc[S^{n-1}]{\bar{D}^{n}}}.
\end{align*}
%on $T(U)_{*}$ and on $T(V)_{*}$.
To construct these stabilization maps, first observe that $U_{\bullet}$ and $V_{\bullet}$ only have simplices in degrees $0$ and $1$. Therefore to construct $g^{l}_{z}$ (likewise for $\gamma^{l}_{z}$ and $\gamma^{r}_{z}$), it suffices to construct a map $$g^{l}_{0,z}\colon \nch{V_{0}}\to \nch{\modc[S^{n-1}]{\bar{D}^{n}}}$$ and a homotopy $H$ from $g^{l}_{0,z}\circ d_{0}$ to $g^{l}_{0,z}\circ d_{1}$.
%induces a map $g\colon \nch{Y}\to \nch{\modc[S^{n-1}]{\bar{D}^{n}}}$.
%Need to write introduction.\footnote{Need to write introduction and fix title}
\begin{definition}
Let $$g^{l}_{0,z}\colon V_{0} = \nch{\bmodc[S^{n-1}]{\bar{D}^{n}}{\leq 1}} \otimes \nch{\mon[S^{n-1}]} \to  \nch{\modc[S^{n-1}]{\bar{D}^{n}}}$$ be the composition $(m_{r})_{*}\circ EZ\circ \big((\text{id}_{\modc[S^{n-1}]{\bar{D}^{n}}})_{*}\otimes t^{l}_{z} \big)$.
%Let $$g^{l}_{0,z}, g^{r}_{0,z} \colon V_{0} = C_{*}(\bmodc[S^{n-1}]{\bar{D}^{n}}{\leq 1};\fieldc) \otimes C_{*}(\mon[S^{n-1}];\fieldc) \to  C_{*}(\modc[S^{n-1}]{\bar{D}^{n}};\fieldc)$$ be the compositions $(m_{r})_{*}\circ EZ\circ \big((\text{id}_{\modc[S^{n-1}]{\bar{D}^{n}}})_{*}\otimes t^{l}_{z} \big)$ and $(m_{r})_{*}\circ EZ\circ \big((\text{id}_{\modc[S^{n-1}]{\bar{D}^{n}}})_{*}\otimes t^{r}_{z} \big)$, respectively.
%\footnote{need to define the notation for the induced map. Should make the notation consistent with $\gamma_{0,z}^{l}$ } 
\end{definition}
\begin{definition}
Let $$\gamma^{l}_{0,z}, \gamma^{r}_{0,z} \colon  C_{*}(\bmodc[S^{n-1}]{\bar{D}^{n}}{\leq 1};\fieldc) \otimes \fieldc[x]\otimes C_{*}(\modc[S^{n-1}]{\bar{D}^{n}};\fieldc) \to  C_{*}(\modc[S^{n-1}]{\bar{D}^{n}};\fieldc)$$ be the compositions $(m_{r})_{*}\circ EZ\circ \big(\text{id}\otimes  (\alpha\circ s^{l}_{z}) \big)$ and $(m_{r})_{*}\circ EZ\circ \big(\text{id}\otimes (\alpha\circ s^{r}_{z}) \big)$, respectively. 
\end{definition}
%Since the Eilenberg--Zilber map is a lax monoidal functor, the map $(m_{r})_{*}\circ EZ$ is a map of right $C_{*}(\mon[S^{n-1}];\fieldc)$-modules.
%because the map $m_{r}$ comes from the right module structure on the submodule $\mon[S^{n-1}]\subset \modc[S^{n-1}]{\bar{D}^{n}}$.
%\footnote{I don't know how to justify the claim that $(m_{r})_{*}\circ EZ$ is a map of right $C_{*}(\mon[S^{n-1}];\fieldc)$-modules in a way that seems clear. I think it's clear because when you restrict $m_{r}$ to the submodule $\mon[S^{n-1}]\subset \modc[S^{n-1}]{\bar{D}^{n}}$, this is the same as the monoid map on $\mon[S^{n-1}]$ (i.e. $(x,y)\mapsto x\cdot y$) and when you restrict $m_{r}$ to the subset $\text{Conf}(D^{n})$, this is just the identity map on $\text{Conf}(D^{n})$ (i.e. $(x,y)\mapsto x$).} 
Since the maps $(m_{r})_{*}\circ EZ$ and $t^{l}_{z}$ are maps of right $C_{*}(\mon[S^{n-1}];\fieldc)$-modules, the map $g^{l}_{0,z}$ is a map of right $C_{*}(\mon[S^{n-1}];\fieldc)$-modules.
%In order for $g^{l}_{0,z}$ to induce a map
%\begin{lemma}\label{stabilization maps from the right side commute}
%The maps $g^{r}_{0,z}\circ d_{0}$ and $g^{r}_{0,z}\circ d_{1}$ agree.
%\end{lemma}
%\begin{proof}
%We can represent $g^{r}_{0,z}\circ d_{0}$ and $g^{r}_{0,z}\circ d_{1}$ with the following figures.%\footnote{include pictures}
%
%Since these pictures are the same, the two maps agree.
%\end{proof}
%Unlike with $g^{r}_{0,z}$, the maps $g^{l}_{0,z}\circ d_{0}$ and $g^{l}_{0,z}\circ d_{1}$ do not agree. But, we will show that these two compositions are chain homotopic (see \cref{p-power stabilization}).
%To show this, we will need the following well-known result.
%To obtain a stabilization map from $g^{'}_{z}$, we will need the following well-known result.
\begin{proposition}\label{p-power stabilization}
%Let $N^{'}$ denote $S^{n-1}$ or $\bar{D}^{n-1}$.
Let $\fieldc$ be a ring
%denote either $\mathbb{F}_{p}$ or $\mathbb{Q}$
and let 
%Suppose 
$e\in H_{0}(\text{Conf}_{1}(\mathbb{R}^{n});\fieldc)$ be the class of a point. Suppose that $z\in H_{*}(\text{Conf}(\mathbb{R}^{n});\fieldc)$ is such that the Browder bracket $\browd{e}{z}=0$. Then the maps $$g^{l}_{0,z}\circ d_{0},\text{ } g^{l}_{0,z} \circ d_{1}\colon V_{1}=C_{*}^{\text{cell}}(\mathbb{S}^{n-1};\fieldc)\otimes C_{*}(\mon[S^{n-1}];\fieldc)\to C_{*}(\modc[S^{n-1}]{D^{n}};\fieldc)$$
are chain homotopic.
%\footnote{need to clean up the proof.}
%agree up to homotopy.
%As a consequence, when $\browd{e}{z}=0$ we can pick a homotopy between $g_{z}\circ (d_{0})_{*} $ and  $g_{z} \circ (d_{1})_{*}$ to get a map of $C_{*}(\mon[S^{n-1}];\fieldc)$-modules\footnote{say why it is a map of $C_{*}(\mon[S^{n-1}];\fieldc)$-module} $$g_{z}\colon C_{*}(Y;\fieldc)\to C_{*}(\modc[S^{n-1}]{\bar{D}^{n}};\fieldc). $$
%As a consequence, the map $g_{z}$ is a map of $C_{*}(\mon[S^{n-1}];\fieldc)$-modules when $\browd{e}{z}=0$.
%As a consequence, $g_{z}$ is a stabilization map when $x=y_{0}\in H_{0}(\text{Conf}_{1}(D);\mathbb{F}_{p})$ is the class of a point and $g_{z}$ is a higher-order stabilization map when the disk $D$ is of dimension $2$ and $x=y_{i}$, for $i$ at least $1$, is a homology class from \cref{thm-Co}.
\begin{comment}
Suppose 1)\footnote{I should make 1) and 2) listed/formatted correctly} $x$ is the class of a point $y_{0}\in H_{0}(\text{Conf}(D);\mathbb{F}_{p})$ and $D$ is a disk of dimension at least $2$ or 2) or $z$ is a class $y_{i}\in H_{2p^{i}-2}(\text{Conf}_{2p^{i}}(D);\mathbb{F}_{p})$ with $i\geq 1$ from \cref{thm-Co} and $D$ is a disk of dimension $2$.\footnote{Is this sentence too long/should I revise/condense it?} As a consequence, $g_{z}$ is a stabilization map in the first case and a higher-order stabilization map in the second.
\end{comment}
\end{proposition}
\begin{proof}
Suppose $\sigma\otimes (\emptyset ,0)\in V_{1}$. 
It suffices to check that $g^{l}_{0,z}\circ d_{0}$ and $g^{l}_{0,z}\circ d_{1}$ are chain homotopic when restricting to $\sigma\otimes (\emptyset ,0)\in V_{1}$, since $g^{l}_{0,z}\circ d_{0}$ and $g^{l}_{0,z}\circ d_{1}$ are maps of right $C_{*}(\mon[S^{n-1}];\fieldc)$-modules. 
%Given $\sigma \otimes\tau\in C^{\text{cell}}_{*}(S^{n-1};\fieldc) \otimes C_{*}(\mon[S^{n-1}];\fieldc) $, 
We can represent $g^{l}_{0,z}\circ d_{0}(\sigma \otimes (\emptyset ,0))$ and $g^{l}_{0,z}\circ d_{1}(\sigma \otimes (\emptyset ,0))$ via \cref{comparing the maps on the one-simplices}.
 \begin{figure}[htb]
  \centering
\includegraphics[scale=.8]{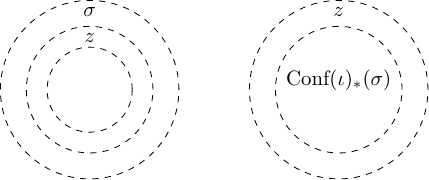}
  \caption{The maps $g^{l}_{0,z}\circ d_{0}$ and $g^{l}_{0,z}\circ d_{1}$ }
  \label{comparing the maps on the one-simplices}
\end{figure}
%Algebraically, we can express the image of $(\sigma,\tau)$ under these two maps as
%\begin{align*}
%    g^{'}_{z}\circ (d_{0})_{*} \colon & (\sigma,\tau) \mapsto ((\emptyset ,0), (\sigma\cdot\tau)) \mapsto  \mapsto\big (t_{z}(\emptyset ,0)\big ) \otimes (\sigma\cdot \tau)\mapsto (m_{r})_{*}(z, \sigma\cdot\tau),\\
%    g^{'}_{z}\circ (d_{1})_{*} \colon & \sigma\otimes\tau\mapsto \text{Conf}(\embed)_{*}(\sigma) \otimes \tau\mapsto t_{z}(\text{Conf}(\embed)_{*}\sigma)\otimes\tau \mapsto (m_{r})_{*}(t_{z}(\text{Conf}(\embed)_{*}\sigma ),\tau).
%\end{align*}

We can express the image of $\sigma\otimes (\emptyset ,0)$ under these two maps as
\begin{align*}
    g^{l}_{0,z}\circ d_{0} \colon & \sigma \otimes (\emptyset ,0) \mapsto (\emptyset ,0)\otimes (-\cdot -)_{*}\circ EZ(\sigma\otimes (\emptyset ,0)) \mapsto (\emptyset ,0)\otimes t_{z}^{l}(\sigma)\mapsto (m_{r})_{*}\circ EZ\big((\emptyset ,0)\otimes t_{z}^{l}(\sigma)\big),\\
    g^{l}_{0,z}\circ d_{1} \colon & \sigma \otimes (\emptyset ,0)\mapsto \text{Conf}(\embed)_{*}(\sigma) \otimes (\emptyset ,0)\mapsto 
    %\text{Conf}(\embed)_{*}(\sigma)\otimes t_{z}^{l}((\emptyset ,0)) \mapsto 
    (m_{r})_{*}\circ EZ(\text{Conf}(\embed)_{*}(\sigma)\otimes t^{l}_{z}((\emptyset ,0))).
\end{align*}
Since $\sigma$ is an element of $C_{*}^{\text{cell}}(\mathbb{S}^{n-1};\fieldc)$ and $\mathbb{S}^{n-1}$ only has one $0$-cell and one $(n-1)$-cell, $C_{*}^{\text{cell}}(\mathbb{S}^{n-1};\fieldc)$ is generated by a cell in degree $0$ and degree $n-1$.
%Therefore, we may treat $\sigma$ as an element of $C_{*}^{\text{cell}}(S^{n-1};\fieldc)$.  
%the singular chain complex $C_{*}(X)$ of a CW-complex $X$ is chain homotopy equivalent to the cellular chain complex $C_{*}^{\text{cell}}(X)$ of $X$,
%(see, for example, \textcite[Theorem 4.8]{MR1802847}),
%\footnote{see Page 60 of Rational Homotopy Theory by F\'elix, Halperin, Thomas}
%we may treat $\sigma$ as an element of $C_{*}^{\text{cell}}(S^{n-1})$. 
%Since we can give a CW structure on $S^{n-1}$ using one $0$-cell and one $(n-1)$-cell, $C_{*}^{\text{cell}}(S^{n-1};\fieldc)$ is generated by a cell in degree $0$ and degree $n-1$. 
As a result, 
%to show that 
%the maps $(m_{r})_{*}(z,-)$ and $t_{z} (\text{Conf}(\embed)_{*}(-))$
%$(g^{'}_{z}\circ (d_{0})_{*})(\sigma \otimes \tau)$ and $(g^{'}_{z}\circ (d_{1})_{*})(\sigma \otimes \tau)$
%are chain homotopy equivalent, 
we have two cases to consider:
\begin{enumerate}[label={Case \arabic*:}]
    \item\label{sigma is class of point} $\sigma$ is the class of a point in $H_{0}(S^{n-1};\fieldc)=C_{0}^{\text{cell}}(\mathbb{S}^{n-1};\fieldc)$, and
    \item\label{sigma is fundamental class} $\sigma$ is a generator of $ H_{n-1}(S^{n-1};\fieldc)=C_{n-1}^{\text{cell}}(\mathbb{S}^{n-1};\fieldc)$.
\end{enumerate}
%We have to consider two cases involving $\sigma\in H_{*}(\bmon{1})\cong H_{*}(S^{n-1})$ in order to show that $(g_{z}\circ (d_{0})_{*})(\sigma \otimes \tau)$ and $(g_{z}\circ (d_{1})_{*})(\sigma \otimes \tau)$ are the same up to homotopy: $\sigma$ is the class of a point in $H_{0}(S^{n-1})$ and $\sigma$ is a fundamental class in $ H_{n-1}(S^{n-1})$.
%\hfill\begin{enumerate}
%    \item $\tau$ is the class of a point in $H_{0}(S^{n-1})$, and
%    \item $\tau$ is a fundamental class in $ H_{n-1}(S^{n-1})$.
%\end{enumerate}

\ref{sigma is class of point}
Since $\sigma$ is the class of a point, 
%the maps $(m_{r})_{*}(z,\sigma)$ and $t_{z} (\text{Conf}(\embed)_{*}\sigma )$ are chain homotopy equivalent to the chain-level stabilization map $t_{1}(z)$,
%\footnote{include picture}
we can construct a chain homotopy in this case by fixing a path $\omega_{0}$ in $C_{*}(\modc[S^{n-1}]{\bar{D}^{n}};\fieldc)$ that swaps $z$ and the point in \cref{comparing the maps on the one-simplices}.

%the map $$(m_{r})_{*}(-,\sigma)\colon C_{*}(\text{Conf}(\mathbb{R}^{n});\fieldc)\to C_{*}(\text{Conf}(\mathbb{R}^{n});\fieldc)$$ is the chain-level stabilization map $t_{1}$.
%we may swap $z$ and $\sigma$ (see figure).\footnote{include figure} Therefore, $(g_{z}\circ (d_{1})_{*})$ and $(g_{z}\circ (d_{2})_{*})$ agree up to homotopy in this case.
%If $\sigma$ is the class of a point, then $(g_{z}\circ (d_{0})_{*})(\sigma \otimes \tau)=(g_{z}\circ (d_{1})_{*})(\sigma \otimes \tau)$  for any homology class of $\text{Conf}(D)$.\footnote{I don't know if I should explain why this is true or if it is obvious enough for the reader to believe it.}
\begin{comment}
If $\tau = d\omega$ for some chain $w\in C_{*}(S^{n-1})$ (if $\tau$ is a boundary chain), then clearly $\browd{\bar{\tau}}{y}=0$ for all $y\in H_{*}(\text{Conf}(D)$.
\end{comment}
\ref{sigma is fundamental class}
If $\sigma$ is a fundamental class in $H_{n-1}(S^{n-1};\fieldc)$, then $g^{l}_{0, z}\circ d_{1}(\sigma \otimes -) $ is null-homotopic, since $S^{n-1}\hookrightarrow \bar{D}^{n}$ being null-homotopic implies that $\text{Conf}(\embed)_{*}(\sigma)=\partial \omega'$ for some $\omega'\in C_{*}(\text{Conf}(D^{n});\fieldc)$. Since $$g_{0,z}^{l}\circ d_{0}(\sigma\otimes (\emptyset,0))=(m_{r})_{*}\circ EZ\big((\emptyset ,0)\otimes t_{z}^{l}(\sigma)\big) =\browd{e}{z}$$ (see \cref{fig:Browder bracket with a point}) and by assumption $\browd{e}{z}=0$, we have that $g^{l}_{0,z}\circ d_{1}(\sigma\otimes (\emptyset,0) )=\partial\omega$ for some chain $\omega$. Therefore, by picking a chain $\omega\in C_{*}(\text{Conf}(D^{n});\fieldc)$ such that $\partial\omega = \browd{e}{z}$, we can construct a null homotopy for $g^{l}_{0,z}\circ d_{0}(\sigma\otimes (\emptyset,0) )$. This gives a chain homotopy between $g^{l}_{0,z}\circ d_{0}$ and 
$g^{l}_{0,z}\circ d_{1}$ in this case as well. As a result, the map
\begin{align*}
H^{l}\colon C_{*}(V_{1};\fieldc)&\to C_{*}(\modc[S^{n-1}]{\bar{D}^{n}};\fieldc)\\
\sigma\otimes\tau &\mapsto\begin{cases}
     (m_{r})_{*}\circ EZ (\omega_{0}\otimes \tau) & \sigma\text{ a generator of } H_{0}(S^{n-1};\fieldc), \text{ and}\\
     (m_{r})_{*}\circ EZ ((\omega-\omega')\otimes\tau) &\sigma\text{ a generator of } H_{n-1}(S^{n-1};\fieldc)
   \end{cases}
\end{align*}
gives a chain homotopy from $g_{0,z}^{l}\circ d_{0}$ to $g_{0,z}^{l}\circ d_{1}$.
\end{proof}
\begin{definition}\label{definition of new stabilization map for Conf(Disk)}
%Let $\fieldc$ be a ring and 
Let $e\in H_{0}(\text{Conf}_{1}(\mathbb{R}^{n});\fieldc)$ be the class of a point. Suppose that $z\in H_{*}(\text{Conf}(\mathbb{R}^{n});\fieldc)$ is a homology class such that the Browder bracket $\browd{z}{e}=0$. 
%Fix a chain homotopy $H^{l}$ from $g^{l}_{0,z}\circ (d_{0})_{*}$ to $g^{l}_{0,z}\circ (d_{1})_{*}$. 
Let $g^{l}_{z}\colon T(V)_{*}\to C_{*}(\modc[S^{n-1}]{\bar{D}^{n}};\fieldc)$ denote the 
map on $T(V)_{*}$ induced by taking the homotopy from $g^{l}_{0,z}\circ d_{0}$ to $g^{l}_{0,z}\circ d_{1}$ to be the homotopy $H^{l}$ constructed at the end of \cref{p-power stabilization}.
%%induced map on $C_{*}(Y;\fieldc)$. 
%Let $g^{r}_{z}\colon C_{*}(Y;\fieldc)\to C_{*}(\modc[S^{n-1}]{\bar{D}^{n}};\fieldc)$ denote the map on $C_{*}(Y;\fieldc)$ induced by taking the homotopy $H^{r}$ from $g^{r}_{0,z}\circ d_{0}$ to $g^{r}_{0,z}\circ d_{1}$ to be the zero map.
\end{definition}
\begin{comment}
Now that we have an alternative model for the $\text{Conf}(Y)$-module $\text{Conf}(D)$, we can create a stabilization map $H_*(\text{Conf}(D);\mathbb{F}_p)\to  H_*(\text{Conf}(D);\mathbb{F}_{p})$ by forming a stabilization map on our new model for $\text{Conf}(D)$, i.e. a map $s_{x} \colon  H_{*}(\text{hocoeq}(X_1 \to  X_0 )\to  H_{*}(\text{hocoeq}(X_1 \to  X_0 );\mathbb{F}_{p})$.
\end{comment}
%When\footnote{drop all but the last sentence} the Browder bracket $\browd{e}{z}=0$, by \cref{p-power stabilization} the map $g_{z} \colon C_{*}(Y_{0};\fieldc)\to  C_{*}(\modc[S^{n-1}]{\bar{D}^{n}};\fieldc)$ induces a map of chain complexes $$g_{z} \colon \text{Tot-Com}(C_{*}(Y_{\bullet};\fieldc))_{*}\to  C_{*}(\modc[S^{n-1}]{\bar{D}^{n}};\fieldc).$$ Therefore, by \cref{thm-BenGit} we obtain a stabilization map $$g_{z} \colon H_{*}(\modc[S^{n-1}]{\bar{D}^{n}};\fieldc) \to  H_{*}(\modc[S^{n-1}]{\bar{D}^{n}};\fieldc).$$ 
\begin{lemma}\label{face maps composed with gamma maps are homotopic}
The maps $\gamma_{0,z}^{r}\circ d_{0}$ and $\gamma_{0,z}^{r}\circ d_{1}$ agree. If $\browd{e}{z}=0$, then the maps $\gamma_{0,z}^{l}\circ d_{0}$ and $\gamma_{0,z}^{l}\circ d_{1}$ are chain homotopic.
\end{lemma}
\begin{proof}
Given $\sigma\otimes x^{j}\otimes y\in U_{1}$, without loss of generality, we can assume that the degree of $\sigma$ is $n-1$, since $d_{0}$ and $d_{1}$ agree if the degree of $\sigma$ is zero.
Since $$(m_{r})_{*}\circ EZ (\text{Conf}(\embed)_{*}(\sigma)\otimes \alpha(x^{j}\otimes -))=(m_{r})_{*}\circ EZ ((\emptyset,0)\otimes \alpha(x^{j+1}\otimes -)),$$
we have that $\gamma_{0,z}^{r}\circ d_{0}$ and $\gamma_{0,z}^{r}\circ d_{1}$ agree.
%\footnote{maybe try to come up with a clearer explanation}

We now show that $\gamma_{0,z}^{l}\circ d_{0}$ and $\gamma_{0,z}^{l}\circ d_{1}$ are chain homotopic if $\browd{e}{z}=0$. We have that 
\begin{align*}
  \gamma_{0,z}^{l}d_{0}(\sigma\otimes x^{j}\otimes y)&= \gamma_{0,z}^{l}((\emptyset,0)\otimes x^{j+1}\otimes y)\\
  &= (m_{r})_{*}\circ EZ \big((\emptyset,0)\otimes (\alpha\circ s_{z}^{l})(x^{j+1}\otimes y) \big)\\
  &= (m_{r})_{*}\circ EZ \big((\emptyset,0)\otimes \alpha\big(\sum_{i=0}^{j+1}\binom{j+1}{i} x^{j+1-i}\otimes (m_{l})_{*}\circ EZ(\mon[f']_{*}(\lambda^{i}(e,z))\otimes y)\big) \big)
\end{align*}
and 
\begin{align*}
  \gamma_{0,z}^{l}d_{1}(\sigma\otimes x^{j}\otimes y)&= \gamma_{0,z}^{l}(\text{Conf}(\embed)_{*}(\sigma)\otimes x^{j}\otimes y)\\
  &= (m_{r})_{*}\circ EZ \big(\text{Conf}(\embed)_{*}(\sigma)\otimes (\alpha\circ s_{z}^{l})(x^{j}\otimes y) \big)\\
  &= (m_{r})_{*}\circ EZ \big(\text{Conf}(\embed)_{*}(\sigma)\otimes \alpha\big(\sum_{i=0}^{j}\binom{j}{i} x^{j-i}\otimes (m_{l})_{*}\circ EZ(\mon[f']_{*}(\lambda^{i}(e,z))\otimes y)\big) \big).
\end{align*}
Since $(m_{r})_{*}\circ EZ (\text{Conf}(\embed)_{*}(\sigma)\otimes \alpha(x^{j}\otimes -))=(m_{r})_{*}\circ EZ ((\emptyset,0)\otimes \alpha(x^{j+1}\otimes -))$, we have that $\gamma_{0,z}^{l}d_{1}(\sigma\otimes x^{j}\otimes y)-\gamma_{0,z}^{l}d_{0}(\sigma\otimes x^{j}\otimes y)$ is equal to
\begin{align*}
    &(m_{r})_{*}\circ EZ \big((\emptyset,0)\otimes \alpha\big(\sum_{i=0}^{j+1}\bigg(\binom{j+1}{i}-\binom{j}{i}\bigg) x^{j+1-i}\otimes (m_{l})_{*}\circ EZ(\mon[f']_{*}(\lambda^{i}(e,z))\otimes y)\big) \big)\\
    =&(m_{r})_{*}\circ EZ \big((\emptyset,0)\otimes \alpha\big(\sum_{i=1}^{j+1}\bigg(\binom{j+1}{i}-\binom{j}{i}\bigg) x^{j+1-i}\otimes (m_{l})_{*}\circ EZ(\mon[f']_{*}(\lambda^{i}(e,z))\otimes y)\big) \big)
\end{align*}
%$$(m_{r})_{*}\circ EZ \big((\emptyset,0)\otimes \alpha\big(\sum_{i=0}^{j+1}\bigg(\binom{j+1}{i}-\binom{j}{i}\bigg) x^{j+1-i}\otimes (m_{l})_{*}\circ EZ(\mon[f']_{*}(\lambda^{i}(e,z)\otimes y)\big) \big)$$ 
(by abuse of notation, when $i=j+1$ we set $\binom{j}{j+1}$ to be zero).
%$$\gamma_{0,z}^{r}\circ d_{0}(\sigma\otimes x^{j}\otimes y)= \gamma_{0,z}^{r}((\emptyset,0)\otimes x^{j+1}\otimes y)$$
Since $\browd{e}{z}=0$, we can fix a chain $\omega\in C_{*}(\modc[S^{n-1}]{\bar{D}^{n}};\fieldc)$ such that $\partial \omega = \browd{e}{z}$.
The map $H_{l}\colon U_{1}\to C_{*}(\modc[S^{n-1}]{\bar{D}^{n}};\fieldc)$, which sends $ \sigma\otimes x^{j}\otimes y$ to
%\begin{cases}
%     (m_{r})_{*}\circ EZ \big((\emptyset,0)\otimes \alpha\big(\sum_{i=1}^{j+1}\bigg(\binom{j+1}{i}-\binom{j}{i}\bigg) x^{j+1-i}\otimes (m_{l})_{*}\circ EZ(\mon[f']_{*}(\lambda^{i-1}(e,\omega)\otimes y)\big) \big & \text{deg}(\sigma)=n-1,\\
%     0 & \text{deg}(\sigma)=0,
%\end{cases}
$$(m_{r})_{*}\circ EZ \big((\emptyset,0)\otimes \alpha\big(\sum_{i=1}^{j+1}\bigg(\binom{j+1}{i}-\binom{j}{i}\bigg) x^{j+1-i}\otimes (m_{l})_{*}\circ EZ(\mon[f']_{*}(\lambda^{i-1}(e,\omega))\otimes y)\big) \big)$$ if the degree of $\sigma$ is $n-1$ and which is zero otherwise,
%\begin{align*}
%    H_{01}\colon U_{1}&\to C_{*}(\modc[S^{n-1}]{\bar{D}^{n}};\fieldc)\\
%    \sigma\otimes x^{j}\otimes y &\mapsto (m_{r})_{*}\circ EZ \big((\emptyset,0)\otimes \alpha\big(\sum_{i=1}^{j+1}\bigg(\binom{j+1}{i}-\binom{j}{i}\bigg) x^{j+1-i}\otimes (m_{l})_{*}\circ EZ(\mon[f']_{*}(\lambda^{i-1}(e,\omega)\otimes y)\big) \big)
%\end{align*}
gives a homotopy from $\gamma_{0,z}^{l}\circ d_{0}$ to $\gamma_{0,z}^{l}\circ d_{1}$.
%$H_{01}\colon $
%Since $(m_{r})_{*}\circ EZ\circ (\text{id}\otimes (\alpha \circ s_{z}^{r}))\circ d_{1} (\sigma\otimes x^{j}\otimes y)$
\end{proof}
\begin{definition}\label{stabilization maps on non-module model for Conf(D^n)}
%Let $\fieldc$ be a ring and 
Let $e\in H_{0}(\text{Conf}_{1}(\mathbb{R}^{n});\fieldc)$ be the class of a point. Suppose that $z\in H_{*}(\text{Conf}(\mathbb{R}^{n});\fieldc)$ is a homology class such that the Browder bracket $\browd{z}{e}=0$. 
%Fix a chain homotopy $H^{l}$ from $g^{l}_{0,z}\circ (d_{0})_{*}$ to $g^{l}_{0,z}\circ (d_{1})_{*}$. 
Let $$\gamma^{l}_{z}\colon T(U)_{*}\to C_{*}(\modc[S^{n-1}]{\bar{D}^{n}};\fieldc)$$ denote the 
map on $T(U)_{*}$ induced by taking the homotopy from $\gamma^{l}_{0,z}\circ d_{0}$ to $\gamma^{l}_{0,z}\circ d_{1}$ to be the homotopy $H_{l}$ constructed at the end of \cref{face maps composed with gamma maps are homotopic}.
%%induced map on $C_{*}(Y;\fieldc)$. 
Let $\gamma^{r}_{z}\colon T(U)_{*}\to C_{*}(\modc[S^{n-1}]{\bar{D}^{n}};\fieldc)$ denote the map on $T(U)_{*}$ induced by taking the homotopy from $\gamma^{r}_{0,z}\circ d_{0}$ to $\gamma^{r}_{0,z}\circ d_{1}$ to be the zero map.
\end{definition}
%By \cref{stabilization maps from the right side commute}, the maps $g^{r}_{0,z}\circ d_{0}\colon C_{*}(Y_{1};\fieldc)\to C_{*}(\modc[S^{n-1}]{\bar{D}^{n}};\fieldc)$ and $g^{r}_{0,z}\circ d_{1}$ agree.
%By \cref{p-power stabilization}, the maps $g^{l}_{0,z}\circ d_{0}\colon C_{*}(V_{1};\fieldc)\to C_{*}(\modc[S^{n-1}]{\bar{D}^{n}};\fieldc)$ and $g^{l}_{0,z}\circ d_{1}$ are chain homotopic when $\browd{e}{z}=0$. By fixing chain homotopies, we can define stabilization maps on $T(V)_{*}$.\footnote{put this in the intro}
%The following proposition relates the map $g_{z}\colon C_{*}(Y;\fieldc)\to C_{*}(\modc[S^{n-1}]{\bar{D}^{n}};\fieldc)$ to the stabilization map $t_{z}\colon C_{*}(Y;\fieldc)\to C_{*}(\modc[S^{n-1}]{\bar{D}^{n}};\fieldc)$
%By \cref{p-power stabilization}, 
%The map $g_{z}$ is a map of $C_{*}(\mon[S^{n-1}];\fieldc)$-modules. The map $g_{z}$ induces a map of chain complexes $$\Vert \cbarc{\modc[S^{n-1}]{\bar{D}^{n}}}{\mon[S^{n-1}]}{\modc[S^{n-1}]{M\setminus D^{n}}}{\bullet}{\fieldc}\Vert \to \Vert \cxbarc{\bar{D}^{n}}{S^{n-1}}{M\setminus D^{n}}{\bullet}{\fieldc}\Vert$$
%which enables us to define a new chain-level stabilization map for $C_{*}(\text{Conf}(M);\fieldc)$.
%From the map $g_{z}$, we can now define a new chain-level stabilization map for $C_{*}(\text{Conf}(M);\fieldc)$.
%We now fix a chain homotopy $H^{l}$ from $g_{0, z}^{l}\circ d_{0}$ to $g_{0, z}^{l}\circ d_{1}$ once and for all time.\footnote{fix this}
\begin{remark}
The homotopies from $s_{z}^{l}$ to $s_{z}^{r}$ and from $\gamma_{0, z}^{l}\circ d_{0}$ to $\gamma_{0, z}^{l}\circ d_{1}$
%and $g_{0, z}^{l}\circ d_{0}$ to $g_{0, z}^{l}\circ d_{1}$ 
both come from fixing a chain $\omega\in C_{*}(\text{Conf}(D^{n});\fieldc)$ which satisfies $\partial\omega =\browd{e}{z}$. We will assume that the chain $\omega$ is the same for both homotopies because we will need to assume this in the proof of \cref{induced gamma maps are homotopic}.
\end{remark}
%\subsection{Comparing left and right stabilization maps}
%\begin{lemma}\label{comparing g sub z and gamma sub z}
%The maps $g_{z}^{l}\circ (\text{id}\otimes \alpha)$ and $g_{z}^{r}\circ (\text{id}\otimes \alpha)$ are chain homotopic to $(\text{id}\otimes \alpha)\circ \gamma_{z}^{l}$ and $(\text{id}\otimes \alpha)\circ \gamma_{z}^{r}$, respectively.\footnote{need to fix this}
%The maps $\gamma_{z}^{l}$ and $\gamma_{z}^{r}$ factor through $g_{z}^{l}\circ (\text{id}\otimes \alpha)$ and $g_{z}^{l}\circ (\text{id}\otimes \alpha)$, respectively.
%\end{lemma}
%\begin{proof}
%This follows immediately from $t_{z}^{l}\circ\alpha$ and $t_{z}^{r}\circ \alpha$ being chain homotopic to $\alpha \circ s_{z}^{l}$ and $\alpha \circ s_{z}^{r}$ respectively by \cref{comparing stabilization maps on different conf(cyl) models}.
%, and from $\text{id}\otimes \alpha_{\bullet}\colon U_{\bullet}\to V_{\bullet}$ being a map of semi-simplicial chain complexes by 
%\end{proof}
%\begin{proposition}
%The maps $\gamma_{z}^{l}$ and $\gamma_{z}^{r}$ are chain homotopic.
%\end{proposition}
\begin{lemma}\label{induced gamma maps are homotopic}
The maps $\gamma_{z}^{l}$ and $\gamma_{z}^{r}$ are chain homotopic.
%\footnote{need to define $\gamma_{z}$ maps}
%\footnote{probably want to replace $g$ with $\gamma$}
\end{lemma}
\begin{proof}
%Since $g_{z}^{l}\circ (\text{id}\otimes \alpha)$ and $g_{z}^{r}\circ (\text{id}\otimes \alpha)$ are chain homotopic to $(\text{id}\otimes \alpha)\circ \gamma_{z}^{l}$ and $(\text{id}\otimes \alpha)\circ \gamma_{z}^{r}$ by \cref{comparing g sub z and gamma sub z}, respectively, and $\alpha$ is a chain homotopy equivalence by \cref{different models of conf(cyl) are equivalent}, it suffices to show that $\gamma_{z}^{l}$ and $\gamma_{z}^{r}$ are chain homootopic.\footnote{probably want to drop this sentence}
%Let $I_{*}$ denote the interval object in the category of chain complexes of $\fieldc$-modules. By definition, $I_{*}$ is the chain complex with $I_{0}=\fieldc\oplus\fieldc\cong \fieldc[\{(0), (1)\}]$, $I_{1}=\fieldc\cong \fieldc[\{(0\to 1)\}]$, $I_{k}=0$ for all $k>1$, with $\partial (r\{(0\to 1)\})=r\{(1)\}-r\{(0)\}$ for $r\in \fieldc$.
%Recall that the data of a chain homotopy
%between $f, g\colon C_{*}\to D_{*}$ is the same as a map $H\colon C_{*}\otimes I_{*}\to D_{*}$ which satisfies that the following diagram  commutes (in the diagram, $i_{j}(\sigma)\colonequals (\sigma, j)$ for $\sigma\in C_{*}$ and $j=0,1$).
%\begin{center}\begin{tikzcd}
%\centering
%C_{*}\arrow[d,"i_{1}"]\arrow[rd, "f"]&\\
%C_{*}\otimes I_{*}\arrow[r,"H"]& D_{*}\\
%C_{*}\arrow[u,"i_{0}"]\arrow[ru, "g"]&
%\end{tikzcd}\end{center}
%We can think $H$ as a map

The data
%\footnote{need to check if I put bars over disk in mod-conf} 
of a map $f\colon T(U)_{*} \to D_{*}$, where $D_{*}$ is a chain complex, is given by a pair $(f_{0},H_{f_{0}})$, where $f_{0}$ is a map $f_{0}\colon U_{0}\to D_{*}$ and $H_{f_{0}}$ is a homotopy $H_{f_{0}}\colon U_{1}\to D_{*+1}$ from $f_{0}\circ d_{0}$ to $f_{0}\circ d_{1}$. Therefore, to give a chain homotopy between two maps $f=(f_{0},H_{f_{0}}), g=(g_{0},H_{g_{0}})\colon T(U)_{*} \to D_{*}$,
%from $f$ to $g$ 
we need the following data:
\begin{enumerate}
    \item a homotopy $H_{i}\colon U_{1}\to D_{*+1}$ from $f_{0}\circ d_{i}$ to $g_{0}\circ d_{i}$ for each $i=0,1$, and
    \item a map $K\colon U_{1}\to D_{*+2}$ that satisfies $\partial K = H_{1} - H_{f}- H_{0} + H_{g}$
\end{enumerate}
%We can summarize this using the following diagram.\footnote{create diagram}
By \cref{comparing stabilization maps on the cylinder},
the map 
$$H(x^{j}\otimes -)= \sum_{i=1}^{j}\binom{j}{i}  x^{j-i}\otimes (m_{l})_{*}\circ EZ (\mon[f']_{*}(\lambda^{i-1}(e,\omega))\otimes -)$$ gives a homotopy from $s_{z}^{l}$ to $s_{z}^{r}$. Therefore, 
the map
$$H^{l,r}\colonequals (m_{r})_{*}\circ EZ \circ (\text{id}\otimes (\alpha \circ H))$$ gives a homotopy
%we have an induced homotopy $H^{l,r}$ 
from $\gamma_{0, z}^{l}$ to $\gamma_{0, z}^{r}$. Composing $H^{l,r}$ with $d_{i}$ gives a homotopy $H_{i}$
%from $\gamma_{z}^{l}$ to $\gamma_{z}^{r}$, the map $H_{i}\colonequals H^{l,r}\circ d_{i}$ for $i=0,1$ is a homotopy 
from $\gamma_{0, z}^{l}\circ d_{i}$ to $\gamma_{0, z}^{r}\circ d_{i}$. We set $H_{l}$ to be the homotopy $\gamma_{0,z}^{l}\circ d_{0}$ to $\gamma_{0,z}^{l}\circ d_{1}$ given at the end of \cref{face maps composed with gamma maps are homotopic}. Since $\gamma_{0,z}^{l}\circ d_{0}=\gamma_{0,z}^{l}\circ d_{1}$ by \cref{face maps composed with gamma maps are homotopic}, we can set the homotopy $H_{r}$ from $\gamma_{0,z}^{r}\circ d_{0}$ to $\gamma_{0,z}^{r}\circ d_{1}$ to be the zero map.
We now calculate $H_{1} - H_{0} + H_{l}$. Given $\sigma\otimes x^{j}\otimes y\in U_{1}$, if the degree of $\sigma$ is zero then $(H_{1} - H_{0} + H_{l})(\sigma\otimes x^{j}\otimes y)=0$ since $H_{l}(\sigma\otimes x^{j}\otimes y)=0$ by definition and $d_{0}(\sigma\otimes x^{j}\otimes y)=d_{1}(\sigma\otimes x^{j}\otimes y)$. Suppose now that the degree of $\sigma$ is $n-1$. 

By definition, the maps $H_{l}$, $H_{0}$, and $H_{1}$ send $\sigma\otimes x^{j}\otimes y$ to
%The map $H_{l}\colon U_{1}\to C_{*}(\modc[S^{n-1}]{\bar{D}^{n}};\fieldc)$ sending $ \sigma\otimes x^{j}\otimes y$ to
\begin{align*}
    %H_{l}(\sigma\otimes x^{j}\otimes y)=\\
    &(m_{r})_{*}\circ EZ \big((\emptyset,0)\otimes \alpha(\sum_{i=1}^{j+1}\bigg(\binom{j+1}{i}-\binom{j}{i}\bigg) x^{j+1-i}\otimes (m_{l})_{*}\circ EZ(\mon[f']_{*}(\lambda^{i-1}(e,\omega))\otimes y)) \big),\\
    %H_{0}(\sigma\otimes x^{j}\otimes y)=
    &(m_{r})_{*}\circ EZ\circ ((\emptyset,0)\otimes\alpha(\sum_{i=1}^{j+1}\binom{j+1}{i}  x^{j+1-i}\otimes (m_{l})_{*}\circ EZ (\mon[f']_{*}(\lambda^{i-1}(e,\omega))\otimes y))), \text{ and}\\
    %H_{1}(\sigma\otimes x^{j}\otimes y)=
    &(m_{r})_{*}\circ EZ\circ (\text{Conf}(\embed)_{*}(\sigma)\otimes\alpha(\sum_{i=1}^{j}\binom{j}{i}  x^{j-i}\otimes (m_{l})_{*}\circ EZ (\mon[f']_{*}(\lambda^{i-1}(e,\omega))\otimes y))),
\end{align*}
%(here, we are assuming that $\sigma $ is degree $n-1$).\footnote{probably should explain what happens when $\sigma $ is degree $0$} 
respectively. Since $(m_{r})_{*}\circ EZ (\text{Conf}(\embed)_{*}(\sigma)\otimes \alpha(x^{j}\otimes -))=(m_{r})_{*}\circ EZ ((\emptyset,0)\otimes \alpha(x^{j+1}\otimes -))$, we have that $(H_{1} - H_{0} + H_{l})(\sigma\otimes x^{j}\otimes y)=0$. It follows that we can set the map $K$ to be the zero map.
%$$(m_{r})_{*}\circ EZ \big((\emptyset,0)\otimes \alpha\big(\sum_{i=1}^{j+1}\bigg(\binom{j+1}{i}-\binom{j}{i}\bigg) x^{j+1-i}\otimes (m_{l})_{*}\circ EZ(\mon[f']_{*}(\lambda^{i-1}(e,\omega)\otimes y)\big) \big)$$
\end{proof}
\subsection{Stabilization Maps for the Configuration Space of a Manifold}\label{new stab maps for conf(M)}
In this subsection, we give a model for $\nch{\text{Conf}(M)}$
%. We 
%do this by 
%construct this model for $\nch{\text{Conf}(M)}$ by
by using copies of a model for $\nch{\modc[S^{n-1}]{\bar{D}^{n}}}$, the chain complex $T(V)_{*}$ from \cref{main model for chains on Conf(D^n)}. We will need to work with copies of $T(V)_{*}$ instead of a single copy because this will allow us to define iterated mapping cones on our model for $\nch{\text{Conf}(M)}$. Then we use the stabilization map $g^{l}_{z}$ from \cref{definition of new stabilization map for Conf(Disk)} on $T(V)_{*}$ to get a stabilization map $\newstab{z}$ when $\browd{z}{e}=0$ on our model for $\nch{\text{Conf}(M)}$. We conclude by relating this new stabilization map to the stabilization map $t_{z}$ when $M$ is an open connected manifold.
\begin{definition}
%Let $AW\colon \nch{X\times Y}\to \nch{X}\otimes \nch{Y}$ denote the Alexander--Whitney map.
Suppose that $U, V$, and $W$ are spaces, with $V$ a monoid, $U$ a right $V$-module, and $W$ a left $V$-module, and that $Z_{\bullet}$ is the semi-simplicial space $B_{\bullet}(U,V,W)$. Let $N_{*}(Z_{\bullet};\fieldc)$ denote
%\footnote{this is bad notation because I use it to denote something else two or three definitions ago} 
the semi-simplicial chain complex, whose  $p$-simplices in degree $q$ are
$$N_{q}(Z_{p};\fieldc)\colonequals\bigoplus_{q=q_{0}+\ldots + q_{p+1} }C_{q_{0}}(U;\fieldc)\otimes_{\fieldc} C_{q_{1}}(V;\fieldc)\otimes_{\fieldc}\cdots\otimes_{\fieldc} C_{q_{p}}(V;\fieldc)\otimes_{\fieldc} C_{q_{p+1}}(W;\fieldc).$$ 
Given $\sigma\in C_{q_{0}}(U;\fieldc)$, $\tau_{j}\in C_{q_{j}}(V;\fieldc)$, and $\omega\in C_{q_{p+1}}(W;\fieldc)$, the face maps $d_{i} \colon N_{q}(Z_{p};\fieldc)\to  N_{q}(Z_{p-1};\fieldc)$ for $0\leq i \leq p$ send $\sigma\otimes\tau_{1}\otimes\cdots\otimes\tau_{p}\otimes\omega$ to 
%are (for $e\in E, (y_{1},\ldots , y_{p})\in Y^{p}, f\in F$)\footnote{I don't know if I should drop the stuff in parentheses}
\begin{align*}
  d_{0}(\sigma\otimes\tau_{1}\otimes\cdots\otimes\tau_{p}\otimes\omega)&=((m_{r})_{*}\circ EZ(\sigma\otimes\tau_{1}))\otimes\tau_{2}\otimes\cdots\otimes\tau_{p}\otimes\omega\\
  d_{i}(\sigma\otimes\tau_{1}\otimes\cdots\otimes\tau_{p}\otimes\omega)&= \sigma\otimes\tau_{1}\otimes\cdots\otimes\tau_{i-1}\otimes ((-\cdot-)_{*}\circ EZ (\tau_{i}\otimes\tau_{i+1}))\otimes\tau_{i+2}\otimes\cdots\otimes\tau_{p}\otimes\omega, 0<i<p, \\ 
  d_{p}(\sigma\otimes\tau_{1}\otimes\cdots\otimes\tau_{p}\otimes\omega)&=\sigma\otimes\tau_{1}\otimes\cdots\otimes\tau_{p-1}\otimes (m_{l})_{*}\circ EZ(\tau_{p}\otimes\omega),  
\end{align*}
where $m_{r}\colon U\times V\to U$ is the right module map of $U$, $(-\cdot-)$ is the monoid map of $V$, and $m_{l}\colon V\times W\to W$ is the left module map of $W$.
%The face maps $d_{i}\colon N_{*}(Z_{p};\fieldc)\to N_{*}(Z_{p-1};\fieldc)$ are
%the composition $AW\circ d_{i}\circ EZ$.
%\footnote{need to clean up notation for citerxbarc stuff}
%following composition
%$$\cbarc{U}{Y}{V}{p}{\fieldc}\xrightarrow{E} C_{*}(\barc{U}{Y}{V}{p};\fieldc)\xrightarrow{(d_{i})_{*}} C_{*}(\barc{U}{Y}{V}{p-1};\fieldc)\xrightarrow{A} \cbarc{U}{Y}{V}{p-1}{\fieldc}.$$
%The degeneracy maps $s_{i}\colon N_{*}(Z_{p};\fieldc)\to N_{*}(Z_{p+1};\fieldc)$ are
%\begin{align*}
%    s_{0}(\sigma \otimes \tau_{1}\otimes\ldots\otimes\tau_{p}\otimes\omega)&= \sigma \otimes 1\otimes \tau_{1}\otimes\ldots\otimes\tau_{p}\otimes\omega\\
%    s_{i}(\sigma \otimes \tau_{1}\otimes\ldots\otimes\tau_{p}\otimes\omega)&= \sigma \otimes \tau_{1}\otimes\ldots\tau_{i}\otimes 1\otimes\tau_{p}\otimes\omega \text{ for } 0<i\leq p.
%\end{align*}
If $Z_{\bullet}'\subset Z_{\bullet}$ is a sub-semi-simplicial space, we will use $N_{*}(z'_{\bullet};\fieldc)$ to denote the analogous sub-semi-simplicial chain complex for $Z_{\bullet}'$.
%\footnote{fix this sentence}
%When $B_{\bullet}(U,V,W)$ is the simplicial space $\xbarc{M_{1}}{\smon}{M_{2}}{\bullet}$, we use $\cxbarc{M_{1}}{\smon}{M_{2}}{\bullet}{\fieldc}$ to denote the simplicial space $\cbarc{\modc[S^{n-1}]{\bar{D}^{n}}}{\mon[S^{n-1}]}{\modc[S^{n-1}]{M\setminus D^{n}}}{\bullet}{\fieldc}$.
\end{definition}
\begin{lemma}\label{chain model for conf(M)}
Suppose
%\footnote{This lemma isn't stated correctly. I want to state that normalized chains on the bar construction are equivalent to bar construction of normalized chains} 
that $U, V$, and $W$ are spaces, with $V$ a monoid, $U$ a right $V$-module, and $W$ a left $V$-module. 
%Suppose that the inclusion $1\to V$ is a closed Hurewicz cofibration. 
There is a quasi-isomorphism
$$\tot{N_{*}(B_{\bullet}(U,V,W);\fieldc)}\to \nch{\Vert B_{\bullet}(U,V,W)\Vert}.$$
%$$\Vert \cbarc{U}{V}{W}{\bullet}{\fieldc}\Vert\to C_{*}(\Vert B_{\bullet}(U,V,W)\Vert;\fieldc).$$
\end{lemma}
\begin{proof}
By the K\"unneth theorem, the Eilenberg--Zilber map induces a quasi-isomorphism 
%$$\Vert \nch{B_{\bullet}(U,V,W)}\Vert\to \Vert C_{*}(B_{\bullet}(U,V,W);\fieldc) \Vert $$
%\tot{C_{*}(Z_{\bullet};\fieldc)}
$$\tot{N_{*}(B_{\bullet}(U,V,W);\fieldc)}\to \tot{\nch{B_{\bullet}(U,V,W)}}. $$ 
%Since $1\to V$ is a closed Hurewicz cofibration, $B_{\bullet}(U,V,W)$ is a good simplicial space. Therefore, 
We have a quasi-isomorphism
%$$\Vert C_{*}(B_{\bullet}(U,V,W);\fieldc) \Vert \to C_{*}(\Vert B_{\bullet}(U,V,W)\Vert ;\fieldc)$$ 
$$\tot{\nch{B_{\bullet}(U,V,W)}} \to \nch{\Vert B_{\bullet}(U,V,W)\Vert}$$
by \cref{total complex is quasi-isomorphic to geometric realization}.
%by \cref{chain complex of a good space}. 
Since the composition of quasi-isomorphisms is a quasi-isomorphism, there is a quasi-isomorphism $\tot{N_{*}(B_{\bullet}(U,V,W);\fieldc)}\to\nch{\Vert B_{\bullet}(U,V,W)\Vert}$.
%\footnote{this proof doesn't make sense.}
\end{proof}
\begin{definition}
Suppose that $M_{1}$ and $M_{2}$ are two manifolds with boundary and that $(\smon_{1},\ldots, \smon_{\srange})$ is an $\srange$-tuple of disjoint subspaces contained in $\partial M_{1}$ and $\partial M_{2}$. Let $\smon$ denote $\bigsqcup_{j=1}^{\srange}\smon_{j}$. Let $\cpiterxbarc{M_{1}}{\smon}{M_{2}}{\bullet}{\srange}{\fieldc}{\partic}$ and $\citerxbarc{M_{1}}{\smon}{M_{2}}{\bullet}{\srange}{\fieldc}$ denote the semi-simplicial chain complexes $N_{*}(\piterxbarc{M_{1}}{\smon}{M_{2}}{\bullet}{\srange}{\partic};\fieldc)$ and \\$N_{*}(\iterxbarc{M_{1}}{\smon}{M_{2}}{\bullet}{\srange};\fieldc)$, respectively.

Suppose that we have a map $A_{*}\to \nch{\rimodc[\smon]{M_{1}}{\srange}}$ of right $\prod_{j=1}^{\srange}\nch{\mon[\smon_{j}]}$-modules.
%\footnote{need to figure out if I should use $\mon[N]$ or $\mon[\smon]$ notation} 
By abuse of notation, let 
%$\cpiterxbarc{A_{*}}{\smon}{M_{2}}{\bullet}{\srange}{\fieldc}{\partic}$ and 
$\citerxbarc{A_{*}}{\smon}{M_{2}}{\bullet}{\srange}{\fieldc}$ denote the semi-simplicial chain complex
$$\citerxbarc{A_{*}}{\smon}{M_{2}}{\bullet}{\srange}{\fieldc}\colonequals N_{*}\big(B_{\bullet}(A_{*},\prod_{j=1}^{\srange}\nch{\mon[\smon_{j}]},\nch{\rimodc[\smon]{M_{2}}{\srange}})\big).$$
%$$\barc{A_{*}}{\nch{\smon}}{\nch{\rimodc{M_{2}}{\srange}}}.$$ 
Let $\cpiterxbarc{A_{*}}{\smon}{M_{2}}{\bullet}{\srange}{\fieldc}{\partic}\subset \citerxbarc{A_{*}}{\smon}{M_{2}}{\bullet}{\srange}{\fieldc}$ denote the sub-semi-simplicial chain complex that maps to $\cpiterxbarc{M_{1}}{\smon}{M_{2}}{\bullet}{\srange}{\fieldc}{\partic}$ under the map $A_{*}\to \nch{\rimodc[\smon]{M_{1}}{\srange}}$.
%\footnote{I want to modify this definition/notation, so that it's also clear what I mean when I'm working with a different model for $\nch{\modc[S^{n-1}]{\bar{D}^{n}}}$. }
%that is weakly equivalent to  , we use the notation 
%$$\citerxbarc{M_{1}}{\smon}{M_{2}}{\srange}{\bullet}{\fieldc}\colonequals \nch{B_{\bullet}(\rimodc{M_{1}}{\srange},\prod_{j=1}^{\srange}\mon[\smon_{j}],\rimodc{M_{2}}{\srange})}.$$
%\cbarc{\rimodc{M_{1}}{\srange}}{\prod_{j=1}^{\srange}\mon[\smon_{j}]}{\rimodc{M_{2}}{\srange}}{\bullet}{\fieldc}
%We define $\cpiterxbarc{M_{1}}{\smon}{M_{2}}{\bullet}{\srange}{\fieldc}{\partic}$ analogous to how we defined $\piterxbarc{M_{1}}{\smon}{M_{2}}{\bullet}{\srange}{\partic}$ in \cref{iterated bar construction model for conf}.
%$$\cxbarc{M_{1}}{\smon}{M_{2}}{\bullet}{\fieldc}\colonequals \cbarc{\modc[S^{n-1}]{\bar{D}^{n}}}{\mon[S^{n-1}]}{\modc[S^{n-1}]{M\setminus D^{n}}}{\bullet}{\fieldc}.$$
\end{definition}
\begin{definition}\label{notation for iterated configuration space model}
%Let $\srange$ be a positive integer. Let $M$ be a manifold of dimension $n$ and fix $\srange$ disjoint closed disks $\bar{D}^{n}_{1},\ldots, \bar{D}^{n}_{\srange}$ in $M$. 
%For $j=1,\ldots, \srange$, let $Y_{j}$ denote the space $Y$. 
Fix $\srange$ disjoint closed disks $\bar{D}^{n}_{1},\ldots,\bar{D}^{n}_{\srange}$ in $M$.
%\footnote{check if this notation was previously used.} 
Let $\bar{K}_{\srange}$ denote $\bigsqcup_{j=1}^{\srange}\bar{D}^{n}_{j}\subset M$ and let $K_{\srange}$ denote $\text{int}(\bar{K}_{\srange})$. Fix $\srange$ copies $T(V^{1})_{*},\ldots, T(V^{\srange})_{*}$ of $T(V)_{*}$.
\end{definition}
%\begin{definition}
%By abuse of notation, let
%\end{definition}
\begin{definition}
Let $[n]$ denote the set $\{1,\ldots,n\}$.
Given $I\subseteq [\srange]$, let $A^{I}_{*}$ denote the chain complex
$$A^{I}_{*}\colonequals \bigoplus_{i\in I}T(V^{i})_{*}\bigoplus_{j\in I^{c}}\nch{\modc[S^{n-1}_{j}]{\bar{D}^{n}_{j}}}.$$
%and let $F^{J}_{*}$ denote $\bigsqcup^{\infty}_{\partic=0}F^{J}_{\partic,*}$. 
%If $J$ is the empty set, we write $A_{*}$ to denote $A^{\{\emptyset\}}_{*}$.
\end{definition}
\begin{definition}\label{stabilization maps for closed iterated barc}
%Let $\bar{K}_{\srange}$ and $K_{\srange}$ be as in \cref{notation for iterated configuration space model}. Consider the ordered $\srange$-tuple $(S^{n-1}_{1},\ldots,S^{n-1}_{\srange})$ of subspaces contained in $\partial (M\setminus K_{\srange})$. 
%Let $\fieldc$ be a ring and 
Let $e\in H_{0}(\mathbb{R}^{n};\fieldc)$ be the class of a point. 
Suppose that $(z_{1},\ldots, z_{d})$ is a $d$-tuple of classes in $H_{*}(\text{Conf}(\mathbb{R}^{n});\fieldc)$ such that $\browd{z_{b}}{e}=0$ for all $b=1,\ldots, d$.
Given  $I \subseteq [\srange]$ with $[d]\subseteq I$, for each $b=1,\ldots, d$, define 
$$\newstab{z_{b}} \colon \tot{\cpiterxbarc{A^{I}_{*}}{\partial \bar{K}_{\srange}}{M\setminus K_{\srange}}{\bullet}{\srange}{\fieldc}{\partic}}\to \tot{\cpiterxbarc{A^{I\setminus\{b\}}_{*}}{\partial \bar{K}_{\srange}}{M\setminus K_{\srange}}{\bullet}{\srange}{\fieldc}{\partic+\text{par}(z_{b})}}$$
to be the stabilization map induced from the map
%\footnote{can I be less explicit about what this map is}
$g_{z_{b}}^{l}\colon T(V^{b})_{*}\to \nch{\modc[S^{n-1}_{b}]{\bar{D}^{n}_{b}}}$ from \cref{definition of new stabilization map for Conf(Disk)}, i.e. given 
\begin{align*}
    %\text{i.e. given } 
    (\sigma_{1},\ldots,\sigma_{\srange})&\in \bigoplus_{i\in I}T(V^{i})_{*}\bigoplus_{j\in I^{c}}^{\srange} \nch{\modc[S^{n-1}_{j}]{\bar{D}^{n}_{j}}}\\
    \text{and } \tau &\in \nch{\prod_{q=1}^{\srange}\mon[S^{n-1}_{q}]}^{\otimes p}\otimes \nch{\rimodc[\partial \bar{K}_{\srange}]{M\setminus K_{\srange}}{\srange}},
\end{align*}
%\text{i.e. given }
%$ (\sigma_{1},\ldots,\sigma_{\srange})\in \bigoplus_{j\in J}\nch{\modc[S^{n-1}_{j}]{\bar{D}^{n}_{j}}}\bigoplus_{i\in I^{c}}T(V^{i})_{*}$ and $\tau\in \nch{\prod_{l=1}^{\srange}\mon[S^{n-1}_{l}]}^{\otimes p}\otimes \nch{\rimodc[\partial \bar{K}_{\srange}]{M\setminus K_{\srange}}{\srange}}$,
%$$\sigma_{j}\in C_{*}(Y_{j};\fieldc) \text{ for } j=1,\ldots,\srange,$$   
%$j=1,\ldots,\srange$,
%C_{*}(\rimodc[\partial \bar{K}_{\srange}]{\bigsqcup_{k=1}^{\srange}Y_{k}}{\srange};\fieldc)
%\begin{align*}
%    \text{i.e.
%    given }(\sigma_{1},\ldots,\sigma_{\srange})&\in \bigoplus_{r\in J}\nch{\modc[S^{n-1}_{r}]{\bar{D}^{n}_{r}}}\bigoplus_{s\in I^{c}}T(V^{s})_{*}\text{ and}\\
%    \tau&\in \nch{\prod_{i=1}^{\srange}\mon[S^{n-1}_{i}]}^{\otimes p}\otimes \nch{\rimodc[\partial \bar{K}_{\srange}]{M\setminus K_{\srange}}{\srange}}
%\end{align*}
$$\newstab{z_{b}}((\sigma_{1}, \ldots, \sigma_{\srange}) \otimes\tau)\colonequals (\sigma_{1}, \ldots, \sigma_{b-1}, g_{z_{b}}^{l}(\sigma_{b}),  \sigma_{b+1}, \ldots,  \sigma_{\srange})\otimes\tau.$$
\end{definition}
\begin{remark}
Since the map $g_{z}^{l}\colon T(V)_{*}\to \nch{\modc[S^{n-1}]{\bar{D}^{n}}}$ depends on fixing a chain homotopy from $g_{0,z}^{l}\circ d_{0}$ to $g_{0,z}^{l}\circ d_{1}$, the map $\newstab{z}$ also depends on this. If we fix a different chain homotopy, we obtain a different map $\newstab{z}'$. We do not know if $\newstab{z}$ and $\newstab{z}'$ are chain homotopic in general, but if $M$ is an open connected manifold, we show that they are chain homotopic in \cref{new stabilization on open manifold}.
\end{remark}
Qualitatively, the stabilization map $\newstab{z_{i}}$ can be described as stabilizing in the $i$-th disk $D^{n}_{i}$. 
%Given classes $z_{1},\ldots, z_{\srange}$ in $H_{*}(\text{Conf}(\mathbb{R}^{n});\fieldc)$ such that $\browd{z_{j}}{e}=0$ for all $j=1,\ldots, \srange$, 
%Given a subset $\{j_{1},\ldots,j_{q}\}\subseteq \subseteq  \{1,\ldots,\srange\}$, and classes $z_{1},\ldots, z_{q}$ in $H_{*}(\text{Conf}(\mathbb{R}^{n});\fieldc)$ such that $\browd{z_{i}}{e}=0$ for all $i=1,\ldots, a$, 
We have that the stabilization maps $\newstab{z_{i}}$ from \cref{stabilization maps for closed iterated barc} commute with each other in the following sense: given $i\neq j$, the stabilization maps $g_{z_{i}}^{l}$ and $g_{z_{j}}^{l}$ also induce maps
\begin{align*}
    \newstab{z_{i}}' \colon \tot{\cpiterxbarc{A^{I\setminus\{j\}}_{*}}{\partial \bar{K}_{\srange}}{M\setminus K_{\srange}}{\bullet}{\srange}{\fieldc}{\partic}}&\to \tot{\cpiterxbarc{A^{I\setminus\{i,j\}}_{*}}{\partial \bar{K}_{\srange}}{M\setminus K_{\srange}}{\bullet}{\srange}{\fieldc}{\partic+\text{par}(z_{i})}}\text{ and}\\
    \newstab{z_{j}}' \colon \tot{\cpiterxbarc{A^{I\setminus\{i\}}_{*}}{\partial \bar{K}_{\srange}}{M\setminus K_{\srange}}{\bullet}{\srange}{\fieldc}{\partic}}&\to \tot{\cpiterxbarc{A^{I\setminus\{i,j\}}_{*}}{\partial \bar{K}_{\srange}}{M\setminus K_{\srange}}{\bullet}{\srange}{\fieldc}{\partic+\text{par}(z_{j})}}
\end{align*}
respectively, and so we have that $\newstab{z_{i}}'\newstab{z_{j}}=\newstab{z_{j}}'\newstab{z_{i}}$.
%\footnote{should I explain why this is? It seems obvious.} 
Therefore, we can define an iterated mapping cone on $\tot{\cpiterxbarc{A^{I}_{*}}{\partial \bar{K}_{\srange}}{M\setminus K_{\srange}}{\bullet}{\srange}{\fieldc}{\partic}}$ as follows.
\begin{definition}\label{iterated mapping cone def for a general manifold}
%Let $\bar{K}_{\srange}\subset M$, and $K_{\srange}$ be as in \cref{notation for iterated configuration space model}. 
%Let $\srange$ be a positive integer. Let $M$ be a manifold of dimension $n$ and fix $\srange$ disjoint closed disks $\bar{D}^{n}_{1},\ldots, \bar{D}^{n}_{\srange}$ in $M$. 
%Let
%$\fieldc$ be a ring and let
%$e\in H_{0}(\text{Conf}_{1}(\mathbb{R}^{n});\fieldc)$ be the class of a point.
%Suppose that $(z_{1},\ldots, z_{\srange})$ is an $\srange$-tuple of classes in $H_{*}(\text{Conf}(\mathbb{R}^{n});\fieldc)$ such that $\browd{z_{j}}{e}=0$ for all $j=1,\ldots, \srange$. 
%Given $J\subset\{1,\ldots, \srange\}$ and $\{a_{1},\ldots, a_{q}\}= I^{c}$, 
%Given  $I \subseteq [\srange]$ with $[d]\subseteq I$,
%$\{j_{1},\ldots, j_{q}\}\subseteq J \subseteq \{1,\ldots,\srange\}$ 
Suppose that $(z_{1},\ldots, z_{d})$ is a $d$-tuple of classes in $H_{*}(\text{Conf}(\mathbb{R}^{n});\fieldc)$ such that $\browd{z_{b}}{e}=0$ for all $b=1,\ldots, d$.
%Suppose that $z_{1},\ldots, z_{d}$ are classes in $H_{*}(\text{Conf}(\mathbb{R}^{n});\fieldc)$ such that $\browd{z_{b}}{e}=0$ for all $b=1,\ldots, d$. 
Given  $I \subseteq [\srange]$ with $[d]\subseteq I$, 
recursively define the \textbf{d-th iterated mapping cone} $\jpariteratedmap{M}{\newstab{z_{i}}}{d}{\partic}{I}$ associated to the $d$-tuple $(\newstab{z_{1}},\ldots, \newstab{z_{d}})$ and the subset $I$ as follows: set $\jpariteratedmap{M}{\newstab{z_{1}}}{1}{\partic}{I}$ to be the mapping cone
$$\text{Cone}\big(\newstab{z_{1}}\colon \tot{\cpiterxbarc{A^{I}_{*}}{\partial \bar{K}_{\srange}}{M\setminus K_{\srange}}{\bullet}{\srange}{\fieldc}{\partic-\text{par}(z_{1})}}\to \tot{\cpiterxbarc{A^{I\setminus\{1\}}_{*}}{\partial \bar{K}_{\srange}}{M\setminus K_{\srange}}{\bullet}{\srange}{\fieldc}{\partic}}\big)$$
and for $d>1$, set $\jpariteratedmap{M}{\newstab{z_{i}}}{d}{\partic}{I}$ to be the induced mapping cone
$$\text{Cone}\big(\newstab{z_{d}}\colon \jpariteratedmap{M}{\newstab{z_{1}},\ldots, \newstab{z_{d-1}}}{d-1}{\partic-\text{par}(z_{d})}{I}\to \jpariteratedmap{M}{\newstab{z_{1}},\ldots, \newstab{z_{d-1}}}{d-1}{\partic}{I\setminus\{d\}}\big).$$
Let $\jiteratedmap{M}{\newstab{z_{i}}}{d}{I}$ denote $\bigsqcup_{\partic=0}^{\infty}\jpariteratedmap{M}{\newstab{z_{i}}}{d}{\partic}{I}$. 
\end{definition}

Now we relate the stabilization map $t_{z}$ to $\newstab{z}$ when $M$ is an open connected manifold. We have a composition of quasi-isomorphisms $$\tot{\citerxbarc{T(V)_{*}}{S^{n-1}}{M\setminus D^{n}}{\bullet}{1}{\fieldc}}\to \tot{\citerxbarc{\bar{D}^{n}}{S^{n-1}}{M\setminus D^{n}}{\bullet}{1}{\fieldc}}\to C_{*}(\text{Conf}(M);\fieldc)$$ and a diagram
\begin{equation}\label{diagram between different stabilization maps}\begin{tikzcd}
\tot{\citerxbarc{T(V)_{*}}{S^{n-1}}{M\setminus D^{n}}{\bullet}{1}{\fieldc}}
\arrow[r]\arrow[d, "\newstab{z}"] &
C_{*}(\text{Conf}(M);\fieldc)
\arrow[d, "t_{z}"]\\
\tot{\citerxbarc{\bar{D}^{n}}{S^{n-1}}{M\setminus D^{n}}{\bullet}{1}{\fieldc}}
\arrow[r]& 
C_{*}(\text{Conf}(M);\fieldc).
\end{tikzcd}\end{equation}
We will show that this diagram commutes up to chain homotopy. 
Roughly speaking, we will do this by using that certain stabilization maps on models of $\nch{\mon[S^{n-1}]}$ and $\nch{\modc[S^{n-1}]{\bar{D}^{n}}}$ are chain homotopic (\cref{comparing stabilization maps on different conf(cyl) models} and \cref{induced gamma maps are homotopic}, respectively). Then we carefully pick a disk in $M$ and a model for $\nch{\text{Conf}(M)}$ so that we can actually relate $t_{z}$ to $\newstab{z}$ by using these stabilization maps on $\nch{\mon[S^{n-1}]}$ and $\nch{\modc[S^{n-1}]{\bar{D}^{n}}}$ and the homotopies between them. 
%we will do this by assuming that $M$ is the interior of a manifold $\bar{M}$ with non-empty boundary and fixing a closed disk $\bar{D}\subset \bar{M}$ whose boundary intersects $\partial \bar{M}$. 
%Then we will pass from certain stabilization maps on $\mon[S^{n-1}]$ being homotopic (using Lemma) to stabilization maps on $\modc[S^{n-1}]{\bar{D}^{n}}$ being homotopic (using Lemma) in order to ``extend'' to get that the stabilization maps on $\text{Conf}(M)$ are also homotopic.
%$\gamma^{l}_{z}$ and $\gamma^{r}_{z}$ from \cref{stabilization maps on non-module model for Conf(D^n)}
%using that certain stabilization maps on $\nch{\mon[S^{n-1}]}$ are chain homotopic (by Lemma) and that the stabilization maps $\gamma^{l}_{z}$ and $\gamma^{r}_{z}$ from \cref{stabilization maps on non-module model for Conf(D^n)} are homotopic 
\begin{proposition}\label{new stabilization on open manifold}
Suppose that $M$ is an open connected manifold. 
%and that $\fieldc$ is a ring. 
%Let $e\in H_{0}(\text{Conf}_{1}(\mathbb{R}^{n});\fieldc)$ be the class of a point. 
Suppose that $z\in H_{*}(\text{Conf}(\mathbb{R}^{n});\fieldc)$ is a class such that $\browd{z}{e}$ is zero.
Diagram \ref{diagram between different stabilization maps}
%\footnote{probably $T(U)_{*}$ should be $T(V)_{*}$. I need to figure out the notation for this chain complex.}
%Then the maps are chain homotopic.
commutes up to chain homotopy.
%There is a natural isomorphism $H_{*}(\iteratedmap{M}{\newstab{z}}{1};\fieldc)\to H_{*}(\iteratedmap{M}{t_{z}}{1};\fieldc)$.
%The diagram
%\begin{center}\begin{tikzcd}
%\Vert \citerxbarc{\bigsqcup_{i=1}^{\srange}Y_{i}}{\partial E}{M\setminus \text{int}(E)}{\bullet}{\srange}{\fieldc}\Vert
%\arrow[r, "\epsilon"]\arrow[d, "\newstab{z}"] &
%C_{*}(\text{Conf}(M);\fieldc)
%\arrow[d, "t_{z}"]\\
%\Vert \citerxbarc{(\bigsqcup_{i=1, i\neq j}^{\srange}Y_{i})\bigsqcup \bar{D}^{n}_{j}}{\partial E}{M\setminus \text{int}(E)}{\bullet}{\srange}{\fieldc}\Vert
%\arrow[r,"\epsilon"]& 
%C_{*}(\text{Conf}(M);\fieldc)
%\end{tikzcd}\end{center}
%\begin{center}\begin{tikzcd}
%\Vert \citerxbarc{Y}{S^{n-1}}{M\setminus D^{n}}{\bullet}{1}{\fieldc}\Vert
%\arrow[r, "\epsilon"]\arrow[d, "\newstab{z}"] &
%C_{*}(\text{Conf}(M);\fieldc)
%\arrow[d, "t_{z}"]\\
%\Vert \citerxbarc{\bar{D}^{n}}{S^{n-1}}{M\setminus D^{n}}{\bullet}{1}{\fieldc}\Vert
%\arrow[r,"\epsilon"]& 
%C_{*}(\text{Conf}(M);\fieldc)
%\end{tikzcd}\end{center}
%commutes up to chain homotopy.
%Let $t_{z}\colon C_{*}(\text{Conf}(M);\fieldc)\to C_{*}(\text{Conf}(M);\fieldc)$ and $\newstab{z}\colon C_{*}(\text{Conf}(M);\fieldc)\to C_{*}(\text{Conf}(M);\fieldc)$ denote the stabilization maps defined in \cref{stabilization map first def} and \cref{stabilization map second def}, respectively. The stabilization maps $t_{z}$ and $\newstab{z}$ are chain homotopic.
%The new stabilization map is the same as the $p$-th power of the old stabilization map. 
\end{proposition}
\begin{comment}
(since the new map is defined using a weird model for $\text{Conf}(D)$, it's not obvious that the two maps are the same).
\end{comment}
\begin{proof}
%The proof outline is as follows.\footnote{need to give an outline} 
%Since $M$ is an open connected manifold, by \cref{open manifold has exhaustion by compact manifold with finite handle decomposition}, $M$ has an exhaustion by the interiors of compact manifolds. As a result, 
First, we fix some notation. Without loss of generality, we can assume that $M$ is the interior of a manifold $\bar{M}$ with non-empty boundary (see, for example, page 41 of \textcite{MillerJeremy2015Sfoc}). Fix an embedding $\embed\colon M\sqcup D^{n}\to M$ as in \cref{stabilization by embedding} (although we took the domain of $\embed$ in \cref{stabilization by embedding} to be  $M\sqcup \mathbb{R}^{n}$, we can take the domain to be $M\sqcup D^{n}$ by first fixing a homeomorphism $D^{n}\to \mathbb{R}^{n}$ and composing). Fix a closed $n$-dimensional disk
%\footnote{change notation for the space E to D}
$\bar{D}\subset \bar{M}$ whose interior $D$ strictly contains $\embed(D^{n})$ and such that $\partial \bar{M}\cap \partial \bar{D}=\bar{D}^{n-1}\equalscolon \bar{D}^{n-1}_{0}$. Let $D^{n-1}_{0}$ denote the interior of $\bar{D}^{n-1}_{0}$ and
let $M'$ denote $\bar{M}\setminus D^{n-1}_{0}$.
%Let $D_{0}$ denote $\partial \bar{M}\setminus \bar{D}_{1}\cong D^{n-1}$ and 
%Let $D^{'}$ denote $\bar{D}\setminus D^{n-1}_{0}$ and let $\bar{D}^{n-1}_{1}$ denote the boundary of $D^{'}$.
Let $\bar{D}^{n-1}_{1}$ denote the boundary of $\bar{D}\setminus D^{n-1}_{0}$. Fix a closed disk $\bar{D}^{n}_{2}\subset D$ not containing $D^{n}$.
We can summarize our setup with  \cref{fig:new-stabilization-map-on-open-manifold}.
\begin{figure}
  \centering
  \includegraphics[scale=.8]{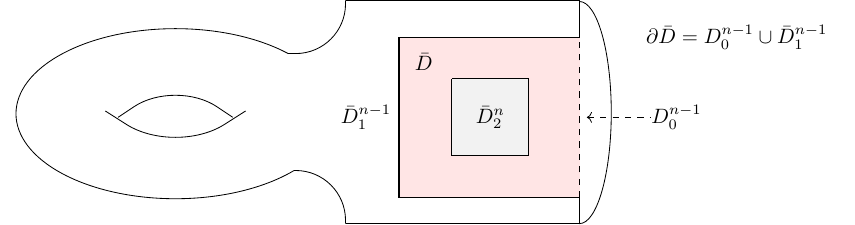}
  \caption{Regions of the manifold $\bar{M}$}
  %     without .tex extension
  % or use \input{mytikz} $t_{1} \colon H_{i}\big (\text{Conf}_{\partic}(M)\big )\to  H_{i}\big (\text{Conf}_{\partic}(M)\big )$}
  \label{fig:new-stabilization-map-on-open-manifold}
\end{figure}

Replace $\bar{M}$ by attaching a collar neighborhood $\partial \bar{M}\times [0,1]$ to $\partial \bar{M}$ along $\partial \bar{M}\times \{0\}$ via the identity map and, by abuse of notation, denote this new space by $\bar{M}$. We attach a collar neighborhood $\partial \bar{M}\times \{0\}$ so that we can consider the space $\modc[S^{n-1}]{\bar{M}\cup_{\partial M}(\partial \bar{M}\times [0,1])\setminus D}$--otherwise the space $\bar{M}\setminus D$ is not a manifold and so $\modc[S^{n-1}]{\bar{M}\setminus D}$ is not defined.
%Take a closed collar neighborhood of $\partial \bar{M}$
%By abuse of notation, we replace  thickened neighborhood of $\bar{M}$ by taking a closed collar neighborhood of $\partial \bar{M}$ , which, , we also denote by . 
By \cref{map from bar construction to conf is a weak equivalence} and the fact that $\text{Conf}(\bar{M})$ and $\text{Conf}(M')$ are homotopy equivalent to $\text{Conf}(M)$, we have weak homotopy equivalences
%\footnote{probably don't need the second line}
\begin{align*}
   \Vert\iterxbarc{\bar{D}^{n}}{S^{n-1}}{\bar{M}\setminus D}{\bullet}{1}\Vert &\to \text{Conf}(M)\\
   \Vert\iterxbarc{\bar{D}}{\bar{D}^{n-1}_{1}}{M'\setminus D}{\bullet}{1}\Vert &\to \text{Conf}(M).
\end{align*}
We have that $\mon[\bar{D}^{n-1}_{1}]$ is a sub-monoid of $\mon[S^{n-1}]$ by identifying $\bar{D}^{n-1}_{1}$ as a subspace of $\partial \bar{D}$.
%so that $\mon[\bar{D}^{n-1}_{1}]$ is a sub-monoid of $\mon[S^{n-1}]$.
%, since $\bar{D}^{n-1}_{1}\subset \partial \bar{D}$. 
As a result, there is an inclusion $$\modc[\bar{D}^{n-1}_{1}]{M'\setminus D}\hookrightarrow\modc[S^{n-1}]{\bar{M}\setminus D}$$ of $\mon[\bar{D}^{n-1}_{1}]$-modules. 
%Since $\text{Conf}(M')$ is homotopy equivalent to $\text{Conf}(\bar{M})$, 
Therefore, we have the following commutative diagram, with the horizontal maps weak homotopy equivalences:
\begin{center}
    \begin{tikzcd}
    \Vert\iterxbarc{\bar{D}}{\bar{D}^{n-1}_{1}}{M'\setminus D}{\bullet}{1}\Vert
    \arrow[d]\arrow[r]&  \text{Conf}(M)\arrow[equals]{d}\\
    \Vert\iterxbarc{\bar{D}^{n}}{S^{n-1}}{\bar{M}\setminus D}{\bullet}{1}\Vert\arrow[r]& \text{Conf}(M).
\end{tikzcd}
\end{center}
%the map $$\Vert\iterxbarc{\bar{D}}{\bar{D}^{n-1}_{1}}{M'\setminus D}{\bullet}{1}\Vert \to \Vert\iterxbarc{\bar{D}^{n}}{S^{n-1}}{\bar{M}\setminus D}{\bullet}{1}\Vert$$ is a weak homotopy equivalence.

%Let\footnote{probably don't need this paragraph} $L_{\bullet}$ denote the simplicial space 
%$$L_{\bullet}\colonequals B_{\bullet}(\rimodc[D^{n-1}_{0}\cup \bar{D}^{n-1}_{1}]{\bar{D}}{2},\mon[\bar{D}^{n-1}_{1}],\modc[\bar{D}_{1}^{n-1}]{M'\setminus D}).$$
%The map 
%\begin{align*}
%    G_{0}\colon \rimodc[D^{n-1}_{0}\sqcup \bar{D}^{n-1}_{1}]{\bar{D}}{2}&\to \modc[S^{n-1}]{\bar{D}}\\
%    (\moconfigone, (\vec{\neck}(1), \vec{\neck}(2))&\mapsto (\moconfigone, \vec{\neck}(1)+\vec{\neck}(2))
%\end{align*}
%is a map of $\mon[\bar{D}^{n-1}_{1}]$-modules. We also have an inclusion $\modc[\bar{D}^{n-1}_{1}]{\bar{D}}\hookrightarrow \rimodc[D^{n-1}_{0}\sqcup \bar{D}^{n-1}_{1}]{\bar{D}}{2}$ of $\mon[\bar{D}^{n-1}_{1}]$-modules.
%Since the inclusion 
%The composition $\modc[\bar{D}^{n-1}_{1}]{\bar{D}}\hookrightarrow \rimodc[D^{n-1}_{0}\sqcup \bar{D}^{n-1}_{1}]{\bar{D}}{2} \xrightarrow{F}\modc[S^{n-1}]{\bar{D}}$ is a composition of $\mon[\bar{D}^{n-1}_{1}]$-module maps. 
%As a result, we have a composition of weak homotopy equivalences\footnote{I should rewrite this because it's a little wordy}
%$$\Vert\iterxbarc{\bar{D}}{\bar{D}^{n-1}_{1}}{M'\setminus D}{\bullet}{1}\Vert \to \Vert L_{\bullet}\Vert\to \Vert\iterxbarc{\bar{D}}{S^{n-1}}{\bar{M}\setminus D}{\bullet}{1}\Vert.$$
%$$\iterxbarc{\bar{D}}{S^{n-1}}{\bar{M}\setminus D}{\bullet}{\srange}$$

To prove that Diagram \ref{diagram between different stabilization maps} commutes up to chain homotopy, it suffices to replace $M$ with $\bar{M}$ in Diagram \ref{diagram between different stabilization maps} and show that the diagram
\begin{equation}\label{modified diagram of stabilization maps}\begin{tikzcd}
\tot{\citerxbarc{T(V)_{*}}{S^{n-1}}{\bar{M}\setminus D}{\bullet}{1}{\fieldc}}
\arrow[r]\arrow[d, "\newstab{z}"] &
C_{*}(\text{Conf}(M);\fieldc)
\arrow[d, "t_{z}"]\\
\tot{\citerxbarc{\bar{D}^{n}}{S^{n-1}}{\bar{M}\setminus D}{\bullet}{1}{\fieldc}}
\arrow[r]& 
C_{*}(\text{Conf}(M);\fieldc)
\end{tikzcd}\end{equation}
commutes up to chain homotopy. In $T(V)_{*}$, we now identify the $\bar{D}^{n}$ in $\nch{\bmodc[S^{n-1}]{\bar{D}^{n}}{\leq 1}}$ with $\bar{D}^{n}_{2}$, the $S^{n-1}$ in $\nch{\mon[S^{n-1}]}$ with $\partial\bar{D}^{n}_{2}$, and the $\mathbb{S}^{n-1}$ in $C_{*}^{\text{cell}}(\mathbb{S}^{n-1};\fieldc)$ with $\partial\bar{D}^{n}_{2}$.

%We now relate the chain complexes in $V_{\bullet}$ to chain complexes coming from the spaces in our figure.
%We now consider $V_{\bullet}$. 
%In $V_{\bullet}$, we identify $\bar{D}^{n}$ in $\nch{\bmodc[S^{n-1}]{\bar{D}^{n}}{\leq 1}}$ with $\bar{D}^{n}_{2}$ and $S^{n-1}$ in $\nch{\mon[S^{n-1}]}$ with $\partial\bar{D}^{n}_{2}$, so that $V_{0}=\nch{\bmodc[\partial\bar{D}^{n}_{2}]{\bar{D}^{n}_{2}}{\leq 1}}\otimes \nch{\mon[\partial\bar{D}^{n}_{2}]}$. 
We now summarize the rest of the proof. Heuristically, we think of $\newstab{z}$ as stabilizing by putting $z$ inside $D\setminus \bar{D}^{n}_{2}$ and $t_{z}$ as stabilizing by bringing $z$ into $D$ from outside. 
%We have the disk $D$, and so 
In $D$ we also have the stabilization maps $\gamma^{l}_{z}$ and $\gamma^{r}_{z}$ from \cref{stabilization maps on non-module model for Conf(D^n)}, which we can use to compare $\newstab{z}$ and $t_{z}$.
%Heuristically, we think of $\gamma^{r}_{z}$ as stabilizing by bringing $z$ in from the right, next to $D^{n-1}_{0}$, and 
We use $\gamma^{l}_{z}$ and $\gamma^{r}_{z}$ to obtain new maps $\hat{\gamma}^{l}_{z}$ and $\hat{\gamma}^{r}_{z}$ on (a model for) $\nch{\text{Conf}(\bar{D})}$ that are also maps of right $\nch{\mon[\bar{D}_{1}^{n-1}]}$-modules. These new maps induce stabilization maps $\hat{t}^{l}_{z}$ and $\hat{t}^{r}_{z}$ on $\nch{\text{Conf}(M)}$. We use the homotopy from $\gamma^{l}_{z}$ to $\gamma^{r}_{z}$ in \cref{induced gamma maps are homotopic} to get an induced homotopy from $\hat{t}^{l}_{z}$ to $\hat{t}^{r}_{z}$ and then we relate $\hat{t}^{l}_{z}$ and $\hat{t}^{r}_{z}$ to $\newstab{z}$ and  $t_{z}$. 
%We will use $\gamma^{l}_{z}$ and $\gamma^{r}_{z}$ to obtain new stabilization maps $\hat{t}^{l}_{z}$ and $\hat{t}^{r}_{z}$ on (a model for) $\nch{\text{Conf}(M)}$ and then relate these new maps to $t_{z}$ and $\newstab{z}$. 
%We are unable to modify $\gamma^{l}_{z}$ and $\gamma^{r}_{z}$ to make them into maps of right $\nch{\mon[S^{n-1}]}$-modules, but we can modify them to make them into maps of right $\nch{\mon[\bar{D}^{n-1}]}$-modules.
%We will use $\gamma^{l}_{z}$ and $\gamma^{r}_{z}$ to obtain new stabilization maps $\hat{t}^{l}_{z}$ and $\hat{t}^{r}_{z}$
%We can think of the stabilization maps $t_{z}$ and $\newstab{z}$ as coming from two different stabilization maps in 
%\tot{\citerxbarc{T(V)_{*}}{S^{n-1}}{\bar{M}\setminus D}{\bullet}{1}{\fieldc}}

We introduce a sub-semi-simplicial chain complex of the semi-simplicial chain complex $V_{\bullet}$ from \cref{main model for chains on Conf(D^n)}. %Let $F(\alpha(\fieldc[x]\otimes \nch{\modc[S^{n-1}]{\bar{D}^{n}}}))\subset \nch{\mon[S^{n-1}]}$
%Let $$F\subset \nch{\mon[\partial \bar{D}^{n}_{2}]}$$
Let $F$ denote the chain complex
%right $\nch{\mon[\bar{D}^{n-1}_{1}]}$-submodule of $\nch{\mon[S^{n-1}]}$ 
$$F\colonequals\alpha(\fieldc[x]\otimes \nch{\modc[S^{n-1}]{\bar{D}^{n}}})\otimes \nch{\mon[\bar{D}^{n-1}_{1}]}.$$ 
From restricting the map
$$\nch{\mon[S^{n-1}]}\otimes\nch{\mon[S^{n-1}]}\to \nch{\mon[S^{n-1}]}$$ to $F$, we can view $F$ as a right $\nch{\mon[\bar{D}^{n-1}_{1}]}$-submodule of $\nch{\mon[\partial\bar{D}]}$.
%denote the right $\nch{\mon[\bar{D}^{n-1}_{1}]}$-submodule of $\nch{\mon[S^{n-1}]}$ generated by the subcomplex $\alpha(\fieldc[x]\otimes \nch{\modc[S^{n-1}]{\bar{D}^{n}}})\subset \nch{\mon[\partial\bar{D}]}$. 
Note that 
%since the fibers of $\alpha$ are contractible, 
%any endomorphism $g$ of $\fieldc[x]\otimes \nch{\modc[S^{n-1}]{\bar{D}^{n}}}$ extends to an endomorphism $\hat{g}$ of $F$ that is also a map of right $\nch{\mon[\bar{D}^{n-1}_{1}]}$-modules.
any map $g\colon \fieldc[x]\otimes \nch{\modc[S^{n-1}]{\bar{D}^{n}}}\to \fieldc[x]\otimes \nch{\modc[S^{n-1}]{\bar{D}^{n}}}$
extends to a map $\hat{g}\colon F\to F$ of right $\nch{\mon[\bar{D}^{n-1}_{1}]}$-modules.
%\footnote{Is this claim true? If it isn't true, I think I have a way of modifying $\alpha$ so that it is injective (and I think the result would be that this sentence is true), but this would require changing some earlier proofs.} 
The construction of stabilization maps on $\fieldc[x]\otimes \nch{\modc[S^{n-1}]{\bar{D}^{n}}}$ all depend on a fixed embedding $f'\colon D^{n}\to D^{n-1}\times (0,1)\subset S^{n-1}\times (0,1)$. We identify $D^{n-1}\times (0,1)$ with a tubular neighborhood of $D^{n-1}_{0}$ in $D$ that contains $\embed(\mathbb{R}^{n})$ and that does not intersect $\bar{D}^{n-1}_{1}$.
%\footnote{should make this more consistent with how we identify spaces in $V_{\bullet}$ }
%Note that $\alpha(\fieldc[x]\otimes \nch{\modc[S^{n-1}]{\bar{D}^{n}}})\subset \nch{\mon[S^{n-1}]}$ is not a right $\nch{\mon[\bar{D}^{n-1}_{1}]}$-module, but we can consider 
Let $\hat{U}_{\bullet}\subset V_{\bullet}$ denote the sub-semi-simplicial chain complex whose $p$-simplices are
\[
\hat{U}_{p}=\begin{cases}
     \nch{\bmodc[S^{n-1}]{\bar{D}^{n}}{\leq 1}}\otimes F & p=0\\
     C_{*}^{\text{cell}}(\mathbb{S}^{n-1};\fieldc)\otimes F & p=1.\\
   \end{cases}
   \]
%\begin{align*}
%\end{align*}
%$\nch{\bmodc[S^{n-1}]{\bar{D}^{n}}{\leq 1}}\otimes F(\alpha(\fieldc[x]\otimes \nch{\modc[S^{n-1}]{\bar{D}^{n}}}))$ and whose $1$-simplices are $C_{*}^{\text{cell}}(\mathbb{S}^{n-1};\fieldc)\otimes F(\alpha(\fieldc[x]\otimes \nch{\modc[S^{n-1}]{\bar{D}^{n}}}))$
The inclusion $\tot{\hat{U}_{\bullet}}\to T(V)_{*}$ is a quasi-isomorphism
%\footnote{Is this obvious?} 
and a map of right $\nch{\mon[\bar{D}^{n-1}_{1}]}$-modules. 
%Let $G_{*}$ denote the chain complex
$$\tot{\citerxbarc{\tot{ \hat{U}_{\bullet}}}{\bar{D}^{n-1}_{1}}{M'\setminus D}{\bullet}{1}{\fieldc}}\colonequals \tot{\citerxbarc{\tot{ \hat{U}_{\bullet}}}{\bar{D}^{n-1}_{1}}{M'\setminus D}{\bullet}{1}{\fieldc}}.$$ 
We now construct stabilization maps on $\tot{\citerxbarc{\tot{ \hat{U}_{\bullet}}}{\bar{D}^{n-1}_{1}}{M'\setminus D}{\bullet}{1}{\fieldc}}$.

Since the stabilization map $s^{l}_{z}$ from \cref{stabilization maps on other model of Conf(cylinder)} is a map from $\fieldc[x]\otimes \nch{\modc[S^{n-1}]{\bar{D}^{n}}}$ to itself, $s^{l}_{z}$ induces a map $\hat{s}^{l}_{z}\colon F\to F$ and the map $\gamma^{l}_{0,z}=(m_{r})_{*}\circ EZ\circ \big(\text{id}\otimes  (\alpha\circ s^{l}_{z}) \big)$ induces a map of right $\nch{\mon[\bar{D}^{n-1}_{1}]}$-modules $$\hat{\gamma}^{l}_{0,z}\colon \hat{U}_{0} \to \nch{\modc[S^{n-1}]{\bar{D}^{n}}}.$$ The homotopy $H_{l}$ from $\gamma^{l}_{0,z} \circ d_{0}$ to $\gamma^{l}_{0,z} \circ d_{1}$ given in \cref{face maps composed with gamma maps are homotopic} 
%comes from the homotopy\footnote{want to make this claim more explicit in \cref{face maps composed with gamma maps are homotopic}}
%\begin{align*}
%    H_{l}'\colon \fieldc[x]&\otimes \nch{\modc[S^{n-1}]{\bar{D}^{n}}}\to \fieldc[x]\otimes \nch{\modc[S^{n-1}]{\bar{D}^{n}}},\\
%    x^{j}&\otimes y \mapsto \sum_{i=1}^{j+1}\bigg(\binom{j+1}{i}-\binom{j}{i}\bigg) x^{j+1-i}\otimes (m_{l})_{*}\circ EZ(\mon[f']_{*}(\lambda^{i-1}(e,\omega)\otimes y)
%\end{align*}
%we have an induced 
induces a homotopy $\hat{H}_{l}$ from $\hat{\gamma}^{l}_{0,z}\circ d_{0}$ to $\hat{\gamma}^{l}_{0,z}\circ d_{1}$ and which 
is also a map of right $\nch{\mon[\bar{D}^{n-1}_{1}]}$-modules. Therefore, the map $\gamma^{l}_{z}\colon T(U)_{*} \to \nch{\modc[S^{n-1}]{\bar{D}^{n}}}$ induces a map of right $\nch{\mon[\bar{D}^{n-1}_{1}]}$-modules $$\hat{\gamma}^{l}_{z}\colon \tot{ \hat{U}_{\bullet}} \to \nch{\modc[S^{n-1}]{\bar{D}^{n}}}.$$ Similarly, the map $\gamma^{r}_{z}\colon T(U)_{*} \to \nch{\modc[S^{n-1}]{\bar{D}^{n}}}$ induces a map of right $\nch{\mon[\bar{D}^{n-1}_{1}]}$-modules $\hat{\gamma}^{r}_{z}\colon \tot{ \hat{U}_{\bullet}} \to \nch{\modc[S^{n-1}]{\bar{D}^{n}}}$.

Therefore, we have stabilization maps
$$\hat{t}^{l}_{z,c}, \hat{t}^{r}_{z,c}\colon \tot{\citerxbarc{\tot{ \hat{U}_{\bullet}}}{\bar{D}^{n-1}_{1}}{M'\setminus D}{\bullet}{1}{\fieldc}}\to \tot{\citerxbarc{\bar{D}^{n}}{S^{n-1}}{\bar{M}\setminus D}{\bullet}{1}{\fieldc}}$$ induced from $\hat{\gamma}^{l}_{z}, \hat{\gamma}^{r}_{z}$, respectively. From the inclusions $\mon[\bar{D}^{n-1}]\hookrightarrow \mon[S^{n-1}]$ and $\hat{U}_{\bullet}\hookrightarrow V_{\bullet}$, we have the following diagram.
\begin{center}\begin{tikzcd}
\tot{\citerxbarc{\tot{ \hat{U}_{\bullet}}}{\bar{D}^{n-1}_{1}}{M'\setminus D}{\bullet}{1}{\fieldc}}\arrow[r]\arrow[swap]{dr}{\hat{t}^{l}_{z,c}, \hat{t}^{r}_{z,c}}&\tot{\citerxbarc{T(V)_{*}}{S^{n-1}}{\bar{M}\setminus D^{n}}{\bullet}{1}{\fieldc}}
\arrow[r]\arrow[d, "\newstab{z}"] &
C_{*}(\text{Conf}(M);\fieldc)
\arrow[d, "t_{z}"]\\
&\tot{\citerxbarc{\bar{D}^{n}}{S^{n-1}}{\bar{M}\setminus D}{\bullet}{1}{\fieldc}}
\arrow[r]& 
C_{*}(\text{Conf}(M);\fieldc)
\end{tikzcd}\end{center}
Since there is a quasi-isomorphism $\tot{\citerxbarc{\tot{ \hat{U}_{\bullet}}}{\bar{D}^{n-1}_{1}}{M'\setminus D}{\bullet}{1}{\fieldc}}\to \nch{\text{Conf}(M')}$ and the following composition $$\tot{\citerxbarc{\tot{ \hat{U}_{\bullet}}}{\bar{D}^{n-1}_{1}}{M'\setminus D}{\bullet}{1}{\fieldc}}\to \tot{\citerxbarc{T(V)_{*}}{S^{n-1}}{\bar{M}\setminus D^{n}}{\bullet}{1}{\fieldc}}
\to
C_{*}(\text{Conf}(M);\fieldc)$$ is a quasi-isomorphism, 
%the map $\tot{\citerxbarc{\tot{ \hat{U}_{\bullet}}}{\bar{D}^{n-1}_{1}}{M'\setminus D}{\bullet}{1}{\fieldc}}\to \tot{\citerxbarc{T(V)_{*}}{S^{n-1}}{\bar{M}\setminus D^{n}}{\bullet}{1}{\fieldc}}$ 
the first map in this composition is a quasi-isomorphism
and it induces an isomorphism on sets of chain homotopy classes of maps
\begin{align*}
   [\tot{\citerxbarc{T(V)_{*}}{S^{n-1}}{\bar{M}\setminus D^{n}}{\bullet}{1}{\fieldc}},\tot{\citerxbarc{\bar{D}^{n}}{S^{n-1}}{\bar{M}\setminus D}{\bullet}{1}{\fieldc}}]&\cong\\ [\tot{\citerxbarc{\tot{ \hat{U}_{\bullet}}}{\bar{D}^{n-1}_{1}}{M'\setminus D}{\bullet}{1}{\fieldc}},\tot{\citerxbarc{\bar{D}^{n}}{S^{n-1}}{\bar{M}\setminus D}{\bullet}{1}{\fieldc}}]&. 
\end{align*}
%$$[\tot{\citerxbarc{T(V)_{*}}{S^{n-1}}{\bar{M}\setminus D^{n}}{\bullet}{1}{\fieldc}},\tot{\citerxbarc{\bar{D}^{n}}{S^{n-1}}{\bar{M}\setminus D}{\bullet}{1}{\fieldc}}]\cong [\tot{\citerxbarc{\tot{ \hat{U}_{\bullet}}}{\bar{D}^{n-1}_{1}}{M'\setminus D}{\bullet}{1}{\fieldc}},\tot{\citerxbarc{\bar{D}^{n}}{S^{n-1}}{\bar{M}\setminus D}{\bullet}{1}{\fieldc}}]. $$
Therefore, to prove that Diagram \ref{modified diagram of stabilization maps} commutes up to chain homotopy, it suffices to show that the diagrams
\begin{equation}\label{triangle}\begin{tikzcd}
\tot{\citerxbarc{\tot{ \hat{U}_{\bullet}}}{\bar{D}^{n-1}_{1}}{M'\setminus D}{\bullet}{1}{\fieldc}}\arrow[r]\arrow[swap]{dr}{\hat{t}^{l}_{z,c}}&\tot{\citerxbarc{T(V)_{*}}{S^{n-1}}{\bar{M}\setminus D^{n}}{\bullet}{1}{\fieldc}}
\arrow[d, "\newstab{z}"]]\\
&\tot{\citerxbarc{\bar{D}^{n}}{S^{n-1}}{\bar{M}\setminus D}{\bullet}{1}{\fieldc}}
\end{tikzcd}\end{equation}
and
\begin{equation}\label{square}\begin{tikzcd}
\tot{\citerxbarc{\tot{ \hat{U}_{\bullet}}}{\bar{D}^{n-1}_{1}}{M'\setminus D}{\bullet}{1}{\fieldc}}\arrow[r]\arrow[swap]{d}{\hat{t}^{r}_{z,c}}&
C_{*}(\text{Conf}(M);\fieldc)
\arrow[d, "t_{z}"]\\
\tot{\citerxbarc{\bar{D}^{n}}{S^{n-1}}{\bar{M}\setminus D}{\bullet}{1}{\fieldc}}
\arrow[r]& 
C_{*}(\text{Conf}(M);\fieldc)
\end{tikzcd}\end{equation} commute up to chain homotopy and that $\hat{t}^{l}_{z,c}$ and $\hat{t}^{r}_{z,c}$ are chain homotopic.

Recall that $\newstab{z}$ is induced from the map $$g_{0,z}^{l}=(m_{r})_{*}\circ EZ\circ \big((\text{id}_{\modc[S^{n-1}]{\bar{D}^{n}}})_{*}\otimes t^{l}_{z} \big)\colon V_{0}\to \nch{\modc[S^{n-1}]{\bar{D}^{n}}}.$$
Since $\alpha \circ s^{l}_{z}$ is homotopic to $t^{l}_{z}\circ\alpha$ by \cref{comparing stabilization maps on different conf(cyl) models}, the restriction of $$t_{z}^{l}\colon \nch{\mon[S^{n-1}]}\to \nch{\mon[S^{n-1}]}$$ to $F$ is homotopic to $\hat{s}^{l}_{z}$. As a result,
%\footnote{should I give more details? I feel like I'm being too hand-wavy here.} 
Diagram \ref{triangle} commutes up to chain homotopy.
%\footnote{should I give more details here?} 
Since the space of orientation-preserving embeddings from $D^{n}$ to itself is path-connected, the embedding $f'\colon D^{n}\to D^{n-1}_{0}\times (0,1)$ is homotopic to $\embed\restr{D^{n}}\colon D^{n}\to D^{n-1}_{0}\times (0,1)$. As a result,
Diagram \ref{square}
commutes up to chain homotopy.
%\footnote{should I include a drawing of this?} 
The homotopy from $\gamma^{l}_{z}$ to $\gamma^{r}_{z}$ given in \cref{induced gamma maps are homotopic} induces a homotopy from $\hat{\gamma}^{l}_{z}$ to $\hat{\gamma}^{r}_{z}$ that is also a right $\nch{\mon[\bar{D}^{n-1}_{1}]}$-module map. Therefore, $\hat{t}^{l}_{z,c}$ and $\hat{t}^{r}_{z,c}$ are chain homotopic.
\end{proof}
\begin{corollary}
%Let $\fieldc$ be a ring and let $e\in H_{0}(\text{Conf}_{1}(\mathbb{R}^{n});\fieldc)$ be the class of a point. 
Suppose that $z\in H_{*}(\text{Conf}(\mathbb{R}^{n});\fieldc)$ is a class such that $\browd{z}{e}$ is zero. There is a natural isomorphism $$H_{*}\big(\text{Cone}(g^{l}_{z}\colon  T(V)_{*}\to \nch{\modc[S^{n-1}]{\bar{D}^{n}}})\big)\to H_{*}(\iteratedmap{D^{n}}{t_{z}}{1};\fieldc).$$
%From the composition of weak homotopy equivalences
%$$Y\xrightarrow{\Vert \alpha_{0}\Vert} \modc[S^{n-1}]{\bar{D}^{n}}\xrightarrow{f} \text{Conf}(\mathbb{R}^{n})$$
%we have that $f\circ g_{z}\colon C_{*}(Y;\fieldc)\to C_{*}(\text{Conf}(\mathbb{R}^{n});\fieldc)$ and $t_{z}\circ f\circ \Vert \alpha_{0}\Vert\colon C_{*}(Y;\fieldc)\to C_{*}(\text{Conf}(\mathbb{R}^{n});\fieldc) $ are chain homotopic.
%%identifications $Y\simeq\text{Conf}(\mathbb{R}^{n})\simeq \modc[S^{n-1}]{\bar{D}^{n}}$,
%%the stabilization map $g_{z}\colon C_{*}(Y;\fieldc)\to C_{*}(\text{Conf}(\modc[S^{n-1}]{\bar{D}^{n}};\fieldc)$ is chain homotopic to the map $t_{z}\colon C_{*}(\text{Conf}(\mathbb{R}^{n});\fieldc)\to C_{*}(\text{Conf}(\mathbb{R}^{n});\fieldc)$.
\end{corollary}
\begin{corollary}\label{new stabilization induces iso on open manifold}
Suppose that $M$ is an open connected manifold 
%and that $\fieldc$ is a ring
%Let $e\in H_{0}(\text{Conf}_{1}(\mathbb{R}^{n});\fieldc)$ be the class of a point. 
and that $z\in H_{*}(\text{Conf}(\mathbb{R}^{n});\fieldc)$ is a class such that $\browd{z}{e}$ is zero. If  $t_{z}$ induces an isomorphism on $H_{*}(\text{Conf}(M);\fieldc)$ in a range, then $\newstab{z}$ induces an isomorphism on $H_{*}(\text{Conf}(M);\fieldc)$ in the same range.
\end{corollary}
%A similar argument allows one to show:\footnote{fix sentence}
\begin{remark}
Suppose that $M$ is an open connected manifold of dimension $n$ and that $z\in H_{i}(\text{Conf}(\mathbb{R}^{n});\fieldc)$ is a class such that $\browd{z}{e}=0$.
%the Browder bracket $\browd{z}{e}$ of $z$ and $e$ vanishes. 
By \cref{new stabilization on open manifold}, $t_{z}$ and $\newstab{z}$ induce the same map on homology. Since the map $$\newstab{z}\colon H_{*}( \tot{\citerxbarc{T(V)_{*}}{S^{n-1}}{M\setminus D^{n}}{\bullet}{1}{\fieldc}})\to  H_{*}(\tot{\citerxbarc{\bar{D}^{n}}{S^{n-1}}{M\setminus D^{n}}{\bullet}{1}{\fieldc}})$$ does not depend on an embedding $\embed\colon M\sqcup \mathbb{R}^{n}\to M$, the stabilization map $$t_{z}\colon H_{*}(\text{Conf}(M);\fieldc)\to  H_{*}(\text{Conf}(M);\fieldc)$$ does not depend on a choice of an embedding $M \sqcup \mathbb{R}^{n} \to M$. 
%This argument also shows that the stabilization map $t_{z}$ is chain homotopy equivalent to $\newstab{z}$ but clearly all weird factorization homology stabilization maps are the same (except possibly by choice of homotopy but this homotopy does not depend on the embedding so basically this says for closed manifolds the choice of homotopy does not matter.)
\end{remark}
\section{The Puncture Resolution}\label{sec:puncture resolution}
%To formulate secondary (and higher-order) homological stability for the configuration space of a closed manifold $M$ of dimension $n$, we need to define an iterated mapping cone for $C_{*}(\text{Conf}(M);\fieldc)$. The construction of an iterated mapping cone for $C_{*}(\text{Conf}(M);\fieldc)$ when $M$ is an open manifold (see \cref{sec:base case} and the beginning of \cref{second stability}) does not work here because that construction depends on an embedding $M\bigsqcup(\bigsqcup_{j=1}^{\srange} D^{n})\to M$. 
%After constructing an iterated mapping cone,
%To conclude our proof of stability for the configuration space $\text{Conf}_{\partic}(M)$ of a closed manifold $M$, we introduce an iterated mapping cone for $C_{*}(\text{Conf}_{\partic}(M))$ 
We prove stability using a puncture resolution argument inspired by \textcite[Section 9.4]{MR3032101}. Our puncture resolution is slightly more similar to the puncture resolution of \textcite[Section 3]{MR3556286} than to that of \textcite[Section 9.4]{MR3032101}. %\footnote{Should also I cite pages 7-8 of \textcite{MR3556286}?} 
The puncture resolution is a semi-simplicial resolution $(\ppunc{\bullet}, \text{Conf}_{\partic}(M), \epsilon )$ of $\text{Conf}_{\partic}(M)$
%. The $p$-
whose simplices
%of the semi-simplicial space $Z_{\partic,\bullet}(M)$ 
involve configuration spaces of open manifolds. We then pass to chain complexes to construct an analogous semi-simplicial chain complex $\cppunc[A^{I}_{*}]{\bullet}{\partic}$ whose simplices involve chain complexes of open manifolds and whose totalization $\tot{\cppunc[A^{I}_{*}]{\bullet}{\partic}}$ is quasi-isomorphic to $\nch{\text{Conf}_{\partic}(M)}$. By the geometric realization spectral sequence, we can compare the homology of the $p$-simplices $\cppunc[A^{I}_{*}]{p}{\partic}$ to the homology of $C_{*}(\text{Conf}_{\partic}(M);\fieldc)$ and leverage stability results for configuration spaces of open manifolds to prove analogous results for configuration spaces of closed manifolds.
\begin{definition}
Given a space $X$, let $$\text{OConf}^{\, \delta}_{\partic}(X)\colonequals \{(x_{1},\ldots,x_{\partic})\subset X^{\partic}\colon x_{i}\neq x_{j}\text{ if } i\neq j\}$$ denote the \emph{set} of ordered configurations of $X$ of cardinality $\partic$. We do \textbf{not} use the topology on $\text{OConf}^{\, \delta}_{\partic}(X)$ as a subspace of $X^{\partic}$.
\end{definition}
\begin{comment}
\begin{align*}
   \text{OConf}^{\, \delta}_{\partic}(X)\colonequals \{ (x_{1},\ldots ,x_{\partic})\in X^{k}\Vert x_{i}\neq x_{j} \text{ for }i\neq j\}. 
\end{align*}
\end{comment}
\begin{comment}
\begin{definition}
Given a space $X$, let $\text{Sym}_{\partic}^{\leq j}(X)$ denote the space\footnote{Should this be set?} of finite sums of the form    $ \sum_{i=1}^{l}n_{i}x_{i} $
\end{comment}
\begin{comment}
\begin{align*}
    \sum_{i=1}^{l}n_{i}x_{i} 
\end{align*}
\end{comment}
\begin{comment}
where $x_{1},\ldots ,x_{l}$ are distinct points in $X$ and $n_{i}$ are positive integers with $n_{i}\leq j$ for all $i$ and $k=\sum_{i=1}^{l}n_{i}$.\footnote{Should I use the following definition of bounded symmetric powers instead of the first: Let $\text{Sym}^{\leq j}_{\partic}(X)$ denote the subspace of points $y$ in the $k$-fold symmetric power of $X$ such that at most $j$ points of $y$ are at the same location in $X$.}
\end{comment}
\begin{comment}
Given a space $X$, let $\text{Sym}_{\partic}(X)$ denote the quotient of $X^{k}$ by the permutation action of the symmetric group $S_{\partic}$ on the $k$ coordinates of $X^{k}$.
\end{comment}
%\end{definition}
\begin{definition}\label{new puncture res for conf}
%Fix $\srange$ disjoint closed disks $\bar{D}_{1},\ldots,\bar{D}_{\srange}\subset M$ and let $E_{\srange}\colonequals \bigsqcup_{j=1}^{\srange}$.
Let $\bar{K}_{\srange}$ and $K_{\srange}$ be as in \cref{notation for iterated configuration space model}.
%\footnote{the notation here and in the next few pages need to be updated} 
%
% Let $A_{r,s}$ denote a disjoint union of $r$ copies of $V$ and $s$ copies of $\modc[S^{n-1}]{\bar{D}^{n}}$ with $r+s=\srange$.
%\footnote{fix this notation} 
%either $\bar{K}_{\srange}$ or $\bigsqcup_{j=1}^{\srange}Y_{j}$
%Let $\punc{\bullet}$ be the semi-simplicial space whose $i$-simplices are
%$$\punc{i}\colonequals\bigsqcup_{(m_{0}, \ldots m_{i})\in \text{OConf}^{\, \delta}_{i+1}(M\setminus \bar{K}_{\srange})}\Vert\iterxbarc{A_{r,s}}{\partial \bar{K}_{\srange}}{M\setminus K_{\srange}\setminus\{m_{0}, \ldots m_{i}\}}{\bullet}{\srange} \Vert.$$
Let $\ppunc{\bullet}$ be the semi-simplicial space whose $p$-simplices are
$$\ppunc{p}\colonequals\bigsqcup_{(m_{0}, \ldots m_{p})\in \text{OConf}^{\, \delta}_{p+1}(M\setminus \bar{K}_{\srange})}\Vert\piterxbarc{\bar{K}_{\srange}}{\partial \bar{K}_{\srange}}{M\setminus K_{\srange}\setminus\{m_{0}, \ldots m_{p}\}}{\bullet}{\srange}{\partic} \Vert.$$
The inclusions $M\setminus \{ m_{0},\ldots , m_{p} \} \hookrightarrow M\setminus \{ m_{0},\ldots ,\widehat{m_j},\ldots , m_{p}\},$ for $j=0,\ldots, p$,  induce the face maps for $\ppunc{\bullet}$. Let $\punc{\bullet}$ denote the semi-simplicial space $$\punc{\bullet}\colonequals\bigsqcup_{\partic=0}^{\infty}\ppunc{\bullet}.$$
\end{definition}
The map $\epsilon\colon \punc{0}\to  \Vert\iterxbarc{\bar{K}_{\srange}}{\partial \bar{K}_{\srange}}{M\setminus K_{\srange}}{\bullet}{\srange} \Vert$ which forgets the ordered configuration $(m_{0})$ satisfies $\epsilon d_{0}=\epsilon d_{1}$, so $(\punc{\bullet},\Vert\iterxbarc{\bar{K}_{\srange}}{\partial \bar{K}_{\srange}}{M\setminus K_{\srange}}{\bullet}{\srange} \Vert, \epsilon) $ is an augmented semi-simplicial space. To prove that it is a semi-simplicial resolution of $\Vert\iterxbarc{\bar{K}_{\srange}}{\partial \bar{K}_{\srange}}{M\setminus K_{\srange}}{\bullet}{\srange} \Vert$, we use a microfibration technique and a result from \textcite{MR3718454}.
\begin{definition}\label{microfibration definition}
A map $g \colon E\to  B$ of spaces is a \textbf{Serre microfibration} if for any $k$ and any commutative diagram 

\begin{center}\begin{tikzcd}
\bar{D}^{k} \times\{ 0 \} \arrow[r,"h"]\arrow[d] & E\arrow[d,"g"]\\
\bar{D}^{k} \times\left [0, 1 \right] \arrow[r,"H"] & B
\end{tikzcd}\end{center}
there exists an $\epsilon > 0$ and a map $\tilde{H} \colon \bar{D}^{k} \times \left [0,\epsilon \right ]\to  E$ with $\tilde{H}(x,0)=h(x)$ and $g \circ \tilde{H}(x,t)=H(x,t)$ for all $(x,t)\in \bar{D}^{k}\times \left [0,\epsilon\right ]$.
\end{definition}
\begin{lemma}[{\textcite[Lemma 2.2]{MR2175298}}]\label{Weiss' Lemma}
Suppose that $g \colon E\to  B$ is a Serre microfibration. If $g$ has weakly contractible fibers, then it is a weak homotopy equivalence.
\end{lemma}
\begin{proposition}[{\textcite[Proposition 2.8]{MR3718454}}]\label{Soren and Oscar's microfibration prop}
Let $W_{\bullet}$ be a semi-simplicial set, and $Z$ be a Hausdorff space. Let $X_{\bullet}\subset Z \times W_{\bullet}$ be a sub-semi-simplicial space which in each degree is an open subset. Then the map $\pi\colon \Vert X_{\bullet}\Vert\to Z$ is a Serre microfibration.
%\footnote{need to move this proposition to section 9}
%For $z\in Z$, let $X_{\bullet}(z) \subset W_{\bullet}$ be the sub-semi-simplicial set defined by $X_{\bullet}\cap (\{z\}\times W_{\bullet})=\{z\}\times  X_{\bullet}(z)$ and suppose that $\Vert X_{\bullet}(z)\Vert$ is $n$-connected for all $z\in Z$. Then the map $\pi\colon \Vert X_{\bullet}\Vert\to Z$ is $(n+1)$-connected.
\end{proposition}
\begin{lemma}\label{Z to iter is Serre}
The map $$\Vert\epsilon \Vert\colon  \Vert \punc{\bullet}\Vert \to \Vert\iterxbarc{\bar{K}_{\srange}}{\partial \bar{K}_{\srange}}{M\setminus K_{\srange}}{\bullet}{\srange} \Vert$$ is a Serre microfibration.
\end{lemma}
\begin{proof}
Let $W_{\bullet}$ denote the semi-simplicial set whose $p$-simplices are $\text{OConf}^{\, \delta}_{p+1}(M\setminus \bar{K}_{\srange})$ and whose face maps $d_{j}\colon W_{p}\to W_{p-1}$ send $(m_{0},\ldots,m_{p})\in W_{p}$ to $(m_{0},\ldots,\widehat{m_{j}},\ldots, m_{p})$. We can represent an element of $\punc{p}$ by a pair
\begin{align*}
    ((m_{0},\ldots,m_{p}),\moconfigone), &\text{ with }(m_{0},\ldots,m_{p})\in \text{OConf}^{\, \delta}_{p+1}(M\setminus \bar{K}_{\srange})\\
    &\text{ and } \moconfigone\in \Vert\iterxbarc{\bar{K}_{\srange}}{\partial \bar{K}_{\srange}}{M\setminus K_{\srange}\setminus\{m_{0}, \ldots m_{p}\}}{\bullet}{\srange} \Vert.
\end{align*}
%$((m_{0},\ldots,m_{i}),\moconfigone)$, with $(m_{0},\ldots,m_{i})\in \text{OConf}^{\, \delta}_{i+1}(M\setminus \bar{K}_{\srange})$ and $\moconfigone\in \Vert\iterxbarc{\bar{K}_{\srange}}{\partial \bar{K}_{\srange}}{M\setminus K_{\srange}\setminus\{m_{0}, \ldots m_{i}\}}{\bullet}{\srange} \Vert$. 
We have a natural open inclusion  $$\rimodc[\partial \bar{K}_{\srange}]{M\setminus K_{\srange}\setminus\{m_{0}, \ldots m_{p}\}}{\srange}\hookrightarrow \rimodc[\partial \bar{K}_{\srange}]{M\setminus K_{\srange}}{\srange}$$ that is also a map of left $\prod_{j=1}^{\srange}\mon[S^{n-1}_{j}]$-modules. As a result, we have an induced open inclusion
$$f\colon \Vert\iterxbarc{\bar{K}_{\srange}}{ \partial\bar{K}_{\srange}}{M\setminus K_{\srange}\setminus\{m_{0},\ldots,m_{p}\}}{\bullet}{\srange} \Vert\hookrightarrow\Vert\iterxbarc{\bar{K}_{\srange}}{\partial \bar{K}_{\srange}}{M\setminus K_{\srange}}{\bullet}{\srange}\Vert .$$ Therefore, the map
\begin{align*}
    g_{p}\colon \punc{p}&\to W_{p}\times \Vert\iterxbarc{\bar{K}_{\srange}}{\partial \bar{K}_{\srange}}{M\setminus K_{\srange}}{\bullet}{\srange} \Vert\\
    ((m_{0},\ldots,m_{p}),\moconfigone)&\mapsto ((m_{0},\ldots,m_{p}),f(\moconfigone)).
\end{align*}
is an open inclusion for all $p$.
The maps $g_{p}$ assemble to define a map of semi-simplicial spaces $$g_{\bullet}\colon \punc{\bullet}\to W_{\bullet}\times \Vert\iterxbarc{\bar{K}_{\srange}}{\partial \bar{K}_{\srange}}{M\setminus K_{\srange}}{\bullet}{\srange} \Vert .$$ Therefore, by \cref{Soren and Oscar's microfibration prop} the map $$\Vert \epsilon\Vert\colon \Vert \punc{\bullet}\Vert \to \Vert\iterxbarc{\bar{K}_{\srange}}{\partial \bar{K}_{\srange}}{M\setminus K_{\srange}}{\bullet}{\srange} \Vert$$ is a Serre microfibration.
\end{proof}
\begin{lemma}\label{puncture res is equivalent to conf}
%Let $M$ be a manifold of dimension $n$ and let $\srange$ be a positive integer. Let $\bar{D}^{n}_{1},\ldots, \bar{D}^{n}_{\srange}$ be $\srange$ disjoint closed disks in $M$ and let $E_{\srange}\colonequals \bigsqcup_{j=1}^{\srange}\bar{D}^{n}_{j}$.
%Suppose that $M$ is an dimension manifold of dimension $n$. 
%Let $M$ and $\bar{E}_{\srange}\subset M$ be as in \cref{notation for iterated configuration space model}.
%Let $(Z_{\partic,\bullet}(M, \bar{E}_{\srange}),\text{Conf}_{\partic}(M),\epsilon )$ be the augmented semi-simplicial space from \cref{puncture res for conf}. The augmentation map induces a weak homotopy equivalence $\Vert \epsilon \Vert  \colon  \Vert Z_{\partic,\bullet}(M, \bar{E}_{\srange})\Vert\to  \text{Conf}_{\partic}(M)$.
There is a weak homotopy equivalence
$\rho\colon  \Vert \punc{\bullet}\Vert\to \text{Conf}(M)$.
%The augmentation map induces a weak homotopy equivalence $$\Vert \epsilon \Vert  \colon  \Vert Z_{\bullet}(M,K_{\srange})\Vert\to  \Vert\iterxbarc{\bigsqcup_{j=1}^{\srange}Y_{j}}{\partial \bar{K}_{\srange}}{M\setminus K_{\srange}}{\bullet}{\srange} \Vert.$$
\end{lemma}
\begin{proof}
Our argument is similar to the proof of \textcite[Proposition 5.8]{MR3546782}, so we just summarize it. The map $\Vert \epsilon\Vert\colon\Vert \punc{\bullet}\Vert\to  \Vert\iterxbarc{\bar{K}_{\srange}}{\partial\bar{K}_{\srange}}{M\setminus K_{\srange}}{\bullet}{\srange} \Vert$ is a Serre microfibration by \cref{Z to iter is Serre}. Any fiber of this map is contractible because it is isomorphic to the complex of injective words for an infinite set $M\setminus \bar{K}_{\srange} \setminus\{ m_{1},\ldots m_{\partic}\}$, with $m_{1},\ldots, m_{\partic}$ a finite set of points in $M\setminus \bar{K}_{\srange}$. Therefore, by 
%\textcite[Lemma 2.2]{MR2175298}, 
\cref{Weiss' Lemma}, the map $\Vert \epsilon\Vert$ is a weak homotopy equivalence.
%\footnote{probably want to say it is weakly equivalent to $\text{Conf}(M)$ so that I can talk about $Z_{\partic,\bullet}(M, K_{\srange})$ }
By \cref{map from bar construction to conf is a weak equivalence}, we have a weak homotopy equivalence $g\colon\Vert\iterxbarc{\bar{K}_{\srange}}{\partial\bar{K}_{\srange}}{M\setminus K_{\srange}}{\bullet}{\srange} \Vert\to \text{Conf}(M)$. Since the composition of weak homotopy equivalences is a weak homotopy equivalence, the map $\rho\colonequals g\circ \Vert\epsilon\Vert$ is a weak homotopy equivalence.
\end{proof}
\begin{definition}
%Let $\cpunc{\bullet}$ denote the semi-simplicial chain complex
%\footnote{probably need to define iterated mapping cones as well} 
%whose $i$-simplices are
%$$\cpunc{i}\colonequals\bigsqcup_{(m_{0}, \ldots m_{i})\in \text{OConf}^{\, \delta}_{i+1}(M\setminus \bar{K}_{\srange})}\Vert\citerxbarc{\bigsqcup_{j=1}^{\srange}Y_{j}}{\partial \bar{K}_{\srange}}{M\setminus K_{\srange}\setminus\{m_{0}, \ldots m_{i}\}}{\bullet}{\srange}{\fieldc} \Vert.$$
%The inclusions $M\setminus \{ m_{0},\ldots , m_{i} \} \hookrightarrow M\setminus \{ m_{0},\ldots ,\widehat{m_j},\ldots , m_{i}\},$ for $j=0,\ldots , i$,  induce the face maps for $\cpunc{\bullet}$. 
%Let $A_{r,s}$ denote the chain complex $$A_{r,s}\colonequals\big(\bigoplus_{i=1}^{r}T(V^{i})_{*}\big)\bigoplus\big(\bigoplus_{i=1}^{s}\nch{\modc[S^{n-1}]{\bar{D}^{n}}}\big),$$ with $r+s=\srange$.
Given $I\subseteq [m]$, let $\cppunc{\bullet}{\partic}$ be the semi-simplicial chain complex
%\footnote{probably need to define iterated mapping cones as well} 
whose $p$-simplices are
$$\cppunc{p}{\partic}\colonequals \bigsqcup_{(m_{0}, \ldots m_{p})\in \text{OConf}^{\, \delta}_{p+1}(M\setminus \bar{K}_{\srange})} \tot{\cpiterxbarc{A^{I}_{*}}{\partial \bar{K}_{\srange}}{M\setminus K_{\srange}\setminus\{m_{0}, \ldots m_{p}\}}{\bullet}{\srange}{\fieldc}{\partic}}$$ and let $\cpunc{\bullet}$ denote the semi-simplicial chain complex $\bigsqcup_{\partic=0}^{\infty}\cppunc{\bullet}{\partic}$.
%\footnote{define $A_{r,s}$ here. Remove it from earlier definitions}
The inclusions $M\setminus \{ m_{0},\ldots , m_{p} \} \hookrightarrow M\setminus \{ m_{0},\ldots ,\widehat{m_j},\ldots , m_{p}\},$ for $j=0,\ldots , p$,  induce the face maps for both semi-simplicial chain complexes.
%Let $\cpunc{\bullet}$ denote the semi-simplicial chain complex $\bigsqcup_{\partic=0}^{\infty}\cppunc{\bullet}{\partic}$.
\end{definition}
Similar to the augmentation map for $\punc{\bullet}$, we have a map $$\epsilon'\colon \cpunc{0}\to \tot{\citerxbarc{A^{I}_{*}}{\partial \bar{K}_{\srange}}{M\setminus K_{\srange}}{\bullet}{\srange}{\fieldc}}$$ induced from the map which forgets the ordered configuration $m_{0}\in \text{OConf}^{\, \delta}_{1}(M\setminus \bar{K}_{\srange})$. As a consequence of $\epsilon'$ and the naturality of the Eilenberg--Zilber map, we have a composition of quasi-isomorphisms
%\footnote{Should I elaborate or is this claim obvious?}
$$\tot{\cpunc{\bullet}} \to \tot{\citerxbarc{A^{I}_{*}}{\partial \bar{K}_{\srange}}{M\setminus K_{\srange}}{\bullet}{\srange}{\fieldc}}\to \nch{\text{Conf}(M)}.$$
%By , we have the following diagram of quasi-isomorphisms
%\begin{center}
%    \begin{tikzcd}
%        \nch{\Vert\cpunc{\bullet}\Vert}
%        \arrow[r, "\Vert\epsilon'\Vert_{*}"]\arrow[d, "EZ"] &
%        \citerxbarc{A_{r,s}}{\partial \bar{K}_{\srange}}{M\setminus K_{\srange}}{\bullet}{\srange}{\fieldc}
        %C_{*}(\text{Conf}(M);\fieldc)&
%        \arrow[d, "EZ"]\\
%        \nch{\Vert \punc{\bullet}\Vert}\Vert
%        \arrow[r,"\Vert\epsilon\Vert_{*}"]& 
%        \nch{\Vert\iterxbarc{A_{r,s}}{\partial \bar{K}_{\srange}}{M\setminus K_{\srange}}{\bullet}{\srange}\Vert}\arrow[r]& \nch{\text{Conf}(M)}.
        %C_{*}(\text{Conf}(M);\fieldc)
%    \end{tikzcd}
%\end{center}
%As a consequence, we have a weak homotopy equivalence $\Vert \cpunc{\bullet} \Vert\to \nch{\text{Conf}(M)}$.
%\begin{definition}
%Let $\bar{K}_{\srange}$ and $K_{\srange}$ be as in \cref{notation for iterated configuration space model}. Let $CZ_{\partic,\bullet}(A_{r,s},M,K_{\srange})$ be the sub-semi-simplicial space of $\cpunc{\bullet}$ mapping to $C_{*}(\text{Conf}_{\partic}(M);\fieldc)$ under the weak homotopy equivalence $\Vert \cpunc{\bullet}\Vert\to C_{*}(\text{Conf}(M);\fieldc)$.
%\end{definition}
\begin{lemma}\label{stabilization map lif to puncture res}
%Given  $I \subseteq [\srange]$ with $[d]\subseteq I$,
%$\{j_{1},\ldots, j_{q}\}\subseteq J \subseteq \{1,\ldots,\srange\}$ 
Suppose that $(z_{1},\ldots, z_{d})$ is a $d$-tuple of classes in $H_{*}(\text{Conf}(\mathbb{R}^{n});\fieldc)$ such that $\browd{z_{b}}{e}=0$ for all $b=1,\ldots, d$. 
%Suppose that $\{j_{1},\ldots, j_{q}\}\subseteq J \subseteq \{1,\ldots,\srange\}$ and that $z_{j_{1}},\ldots, z_{j_{q}}$ are classes in $H_{*}(\text{Conf}(\mathbb{R}^{n});\fieldc)$ such that $\browd{z_{j_{i}}}{e}=0$ for all $i=1,\ldots, j$.
%Suppose that $z_{1},\ldots,z_{\srange}\in H_{*}(\text{Conf}(\mathbb{R}^{n});\fieldc)$ are classes such that $\browd{z_{j}}{e}=0$ for all $j$. Given $J\subset\{1,\ldots, \srange\}$ and $\{a_{1},\ldots ,a_{q}\}\subset I^{c}$, 
Given  $I \subseteq [\srange]$ with $[d]\subseteq I$, the stabilization maps 
$$\newstab{z_{b}} \colon \tot{\cpiterxbarc{A^{I}_{*}}{\partial \bar{K}_{\srange}}{M\setminus K_{\srange}}{\bullet}{\srange}{\fieldc}{\partic-\text{par}(z_{b})}}\to \tot{\cpiterxbarc{A^{I\setminus\{b\}}_{*}}{\partial \bar{K}_{\srange}}{M\setminus K_{\srange}}{\bullet}{\srange}{\fieldc}{\partic}}$$
%$\newstab{z_{j}} \colon F^{0}_{*}\to F^{1}_{*}$
%$$\newstab{z_{j}} \colon \Vert \citerxbarc{\bigsqcup_{i=1}^{\srange}Y_{i}}{\partial \bar{K}_{\srange}}{M\setminus K_{\srange}}{\bullet}{\srange}{\fieldc}\Vert\to \Vert \citerxbarc{(\bigsqcup_{i=1, i\neq j}^{\srange}Y_{i})\sqcup \bar{D}^{n}_{j}}{\partial \bar{K}_{\srange}}{M\setminus K_{\srange}}{\bullet}{\srange}{\fieldc}\Vert$$ 
from \cref{stabilization maps for closed iterated barc} extend to maps
$$(\newstab{z_{b}})_{\bullet}\colon  \cppunc{\bullet}{\partic-\text{par}(z_{b})} \to \cppunc[A^{I\setminus\{b\}}_{*}]{\bullet}{\partic}$$ that all commute with each other.
%Let $M$ be a manifold of dimension $n$ and let $\srange$ be a positive integer. Let $\bar{D}^{n}_{1},\ldots, \bar{D}^{n}_{\srange}$ be $\srange$ disjoint closed disks in $M$ and let $E_{\srange}\colonequals \bigsqcup_{j=1}^{\srange}\bar{D}^{n}_{j}$.
%Let $M$ and $\bar{E}_{\srange}\subset M$ be as in \cref{notation for iterated configuration space model}.
%Let $(Z_{\partic,\bullet}(M, \bar{E}_{\srange}),\text{Conf}_{\partic}(M),\epsilon )$ be the augmented semi-simplicial space from \cref{puncture res for conf}. 
%Given a stabilization map $s_{z} \colon  C_{*}(\text{Conf}_{\partic}(M);\fieldc)\to  C_{*+a}(\text{Conf}_{\partic +b}(M);\fieldc) $ induced from a class $z\in H_{a}(\text{Conf}_{b}(D^{n});\fieldc)$, with $D^{n}\subseteq \bar{D}_{j}^{n}$ for some $j$,
%there is a map $$s_{z, \bullet} \colon  C_{*}(Z_{\partic,\bullet}(M, \bar{E}_{\srange});\fieldc)\to  C_{*+a}(Z_{\partic +b,\bullet}(M, \bar{E}_{\srange});\fieldc) $$ of chain complexes of semi-simplicial spaces extending $s_{z}$ to $C_{*}(Z_{\partic,\bullet}(M, \bar{E}_{\srange});\fieldc)$.
\end{lemma}
\begin{proof}
Let $\epsilon'_{p} \colon \cpunc{p}\to \citerxbarc{A^{I}_{*}}{\partial \bar{K}_{\srange}}{M\setminus K_{\srange}}{p}{\srange}{\fieldc}$ denote the chain map coming from the map $\epsilon'$. We can represent a chain in $\cpunc{p}$ by a pair
$((m_{0},\ldots,m_{p}),\sigma)$ with $(m_{0},\ldots,m_{p})\in \text{OConf}^{\, \delta}_{p+1}(M\setminus K_{\srange})$ and $\sigma\in \citerxbarc{A^{I}_{*}}{\partial \bar{K}_{\srange}}{M\setminus K_{\srange}\setminus\{m_{0}, \ldots m_{p}\}}{\bullet}{\srange}{\fieldc}$.
%\begin{align*}
%    ((m_{0},\ldots,m_{p}),\sigma), &\text{ with } (m_{0},\ldots,m_{p})\in \text{OConf}^{\, \delta}_{p+1}(M\setminus K_{\srange})\\
%    &\text{ and } \sigma\in \citerxbarc{A^{I}_{*}}{\partial \bar{K}_{\srange}}{M\setminus K_{\srange}\setminus\{m_{0}, \ldots m_{p}\}}{\bullet}{\srange}{\fieldc}.
%\end{align*}
%$((m_{0},\ldots,m_{p}),\sigma)$ with $(m_{0},\ldots,m_{p})\in \text{OConf}^{\, \delta}_{p+1}(M\setminus K_{\srange})$ and $\sigma\in \Vert\citerxbarc{\bigoplus_{i=1}^{\srange}T(V^{i})_{*}}{\partial \bar{K}_{\srange}}{M\setminus K_{\srange}\setminus\{m_{0}, \ldots m_{p}\}}{\bullet}{\srange}{\fieldc}\Vert$.
%$$\{((m_{0},\ldots,m_{p}),\sigma)\vert (m_{0},\ldots,m_{p})\in \text{OConf}^{\, \delta}_{p+1}(M\setminus K_{\srange}), \sigma\in \Vert\citerxbarc{A_{r,s}}{\partial \bar{K}_{\srange}}{M\setminus K_{\srange}\setminus\{m_{0}, \ldots m_{p}\}}{\bullet}{\srange}{\fieldc}\Vert\}$$ with $(m_{0},\ldots,m_{p})\in \text{OConf}^{\, \delta}_{p+1}(M\setminus K_{\srange})$,  $\sigma\in \Vert\citerxbarc{A_{r,s}}{\partial \bar{K}_{\srange}}{M\setminus K_{\srange}\setminus\{m_{0}, \ldots m_{p}\}}{\bullet}{\srange}{\fieldc}\Vert$. 
Since the points $m_{0},\ldots , m_{p}$ are not in $\bar{K}_{\srange}$, we have that 
%the chain $\newstab{z_{j}}\circ \epsilon'_{p}(\sigma )$ 
%\in \Vert\citerxbarc{(\bigsqcup_{i=1, i\neq j}^{\srange}Y_{i})\bigsqcup \bar{D}^{n}_{j}}{\partial \bar{K}_{\srange}}{M\setminus K_{\srange}}{\bullet}{\srange}{\fieldc}\Vert
%can be treated as a chain in 
$$\newstab{z_{b}}\circ \epsilon'_{p}(\sigma )\in \citerxbarc{A^{I\setminus\{b\}}_{*}}{\partial \bar{K}_{\srange}}{M\setminus K_{\srange}\setminus \{m_{0},\ldots , m_{p}\}}{\bullet}{\srange}{\fieldc}.$$ Therefore, $$(\newstab{z_{b}})_{p} ((m_{0},\ldots , m_{p}),\sigma)\colonequals  (\newstab{z_{b}}\circ \epsilon_{p}'(\sigma ), (m_{0},\ldots , m_{p}))$$ is a chain in $\cpunc[A^{I\setminus\{b\}}_{*}]{p}$.
%$\citerxbarc{A^{I\setminus\{b\}}_{*}}{\partial \bar{K}_{\srange}}{M\setminus K_{\srange}}{p}{\srange}{\fieldc}$ and $$(\newstab{z_{b}})_{p}((m_{0},\ldots , m_{p}),\sigma)\colonequals  (\newstab{z_{b}}\circ (\epsilon_{p})_{*}(\sigma ), (m_{0},\ldots , m_{p}))$$ is a chain in $\cpunc[A^{I\setminus\{b\}}_{*}]{p}.$

To show that the maps $(\newstab{z_{b}})_{p}$ assemble to give a map $$(\newstab{z_{b}})_{\bullet} \colon \cpunc[A^{I}_{*}]{\bullet}\to \cpunc[A^{I\setminus\{b\}}_{*}]{\bullet}$$ of semi-simplicial chain complexes, we need to show that the following diagram

\begin{center}\begin{tikzcd}
\cpunc[A^{I}_{*}]{p} \arrow[d, "(\newstab{z_{b}})_{p}"]\arrow[r,"d_{q}"] & \cpunc[A^{I}_{*}]{p-1}\arrow[d, "(\newstab{z_{b}})_{p-1}"]\\
\cpunc[A^{I\setminus\{b\}}_{*}]{p}\arrow[r, "d_{q}"] & \cpunc[A^{I\setminus\{b\}}_{*}]{p-1}
\end{tikzcd}\end{center}
%\cpunc[\bigoplus_{i=1}^{\srange}T(V^{i})_{*}]{p-1}\arrow[d, "(\newstab{z_{j}})_{p-1}"]
%\begin{center}\begin{tikzcd}
% \cpunc[A^{I}_{*}]{p} \arrow[r, "(\newstab{z_{i}})_{p}"]\arrow[d,"d_{q}"] & \cpunc[A^{I\setminus\{b\}}_{*}]{p}\arrow[d, "d_{q}"]\\
% \cpunc[A^{I}_{*}]{p-1}\arrow[r, "(\newstab{z_{i}})_{p-1}"]& \cpunc[A^{I\setminus\{b\}}_{*}]{p-1}
%\end{tikzcd}\end{center}
commutes. Since the restriction of $d_{q}$ to $\citerxbarc{A^{I}_{*}}{\partial \bar{K}_{\srange}}{M\setminus K_{\srange}\setminus \{m_{0},\ldots , m_{p}\}}{\bullet}{\srange}{\fieldc}$ commutes with $(\newstab{z_{b}})_{\bullet}$ for all $m_{0},\ldots , m_{p}\in M\setminus \bar{K}_{\srange}$, i.e.
$d_{q}\circ (\newstab{z_{b}})_{p}= (\newstab{z_{b}})_{p-1}\circ d_{q}$, it follows that the diagram commutes as well. The maps $ (\newstab{z_{b}})_{\bullet}$ all commute with each other due to how $(\newstab{z_{b}})_{\bullet}$ was constructed and the stabilization maps $\newstab{z_{b}}$ all commuting with each other. 
\end{proof}
\begin{definition}\label{iterated mapping cone def for puncture resolution}
Suppose that $(z_{1},\ldots, z_{d})$ is a $d$-tuple of classes in $H_{*}(\text{Conf}(\mathbb{R}^{n});\fieldc)$ such that $\browd{z_{b}}{e}=0$ for all $b=1,\ldots, d$.
%Let $\srange$ be a positive integer. Let $M$ be a manifold of dimension $n$ and fix $\srange$ disjoint closed disks $\bar{D}^{n}_{1},\ldots, \bar{D}^{n}_{\srange}$ in $M$. 
%Let
%$\fieldc$ be a ring and let
%$e\in H_{0}(\text{Conf}_{1}(\mathbb{R}^{n});\fieldc)$ be the class of a point. 
%Suppose that $(z_{1},\ldots, z_{\srange})$ is an $\srange$-tuple of classes in $H_{*}(\text{Conf}(\mathbb{R}^{n});\fieldc)$ such that $\browd{z_{j}}{e}=0$ for all $j=1,\ldots, \srange$. Given $J\subset\{1,\ldots, \srange\}$ and $\{a_{1},\ldots, a_{q}\}\subset I^{c}$, 
%Suppose that $\{j_{1},\ldots, j_{q}\}\subseteq J \subseteq \{1,\ldots,\srange\}$ and that $(z_{1},\ldots, z_{q})$ is a $a$-tuple of classes in $H_{*}(\text{Conf}(\mathbb{R}^{n});\fieldc)$ such that $\browd{z_{i}}{e}=0$ for all $i=1,\ldots, j$.
Given  $I \subseteq [\srange]$ with $[d]\subseteq I$, let $\lnewstab{z_{b}}$ denote the induced map $\tot{(\newstab{z_{b}})_{\bullet}}$ obtained by passing to the totalization of $\cpunc{\bullet}$.
%\footnote{move this notation to the previous lemma}
%$(\srange-1)$-tuple $(\Vert(\newstab{z_{1}})_{\bullet}\Vert,\ldots,\Vert(\newstab{z_{j}})_{\bullet}\Vert)$. 
%$(\Vert\newstab{z_{i}})_{\bullet}\Vert)_{j}$ denote the $j$-tuple $(\Vert(\newstab{z_{1}})_{\bullet}\Vert,\ldots,\Vert(\newstab{z_{j}})_{\bullet}\Vert)$.
Recursively define the \textbf{d-th iterated mapping cone} $\jpariteratedmap{M}{\lnewstab{z_{i}}}{d}{\partic}{I}$ associated to the $d$-tuple $(z_{1},\ldots, z_{d})$ and the subset $I$ as follows:  
%by setting 
set $\jpariteratedmap{M}{\lnewstab{z_{1}}}{1}{\partic}{I}$ to be the mapping cone
$$\text{Cone}\big(\lnewstab{z_{1}}\colon \tot{\cppunc{\bullet}{\partic-\text{par}(z_{1})}}\to \tot{\cppunc[A^{I\setminus\{1\}}_{*}]{\bullet}{\partic}}\big)$$
and for $d>1$, set $\jpariteratedmap{M}{\lnewstab{z_{i}}}{d}{\partic}{I}$ to be the induced mapping cone
$$\text{Cone}\big(\lnewstab{z_{d}}\colon \jpariteratedmap{M}{\lnewstab{z_{1}},\ldots, \lnewstab{z_{d-1}}}{d-1}{\partic-\text{par}(z_{d})}{I}\to \jpariteratedmap{M}{\lnewstab{z_{1}},\ldots, \lnewstab{z_{d-1}}}{d-1}{\partic}{I\setminus\{d\}}\big).$$
\end{definition}
The following lemma lists some homology classes $z\in H_{*}(\text{Conf}(\mathbb{R}^{n});\fieldc)$ for which $\browd{z}{e}=0$.
%\footnote{I should move this lemma to some other place}
%We conclude this subsection by explicitly showing that there exist homology classes $z\in H_{*}(\text{Conf}(\mathbb{R}^{n}))$ with $\browd{z}{e}=0$ and comparing this new stabilization map
\begin{lemma}\label{homology classes which give a stabilization map}
Let $e\in H_{0}(\text{Conf}_{1}(\mathbb{R}^{n});\fieldc)$ be the class of a point.
\hfill\begin{enumerate}
    \item\label{class is pth power} If $\fieldc$ is $\mathbb{F}_{p}$ and $\omega \in H_{*}(\text{Conf}(\mathbb{R}^{n});\mathbb{F}_{p})$ is a homology class such that $\omega =z^{p}$ for some class $z$, then $\browd{\omega}{e}=0$. It follows that $\browd{e^{p}}{e}=0$ for all $n$.
    \item\label{Browder bracket applied to Dyer-Lashof is zero} If $\fieldc$ is $\mathbb{F}_{p}$ and $\omega \in H_{*}(\text{Conf}(\mathbb{R}^{n});\mathbb{F}_{p})$, then $\browd{Q^{s}\omega}{e}=0$.
    %of $e$ with the Dyer-Lashof operation $Q^{s}$ applied to $\omega$ is zero.
    \item\label{class of a point} If  $\fieldc$ is $\mathbb{F}_{2}$ or $n$ is odd, then $\browd{e}{e}=0$.
    %If  , then $\browd{e}{e}=0$ for all $n$.
    \item\label{classes in dim two} 
    %Suppose that 
    If $\fieldc$ is $\mathbb{F}_{p}$ and $x_{i},y_{i},z_{i}\in H_{*}(\text{Conf}(\mathbb{R}^{2});\mathbb{F}_{p})$ are the homology classes from \cref{thm-Co}, then $\browd{x_{i}}{e}=\browd{y_{i}}{e}=\browd{z_{i}}{e}=0$.
    %$\browd{z_{i}}{e}=0$ for $i\geq 0$ and $\browd{x_{i}}{e}=\browd{y_{i}}{e}=0$ for $i>0$.
    %For $i\geq 0$, we have that $\browd{x_{i}}{e}=0$ and for  If $\omega =y_{i}$ or $z_{i}$, then $\browd{\omega}{e}.$
\end{enumerate}
%We have that $\browd{e}{x_{0}}=0$ and for $i$ positive, $\browd{e}{x_{i}}=\browd{e}{y_{i}}=0$. 
%If $z=y_{i}$ or $z=e^{p}$, then $\browd{z}{e}=0$.
\end{lemma}
\begin{proof}
Parts \ref{class is pth power} and  \ref{class of a point} appear in \textcite[Proposition 2.5]{MR3556286} and \textcite[p.~16]{MR3032101}, respectively. Parts \ref{class is pth power} and \ref{Browder bracket applied to Dyer-Lashof is zero} are immediate consequences of the derivation property of the Browder bracket (Property \ref{derivation property of Browder bracket}) and the relationship of a Dyer-Lashof operation with a Browder bracket (Property \ref{Browder bracket with Dyer-Lashof}), respectively.
%of the Browder bracket  . Part \ref{Browder bracket applied to Dyer-Lashof is zero} is an i 
Part \ref{class of a point} follows
%since $\browd{e}{e}$ is a class in $H_{n-1}(\text{Conf}_{2}(\mathbb{R});\fieldc)$. 
from the geometric description of the Browder bracket. The class $\browd{e}{e}$ corresponds to twice a generator of $H_{n-1}(\mathbb{R}P^{n-1};\fieldc)$, so it vanishes when $n$ is odd or $\fieldc=\mathbb{F}_{2}$.
%due to the definition of the Browder bracket and  $\text{Conf}_{2}(\mathbb{R}^{n})$ being homotopy equivlaent to $\mathbb{R}P^{n-1}$, the Browder bracket $\browd{e}{e}$ corresponds to a generator of $H_{n-1}(\mathbb{R}P^{n-1};\fieldc)$, which is zero when $n$ is odd or $\fieldc=\mathbb{F}_{2}$.
%$\browd{e}{e}\in H_{n-1}(\text{Conf}_{2}(\mathbb{R}^{n});\fieldc)\cong H_{n-1}(\mathbb{R}P^{n};\fieldc)$ so $\browd{e}{e}=0 $ when $\fieldc=\mathbb{F}_{2}$ because $\browd{x}{x}=0$ for all homology classes in this case. When $p$ and $n$ are odd, $\browd{e}{e}=0$ by Property \ref{Browder bracket symmetric up to sign} of the Browder bracket. 
%Due to $\text{Conf}_{2}(\mathbb{R}^{n})$ being homotopy equivalent to $\mathbb{R}P^{n}$ and the definition of the Browder bracket, $\browd{e}{e}$ is a generator of $H_{n-1}(\mathbb{R}P^{n};\mathbb{F}_{p})$ and is nonzero in the case $p$ is odd and $n$ is even.
%Since the Browder bracket is a derivation of the product in each variable by\footnote{I need to cref the derivation property for the Browder bracket}, if $z=e^{p}$, then $\browd{z}{e}=p\browd{e}{e}=0$.

For Part \ref{classes in dim two}, we first show that $\browd{x_{i}}{e}$ vanishes. Recall that $x_{i}$ is the class $\xi \cdots \xi e\in H_{*}(\text{Conf}(\mathbb{R}^{2});\mathbb{F}_{2})$. Since $\browd{x}{x}=0$ for all $x\in H_{*}(\text{Conf}(\mathbb{R}^{2});\mathbb{F}_{2})$ (Property \ref{Jacobi identity of Browder bracket} of the Browder bracket), we have that $\browd{e}{e}=0$. Since $$\browd{e}{x_{1}}=\browd{e}{\xi e}=\browd{e}{\browd{e}{e}}$$ (by Property \ref{iterated Browder bracket} of the Browder bracket) and $\browd{e}{e}=0$, we have that $\browd{x_{1}}{e}=0$. Since $x_{i}=\xi x_{i-1}$, the same argument recursively shows that $\browd{x_{i}}{e}=0$.
%For $i>0$, since 
%$$\browd{e}{x_{i}}=\browd{e}{
%\xi x_{i-1}}=\browd{e}{\browd{x_{i-1}}{e}} $$
%(by Property \ref{iterated Browder bracket} of the Browder bracket) and $\browd{x_{i-1}}{e}=0$ by recursion, we have that $\browd{x_{i}}{e}=0$.

We now deal with the case that $p$ is odd. we first show that $\browd{z_{i}}{e}=0$ for all $i$. Recall that $z_{i}=\xi\cdots\xi \browd{e}{e}$ and $y_{i}=\beta \xi\cdots\xi \browd{e}{e}$.
By the Jacobi identity of the Browder bracket (Property \ref{Jacobi identity of Browder bracket} of the Browder bracket), $\browd{e}{\browd{e}{e}}=\browd{e}{z_{0}}=0$.
%For $i>0$, we show that $\browd{z_{j}}{e}$
By the relation of the top homology operation $\xi$ and the Browder bracket and the Jacobi identity (Properties \ref{iterated Browder bracket} and \ref{Jacobi identity of Browder bracket} of the Browder bracket), $\browd{e}{z_{i}}$ can be recursively expressed as 
\begin{align*}
    \browd{e}{z_{i}}&=\browd{e}{\xi z_{i-1}}\\
    &=\text{ad}^{p}(z_{i-1})(e) \text{ by Property }\ref{iterated Browder bracket}\text{ of the Browder bracket}\\
    &=\browd{z_{i-1}}{\browd{z_{i-1}}{\cdots ,\browd{z_{i-1}}{\browd{z_{i-1}}{e}}\cdots}}\text{ by definition of }\text{ad}^{p}.
\end{align*}
Since $\browd{z_{0}}{e}=0$, we have that $\browd{z_{i-1}}{e}=0$ by recursion 
%\footnote{Should I say induction instead of recursion?}
and so $\browd{z_{i}}{e}=0$.

We now show that $\browd{y_{i}}{e}=0$ for $i>0$. First, we show that $\browd{y_{1}}{e}=0$. Recall the homology operation $\zeta (-)\colonequals \beta\xi (-)-\text{ad}^{p-1}(-)(\beta (-))$ and that $\browd{x}{\zeta y}=0$ by Property \ref{iterated Browder bracket} of the Browder bracket. Since  $y_{i}=\beta\xi z_{i-1}$ and the Browder bracket is linear in both entries, we have that
\begin{align*}
   0&=\browd{e}{ \zeta z_{i-1}}= \browd{e}{\beta\xi (z_{i-1})-\text{ad}^{p-1}(z_{i-1})(\beta (z_{i-1}))}\\
   &=\browd{e}{y_{i}-\text{ad}^{p-1}(z_{i-1})(\beta (z_{i-1}))}\\
   &=\browd{e}{y_{i}}-\browd{e}{\text{ad}^{p-1}(z_{i-1})(\beta (z_{i-1}))}.
\end{align*}
Therefore, it suffices to show that $\browd{e}{\text{ad}^{p-1}(z_{i-1})(\beta (z_{i-1}))}=0$. For $i=1$, we have that $\beta (z_{i-1})=\beta (z_{0})=\beta (\browd{e}{e})$. By Property \ref{Bockstein applied to Browder bracket} of the Browder bracket, $$\beta \browd{e}{e}=\browd{\beta e}{e}\pm \browd{e}{\beta e}. $$
We have that $\beta e=0$ since $e$ is a homology class of degree $0$. Therefore, $\beta z_{0} =0$ and so $\browd{e}{y_{1}}=0$.

Suppose now that $i>1$. We have that $\beta z_{i-1} =y_{i-1}$ since $i>1$. Note that 
\begin{align*}
    \browd{e}{\text{ad}^{p-1}(z_{i-1})(y_{i-1})}&=\browd{e}{\text{ad}(z_{i-1})(\text{ad}^{p-2}(z_{i-1})(y_{i-1}))}\\
    &=\browd{e}{\browd{z_{i-1}}{\text{ad}^{p-2}(z_{i-1})(y_{i-1})}}
\end{align*}
By the Jacobi identity (Property \ref{Jacobi identity of Browder bracket} of the Browder bracket), we have that
\begin{align*}
   0&=\pm \browd{e}{\browd{z_{i-1}}{\text{ad}^{p-2}(z_{i-1})(y_{i-1})}}\pm \browd{z_{i-1}}{\browd{e}{\text{ad}^{p-2}(z_{i-1})(y_{i-1})}}\\
   &\pm  \browd{\text{ad}^{p-2}(z_{i-1})(y_{i-1})}{\browd{z_{i-1}}{e}}. 
\end{align*}

Since $\browd{z_{i-1}}{e}=0$, it suffices to show that $\browd{z_{i-1}}{\browd{e}{\text{ad}^{p-2}(z_{i-1})(y_{i-1})}}=0$. By iterating the Jacobi identity, it suffices to show that 
$$\browd{e}{\text{ad}(z_{i-1})(y_{i-1})}=\browd{e}{\browd{z_{i-1}}{y_{i-1}}}=0.$$ But this follows from an another application of the Jacobi identity
$$0=\pm \browd{e}{\browd{z_{i-1}}{y_{i-1}}}\pm \browd{z_{i-1}}{\browd{e}{y_{i-1}}}\pm \browd{y_{i-1}}{\browd{z_{i-1}}{e}} $$
and the fact that $\browd{e}{z_{i-1}}$ and $\browd{e}{y_{i-1}}$ are both zero ($\browd{e}{y_{i-1}}$ is zero by recursion).
\begin{comment}
By induction, suppose that $\browd{y_{j}}{e}=0$ for all $j$ less than $i$.

Recall that $z_{i}=\xi\cdots\xi \browd{e}{e}$ and $y_{i}=\beta \xi\cdots\xi \browd{e}{e}$. 

Recall that $y_{i}=\beta\xi_{1}\cdots\xi_{1}\browd{y_{0}}{y_{0}}$, where $\xi_{1}$ is applied $p$ times. 

Since the Browder bracket $\browd{x}{\beta Q^{s}y}$ vanishes for all $x\in H_{*}(\text{Conf}(\mathbb{R}^{n})$ and $y\in H_{i}(\text{Conf}(\mathbb{R}^{n})$ with $i+n$ odd, $\browd{y_{0}}{y_{i}}=0$ for all $i\geq 1$.\footnote{This argument doesn't work because I can't assume $\browd{x}{\beta Q^{(\text{deg}(y)+n-1)/2}y} $ is zero.} 

Let $z_{i}\colonequals \xi_{1} \cdots \xi_{1} \browd{y_{0}}{y_{0}}$, where $\xi_{1}$ is applied $p-1$ times. Note that $z_{i}$ is a homology class of odd degree, so $\zeta (z_{i})=\beta\xi_{1}(z_{i})-\text{ad}^{p-1}(z_{i})(\beta z_{i})=y_{i}-\text{ad}^{p-1}(z_{i})(\beta z_{i})$ is defined. By Property $4$ of the Browder bracket, we have that $\browd{y_{0}}{\zeta(z_{i})}=0$, so it suffices to show that $\browd{y_{0}}{\text{ad}^{p-1}(z_{i})(\beta z_{i})}=0$.

Following Cohen--Lada--May's notation on page 217 (page 114 of the pdf), let $Z(x)=B\Psi(x)- ad^{p-1}(x)(Bx)$. By property 4 on page 218, lambda (e, Z(Lambda(e,e) )=0, so it suffices to show 
$\Lambda (e, ad^{p-1}(Lambda(e,e) )(BLambda (e,e) )=0$.
By property 7 on page 216, $B\Lambda(e,e)=0$ since e is a class in degree 0.

\end{comment}
\end{proof}
\begin{proposition}\label{stable hom for a closed manifold}
Suppose that $M$ is a closed connected surface. Let $p$ be an odd prime and let $D(p, \srange , \partic)$ denote the constant from \cref{stable ranges for the plane}. The stabilization map
$$\newstab{e^{p}} \colon \tot{\cpiterxbarc{T(V)_{*}}{S^{n-1}}{M\setminus D^{n}}{\bullet}{1}{\mathbb{F}_{p}}{\partic}}\to \tot{\cpiterxbarc{\modc[S^{n-1}]{\bar{D}^{n}}}{S^{n-1}}{M\setminus D^{n}}{\bullet}{1}{\mathbb{F}_{p}}{\partic+p}}$$
from \cref{stabilization maps for closed iterated barc} induces a map
$\newstab{e^{p}} \colon H_{i}(\text{Conf}_{\partic}(M);\mathbb{F}_{p})\to  H_{i}(\text{Conf}_{\partic +p}(M);\mathbb{F}_{p})$
%and
%$$\newstab{e^{p}} \colon H_{*}(\Vert\cpiterxbarc{Y}{S^{n-1}}{M\setminus D^{n}}{\bullet}{1}{\mathbb{F}_{p}}{\partic}\Vert)\to H_{*}(\Vert \cpiterxbarc{\bar{D}^{n}}{S^{n-1}}{M\setminus D^{n}}{\bullet}{1}{\mathbb{F}_{p}}{\partic+p}\Vert)$$
that is an isomorphism for $i<D(p,1,\partic )$ and a surjection for $i=D(p,1,\partic )$.
\end{proposition}
\begin{proof}
%We prove the result for the case the prime $p$ is odd (the case $p=2$ follows from \textcite[Theorem 9.1]{MR3032101}). 
Fix a closed disk $\bar{D}^{n}$ in $M$. 
Showing that
%the map 
$\newstab{e^{p}} \colon  H_{i}(\text{Conf}_{\partic}(M);\mathbb{F}_{p})\to  H_{i}(\text{Conf}_{\partic +p}(M);\mathbb{F}_{p})$ is an isomorphism for $i<D(p,1,\partic )$ and a surjection for $i=D(p,1,\partic )$
is equivalent to showing that 
$H_{i}(\jpariteratedmap{M}{\newstab{e^{p}}}{1}{\partic}{[1]};\mathbb{F}_{p})$ is zero for $i\leq D(p,1,\partic )$.
%is equivalent to showing that $$H_{i}(\pariteratedmap{M}{\newstab{e^{p}}}{1}{\partic};\mathbb{F}_{p})=H_{i}( \tot{\cpiterxbarc{\bar{D}^{n}}{S^{n-1}}{M\setminus D^{n}}{\bullet}{1}{\mathbb{F}_{p}}{\partic+p}},\tot{\cpiterxbarc{T(V)_{*}}{S^{n-1}}{M\setminus D^{n}}{\bullet}{1}{\mathbb{F}_{p}}{\partic}};\mathbb{F}_{p})$$ is zero for $i\leq D(p,1,\partic )$. 
%For $A_{1}=T(V)_{*}$ or $C_{*}(\modc[S^{n-1}]{\bar{D}^{n}};\mathbb{F}_{p})$, let $CZ_{\partic,\bullet}(M,A_{1})$ denote $CZ_{\partic,\bullet}(A_{1},M,K_{1};\mathbb{F}_{p})$.

By \cref{stabilization map lif to puncture res}, the stabilization map $\newstab{e^{p}}$ extends to a map of semi-simplicial chain complexes $$(\newstab{e^{p}})_{\bullet} \colon CZ^{1}_{\partic,\bullet}(T(V)_{*},M; \mathbb{F}_{p})\to  CZ^{1}_{\partic+p,\bullet}(C_{*}(\modc[S^{n-1}]{\bar{D}^{n}};\mathbb{F}_{p}),M; \mathbb{F}_{p}).$$ 
We
%\footnote{should figure out where I'm using punctured chain resolution stuff} 
can apply a relative version of the geometric realization spectral sequence:
%to $CZ_{\partic,\bullet}(C_{*}(\modc[S^{n-1}]{\bar{D}^{n}};\mathbb{F}_{p}), M;\mathbb{F}_{p})$:
\begin{align*}
    E^{1}_{s,t}&=H_{t}\big( CZ^{1}_{\partic+p,s}(C_{*}(\modc[S^{n-1}]{\bar{D}^{n}};\mathbb{F}_{p}),M;\mathbb{F}_{p} ),CZ^{1}_{\partic ,s}(T(V)_{*}, M;\mathbb{F}_{p})\big)\\
    &\Longrightarrow H_{s+t}( \jpariteratedmap{M}{\lnewstab{e^{p}}}{1}{\partic+p}{[1]};\mathbb{F}_{p}).
\end{align*}
%$$E^{1}_{s,t}=H_{t}( CZ_{\partic+p,s}(M,C_{*}(\modc[S^{n-1}]{\bar{D}^{n}};\mathbb{F}_{p})),CZ_{\partic ,s}(M,T(V)_{*});\mathbb{F}_{p})\Longrightarrow H_{s+t}( \pariteratedmap{M}{\lnewstab{e^{p}}}{1}{\partic+p};\mathbb{F}_{p}).$$
%By \cref{puncture res is equivalent to conf}, 
We have that $\tot{CZ^{1}_{\partic,\bullet}(A^{I}_{*},M;\mathbb{F}_{p})}$ is quasi-isomorphic to $\tot{\cpiterxbarc{A^{I}_{*}}{S^{n-1}}{M\setminus D^{n}}{\bullet}{1}{\mathbb{F}_{p}}{\partic}}$. Therefore, to show that $H_{i}( \jpariteratedmap{M}{\lnewstab{e^{p}}}{1}{\partic+p}{[1]})$, which is isomorphic to
$$ H_{i}\big(\tot{\cpiterxbarc{\bar{D}^{n}}{S^{n-1}}{M\setminus D^{n}}{\bullet}{1}{\mathbb{F}_{p}}{\partic+p}},\tot{\cpiterxbarc{T(V)_{*}}{S^{n-1}}{M\setminus D^{n}}{\bullet}{1}{\mathbb{F}_{p}}{\partic}}\big),$$
%$H_{i}(\text{Conf}_{\partic+p}(M) ,\text{Conf}_{\partic}(M);\mathbb{F}_{p})$ 
vanishes for $i\leq D(p,1,\partic )$, it suffices to show that the relative homology $$H_{t}\big(CZ^{1}_{\partic+p,s}(C_{*}(\modc[S^{n-1}]{\bar{D}^{n}};\mathbb{F}_{p}), M;\mathbb{F}_{p}),CZ^{1}_{\partic ,s}(T(V)_{*}, M;\mathbb{F}_{p})\big)$$ vanishes for $s+t=i\leq D(p,1,\partic )$. Since the manifolds $M\setminus \bar{D}^{n}\setminus \{m_{0},\ldots, m_{s}\}$ appearing in $$\tot{\cpiterxbarc{A^{I}_{*}}{S^{n-1}}{M\setminus D^{n}\setminus\{m_{0}, \ldots m_{s}\}}{\bullet}{1}{\mathbb{F}_{p}}{\partic}}\subset CZ^{1}_{\partic, s}(A^{I}_{*},M; \mathbb{F}_{p})$$ are open for all nonnegative $s$, by \cref{new stabilization on open manifold}, \cref{new stabilization induces iso on open manifold}, and \cref{homological stability in dimension 2}, $E^{1}_{s,t}$ vanishes for $s\leq D(p,1,\partic )$ and hence also for $s+t\leq D(p,1,\partic )$.
\end{proof}
%By the argument of Proposition, to obtain a stability result for the configuration spaces of compact manifolds, it suffices to prove the result for the configuration spaces of non-compact manifolds. Before listing such stability
By applying the argument of \cref{stable hom for a closed manifold}, we can obtain additional stability results for the configuration spaces of a closed manifold and prove \cref{closed periodic secondary stability} in the case that homology is with $\mathbb{F}_{p}$ coefficients. 
%\cref{closed periodic secondary stability} is the case $\srange=1$ of 
%The following proposition, in the case $\srange=1$, proves part of \cref{closed periodic secondary stability}.
%and prove \cref{closed periodic secondary stability} (\cref{closed periodic secondary stability}
\begin{proposition}\label{stability results for closed manifold}
Suppose that $M$ is a connected manifold of dimension $n\geq 2$.
%Let $\fieldc$ be a ring and 
%Let $\srange$ be a positive number 
%Let $p$ be a prime number 
Let $D(p, \srange , \partic)$ denote the constant from \cref{stable ranges for the plane}. Let $e\in H_{0}(\text{Conf}_{1}(\mathbb{R}^{n}))$ be the class of a point.
\hfill\begin{enumerate}
    \item\label{higher dim and p is two} 
    %Suppose $p$ is $2$.
    The map
    $$ \newstab{e} \colon  H_{i}(\text{Conf}_{\partic}(M);\mathbb{F}_{2})\to  H_{i}(\text{Conf}_{\partic +1}(M);\mathbb{F}_{2}) $$ is an isomorphism for $i < k/2$ and a surjection for $i=k/2$.
    \item\label{closed stability for surface F2 coefficients} Suppose that $n=2$.
    %and that $\fieldc= \mathbb{F}_{2}$. 
    Let $x_{j}\in H_{*}(\text{Conf}(\mathbb{R}^{2});\mathbb{F}_{2})$ denote the classes from \cref{thm-Co}. Let $\jpariteratedmap{M}{\newstab{\displaystyle x_{j}}}{\srange}{\partic}{I}$ denote the \srange-th iterated mapping cone associated to the $\srange$-tuple 
    \\$(\newstab{\displaystyle e^{p}},  \newstab{\displaystyle x_{1}}, \ldots, \newstab{\displaystyle x_{\srange-1}})$ and the subset $I\subseteq [\srange+1]$, with $I=[\srange]$ or $[\srange+1]$.
    The stabilization map
$$\newstab{\displaystyle x_{\srange}} \colon \tot{\cpiterxbarc{A^{[\srange+1]}_{*}}{S^{n-1}}{M\setminus D^{n}}{\bullet}{\srange+1}{\mathbb{F}_{p}}{\partic}}\to \tot{\cpiterxbarc{A^{[\srange]}_{*}}{S^{n-1}}{M\setminus D^{n}}{\bullet}{\srange+1}{\mathbb{F}_{p}}{\partic+2^{\srange}}}$$
induces a map
    %as defined in \cref{iterated mapping cone def for a general manifold}.
%    
    %Let $y_{j}\in H_{*}(\text{Conf}(\mathbb{R}^{2});\mathbb{F}_{p})$ denote the classes from \cref{thm-Co}.
    %Let $J_{\srange , c}\colonequals \text{im}(\newstab{e^{p}})\bigcup_{j=1}^{\srange} \text{im}(\newstab{y_{j}})$ denote the subcomplex of $C_{*}(\text{Conf}(M);\mathbb{F}_{p})$ generated by the images of $\newstab{y_{0}^{p}},\newstab{y_{1}},\ldots ,\newstab{y_{\srange -1}}$ and let $J_{\srange , c}(\partic )\subset J_{\srange , c}$ denote the subcomplex of chains in $C_{*}(\text{Conf}_{\partic}(M);\mathbb{F}_{p})$.
%    The map
    $$\newstab{\displaystyle x_{m}} \colon  H_{i}\big (\jpariteratedmap{M}{\newstab{\displaystyle x_{j}}}{\srange}{\partic}{[\srange+1]};\mathbb{F}_{2}\big )\to  H_{i+ 2^{\srange}-1}\big (\jpariteratedmap{M}{\newstab{\displaystyle x_{j}}}{\srange}{\partic+2^{\srange}}{[\srange]};\mathbb{F}_{2}\big )$$ that is an isomorphism for $i<D(2,\srange+1 , \partic)$ and a surjection for $i= D(2, \srange+1, \partic)$.
    %Let $I_{\srange , c}\colonequals \bigcup_{j=0}^{\srange} \text{im}(\newstab{x_{j}})\subset C_{*}(\text{Conf}(M);\mathbb{F}_{2})$ denote the subcomplex generated by the images of $\newstab{x_{0}},\ldots , \newstab{x_{\srange -1}}$ and let $I_{\srange , c}(\partic )\subset I_{\srange , c}$ denote the subcomplex of chains in $C_{*}(\text{Conf}_{\partic}(M);\mathbb{F}_{2})$. The map
    %$$\newstab{x_{m}} \colon  H_{i}\big (\text{Conf}_{\partic}(M), I_{\srange-1, c}(\partic )\big )\to  H_{i+ 2^{\srange}-1}\big (\text{Conf}_{\partic + 2^{\srange}}(M), I_{\srange-1 , c}(\partic +2^{\srange})\big )$$ is an isomorphism for $i<D(2,\srange+1 , \partic)$ and a surjection for $i= D(2,\srange+1 , \partic)$.
    \item\label{closed stability for surface odd coefficients} Suppose that $n$ is $2$ and that $p$ is an odd prime.
    %and that $\fieldc =\mathbb{F}_{p}$ with $p>2$. 
    Let $y_{j}\in H_{*}(\text{Conf}(\mathbb{R}^{2});\mathbb{F}_{p})$ denote the classes from \cref{thm-Co}.
    Let $\jpariteratedmap{M}{\newstab{\displaystyle y_{j}}}{\srange}{\partic}{I}$ denote the \srange-th iterated mapping cone associated to the $\srange$-tuple $(\newstab{\displaystyle e^{p}},  \newstab{\displaystyle y_{1}}, \ldots, \newstab{\displaystyle y_{\srange-1}})$ and the subset $I\subseteq [\srange+1]$, with $I=[\srange]$ or $[\srange+1]$.
    
    %Let $y_{j}\in H_{*}(\text{Conf}(\mathbb{R}^{2});\mathbb{F}_{p})$ denote the classes from \cref{thm-Co}.
    %Let $J_{\srange , c}\colonequals \text{im}(\newstab{e^{p}})\bigcup_{j=1}^{\srange} \text{im}(\newstab{y_{j}})$ denote the subcomplex of $C_{*}(\text{Conf}(M);\mathbb{F}_{p})$ generated by the images of $\newstab{y_{0}^{p}},\newstab{y_{1}},\ldots ,\newstab{y_{\srange -1}}$ and let $J_{\srange , c}(\partic )\subset J_{\srange , c}$ denote the subcomplex of chains in $C_{*}(\text{Conf}_{\partic}(M);\mathbb{F}_{p})$.
    The map
    $$\newstab{\displaystyle y_{m}} \colon  H_{i}\big (\jpariteratedmap{M}{\newstab{\displaystyle y_{j}}}{\srange}{\partic}{[\srange+1]};\mathbb{F}_{p}\big )\to  H_{i+ 2p^{\srange}-2}\big (\jpariteratedmap{M}{\newstab{\displaystyle y_{j}}}{\srange}{\partic+2p^{\srange}}{[\srange]};\mathbb{F}_{p}\big )$$ is an isomorphism for $i<D(p,\srange+1 , \partic)$ and a surjection for $i= D(p, \srange+1, \partic)$.
    \item\label{odd dim and Z} Suppose that $n$ is odd. The map
    $$ \newstab{e} \colon  H_{i}(\text{Conf}_{\partic}(M);\mathbb{Z})\to  H_{i}(\text{Conf}_{\partic +1}(M);\mathbb{Z}) $$ is an isomorphism for $i < k/2$ and a surjection for $i=k/2$.
    \item\label{odd dim and p is odd} Suppose that $n$ is odd. Suppose that $\fieldc$ is $\mathbb{Q}$ or $\mathbb{F}_{p}$ with $p>2$. The map
    $$ \newstab{e} \colon  H_{i}(\text{Conf}_{\partic}(M);\fieldc)\to  H_{i}(\text{Conf}_{\partic +1}(M);\fieldc) $$ is an isomorphism for $i < k$ and a surjection for $i=k$.
    \item\label{higher stabilization in odd dimensions and p is 2}
    Suppose that $n$ is greater than $2$.
    %and $\fieldc=\mathbb{F}_{2}$. 
    Let $\omega_{j}\in H_{*}(\text{Conf}(\mathbb{R}^{n});\mathbb{F}_{2})$ denote the classes from \cref{stable ranges for higher dimensions with F2 coefficients}. 
    %\footnote{might need to change this cref if i put the classes in a definition.}
    Let $\jpariteratedmap{M}{\newstab{\displaystyle \omega_{j}}}{\srange}{\partic}{I}$ denote the \srange-th iterated mapping cone associated to the $\srange$-tuple \\$(\newstab{\displaystyle e},  \newstab{\displaystyle \omega_{1}}, \ldots, \newstab{\displaystyle \omega_{\srange-1}})$ and the subset $I\subseteq [\srange+1]$, with $I=[\srange]$ or $[\srange+1]$.
    %$[\srange]\subset [\srange+1]$.
    %, with $J=\{1,\ldots,\srange+1\}$ or $\{1,\ldots,\srange+1\}$.
    %Let $y_{j}\in H_{*}(\text{Conf}(\mathbb{R}^{2});\mathbb{F}_{p})$ denote the classes from \cref{thm-Co}.
    %Let $J_{\srange , c}\colonequals \text{im}(\newstab{e^{p}})\bigcup_{j=1}^{\srange} \text{im}(\newstab{y_{j}})$ denote the subcomplex of $C_{*}(\text{Conf}(M);\mathbb{F}_{p})$ generated by the images of $\newstab{y_{0}^{p}},\newstab{y_{1}},\ldots ,\newstab{y_{\srange -1}}$ and let $J_{\srange , c}(\partic )\subset J_{\srange , c}$ denote the subcomplex of chains in $C_{*}(\text{Conf}_{\partic}(M);\mathbb{F}_{p})$.
    
    The map
    $$\newstab{\displaystyle \omega_{m}} \colon  H_{i}\big (\jpariteratedmap{M}{\newstab{\displaystyle \omega_{j}}}{\srange}{\partic}{[\srange+1]};\mathbb{F}_{2}\big )\to  H_{i+2^{\srange}-1}\big (\jpariteratedmap{M}{\newstab{\displaystyle \omega_{j}}}{\srange}{\partic+2^{\srange}}{[\srange]};\mathbb{F}_{2}\big )$$ is an isomorphism for $i<D(2,\srange+1 , \partic)$ and a surjection for $i= D(2, \srange+1, \partic)$.
    %Let $\omega_{j}\in H_{*}(\text{Conf}(\mathbb{R}^{n});\mathbb{F}_{2})$ denote the homology classes defined in \cref{stable ranges for higher dimensions with F2 coefficients}. Let $K_{\srange, c}\colonequals \bigcup_{j=0}^{\srange}\text{im}(\newstab{\omega_{j}})\subset C_{*}(\text{Conf}(M);\mathbb{F}_{2})$ denote the subcomplex generated by the images of $\newstab{\omega_{0}},\ldots , \newstab{\omega_{\srange}}.$ Let $K_{\srange, c}(\partic )$ denote the subcomplex of $K_{\srange, c}$ of chains in $C_{*}(\text{Conf}_{\partic}(M);\mathbb{F}_{2})$. The map of relative homology groups 
    %$$ \newstab{\omega_{\srange}}\colon H_{i}\big (\text{Conf}_{\partic }(M),K_{\srange -1, c}(\partic)\big )\to  H_{i+2^{\srange}-1}\big (\text{Conf}_{\partic +2^{\srange}}(M),K_{\srange-1, c}(\partic +2^{\srange})\big )$$
    %is an isomorphism for $i<D(2, \srange+1, \partic)$ and a surjection for $i= D(2,\srange+1, \partic)$.
    \item\label{even dim and p is odd} Suppose that $n$ is even and greater than $2$ and that $p$ is an odd prime.
    %$\fieldc=\mathbb{F}_{p}$ with $p$ odd. 
    %is odd and $n$ is even and greater than $2$.  
    The map
    $$ \newstab{e^{p}} \colon  H_{i}(\text{Conf}_{\partic}(M);\mathbb{F}_{p})\to  H_{i}(\text{Conf}_{\partic +p}(M);\mathbb{F}_{p}) $$ is an isomorphism for $i < k$ and a surjection for $i=k$.
\end{enumerate}
\end{proposition}
%We conclude this section by proving \cref{closed periodic secondary stability}.
\begin{proof}
%[Proof of \cref{closed periodic secondary stability}]
%Case \ref{higher dim and p is two} is due to \textcite[Theorem 9.1]{MR3032101}.
The argument is exactly the same as
%similar to 
the proof of \cref{stable hom for a closed manifold}. 
%The only difference is that to prove Parts \ref{closed stability for surface F2 coefficients}, \ref{closed stability for surface odd coefficients}, and \ref{higher stabilization in odd dimensions and p is 2}, we replace
%$\cppunc[T(V)_{*}]{\bullet}{\partic}$ and $\cppunc[\nch{\modc[S^{n-1}]{\bar{D}^{n}}}]{\bullet}{\partic}$
%
%$CZ_{\partic,\bullet}(T(V)_{*},M,K_{1};\fieldc)$ and $CZ_{\partic,\bullet}(\nch{\modc[S^{n-1}]{\bar{D}^{n}}},M,K_{1};\fieldc)$ 
%with $CZ_{\partic,\bullet}(A_{1,\srange -1},M,K_{\srange};\fieldc)$ and $CZ_{\partic,\bullet}(A_{0,\srange},M,K_{\srange};\fieldc)$ respectively.
%$CZ_{\partic,\bullet}(A_{1},M,K_{1};\fieldc)$ with $CZ_{\partic,\bullet}(A_{r,s},M,K_{\srange};\fieldc)$, with $A_{r,s}=(\bigsqcup_{j=1}^{\srange-1}\bar{D}^{n}_{j}) \bigsqcup Y$ or $\bigsqcup_{j=1}^{\srange}\bar{D}^{n}_{j}$. 
%$Z_{\partic,\bullet}(M, \bar{D}^{n})$ with $Z_{\partic,\bullet}(M, \bar{E}_{\srange})$. 
The only difference is that we implicitly use the fact that for an open connected manifold $M$, there is a zig-zag of quasi-isomorphisms $H_{*}(\jiteratedmap{M}{\newstab{z_{i}}}{\srange+1}{[\srange+1]};\fieldc)\to H_{*}(\iteratedmap{M}{t_{z_{i}}}{\srange+1};\fieldc)$
%the diagram
%\begin{center}\begin{tikzcd}
%C_{*}(\iteratedmap{M}{\newstab{z_{i}}}{\srange};\fieldc)
%\arrow[r, "\epsilon"]\arrow[d, "\newstab{z_{\srange}}"] &
%C_{*}(\iteratedmap{M}{t_{z_{i}}}{\srange};\fieldc)
%\arrow[d, "t_{z_{\srange}}"]\\
%C_{*}(\iteratedmap{M}{\newstab{z_{i}}}{\srange};\fieldc)
%\arrow[r,"\epsilon"]& 
%C_{*}(\iteratedmap{M}{t_{z_{i}}}{\srange};\fieldc)
%\end{tikzcd}\end{center}
%commutes up to chain homotopy
%the stabilization map $$\newstab{\displaystyle y_{m}} \colon  C_{i}\big (\pariteratedmap{M}{\newstab{\displaystyle y_{j}}}{\srange}{\partic};\mathbb{F}_{p}\big )\to  C_{i+ 2p^{\srange}-2}\big (\pariteratedmap{M}{\newstab{\displaystyle y_{j}}}{\srange}{\partic+2p^{\srange}};\mathbb{F}_{p}\big )$$ is chain homotopy equivalent to the map $$t_{\displaystyle y_{m}} \colon  C_{i}\big (\pariteratedmap{M}{\newstab{\displaystyle y_{j}}}{\srange}{\partic};\mathbb{F}_{p}\big )\to  C_{i+ 2p^{\srange}-2}\big (\pariteratedmap{M}{t_{\displaystyle y_{j}}}{\srange}{\partic+2p^{\srange}};\mathbb{F}_{p}\big )$$ 
by a stronger version of \cref{new stabilization on open manifold} for iterated mapping cones (the proof of this stronger version is similar to the proof of \cref{new stabilization on open manifold}, so we omit it).
%The only differences are that we also need a relative version of the geometric realization spectral sequence
%$$E^{1}_{s,t}=H_{t}( Z_{\partic,s}(M),Z_{\partic -p,s}(M);\fieldc)\Longrightarrow H_{s+t}( \Vert Z_{\partic,s}(M)\Vert ,\Vert Z_{\partic -p,s}(M)\Vert;\fieldc)$$
%for proving Cases \ref{closed stability for surface F2 coefficients} and \ref{closed stability for surface odd coefficients}.
%The ranges for these cases follow from \cref{prop:conf-stab}.
The range for Part \ref{higher dim and p is two} follows from \cref{homological stability in dimension 2} and \cref{higher stab open higher dimensional manifold}.
The ranges for Parts \ref{closed stability for surface F2 coefficients} and \ref{closed stability for surface odd coefficients} follow from \cref{higher stab open surface}.
%\cref{best known stable ranges}, and \cref{higher stab open higher dimensional manifold}.
The ranges for Parts \ref{odd dim and Z}, \ref{odd dim and p is odd}, and \ref{even dim and p is odd} follow from \cref{best known stable ranges}. The ranges for Part \ref{higher stabilization in odd dimensions and p is 2} follow from \cref{higher stab open higher dimensional manifold}.
\end{proof}
When the dimension of the manifold is odd, one can apply Parts \ref{odd dim and p is odd} and \ref{higher stabilization in odd dimensions and p is 2} to obtain the following secondary homological stability result for the integral relative homology $H_{*}(\text{Conf}_{\partic}(M),\text{Conf}_{\partic-1}(M);\mathbb{Z})$, which proves \cref{closed periodic secondary stability} in the case that homology is integral homology, exactly like \cref{integral secondary stab open manifold}.
%\footnote{this result might be wrong}
\begin{corollary}\label{integral secondary stab closed manifold}
Suppose that $M$ is a connected manifold of odd dimension greater than $2$. 
%Let $H_{*}(-)$ denote homology with $\mathbb{Z}$ coefficients. 
The map 
    $$\newstab{\displaystyle \omega_{1}} \colon  C_{i}(\text{Conf}_{\partic}(M),\text{Conf}_{\partic-1}(M);\mathbb{Z})\to C_{i+1}(\text{Conf}_{\partic+2}(M),\text{Conf}_{\partic+1}(M);\mathbb{Z})$$
    induces an isomorphism for $i<D(2, 2, \partic)$ and a surjection for $i= D(2, 2, \partic)$.
\end{corollary}
We now prove \cref{closed periodic secondary stability}.
\begin{proof}[Proof of \cref{closed periodic secondary stability}]
\cref{secondary stability open case} immediately follows from \cref{stability results for closed manifold} and \cref{integral secondary stab closed manifold} in the case $\srange=1$.
%The part of  homology with $\mathbb{F}_{2}$ coefficients
%\cref{secondary stability open case} involving homology follows from 
%Secondary stability for homology with $\mathbb{F}_{2}$ coefficients follows from 
\end{proof}

\begin{remark}
%If\footnote{I don't know which versions of this remark I should use. I would prefer to use the second version since it's more detailed, but not if it's inaccurate/wrong.} $M$ is a closed connected manifold of even dimension, we do not know if there is a similar statement to \cref{integral secondary stab closed manifold}. In this case, the issue is that we do not know how to define relative homology $H_{*}(\text{Conf}_{\partic }(M),\text{Conf}_{\partic -1}(M);\mathbb{Z})$.
%Possibly it should be defined in terms of the action of a dived power algebra described by \textcite{MR3845723}.
If $M$ is a closed connected manifold of even dimension, we do not know if there is a similar statement to \cref{integral secondary stab closed manifold}. In this case, the issue is that we do not know how to define relative homology $H_{*}(\text{Conf}_{\partic }(M),\text{Conf}_{\partic -1}(M);\mathbb{Z})$.
When $M$ is a closed connected manifold of even dimension, we can only consider $H_{*}(\text{Conf}_{\partic }(M),\text{Conf}_{\partic -p}(M);\mathbb{F}_{p})$, not $H_{*}(\text{Conf}_{\partic }(M),\text{Conf}_{\partic -1}(M);\mathbb{F}_{p})$, and there is no analogue of $H_{*}(\text{Conf}_{\partic }(M),\text{Conf}_{\partic -p}(M);\mathbb{F}_{p})$ for integral homology.

It might still be possible to obtain a secondary stability statement result by working in the setting of divided power algebras. Let $D$ denote the divided power algebra over $\mathbb{Z}$ in a single variable. Suppose that $M$ is an open manifold. The chain complex $C_{*}(\text{Conf}(M);\mathbb{Z})$ can be viewed as a $\mathbb{Z}[x]$-module, where the variable $x$ acts on $C_{*}(\text{Conf}(M);\mathbb{Z})$ by stabilization,
%\footnote{should I explain how $x$ acts?}
%, with $x$ acting on $C_{*}(\text{Conf}(\smon);\mathbb{Z})$ by the stabilization map $t_{z}\colon C_{*}(\text{Conf}(\smon);\mathbb{Z})\to C_{*}(\text{Conf}(\smon);\mathbb{Z})$. 
and we have that $$\text{Tor}_{i}^{\mathbb{Z}[x]}(C_{*}(\text{Conf}(M);\mathbb{Z}),\mathbb{Z})\cong \bigoplus_{\partic} H_{i}(\text{Conf}_{\partic}(M),\text{Conf}_{\partic-1}(M);\mathbb{Z}).$$ 
When $M$ is a closed manifold, $\mathbb{Z}[x]$ no longer acts on $C_{*}(\text{Conf}(M);\mathbb{Z})$, but $D$ acts on $C^{*}(\text{Conf}(M);\mathbb{Z})$. So instead consider
%We can pass from $\mathbb{Z}[x]$ to $D$ and consider 
the hyper Tor groups $\text{Tor}_{i}^{D}(C^{*}(\text{Conf}(M);\mathbb{Z}),\mathbb{Z})$. These hyper Tor groups are an analogue of the relative homology $H_{i}(\text{Conf}_{\partic}(M),\text{Conf}_{\partic-1}(M);\mathbb{Z})$. We conjecture that the hyper Tor groups $\text{Tor}_{*}^{D}(C^{*}(\text{Conf}(M);\mathbb{Z}),\mathbb{Z})$ stabilize for a closed manifold $M$. We also ask if there is an analogue of the secondary stabilization map $$t_{2} \colon C_{*}(\text{Conf}_{k}(M),\text{Conf}_{k-1}(M);\mathbb{Z})\to C_{*+1}(\text{Conf}_{k+2}(M),\text{Conf}_{k+1}(M);\mathbb{Z})$$ for the hyper Tor groups $\text{Tor}_{*}^{D}(C^{*}(\text{Conf}(M);\mathbb{Z}),\mathbb{Z})$.
\end{remark}
\begin{remark}\label{remark on applying BCT}
In the case that $M$ is a manifold of finite type and homology is with coefficients over the field $\mathbb{F}_{2}$, the secondary stability results for this case can also be deduced from \textcite{MR991102}. B\"{o}digheimer--Cohen--Taylor show that the homology $H_{*}(\text{Conf}(M);\mathbb{F}_{2})$ can be explicitly calculated from a graded vector space $\mathscr{C}(H_{*}(M;\mathbb{F}_{2}))$ depending only on the homology $H_{*}(M;\mathbb{F}_{2})$ and the dimension of $M$ (see \textcite[Theorem A]{MR991102}). They give an explicit filtration of $\mathscr{C}(H_{*}(M;\mathbb{F}_{2}))$ which calculates $H_{*}(\text{Conf}_{\partic}(M);\mathbb{F}_{2})$ (see \textcite[Theorem B and Section 4]{MR991102}).
As a result, we can use the vector space $\mathscr{C}(H_{*}(M;\mathbb{F}_{2}))$ to deduce secondary stability in this case. Since the stabilization map $t_{1}\colon H_{*}(\text{Conf}_{\partic}(M);\mathbb{F}_{2})\to H_{*}(\text{Conf}_{\partic+1}(M);\mathbb{F}_{2})$ is injective when $M$ is an open manifold, the relative homology $H_{*}(\text{Conf}_{\partic}(M),\text{Conf}_{\partic-1}(M);\mathbb{F}_{2})$ is the cokernel of this stabilization map. Therefore, secondary stability can almost certainly be deduced from the vector space $\mathscr{C}(H_{*}(M;\mathbb{F}_{2}))$ in this case. The only hard part is checking that the obvious guess for the secondary stabilization map on $\mathscr{C}(H_{*}(M;\mathbb{F}_{2}))$ agrees with the secondary stabilization map $t_{2}\colon H_{*}(\text{Conf}_{\partic}(M),\text{Conf}_{\partic-1}(M);\mathbb{F}_{2})\to H_{*+1}(\text{Conf}_{\partic+2}(M),\text{Conf}_{\partic+1}(M);\mathbb{F}_{2})$.
\end{remark}
\printbibliography
\end{document}